\documentclass{article}
\usepackage{amsmath,amsthm,amssymb}
\usepackage[usenames,dvipsnames]{xcolor}
\usepackage{enumerate}
\usepackage{graphicx}
\usepackage{cite}
\usepackage{comment}
\usepackage{oands}
\usepackage{tikz}
\usepackage{changepage}
\usepackage{bbm}
\usepackage{mathtools}
\usepackage[margin=1in]{geometry}
\usepackage[pagewise,mathlines]{lineno}
\usepackage{appendix}
\usepackage{microtype}
\usepackage{tabularx}
\usepackage{stmaryrd}
\usepackage{multicol}
\theoremstyle{plain}
\newtheorem{thm}{Theorem}[section]

\newtheorem{lem}[thm]{Lemma}
\newtheorem{prop}[thm]{Proposition}

\def\@rst #1 #2other{#1}
\newcommand\MR[1]{\relax\ifhmode\unskip\spacefactor3000 \space\fi
  \MRhref{\expandafter\@rst #1 other}{#1}}
\newcommand{\MRhref}[2]{\href{http://www.ams.org/mathscinet-getitem?mr=#1}{MR#2}}

\usepackage[pdftitle={Convergence of the self-avoiding walk on random quadrangulations to SLE(8/3) on sqrt{8/3}-Liouville quantum gravity},
  pdfauthor={Ewain Gwynne and Jason Miller},
colorlinks=true,linkcolor=NavyBlue,urlcolor=RoyalBlue,citecolor=PineGreen,bookmarks=true,bookmarksopen=true,bookmarksopenlevel=2,unicode=true,linktocpage]{hyperref}

\theoremstyle{definition}
\newtheorem{defn}[thm]{Definition}
\newtheorem{example}[thm]{Example}
\newtheorem{remark}[thm]{Remark}

\numberwithin{equation}{section}

\newcommand{\dsb}{\begin{adjustwidth}{2.5em}{0pt}
\begin{footnotesize}}
\newcommand{\dse}{\end{footnotesize}
\end{adjustwidth}}

\newcommand{\ssb}{\begin{adjustwidth}{2.5em}{0pt}}
\newcommand{\sse}{\end{adjustwidth}}

\newcommand{\aryb}{\begin{eqnarray*}}
\newcommand{\arye}{\end{eqnarray*}}
\def\alb#1\ale{\begin{align*}#1\end{align*}}
\newcommand{\eqb}{\begin{equation}}
\newcommand{\eqe}{\end{equation}}
\newcommand{\eqbn}{\begin{equation*}}
\newcommand{\eqen}{\end{equation*}}

\newcommand{\BB}{\mathbbm}
\newcommand{\ol}{\overline}
\newcommand{\ul}{\underline}
\newcommand{\op}{\operatorname}

\newcommand{\frk}{\mathfrak}
\newcommand{\eqD}{\overset{d}{=}}
\newcommand{\ep}{\epsilon}
\newcommand{\rta}{\rightarrow}

\newcommand{\wt}{\widetilde}
\newcommand{\wh}{\widehat} 
\newcommand{\mcl}{\mathcal}

\newcommand{\bdy}{\partial}

\let\originalleft\left
\let\originalright\right
\renewcommand{\left}{\mathopen{}\mathclose\bgroup\originalleft}
\renewcommand{\right}{\aftergroup\egroup\originalright}

\title{Convergence of the self-avoiding walk on random quadrangulations to SLE$_{8/3}$ on $\sqrt{8/3}$-Liouville quantum gravity}
\date{  }
\author{
\begin{tabular}{c} Ewain Gwynne\\[-5pt]\small MIT \end{tabular}
\begin{tabular}{c} Jason Miller\\[-5pt]\small Cambridge \end{tabular}
}

\begin{document}

\maketitle

\begin{abstract} 
We prove that a uniform infinite quadrangulation of the half-plane decorated by a self-avoiding walk (SAW) converges in the scaling limit to the metric gluing of two independent Brownian half-planes identified along their positive boundary rays.  Combined with other work of the authors, this implies the convergence of the SAW on a random quadrangulation to SLE$_{8/3}$ on a certain $\sqrt{8/3}$-Liouville quantum gravity surface.  The topology of convergence is the local Gromov-Hausdorff-Prokhorov-uniform topology, the natural generalization of the local Gromov-Hausdorff topology to curve-decorated metric measure spaces.  We also prove analogous scaling limit results for uniform infinite quadrangulations of the whole plane decorated by either a one-sided or two-sided SAW. Our proof uses only the peeling procedure for random quadrangulations and some basic properties of the Brownian half-plane, so can be read without any knowledge of SLE or LQG. 
\end{abstract}

\tableofcontents

\section{Introduction}
\label{sec-intro}

\subsection{Overview}
\label{sec-overview}

\subsubsection{Self-avoiding walk}
\label{sec-overview-saw}

Suppose that $G = (\mcl V(G) , \mcl E(G) )$ is a graph and $x,y \in \mcl V(G)$ are distinct vertices.  The \emph{self-avoiding walk} (SAW) on $G$ from $x$ to $y$ of length $n$ is the uniform measure on paths from $x$ to $y$ in $G$ of length $n$ which do not self-intersect.  The SAW was first introduced as a model for a polymer by Flory~\cite{flory-book}.  There is a vast literature on the SAW in both mathematics and physics and we will not attempt to survey it in its entirety here, but we will now briefly mention a few of the basic results which are most closely related to the present work.

The first question that one is led to ask about the SAW is \emph{how many are there}?  
If $G$ is an infinite, vertex transitive graph (such as $\BB Z^d$) and $c_n$ denotes the number of SAWs in $G$ starting from a given vertex with length $n$, then it is not difficult to see that $c_{m+n} \leq c_m c_n$ for each $m,n \in \BB N$.  Consequently, the limit $\mu = \lim_{n \to \infty} c_n^{1/n}$ exists and is the so-called \emph{connective constant}~\cite{hammersley-cc}.  There is an extensive literature on the connective constant for various graphs. See, e.g., the survey provided in~\cite{gl-connective3,gl-saw-survey-new} and the references therein. We mention that the connective constant in the case of the two-dimensional hexagonal lattice was shown to be $\sqrt{2+\sqrt 2}$ in~\cite{dc-smirnov-connective-constant}, but identifying this constant for other lattices remains an open problem.

The next natural question that one is led to ask is whether the SAW possesses a \emph{scaling limit}, and this is the question which will be the focus of the present work.  Building on work of Brydges and Spencer~\cite{brydges-spencer-saw5}, it was shown by Hara and Slade that the SAW on the integer lattice in dimension $d \geq 5$ converges to Brownian motion when one performs a diffusive scaling~\cite{hara-slade-saw-bm}.  The scaling limit of the SAW is also conjectured to be given by Brownian motion when $d=4$, but with an extra logarithmic correction in the scaling.  This has not yet been proved, although a number of theorems about weakly self-avoiding walk have been proven; see~\cite{slade-4d-survey} for a recent survey.  It is not known what the scaling limit (or factor) should be for $d=3$, though various exponents associated with this case have been derived numerically.  We refer to the survey articles~\cite{slade-saw-survey,saw-lectures,gl-saw-survey-new} and the book \cite{ms-saw-book} and the references therein for more results on the SAW.

The main focus of the present work is the case $d=2$.  It was conjectured by Lawler, Schramm, and Werner~\cite{lsw-saw} that in this case the SAW converges upon appropriate rescaling to the Schramm-Loewner evolution (SLE) \cite{schramm0} with parameter $\kappa = 8/3$.  This conjecture was derived by making the ansatz that the scaling limit of the SAW should be conformally invariant and satisfy a certain Markov property.  The value $\kappa=8/3$ arises because SLE$_{8/3}$ satisfies the so-called \emph{restriction property}~\cite{lsw-restriction}, which is the continuum analog of the fact that a SAW conditioned to stay in a subgraph is the same as a SAW on that subgraph.  This conjecture has been supported by extensive numerical simulations due to Tom Kennedy~\cite{kennedy-saw-sim}.  Prior to the present work, no scaling limit result for the SAW in two dimensions has been proved, however.

We will study and prove scaling limit results for the SAW in two dimensions on certain types of \emph{random planar maps}.  The SAW in this context was first studied (non-rigorously) by Duplantier and Kostov~\cite{dup-kos-saw,dup-kos-saw-long} as a test case for the KPZ formula~\cite{kpz-scaling}, which relates exponents for critical models on random surfaces with the corresponding exponents on planar lattices.  We will establish the existence of the scaling limit of the SAW on a \emph{random planar quadrangulation}, viewed as a curve-decorated metric measure space equipped with the SAW, the graph distance, and the counting measure on edges. Although the proof of this scaling limit result uses only the theory of random planar maps, the results of~\cite{gwynne-miller-gluing} allow us to identify the limiting object with SLE$_{8/3}$ on a \emph{$\sqrt{8/3}$-Liouville quantum gravity wedge}, 
a certain random metric measure space with the topology of the upper half-plane. We will discuss this identification further in Section~\ref{sec-overview-sle}, but first let us say more about SAW on random planar maps. 
 
Recall that a planar map is a graph together with an embedding in the plane so that no two edges cross.  Two such maps are said to be equivalent if there exists an orientation preserving homeomorphism which takes one to the other.  A map is said to be a quadrangulation if every face has exactly four adjacent edges.

The theories of statistical mechanics models like the SAW on random planar maps and on deterministic lattices are equally important: both are well-motivated physically and have been studied extensively in the math and physics literature. There are many questions (such as scaling limits of various curves toward SLE) which can be asked in both the random planar map and deterministic lattice settings (in the former setting, one has to specify a topology). It is not in general clear which setting is easier to analyze. %For example, a type of convergence known as peanosphere convergence~\cite{wedges,shef-burger,kmsw-bipolar,gkmw-burger,lsw-schnyder-wood} for many different models has so far only been established in the random planar map setting; whereas the scaling limit of loop-erased random walk toward SLE$_2$ on a deterministic lattice has been established in~\cite{lsw-lerw-ust}, but so far this convergence has not been shown for any planar map in either the metric sense or as a curve embedded into the plane via some canonical embedding of the planar map.

The convergence of the SAW toward SLE$_{8/3}$ is particularly interesting since in both the random planar map and deterministic lattice settings, the SAW is easy to define and important both mathematically and physically; the convergence toward SLE$_{8/3}$ is supported by heuristic evidence; and, prior to this work, the convergence was not proven rigorously in either setting. 

\subsubsection{Gluing together random quadrangulations}
\label{sec-overview-gluing}

We will now describe a simple construction of a finite quadrangulation decorated with a SAW and then describe the corresponding infinite volume versions of this construction.

Suppose we sample two independent uniformly random quadrangulations of the disk with simple boundary with $n$ quadrilaterals and perimeter $2l$ and then glue them together along a boundary segment of length $2s < 2l$ by identifying the corresponding edges (Figure~\ref{fig-saw-gluing}, left). The conditional law of the gluing interface given the overall glued map will then be that of a SAW of length $2s$ conditioned on its left and right complementary components both containing $n$ quadrilaterals.  One can also glue the \emph{entire} boundaries of the two disks to obtain a map with the topology of the sphere decorated by a path whose conditional law given the map is that of a self-avoiding loop on length $2l$ conditioned on the two complementary components both containing $n$ quadrilaterals. See, for example, the discussion in \cite[Section~8.2]{bet-disk-tight} (which builds on~\cite{bg-simple-quad,bbg-recursive-approach}) for additional explanation.
 
The \emph{uniform infinite half-planar quadrangulation with simple boundary} (UIHPQ$_{\op{S}}$) is the infinite-volume local limit of uniform quadrangulations of the disk with simple boundary rooted at a boundary edge as the total number of interior faces (or interior vertices), and then the number of boundary edges, is sent to~$\infty$~\cite{curien-miermont-uihpq,caraceni-curien-uihpq}.   

It is shown by Caraceni and Curien~\cite[Section 1.4]{caraceni-curien-saw} that the infinite volume limit of the aforementioned random SAW-decorated quadrangulations can be constructed by starting off with two independent UIHPQ$_{\op{S}}$'s and then gluing them together along their boundary (Figure~\ref{fig-saw-gluing}, right).
In this case, the gluing interface is an infinite volume limit of a SAW. 
There are several natural constructions leading to SAW decorated quadrangulations that one can build with these types of gluing operations: 
\begin{itemize}
\item \emph{Chordal SAW on a half-planar quadrangulation from $0$ to $\infty$:}  Glue two independent UIHPQ$_{\op{S}}$'s along their positive boundaries, yielding a random quadrangulation of the upper half-plane with a distinguished path from the boundary to $\infty$.
\item \emph{Two-sided SAW on a whole-plane quadrangulation from $\infty$ to $0$ and back to $\infty$:} Glue two independent UIHPQ$_{\op{S}}$'s along their entire boundaries, yielding a random quadrangulation of the plane together with a two-sided path from $\infty$ through the root vertex and then back to $\infty$.
\item \emph{Whole-plane SAW from $0$ to $\infty$ on a whole-plane quadrangulation:} Glue together the two complementary rays of the boundary of a single UIHPQ$_{\op{S}}$, yielding a quadrangulation of the plane together with a distinguished path from the root vertex to $\infty$.
\end{itemize} 
In each case, the infinite random planar map, decorated by the curve corresponding to the gluing interface, is the infinite volume local limit of finite random planar maps decorated by a self-avoiding walk~\cite{caraceni-curien-saw}. Hence these infinite random maps can be viewed as infinite SAW-decorated maps. 

\begin{figure}[ht!]
 \begin{center}
\includegraphics[scale=1]{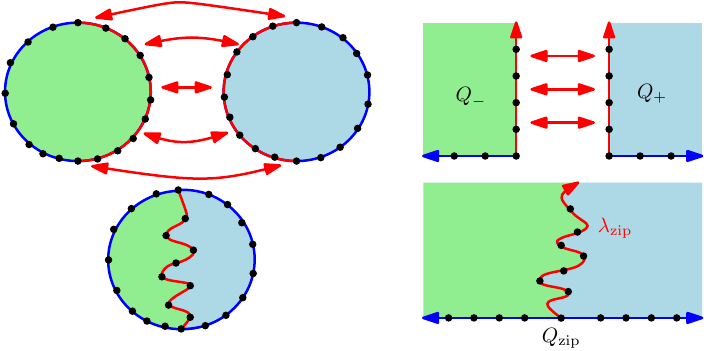} 
\caption[Random quadrangulations glued together along their boundaries to get a SAW]{\textbf{Left:} Two independent uniformly random finite quadrangulations with boundary glued together along a boundary arc to get a uniformly random SAW-decorated quadrangulation with boundary.
\textbf{Right:} The infinite-volume and boundary length limit of the left panel: two independent UIHPQ$_{\op{S}}$'s glued together along their positive boundary rays to obtain an infinite-volume uniform SAW-decorated quadrangulation with boundary. %We prove that the scaling limit of the picture on the right exists and is equal to the metric space quotient of a pair of independent Brownian half-planes glued together along their positive boundaries. By the results of~\cite{gwynne-miller-gluing}, this limiting space can equivalently be described as a weight-$4$ Liouville quantum gravity wedge decorated by an independent chordal SLE$_{8/3}$ curve.
}\label{fig-saw-gluing}
\end{center}
\end{figure}

\subsubsection{Gluing together Brownian half-planes}
\label{sec-overview-bhp}

Building on the scaling limit result for finite uniform quadrangulations with boundary in~\cite{bet-mier-disk}, it was proved in~\cite{gwynne-miller-uihpq} that the UIHPQ$_{\op{S}}$ converges in the scaling limit to the so-called \emph{Brownian half-plane} (see also~\cite{blr-exponents} for more general scaling limit results for half-plane quadrangulations with general boundary).  This is a random metric space with boundary which has the topology of the upper-half-plane whose definition is reviewed in Section~\ref{sec-bhp-prelim}.  This metric space comes with some additional structure: an area measure and a boundary length measure.
One can perform each of the aforementioned gluing operations with the Brownian half-plane in place of the UIHPQ$_{\op{S}}$ by identifying Brownian half-planes together along their boundaries and taking a metric space quotient (see Figure~\ref{fig-thick-gluing}).

The main results of the present work, stated precisely in Section~\ref{sec-main-result}, are that in each of the above three itemized cases the construction built from the UIHPQ$_{\op{S}}$ converges in the scaling limit to the corresponding construction built from the Brownian half-plane (see Remark~\ref{remark-finite-volume} for a discussion of the case where we glue together finite quadrangulations with simple boundary).  Combining this with the main results of \cite{gwynne-miller-gluing}, we conclude that the SAW on random quadrangulations converges to SLE$_{8/3}$ on $\sqrt{8/3}$-Liouville quantum gravity (LQG).  We will explain this latter point in more detail just below.
 
The topology in which the scaling limits in this paper take place is the one induced by the \emph{local Gromov-Hausdorff-Prokhorov-uniform (GHPU) metric} on curve-decorated metric measure spaces, which is introduced in~\cite{gwynne-miller-uihpq} and reviewed in Section~\ref{sec-ghpu-prelim} below.  The local GHPU metric is the natural analog of the local Gromov-Hausdorff metric when we study metric spaces with a distinguished measure and curve.  Roughly speaking, two compact curve-decorated metric measure spaces are said to be close in the GHPU metric if they can be isometrically embedded into a common metric space in such a way that the spaces are close in the Hausdorff distance, the measures are close in the Prokhorov distance, and the curves are close in the uniform distance. Here, the Hausdorff, Prokhorov, and uniform distance are all defined w.r.t.\ the common metric space into which our given spaces are isometrically embedded.  Two non-compact curve decorated metric measure spaces are close in the local GHPU topology if their metric balls of radius $r$ are close in the GHPU topology for a large value of $r$. See Section~\ref{sec-ghpu-prelim} below for a precise definition of the local GHPU metric. (See also~\cite{gwynne-miller-perc} for an analogous GHPU convergence result for percolation on random quadrangulations with simple boundary to SLE$_6$ on $\sqrt{8/3}$-LQG.)

Since we already know that the scaling limit of the UIHPQ$_{\op{S}}$ is the Brownian half-plane, to prove our main results we need to show that the operation of passing to the scaling limit of two independent UIHPQ$_{\op{S}}$'s to get two independent Brownian half-planes commutes with the operation of gluing the surfaces together along their boundaries. It is natural to expect this to be the case, but proving this commutation of scaling limits and gluing operations is highly non-trivial. Indeed, it is a priori possible that paths which cross the gluing interface more than a constant-order number of times are typically much shorter than paths which cross only a constant-order number of times (see Example~\ref{example-gluing} below). If this is the case, then a subsequential scaling limit of the discrete glued maps in the GHPU topology might not coincide with the metric gluing of the scaling limits of the UIHPQ$_{\op{S}}$'s on either side of the SAW. Here we emphasize that distances in the continuum metric gluing are given by the infimum of the lengths of paths which cross the gluing interface at most finitely many times; see the definition of the quotient metric in Section~\ref{sec-metric-prelim}.

For similar reasons, it is not a priori clear that gluing together Brownian half-planes along their boundaries produces a metric space decorated by a simple curve. The results of~\cite{gwynne-miller-gluing} imply that this is indeed the case (and identifies the law of the curve-decorated metric space with a certain type of SLE$_{8/3}$-decorated $\sqrt{8/3}$-LQG surface).
As a by-product of the arguments in the present paper, we obtain another proof that the gluing interface is simple, and show that it is in fact locally reverse H\"older continuous of any exponent $ p > 3/2$, i.e., for each fixed $L > 0$ there exists $c >0$ such that the interface $\eta_{\op{zip}}$ satisfies $d_{\op{zip}}(\eta_{\op{zip}}(t_t) , \eta_{\op{zip}}(t_2)) \geq c |t_2-t_1|^p$ for each $t_1,t_2 \in [0,L]$ (see Lemma~\ref{prop-simple-curve} below).
See~\cite[Section 2.2]{gwynne-miller-gluing} for some additional discussion of the issues which can arise when performing metric gluings.

\begin{figure}[ht!!]
\begin{center}
\includegraphics[page=2,scale=0.7]{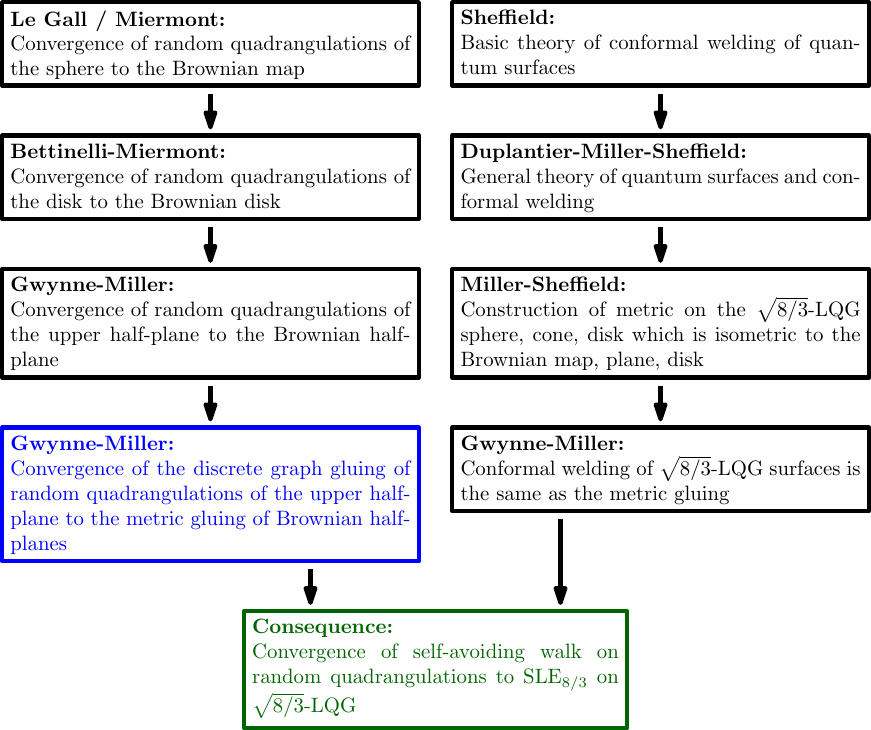}	
\end{center}
\vspace{-0.02\textheight}
\caption[Relationships between papers related to Chapter~\ref{chap-saw}]{\label{fig-chart} A chart of the different components which serve as input into the proof that self-avoiding walk on random quadrangulations converges to SLE$_{8/3}$ on $\sqrt{8/3}$-LQG.  The present article corresponds to the blue box and implies that the discrete graph gluing of random quadrangulations of the upper half-plane converge to the metric gluing of Brownian half-planes.  Combined with \cite{gwynne-miller-gluing} (i.e., the article indicated in the box immediately to the right of the blue box), this implies that the self-avoiding walk on random quadrangulations converges to an SLE$_{8/3}$-type path on $\sqrt{8/3}$-LQG.}
\end{figure}

\subsubsection{Identification with SLE$_{8/3}$-decorated $\sqrt{8/3}$-LQG}
\label{sec-overview-sle}

In order to explain how the main results of this article allow us to identify the scaling limit of the SAW with SLE$_{8/3}$ on $\sqrt{8/3}$-Liouville quantum gravity (LQG), we first need to briefly discuss the basics of LQG surfaces (see Section~\ref{sec-lqg-prelim} and the references therein for more detail).  Such a surface is formally described by the metric $e^{\sqrt{8/3} h} dx \otimes dy$ where $dx \otimes dy$ is the Euclidean metric tensor and $h$ is an instance of some form of the Gaussian free field (GFF)~\cite{shef-gff,ss-contour} on a domain $D\subset \BB C$. 
This metric tensor does not make rigorous sense since $h$ is a distribution, not a function. However, it is shown in~\cite{shef-kpz} that one can make rigorous sense of the volume form $\mu_h = e^{\sqrt{8/3} h} \, dz$, where $dz$ denotes Lebesgue measure, via a regularization procedure. 

It was further shown in \cite{tbm-characterization,lqg-tbm1,sphere-constructions,lqg-tbm2,lqg-tbm3}, building on \cite{qle}, that every $\sqrt{8/3}$-LQG surface can be endowed with a canonical metric space structure.  
Certain special $\sqrt{8/3}$-LQG surfaces are equivalent to Brownian surfaces, like the aforementioned Brownian half-plane and the Brownian map~\cite{legall-uniqueness,miermont-brownian-map}, in the sense that they differ by a measure-preserving isometry.
In particular, the Brownian half-plane is equivalent to the so-called \emph{weight-2 quantum wedge}, a $\sqrt{8/3}$-LQG surface described by a certain variant of the GFF on the upper half-plane $\BB H$. 

By the main result of \cite{lqg-tbm3}, the metric measure space structure of a $\sqrt{8/3}$-LQG surface a.s.\ determines the surface.
This implies in particular that the Brownian map has a canonical embedding into $\BB H$ (modulo conformal automorphisms) obtained by identifying it with a weight-2 quantum wedge parameterized by $\BB H$. 
Furthermore, there is an infinite family of random metric measure spaces which locally look like Brownian surfaces, obtained by considering different variants of the GFF on different domains.
We provide in Section~\ref{sec-lqg-prelim} below a more detailed exposition of LQG and its relationship to Brownian surfaces.

It is shown in~\cite{shef-zipper,wedges} that one can conformally weld two $\sqrt{8/3}$-LQG surfaces according to the $\sqrt{8/3}$-LQG length measure along their boundaries to get a new LQG surface, and the interface between such surfaces after welding is an SLE$_{8/3}$-type curve~\cite{shef-zipper,wedges}.  It was proved in~\cite{gwynne-miller-gluing} that the $\sqrt{8/3}$-LQG metric on the welded surface coincides with the metric quotient of the two smaller surfaces; such a statement is not at all obvious from the construction of the $\sqrt{8/3}$-LQG metric in \cite{lqg-tbm1,lqg-tbm2,lqg-tbm3}.
 
The preceding paragraph and the equivalence of the Brownian half-plane and the weight-2 quantum wedge together imply that the curve-decorated metric measure spaces obtained by gluing together Brownian half-planes which arise as the scaling limits of random SAW-decorated quadrangulations are equivalent to certain explicit $\sqrt{8/3}$-LQG surfaces decorated by SLE$_{8/3}$-type curves.

We emphasize, however, that the present work \emph{does not} use any LQG machinery (see Figure~\ref{fig-chart} for the dependencies).  The LQG machinery in~\cite{lqg-tbm1,lqg-tbm2,lqg-tbm3,gwynne-miller-gluing} is what allows us to deduce the correspondence with SLE$_{8/3}$ on $\sqrt{8/3}$-LQG from the results proved here.

\bigskip

\noindent{\bf Acknowledgements} 
We thank an anonymous referee for a careful reading and helpful comments on an earlier version of this manuscript which has led to many improvements in the exposition.  E.G.\ was supported by the U.S. Department of Defense via an NDSEG fellowship.  E.G.\ also thanks the hospitality of the Statistical Laboratory at the University of Cambridge, where this work was started.  J.M.\ thanks Institut Henri Poincar\'e for support as a holder of the Poincar\'e chair, during which this work was completed.

\subsection{Main results}
\label{sec-main-result}

In this subsection we state our main results, which say that infinite random quadrangulations decorated by self-avoiding walks converge to $\sqrt{8/3}$-LQG surfaces decorated by SLE$_{8/3}$ curves, in the metric space sense, in the half-plane (chordal), two-sided whole-plane, and one-sided whole-plane cases. The theorem statements in the three cases are similar, so the reader may wish to read just one of the statements (probably the chordal case) and skim the others. See Figure~\ref{fig-thick-gluing} for an illustration of the limiting objects in each of the three statements.

Since our convergence results are with respect to the Gromov-Hausdorff-Prokhorov-uniform (GHPU) metric, we need to work with continuous curves.  To do this, we view graphs as connected metric spaces by identifying each edge with an isometric copy of the unit interval, and extend the definitions of curves from discrete intervals to continuum intervals by linear interpolation; c.f.\ Remark~\ref{remark-ghpu-graph} below.  A review of the definitions of the objects involved in the theorem statements (in particular, the GHPU metric, the Brownian half-plane, and the particular $\sqrt{8/3}$-LQG surfaces obtained by gluing together Brownian half-planes) can be found in Section~\ref{sec-prelim}.
 
\subsubsection{Chordal case}

Let $(Q_- , \BB e_-)$ and $(Q_+ , \BB e_+)$ be independent UIHPQ$_{\op{S}}$'s. Let $Q_{\op{zip}}$\footnote{The reason for the subscript $\op{zip}$ is that $(Q_{\op{zip}} , \lambda_{\op{zip}})$ is the discrete analog of the so-called \emph{quantum zipper}~\cite{shef-zipper} obtained by gluing together two LQG surfaces along an SLE$_{8/3}$ curve.}
 be the infinite quadrangulation with boundary obtained by identifying each edge on the positive infinite ray of $\bdy Q_-$ (i.e., each edge to the right of $\BB e_-$) with the corresponding edge of $\bdy Q_+$. 
Let $\lambda_{\op{zip}} : \{0,1,2,\dots\} \rta \mcl E( Q_{\op{zip}})$ be the path in $Q_{\op{zip}}$ corresponding to the identified boundary rays of $Q_\pm$. 
Then $(Q_{\op{zip}} , \lambda_{\op{zip}})$ is the infinite-volume limit of uniform SAW-decorated quadrangulations with boundary based at the starting point of the SAW~\cite{caraceni-curien-saw}.

For $n\in\BB N$, let $d_{\op{zip}}^n$ be the graph metric on $Q_{\op{zip}} $, re-scaled by $(9/8)^{1/4} n^{-1/4}$.   
Let $\mu_{\op{zip}}^n$ be the measure on $Q_{\op{zip}}^n$ which assigns to each vertex a mass equal to $(4n)^{-1}$ times its degree. 
Extend the path $\lambda_{\op{zip}}$ to $[0,\infty)$ by linear interpolation (in the manner discussed above) and let $\eta_{\op{zip}}^n(t) := \lambda_{\op{zip}}\left( \frac{2^{3/2}}{3} n^{1/2} t \right)$ for $t\geq 0$. 
  
Let $(X_- , d_- , x_- )$ and $(X_+ , d_+ , x_+)$ be a pair of independent Brownian half-planes (weight-$2$ quantum wedges) with marked boundary points.  
Let $\mu_\pm$ be the canonical area measure\footnote{
In the Schaeffer-type construction of the Brownian half-plane, the area measure is the pushforward of Lebesgue measure under the quotient map $\BB R\rta X_\pm$. The boundary length measure is the pushforward under the quotient map of the local time of the ``contour function" encoding process at its running minimum. See Section~\ref{sec-bhp-prelim} for details.} 
on $X_\pm$. 
Also let $\eta_\pm : \BB R\rta \bdy X_\pm$ be the curve which traces $\bdy X_\pm$ in such a way that $\eta_\pm(0) = 0$ and for each $t_1<t_2$, the length of $\eta_\pm([t_1,t_2])$ with respect to the canonical boundary length measure on $\bdy X_\pm$ is $t_2-t_1$. 

Let $(X_{\op{zip}} , d_{\op{zip}})$ be the metric space quotient of the disjoint union of $(X_- ,d_-)$ and $(X_+ , d_+)$ under the equivalence relation which identifies their positive boundary rays according to boundary length: that is, $\eta_-(t) \sim \eta_+(t)$ for each $t\geq 0$. 
Let $\mu_{\op{zip}}$ be the measure on $X_{\op{zip}}$ inherited from the area measures on $X_\pm$, i.e., the sum of the pushforwards of $\mu_-$ and $\mu_+$ under the quotient map $X_-\sqcup X_+ \rta X_{\op{zip}}$. 
Let $\eta_{\op{zip}} : [0,\infty) \rta X_{\op{zip}}$ be the path which is the image of $\eta_-|_{[0,\infty)}$ (equivalently, $\eta_+|_{[0,\infty)}$) under the quotient map. 

By~\cite[Corollary~1.2]{gwynne-miller-gluing}, $(X_{\op{zip}},d_{\op{zip}},\mu_{\op{zip}} , \eta_{\op{zip}})$ is equivalent as a curve-decorated metric measure space to a certain $\sqrt{8/3}$-LQG surface called a \emph{weight-$4$ quantum wedge} decorated by an independent chordal SLE$_{8/3}$ curve from 0 to $\infty$. 
That is, there is a GFF-type distribution $h_{\op{zip}}$ on $\BB H$, which is a.s.\ determined by $(X_{\op{zip}},d_{\op{zip}},\mu_{\op{zip}})$, and a map $X_{\op{zip}} \rta \BB H$ which a.s.\ takes $d_{\op{zip}}$ and $\mu_{\op{zip}}$, respectively, to the $\sqrt{8/3}$-LQG metric and $\sqrt{8/3}$-LQG area measure, respectively, induced by $h_{\op{zip}}$ and takes $\eta_{\op{zip}}$ to a chordal SLE$_{8/3}$ curve from 0 to $\infty$ in $\BB H$ sampled independently from $h_{\op{zip}}$ then parameterized by $\sqrt{8/3}$-LQG length with respect to $h_{\op{zip}}$. See Section~\ref{sec-lqg-prelim} below for more details.

\begin{thm} \label{thm-saw-conv-wedge}
In the setting described just above,
\eqb
\left( Q_{\op{zip}}  , d_{\op{zip}}^n , \mu_{\op{zip}}^n , \eta_{\op{zip}}^n \right) \rta \left( X_{\op{zip}} , d_{\op{zip}}  , \mu_{\op{zip}} , \eta_{\op{zip}}  \right)
\eqe
in law in the local Gromov-Hausdorff-Prokhorov-uniform topology. In other words, the scaling limit of uniform random SAW-decorated half-planar maps in the local GHPU topology is a weight-$4$ quantum wedge decorated by an independent chordal SLE$_{8/3}$ parameterized by $\sqrt{8/3}$-LQG length, viewed as a curve-decorated metric measure space equipped with the $\sqrt{8/3}$-LQG metric and area measure.
\end{thm}

It follows from~\cite[Theorem~1.12]{gwynne-miller-uihpq} that the independent UIHPQ$_{\op{S}}$'s $Q_\pm$, equipped with their graph metric, area measure, and boundary path, (with the aforementioned scaling) converge in law to a pair of independent Brownian half-planes. Theorem~\ref{thm-saw-conv-wedge} says that the metric gluing operation for the UIHPQ$_{\op{S}}$'s (or Brownian half-planes) commutes with the operation of taking the limit as $n\rta\infty$.  A similar statement holds in the settings of Theorems~\ref{thm-saw-conv-2side} and~\ref{thm-saw-conv-cone} below.

\begin{figure}[ht!]
 \begin{center}
\includegraphics{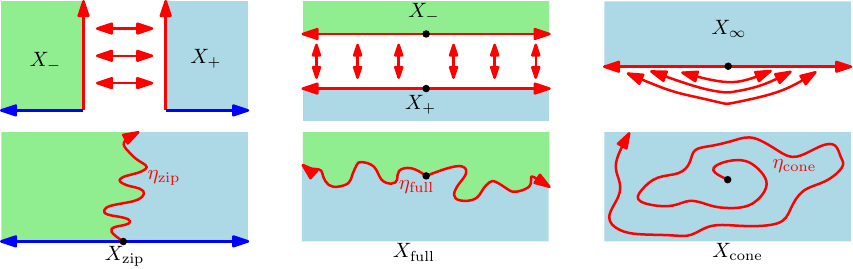} 
\caption[Brownian half-planes glued together to get an SLE$_{8/3}$]{\textbf{Left:} The limiting space $X_{\op{zip}}$ in Theorem~\ref{thm-saw-conv-wedge}, which is a weight-$4$ quantum wedge decorated by an independent chordal SLE$_{8/3}$ and is obtained by gluing two independent Brownian half-planes $X_\pm$ along their positive boundary rays according to boundary length. 
\textbf{Middle:} The limiting space $X_{\op{full}}$ in Theorem~\ref{thm-saw-conv-2side}, which is a weight-$4$ quantum cone decorated by a two-sided SLE$_{8/3}$-type curve and is obtained by gluing two independent Brownian half-planes $X_\pm$ along their full boundaries according to boundary length.  (This SLE$_{8/3}$-type path can be described as a pair of GFF flow lines \cite{ig1,ig4}.)
\textbf{Right:} The limiting space $X_{\op{cone}}$ in Theorem~\ref{thm-saw-conv-cone}, which is a weight-$2$ quantum cone decorated by a whole-plane SLE$_{8/3}$ curve and is obtained by gluing together the left and right boundary rays of a single Brownian half-plane $X_\infty$ according to boundary length.  
}\label{fig-thick-gluing}
\end{center}
\end{figure} 

\subsubsection{Two-sided whole-plane case}
 
Next we state a variant of Theorem~\ref{thm-saw-conv-wedge} for the case when we identify two UIHPQ$_{\op{S}}$'s along their entire boundary (not just their positive boundary rays). 

Let $(Q_\pm , \BB e_\pm)$ and $(X_\pm , d_\pm , x_\pm)$, respectively, be UIHPQ$_{\op{S}}$'s and Brownian half-planes as above. Let $Q_{\op{full}}$ be the quadrangulation without boundary obtained by identifying every edge on $\bdy Q_-$ to the corresponding edge on $\bdy Q_+$ (equivalently, the map obtained by identifying the left and right boundary rays of $Q_{\op{zip}}$). 
Let $\lambda_{\op{full}} : \BB Z\rta \mcl E(Q_{\op{full}})$ be the two-sided path corresponding to the identified boundary paths of $Q_\pm$. 
Then $(Q_{\op{full}},  \eta_{\op{full}})$ is the local limit of uniformly random SAW-decorated quadrangulations of the sphere based at a typical point of the SAW~\cite{caraceni-curien-saw}. 

For $n\in\BB N$, let $d_{\op{full}}^n$ be the graph metric on $Q_{\op{full}} $, re-scaled by $(9/8)^{1/4} n^{-1/4}$.   
Let $\mu_{\op{full}}^n$ be the measure on $Q_{\op{full}}^n$ which assigns to each vertex a mass equal to $(4n)^{-1}$ times its degree. 
Let $\eta_{\op{full}}^n(t) := \lambda_{\op{full}}\left( \frac{2^{3/2}}{3} n^{1/2} t \right)$ for $t\in\BB R$, where here we have extended $\lambda_{\op{full}}$ be linear interpolation in the manner discussed above. 

Let $(X_{\op{full}} , d_{\op{full}})$ be the metric space quotient of the disjoint union of $(X_- ,d_-)$ and $(X_+ , d_+)$ under the equivalence relation which identifies their entire boundaries according to boundary length in such a way that the marked points $x_-$ and $x_+$ are identified. 
In other words, if we define the boundary-tracing curves $\eta_\pm$ as in the preceding subsection, then $\eta_-(s)$ is identified with $\eta_+(s)$ for each $s\in\BB R$.
Let $\mu_{\op{full}}$ be the measure on $X_{\op{full}}$ inherited from the area measures $\mu_\pm$ on $X_\pm$. 
Let $\eta_{\op{full}} : \BB R \rta X_{\op{full}}$ be the path which is the image of $\eta_-$ (equivalently, $\eta_+$) under the quotient map.

By~\cite[Corollary~1.5]{gwynne-miller-gluing}, $(X_{\op{full}},d_{\op{full}},\mu_{\op{full}} ,\eta_{\op{full}})$ is equivalent as a curve-decorated metric measure space to a $\sqrt{8/3}$-LQG surface called a \emph{weight-$4$ quantum cone} decorated by a two-sided SLE$_{8/3}$-type curve in $\BB C$ passing through the origin.
More precisely, there is a GFF-type distribution $h_{\op{full}}$ on $\BB C$ which is a.s.\ determined by $(X_{\op{full}},d_{\op{full}},\mu_{\op{full}} )$ and a map $X_{\op{full}} \rta \BB C$ which a.s.\ takes $d_{\op{full}}$ and $\mu_{\op{full}}$, respectively, to the $\sqrt{8/3}$-LQG metric and $\sqrt{8/3}$-LQG area measure, respectively, induced by $h_{\op{full}}$ and which takes $\eta_{\op{full}}$ to a two-sided SLE$_{8/3}$-type curve sampled independently from $h_{\op{full}}$ then parameterized according to $\sqrt{8/3}$-LQG length with respect to $h_{\op{full}}$.
The law of this SLE$_{8/3}$-type curve can be sampled from as follows: first sample a whole-plane SLE$_{8/3}(2)$ curve $\eta_1$ from $\infty$ to $0$; then, conditional on $\eta_1$, sample a chordal SLE$_{8/3}$ curve~$\eta_2$ from~$0$ to~$\infty$ in $\BB C\setminus \eta_1$. Then concatenate these two curves. (These two curves can also be described as a pair of GFF flow lines \cite{ig1,ig4}.)

\begin{thm} \label{thm-saw-conv-2side}
In the setting described just above,
\eqb
\left( Q_{\op{full}}  , d_{\op{full}}^n , \mu_{\op{full}}^n , \eta_{\op{full}}^n \right) \rta \left( X_{\op{full}} , d_{\op{full}}  , \mu_{\op{full}} , \eta_{\op{full}}  \right)
\eqe
in law in the local Gromov-Hausdorff-Prokhorov-uniform topology. In other words, the scaling limit of uniform random full-planar maps decorated by a two-sided SAW in the local GHPU topology is a weight-$4$ quantum cone decorated by an independent two-sided SLE$_{8/3}$-type curve as described above parameterized by $\sqrt{8/3}$-LQG length, viewed as a curve-decorated metric measure space equipped with the $\sqrt{8/3}$-LQG metric and area measure.
\end{thm}

\subsubsection{One-sided whole-plane case}

We next state a variant of Theorem~\ref{thm-saw-conv-wedge} for the case when we glue a single UIHPQ$_{\op{S}}$ to itself along the two sides of its boundary. 

Let $(Q_{\op{S}}, \BB e_{\op{S}})$ be a UIHPQ$_{\op{S}}$. Let $Q_{\op{cone}}$ be the quadrangulation without boundary obtained by identifying every edge on the positive ray of $\bdy Q_{\op{S}}$ (i.e., the ray to the right of $\BB e_{\op{S}}$) to the corresponding edge on the negative ray of $\bdy Q_{\op{S}}$.
Let $\lambda_{\op{cone}} : \{0,1,2,\dots\} \rta \mcl E(Q_{\op{cone}})$ be the one-sided path corresponding to the identified boundary rays of $Q_{\op{S}}$. 
Then $(Q_{\op{cone}},  \eta_{\op{cone}})$ is the local limit of uniformly random SAW-decorated quadrangulations of the sphere based at the starting point of the SAW~\cite{caraceni-curien-saw}. 
 
For $n\in\BB N$, let $d_{\op{cone}}^n$ be the graph metric on $Q_{\op{cone}} $, re-scaled by $(9/8)^{1/4} n^{-1/4}$.   
Let $\mu_{\op{cone}}^n$ be the measure on $Q_{\op{cone}}^n$ which assigns to each vertex a mass equal to $(4n)^{-1}$ times its degree. 
Let $\eta_{\op{cone}}^n(t) := \lambda_{\op{cone}}\left( \frac{2^{3/2}}{3} n^{1/2} t \right)$ for $t\in\BB R$, with $\lambda_{\op{cone}}$ viewed as a continuous curve via linear interpolation, as discussed at the beginning of this subsection.  

Let $(X_\infty,d_\infty,x_\infty)$ be a Brownian half-plane with marked boundary point. 
Let $\eta_\infty  : \BB R\rta \bdy X_\infty$ be the path which parameterizes the boundary according to its natural length measure, as described just before Theorem~\ref{thm-saw-conv-cone}.
Let $(X_{\op{cone}} , d_{\op{cone}})$ be the metric space quotient of $(X_\infty,d_\infty)$ under the equivalence relation which identifies the positive and negative rays (i.e., the rays to the left and right of $x_\infty$) of $\bdy X_\infty$ according to boundary length. 
That is, $\eta_\infty(s) \sim \eta_\infty(-s)$ for each $s\geq 0$. 
Let $\mu_{\op{cone}}$ be the measure on $X_{\op{cone}}$ which is the pushforward under the quotient map of the area measure on $X_\infty$. 
Let $\eta_{\op{cone}} : [0,\infty) \rta X_{\op{cone}}$ be the path which is the image of $\eta_\infty|_{[0,\infty)}$ (equivalently, $\eta_\infty(-\cdot)|_{[0,\infty)}$) under the quotient map.

By~\cite[Corollary~1.4]{gwynne-miller-gluing}, the metric measure space $(X_{\op{cone}},d_{\op{cone}},\mu_{\op{cone}})$ is equivalent to a curve-decorated metric measure space to a $\sqrt{8/3}$-LQG surface called a \emph{weight-2 quantum cone} decorated by an independent whole-plane SLE$_{8/3}$ curve. 
That is, there is a GFF-type distribution $h_{\op{cone}}$ on $\BB C$ which is a deterministic functional of $(X_{\op{cone}},d_{\op{cone}} , \mu_{\op{cone}} )$ and a map $X_{\op{cone}} \rta \BB C$ which a.s.\ takes $d_{\op{cone}}$ and $\mu_{\op{cone}}$, respectively, to the $\sqrt{8/3}$-LQG metric and $\sqrt{8/3}$-LQG area measure, respectively, induced by $h_{\op{full}}$ and which takes $\eta_{\op{zip}}$ to a whole-plane SLE$_{8/3}$ curve from 0 to $\infty$ sampled independently from $h_{\op{cone}}$ then parameterized according to $\sqrt{8/3}$-LQG length with respect to $h_{\op{cone}}$.

\begin{thm} \label{thm-saw-conv-cone}
In the setting described just above,
\eqb
\left( Q_{\op{cone}}  , d_{\op{cone}}^n , \mu_{\op{cone}}^n , \eta_{\op{cone}}^n \right) \rta \left( X_{\op{cone}} , d_{\op{cone}}  , \mu_{\op{cone}} , \eta_{\op{cone}}  \right)
\eqe
in law in the local Gromov-Hausdorff-Prokhorov-uniform topology. In other words, the scaling limit of uniform random full-planar maps decorated by a one-sided SAW in the local GHPU topology is a weight-$2$ quantum cone decorated by an independent whole-plane SLE$_{8/3}$ parameterized by $\sqrt{8/3}$-LQG length,  viewed as a curve-decorated metric measure space equipped with the $\sqrt{8/3}$-LQG metric and area measure. 
\end{thm}

\begin{remark} \label{remark-finite-volume}
In~\cite{gwynne-miller-simple-quad}, we obtain analogs of the results of this paper for the \emph{finite} uniform SAW-decorated planar quadrangulations obtained by gluing together finite quadrangulations with simple boundary along their boundaries. The main inputs in the proof in this case are the results of the present paper and a scaling limit result for free Boltzmann quadrangulations with simple boundary toward the Brownian disk, which is also proven in~\cite{gwynne-miller-simple-quad}.
\end{remark}

\subsection{Outline}
\label{sec-outline}

\begin{figure}[ht!!]
\begin{center}
\includegraphics[scale=0.7]{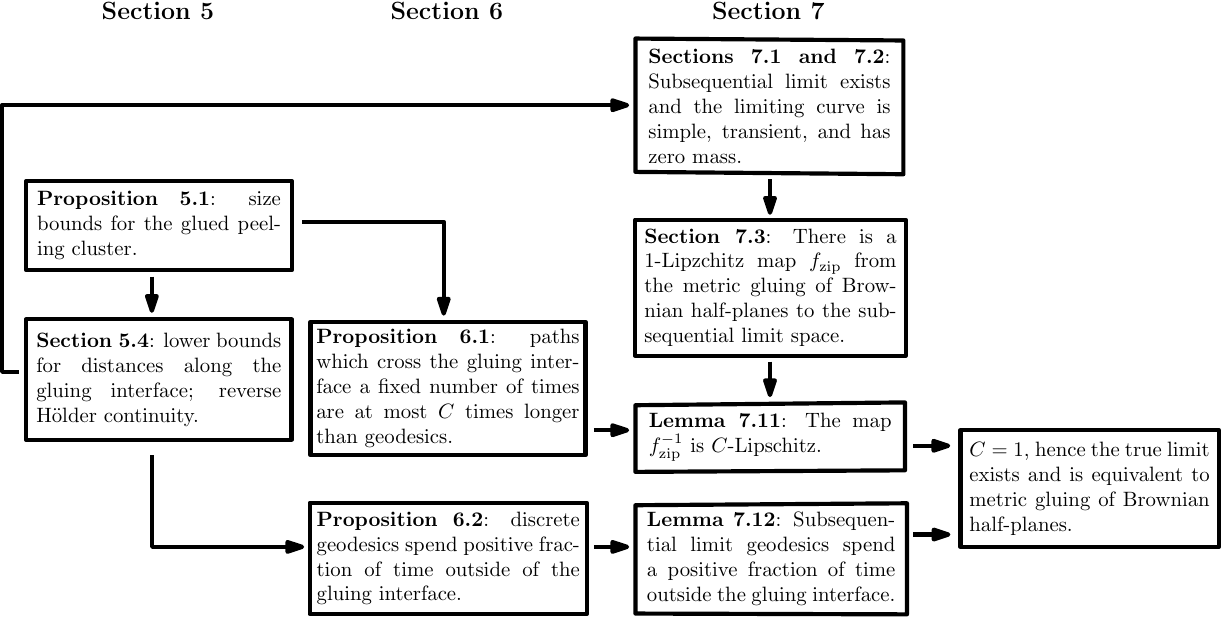}	
\end{center}
\vspace{-0.02\textheight}
\caption{\label{fig-map} Map of the main statements in the core of the paper (Sections~\ref{sec-peeling-moment}--\ref{sec-saw-conv}) and how they fit together.  It is not necessary to read the proof of each of these main statements in order to understand the proofs of the others.}
\end{figure}

In this subsection we give a moderately detailed overview of the main ideas of our proof and the content of the remainder of this article.   
We will only give a detailed proof of Theorem~\ref{thm-saw-conv-wedge}. The proofs of Theorems~\ref{thm-saw-conv-2side} and~\ref{thm-saw-conv-cone} are essentially identical. We will remark briefly on the proofs of the latter two theorems in Remark~\ref{remark-other-thm}. 

Before we describe our proof, we make some general comments.
\begin{itemize}
\item Our proof does not use anything from the theory of SLE or Liouville quantum gravity. In fact, the only non-trivial outside inputs are the definition of the GHPU topology, the scaling limit of the UIHPQ$_{\op{S}}$~\cite{gwynne-miller-uihpq}, and some basic estimates for the peeling procedure of the UIHPQ$_{\op{S}}$ (see Section~\ref{sec-peeling-prelim}).
\item By \cite{gwynne-miller-uihpq}, we know that the two UIHPQ$_{\op{S}}$'s $(Q_\pm , \BB e_\pm)$ converge in law in the local GHPU topology to the two Brownian half-planes $(X_\pm , d_\pm)$. Due to the universal property of the quotient metric (recall Section~\ref{sec-metric-prelim}), we expect that the metric on any subsequential scaling limit of our glued maps $(Q_{\op{zip}}, d_{\op{zip}}^n , \mu_{\op{zip}}^n , \eta_{\op{zip}}^n)$ is in some sense no larger than the metric $d_{\op{zip}}$ on $X_{\op{zip}}$.  It could \emph{a priori} be strictly smaller if paths in $Q_{\op{zip}}$ which cross the SAW $\eta_{\op{zip}}^n$ more than a constant order number of times are shorter than paths which cross only a constant order number of times. Hence most of our estimates are devoted to proving lower bounds for distances in $Q_{\op{zip}}$ (equivalently upper bounds for the size of metric balls) and upper bounds for how often $Q_{\op{zip}}$-geodesics cross the SAW. 
\item Similarly to Brownian surfaces, the random planar maps considered in this paper satisfy a scaling rule. Heuristically, a graph distance ball of radius $r \in \BB N$ typically has boundary length $\approx r^2$ and contains at most $\approx r^2$ edges of the SAW or the boundary of the map; and contains $\approx r^4$ total edges. 
\end{itemize} 
\medskip

\noindent
Before beginning the proofs of our main theorems, in \textbf{Section~\ref{sec-prelim}} we will establish some standard notational conventions and review some background on several objects which are relevant to this paper, including the Gromov-Hausdorff-Prokhorov-uniform metric, the Brownian half-plane, and the theory of Liouville quantum gravity surfaces. The sections on the Brownian half-plane and on LQG are not used in our proofs and are provided only to make the statements and interpretations of our main results more self-contained. 
\medskip

\noindent
The main tool in this paper is the peeling procedure for the UIHPQ$_{\op{S}}$, which is a means of exploring a UIHPQ$_{\op{S}}$ one quadrilateral at a time in such a way that the law of the unbounded connected component of the unexplored region is always that of a UIHPQ$_{\op{S}}$.  In \textbf{Section~\ref{sec-peeling-prelim}}, we will review the peeling procedure and some of the estimates for peeling which have been proven elsewhere in the literature. We will also use peeling to prove some basic estimates for the UIHPQ$_{\op{S}}$ which will be needed later.
\medskip

\noindent
In \textbf{Section~\ref{sec-peeling-glued}}, we will introduce the \emph{glued peeling process}, a peeling process for the glued map $Q_{\op{zip}} = Q_- \cup Q_+$ appearing in Theorem~\ref{thm-saw-conv-wedge} which approximates the sequence of $Q_{\op{zip}}$-graph metric neighborhoods $B_r\left(\BB A ; Q_{\op{zip}}\right)$ for $r\in\BB N$ together with the points they disconnect from $\infty$ on either side of the SAW. This will be the main peeling process used in our proofs. 

Roughly speaking, if one is given a bounded connected initial edge set $\BB A\subset \bdy Q_-\cup \bdy Q_+$, the glued peeling process started from $\BB A$ is the family of quadrangulations $\{\dot Q^j\}_{j \geq 0}$ obtained as follows. We start by peeling some quadrilateral of $Q_-$ or $Q_+$ which shares a vertex with $\BB A$, and define $\dot Q^1$ to be the quadrangulation consisting of the union of this quadrilateral and all of the vertices and edges it disconnects from $\infty$ in either $Q_-$ or $Q_+$. We continue this procedure until the first time $J_1 \in \BB N$ that every quadrilateral which shares a vertex with $\BB A$ belongs to $\dot Q^{J_1}$. We then continue in the same manner, except we peel quadrilaterals incident to $\bdy \dot Q^{J_1}$ instead of quadrilaterals incident to $\BB A$. There is a natural sequence of stopping times $\{J_r\}_{r\geq 0}$ associated with the glued peeling process, with the property that $J_0 =0$ and $J_r$ is the smallest $r\in\BB N$ such that $\dot Q^{J_r}$ contains every quadrilateral of $Q_{\op{zip}}$ incident to $\bdy\dot Q^{J_{r-1}}$. 
One easily checks (Lemma~\ref{prop-peel-ball}) that the $Q_{\op{zip}}$-graph metric ball satisfies
\eqb \label{eqn-peel-ball0}
B_r\left( \BB A ; Q_{\op{zip}} \right) \subset \dot Q^{J_r} ,\quad \forall r = 0,1,2,\dots ,
\eqe
although the inclusion is typically strict.
Hence we can use the precise estimates for peeling described in Section~\ref{sec-peeling-prelim} to obtain upper bounds for the size of graph metric balls in $Q_{\op{zip}}$. 

The glued peeling process is similar in spirit to the peeling by layers algorithm studied in~\cite{curien-legall-peeling}; c.f.\ Remark~\ref{remark-peeling-by-layers}. This peeling process is also introduced and studied independently in~\cite{caraceni-curien-saw}, where it is shown that the number of SAW edges contained in radius-$r$ glued peeling cluster is typically at most $O_r(r^2)$. Our estimates for the glued peeling process, described just below, are sharper than those of \cite{caraceni-curien-saw}.

We write $\wh Y^j$, $j\in\BB N$, for the number of edges of $\bdy Q_-\cup \bdy Q_+$ which are contained $\dot Q^j$. Note that $\wh Y^j$ is at least the number of SAW edges belonging to $\dot Q^j$, but it could be larger since not every edge of $\bdy Q_-$ is identified with an edge of $\bdy Q_+$.
In keeping with the fact that SAW lengths should behave like the square of distances, we expect that $\wh Y^{J_r}$ is typically of order $r^2$. 
A key task in our proofs is to prove that this is indeed the case, in a sufficiently quantitative sense.
Section~\ref{sec-peel-jump} contains some basic estimates for $\wh Y$ which are proven using the basic peeling estimates from Section~\ref{sec-peeling-prelim}.  
\begin{itemize}
\item (Bound for the sum of the small jumps) For each fixed $p\geq 1$, each $r\in\BB N$, and each $n\in\BB N$, the $p$th moment of the sum of the truncated ``jumps" $(\wh Y_j - \wh Y_{j-1}) \wedge n$ up to time $J_r$ is bounded above by a constant times $(r^2 \vee n)^p$; see Lemma~\ref{prop-small-bubble-moment}.  
\item (Bound for the number of large jumps) For each $r\in\BB N$ and each $n\in\BB N$, the number of $j \leq J_r$ for which $\wh Y_j - \wh Y_{j-1} \geq n$ is stochastically dominated by a geometric random variable with success probability proportional to $n^{-1/2} r$; see Lemma~\ref{prop-big-jump}.
\end{itemize}   
\medskip

\noindent
Sections~\ref{sec-peeling-moment}--\ref{sec-saw-conv} form the core of the paper.  See Figure~\ref{fig-map} for a map of how the main statements of these sections fit together.  In what follows, we will provide a more detailed outline of how these statements are proved.
\medskip

\noindent
In \textbf{Section~\ref{sec-peeling-moment}}, we will prove our key estimate for the glued peeling process (Proposition~\ref{prop-hull-moment}), which says that for $r\in \BB N$ and $p \in [1,3/2)$, both $\BB E[(\wh Y^{J_r})^p]$ (in the notation just above) and the $p$th moment of the length of $\bdy\dot Q^{J_r}$ are bounded above by a universal constant times $r^{2p}$.  
This estimate is proven by using results from Sections~\ref{sec-peeling-prelim} and~\ref{sec-peeling-glued} and the inductive manner in which the glued peeling clusters are constructed to set up various recursive relations between quantities related to the glued peeling process, then solving the recursions to obtain estimates. 

To be more precise, we will first show that the first moments satisfy $\BB E[\wh Y^{J_r}]  = O(r^2)$ and $\BB E[J_r]  = O( r^3)$. To do this, we use basic peeling estimates to bound the conditional expectations of the increments $\wh Y^{J_r} - \wh Y^{J_{r-1}}$ and $J_r - J_{r-1}$ given the peeling process up to time $J_{r-1}$ in terms of $\wh Y^{J_{r-1}}$ and the number of edges in $\bdy \dot Q^{J_{r-1}}$. This gives us a recursive relation between the expectations of these quantities which we then solve to get the desired moment bounds. A key tool in setting up these recursions is the fact that the net number of edges added to the boundary of the unexplored UIHPQ$_{\op{S}}$ when we peel a single quadrilateral is zero. This implies that the net number of edges added to the boundary after $j$ peeling steps, i.e., (number of edges of $\bdy \dot Q^j\cap Q_\pm$) $-$ (number of edges of $  \bdy Q_\pm \cap \dot Q^j$), is a martingale. 
 In Sections~\ref{sec-mart-moment} and~\ref{sec-hull-moment-proof}, these estimates will be combined with the estimates for the jumps of $\wh Y^j$ proven in Section~\ref{sec-peel-jump} to set up another recursive bound (Proposition~\ref{prop-big-jump-recursion}) in terms of the times of the big jumps of $\wh Y^{J_r}$ (i.e., those of size proportional to $r^2$). Solving this last recursion gives the desired moment bounds for $\wh Y^{J_r}$ and the length of $\bdy\dot Q^{J_r}$. 

Section~\ref{sec-hull-moment-misc} contains several consequences of Proposition~\ref{prop-hull-moment} which imply qualitative statements about subsequential limits of the curve-decorated metric measure spaces $(Q_{\op{zip}}^n ,\eta_{\op{zip}}^n , \mu_{\op{zip}}^n , \eta_{\op{zip}}^n)$ of Theorem~\ref{thm-saw-conv-wedge} in the GHPU topology. 
These estimates include the following.
\begin{itemize}
\item A reverse H\"older continuity estimate for the SAW (Lemma~\ref{prop-reverse-holder}) which implies that the curve in any subsequential GHPU limit is simple.
\item An upper bound for the diameter of a $Q_{\op{zip}}$-metric ball with respect to the metrics on the two UIHPQ$_{\op{S}}$'s $Q_\pm$ (Lemma~\ref{prop-hull-diam}).
\item A bound which says that when $\rho > 0$ is small, any path in $Q_{\op{zip}}$ which stays in the $\rho r$-neighborhood of the gluing interface (i.e., the SAW) has to be much longer than a $Q_{\op{zip}}$-geodesic with the same endpoints (Lemma~\ref{prop-saw-neighborhood-dist}). 
This bound is useful to prevent  $Q_{\op{zip}}$-geodesics from staying close to the gluing interface. 
\end{itemize}  
\medskip

\noindent
In \textbf{Section~\ref{sec-geodesic-properties}}, we will prove two estimates which will be used to identify the law of a subsequential limit of our SAW-decorated quadrangulations in the GHPU topology.

Proposition~\ref{prop-lipschitz-path} tells us that two given points of the re-scaled SAW $\eta_{\op{zip}}^n$ can typically be joined by a path which crosses $\eta_{\op{zip}}^n$ at most a number of times which can be bounded independently of $n$; and whose length is at most a universal constant $C$ times the $d_{\op{zip}}^n$-distance between the two points. 
Recall that the quotient metric $d_{\op{zip}}$ on $X_{\op{zip}}$ is defined in terms of paths which cross the gluing interface $\eta_{\op{zip}}$ at most a finite number of times (see Section~\ref{sec-metric-prelim}). Hence this result will imply that if $(\wt X , \wt d  , \wt\mu , \wt \eta)$ is a subsequential limit of $\{(Q_{\op{zip}}^n ,\eta_{\op{zip}}^n , \mu_{\op{zip}}^n , \eta_{\op{zip}}^n)\}_{n\in\BB N}$ in the GHPU topology, then there is a $C$-Lipschitz map $(\wt X , \wt d) \rta (X_{\op{zip}} , d_{\op{zip}})$. 

Proposition~\ref{prop-geodesic-away} tells us that there is a universal constant $\beta \in (0,1)$ such that a $d_{\op{zip}}^n$-geodesic between two given points of $\eta_{\op{zip}}^n$ typically spends at least a $\beta$ fraction of times time away from $\eta_{\op{zip}}^n$. 

To prove the above two propositions, we will show that if we run the glued peeling process from a given initial edge set, then with high probability there is a radius $R$, which is not too large, for which a certain ``good" event occurs for the cluster $\dot Q^{J_R}$. 
We will then cover a given segment of the SAW by such good clusters and study the behavior of a $Q_{\op{zip}}$-geodesic when it passes through them. 
In the case of Proposition~\ref{prop-lipschitz-path}, the ``good" event corresponds to the condition that the $Q_-$-diameter of $\bdy \dot Q^{J_R} \cap Q_-$ and the $Q_+$-diameter of $\bdy \dot Q^{J_R} \cap Q_+$ are each most $C R$, for a constant $C>1$. This means that there is a path between any two points of $\bdy \dot Q^{J_R}$ of length at most $2 C R$ which crosses the gluing interface at most once. In the case of Proposition~\ref{prop-geodesic-away}, the ``good" event corresponds to the condition that a $Q_{\op{zip}}$-geodesic from $\bdy \dot Q^{J_R}$ to the initial edge set has to spend at least $\beta R$ units of time outside of a small neighborhood of the SAW.
  
The existence of the desired radius $R$ is deduced from the estimates of Section~\ref{sec-peeling-moment} together with a multi-scale argument. 
In particular, we start the glued peeling process at a given initial edge set $\BB A$ and consider certain random radii $1 = r_0 < r_1 < r_2 < \dots$ such that the $J_{r_k}$'s are stopping times for the glued peeling process and $r_k$ typically grows like an exponential function of $k$ (see~\eqref{eqn-radius-iterate-r-L-def} for a precise definition of these radii). We show that, for each $k$, the conditional probability given the peeling process up to time $J_{r_{k-1}}$ that the annulus $\dot Q^{J_{r_k}} \setminus \dot Q^{J_{r_{k-1}}}$ is ``good" is close to 1 when $C$ is large (or $\beta$ is small) uniformly over the choices of $k$ and $\BB A$.  Hence, by multiplying over $k$, we get that the smallest $k$ for which the radius $r_k$ is ``good" has an exponential tail, and the coefficient inside the exponential can be made arbitrarily large by making $C$ large enough (or $\beta$ small enough).  The estimates from Section~\ref{sec-peeling-moment} are used both to lower-bound the probability of the event at each scale and to control the ratios $r_k / r_{k-1}$, so as to upper-bound the size of the smallest good radius. 

In the end, we get that for $p$ slightly smaller than $3/2$, the smallest good radius $R$ for the glued peeling process started from $\BB A$ satisfies
\eqbn
\BB P\left[ R > S (\#\BB A)^{1/2} \right]  = O(S^{-2p}) ,\quad\forall S > 1 ;
\eqen
see Lemmas~\ref{prop-good-radius} and~\ref{prop-good-radius'}.
The $3/2$ appearing here is related to the fact that we get moments up to order $3/2$ in Section~\ref{sec-peeling-moment}.
We then apply this estimate to $O(\delta^{-2})$ initial edge sets $\BB A$ which each have $\#\BB A = \delta^2 r^2$, for $\delta > 0$ small but independent of $r$, and take a union bound. This allows us to cover a given segment of the SAW by good clusters of the form $\dot Q^{J_R}$, for varying choices of the initial edge set, in such a way that most of the clusters do not contain the endpoints of the SAW segment (Lemma~\ref{prop-good-radius-exist}). As explained in Section~\ref{sec-geo-proof}, this leads to the conclusion that a $Q_{\op{zip}}$-geodesic has to cross between the boundary of one of these good clusters and its initial edge set (nearly) every time it hits the gluing interface. This, in turn, leads to Propositions~\ref{prop-lipschitz-path} and~\ref{prop-geodesic-away}.  
\medskip

\noindent
As explained in the earlier parts of \textbf{Section~\ref{sec-saw-conv}}, the results of Section~\ref{sec-peeling-moment} together with the scaling limit result for the UIHPQ$_{\op{S}}$~\cite[Theorem~1.12]{gwynne-miller-uihpq} already imply the convergence of $(Q_{\op{zip}}^n ,\eta_{\op{zip}}^n , \mu_{\op{zip}}^n , \eta_{\op{zip}}^n)$ along subsequences to a non-degenerate limiting curve-decorated metric measure space. 
Using this, the universal property of the quotient metric (Remark~\ref{remark-quotient-universal}), and elementary limiting arguments based on the general deterministic properties of the GHPU metric from~\cite{gwynne-miller-uihpq}, one can show that if $(\wt X , \wt d , \wt \mu , \wt\eta)$ is a subsequential limit, then the following is true.
There exists a bijective $1$-Lipschitz map $f_{\op{zip}} : X_{\op{zip}} \rta \wt X$ satisfying $(f_{\op{zip}})_* \mu_{\op{zip}} = \wt\mu$ and $f_{\op{zip}} \circ\eta_{\op{zip}}  =\wt\eta$ which preserves the length of any path in $X_{\op{zip}}$ which does not hit $\eta_{\op{zip}}$ (Proposition~\ref{prop-zip-map}).  

In Section~\ref{sec-saw-proof}, we will show that the map $f_{\op{zip}}$ is an isometry as follows. The results of Section~\ref{sec-geodesic-properties} discussed above imply that there are universal constants $C \geq 1$ and $\beta \in (0,1)$ such that the following is true. 
The map $f_{\op{zip}}^{-1}$ is a.s.\ Lipschitz with Lipschitz constant $C$; and almost every pair of points on the gluing interface $\wt\eta$ can be joined by a $\wt d$-geodesic which spends at least a $\beta$-fraction of its time away from $\wt\eta$. 
Note that at this point we have not ruled out the possibility that a $\wt d$-geodesic traces $\wt\eta$ for a positive Lebesgue measure set of times (\emph{a posteriori}, this follows from the SLE-decorated LQG description of the limiting object). 
We can take $C$ to be the smallest constant for which the above Lipschitz property holds. We want to show that $C=1$.

Suppose $\gamma$ is a $\wt d$-geodesic between two points of $\wt \eta$ which spends at least a $\beta$-fraction of its time away from $\wt\eta$.
We decompose $\gamma$ into finitely many segments of total length at least $(\beta/2)|\gamma|$ during which it does not hit $\wt\eta$ and finitely many complementary segments during which it may hit or cross $\wt\eta$, with total length at most $(1-\beta/2)|\gamma|$. 
The map $f_{\op{zip}}^{-1}$ is an isometry away from $\wt\eta$, so the $\wt d$-length of each segment of $\gamma$ which does not hit $\wt\eta$ is the same as the $d_{\op{zip}}$-length of the image of this segment under $f_{\op{zip}}^{-1}$. On the other hand, since $f_{\op{zip}}^{-1}$ is $C$-Lipschitz, the $\wt d$-length of any segment of $\gamma$ is at most $C$ times the $d_{\op{zip}}$-length of its image under $f_{\op{zip}}^{-1}$. 
By summing over the finitely many intervals in our decomposition of $\gamma$, we see that the $\wt d$-length of $\gamma$ is at most $(1-\beta/2 )C + \beta/2$ times the $d_{\op{zip}}$-length of $f_{\op{zip}}^{-1}(\gamma)$. Since $C$ was chosen to be the optimal Lipschitz constant for $f_{\op{zip}}^{-1}$, this shows that $C \leq (1-\beta/2) C + \beta/2$, so since $C\geq 1$ we must have $C=1$. Thus any subsequential limit of the SAW-decorated quadrangulations agrees with $(X_{\op{zip}} , d_{\op{zip}} , \mu_{\op{zip}} , \eta_{\op{zip}} )$ as curve-decorated metric measure spaces. 
 \medskip

\noindent
\textbf{Appendix~\ref{sec-index}} contains an index of the commonly used symbols in the paper.

\section{Preliminaries}
\label{sec-prelim}

In this section we will introduce some notation and review several objects from other places in the literature which are relevant to the results of this paper.  In Section~\ref{sec-notation-prelim}, we will fix some (essentially standard) notation which we will use throughout the remainder of this article. In Section~\ref{sec-metric-prelim}, we review some notation and definitions converging metric spaces. In Section~\ref{sec-ghpu-prelim}, we review the definition of the Gromov-Hausdorff-Prokhorov-uniform metric from~\cite{gwynne-miller-uihpq} and some of its basic properties. This is the metric with respect to which the convergence in our main theorems takes place.  In Section~\ref{sec-bhp-prelim}, we recall the definition of the Brownian half-plane, which can be used to construct the limiting objects in our main theorems.  In Section~\ref{sec-lqg-prelim}, we review the theory of Liouville quantum gravity and explain why the limiting objects in our main theorems are equivalent to $\sqrt{8/3}$-LQG surfaces decorated by independent SLE$_{8/3}$-type curves.

Most of the content of this paper can be understood independently of this section.  In order to understand the proofs in Sections~\ref{sec-peeling-prelim}--\ref{sec-geodesic-properties}, one only needs to be familiar with the notation described in Section~\ref{sec-notation-prelim}.  In order to also understand the proofs in Section~\ref{sec-saw-conv}, one only needs to be familiar with Sections~\ref{sec-notation-prelim}, \ref{sec-metric-prelim}, and~\ref{sec-ghpu-prelim}.

\subsection{Notational conventions}
\label{sec-notation-prelim}

\subsubsection{Basic notation}
\label{sec-basic-notation}

\noindent
We write $\BB N$ for the set of positive integers and $\BB N_0 = \BB N\cup \{0\}$. 
\vspace{6pt}

\noindent
For $a < b \in \BB R$, we define the discrete intervals $[a,b]_{\BB Z} := [a, b]\cap \BB Z$ and $(a,b)_{\BB Z} := (a,b)\cap \BB Z$.
\vspace{6pt}

\noindent
If $a$ and $b$ are two quantities, we write $a\preceq b$ (resp.\ $a \succeq b$) if there is a constant $C$ (independent of the parameters of interest) such that $a \leq C b$ (resp.\ $a \geq C b$). We write $a \asymp b$ if $a\preceq b$ and $a \succeq b$.

\subsubsection{Graphs and maps}
\label{sec-graph-notation}
 
\noindent
For a planar map $G$, we write $\mcl V(G)$ for its set of vertices, $\mcl E(G)$ for its set of edges, and $\mcl F(G)$ for its set of faces.
\vspace{6pt}

\noindent
By a \emph{path} in $G$, we mean a function $ \lambda : I \rta \mcl E(G)$ for some (possibly infinite) discrete interval $I\subset \BB Z$, with the property that the edges $\{\lambda(i)\}_{i\in I}$ can be oriented in such a way that the terminal endpoint of $\lambda(i)$ coincides with the initial endpoint of $\lambda(i+1)$ for each $i \in I$ other than the right endpoint of~$I$. We define the \emph{length} of~$\lambda$, denoted $|\lambda|$, to be the integer~$\# I$.  
It is convenient to require that the edges can be oriented in a consistent manner since it allows us to ``linearly interpolate" along the path in a canonical way; see Remark~\ref{remark-ghpu-graph}.
\vspace{6pt}

\noindent
For sets $A_1,A_2$ consisting of vertices and/or edges of~$G$, we write $\op{dist}\left(A_1 , A_2 ; G\right)$ for the graph distance from~$A_1$ to~$A_2$ in~$G$, i.e.\ the minimum of the lengths of paths in $G$ whose initial edge either has an endpoint which is a vertex in $A_1$ or shares an endpoint with an edge in $A_1$; and whose final edge satisfies the same condition with $A_2$ in place of $A_1$.  If $A_1$ and/or $A_2$ is a singleton, we do not include the set brackets. Note that the graph distance from an edge $e$ to a set $A$ is the minimum distance between the endpoints of $e$ and the set $A$.
\vspace{6pt}

\noindent
For $r>0$, we define the graph metric ball $B_r\left( A_1 ; G\right)$ to be the subgraph of $G$ consisting of all vertices of $G$ whose graph distance from $A_1$ is at most $r$ and all edges of $G$ whose endpoints both lie at graph distance at most $r$ from $A_1$.  If $A_1 = \{x\}$ is a single vertex or edge, we write $B_r\left( \{x\} ; G\right) =  B_r\left( x ; G\right)$. 

\subsubsection{Quadrangulations with boundary}
\label{sec-quad-prelim} 
 
\noindent
A \emph{quadrangulation with boundary} is a (finite or infinite) planar map~$Q$ with a distinguished face~$f_\infty$, called the \emph{exterior face}, such that every face of~$Q$ other than~$f_\infty$ has degree $4$. The \emph{boundary} of $Q$, denoted by $\bdy Q$, is the smallest subgraph of~$Q$ which contains every edge of $Q$ incident to $f_\infty$. The \emph{perimeter} $\op{Perim}(Q) $ of~$Q$ is defined to be the degree of the exterior face. 
\vspace{6pt}

\noindent
We say that $\bdy Q$ is \emph{simple} if the exterior face has no vertices of multiplicity strictly larger than $1$. In this paper we will only consider quadrangulations with simple boundary. 
\vspace{6pt} 

\noindent
A \emph{boundary path} of $Q$ is a path $\lambda$ from $[1, \op{Perim}(Q) ]_{\BB Z}$ (if $\bdy Q$ is finite) or $\BB Z$ (if $\bdy Q$ is infinite) to $\mcl E(\bdy Q)$ which traces the edges of $\bdy Q$ (counted with multiplicity) in cyclic order. Choosing a boundary path is equivalent to choosing an oriented root edge on the boundary. This root edge is $\lambda(\op{Perim}(Q))$, oriented toward $\lambda(1)$ in the finite case; or $\lambda(0)$, oriented toward $\lambda(1)$, in the infinite case.  
\vspace{6pt}

\noindent
The \emph{uniform infinite planar quadrangulation with simple boundary} (UIHPQ$_{\op{S}}$) is the infinite boundary-rooted quadrangulation $(Q_{\op{S}} , \BB e_{\op{S}})$ with simple boundary which is the limit in law with respect to the Benjamini-Schramm topology~\cite{benjamini-schramm-topology} of a uniformly random quadrangulation with simple boundary (rooted at a uniformly random boundary edge) with $n$ interior vertices and $2l$ boundary edges if we first send $n \rta \infty$ and then $l\rta\infty$~\cite{curien-miermont-uihpq,caraceni-curien-uihpq}. It can also be constructed from the uniform infinite planar quadrangulation with general boundary (UIHPQ) by ``pruning" quadrangulations which can be disconnected from $\infty$ by removing a single vertex; see~\cite{curien-miermont-uihpq,caraceni-curien-uihpq,gwynne-miller-uihpq}.

\subsection{Metric spaces}
\label{sec-metric-prelim}
 
Here we introduce some notation for metric spaces and recall some basic constructions.
Throughout, let $(X,d_X)$ be a metric space. 
\vspace{6pt}

\noindent
For $A\subset X$ we write $\op{diam} (A ; d_X )$ for the supremum of the $d_X$-distance between points in $A$.
\vspace{6pt}

\noindent
For $r>0$, we write $B_r(A;d_X)$ for the set of $x\in X$ with $d_X (x,A) \leq r$. We emphasize that $B_r(A;d_X)$ is closed (this will be convenient when we work with the local GHPU topology). 
If $A = \{y\}$ is a singleton, we write $B_r(\{y\};d_X) = B_r(y;d_X)$.  
\vspace{6pt}
 
Let $\sim$ be an equivalence relation on $X$, and let $\ol X = X/\sim$ be the corresponding topological quotient space. For equivalence classes $\ol x , \ol y\in \ol X$, let $\mcl Q(\ol x , \ol y)$ be the set of finite sequences $(x_1 , y_1 ,    \dots , x_n , y_n)$ of elements of $X$ such that $x_1 \in \ol x$, $y_n \in \ol y$, and $y_i \sim x_{i+1}$ for each $i \in [1,n-1]_{\BB Z}$. Let
\eqb \label{eqn-quotient-def}
\ol d_X (\ol x , \ol y) := \inf_{\mcl Q(\ol x , \ol y)} \sum_{i=1}^n d_X (x_i ,y_i ) .
\eqe  
Then $\ol d_X$ is a pseudometric on $\ol X$ (i.e., it is symmetric and satisfies the triangle inequality), which we call the \emph{quotient pseudometric}.

To define the limiting metric space $(X_{\op{zip}} , d_{\op{zip}})$ in Theorem~\ref{thm-saw-conv-wedge}, we are applying this definition with $X$ equal to the disjoint union of the two Brownian half-planes $(X_-,d_-)$ and $(X_+ ,d_+)$ and with the equivalence relation which identifies the boundary points $\eta_-(s)$ and $\eta_+(s)$ for each $s\geq 0$. 
We know a priori that $(X_{\op{zip}} , d_{\op{zip}})$ is a metric space instead of just a pseudometric space thanks to the results of~\cite{gwynne-miller-gluing}, but this statement also follows from our proof. Similar statements apply in the settings of Theorems~\ref{thm-saw-conv-2side} and~\ref{thm-saw-conv-cone}.  

\begin{remark}[Universal property of the quotient metric] \label{remark-quotient-universal}
The quotient pseudometric possesses the following universal property. Suppose $f : (X,d_X) \rta (Y , d_Y)$ is a $1$-Lipschitz map which is compatible with $\sim$ in the sense that such that $f(x) = f(y)$ whenever $x,y\in X$ with $x\sim y$. Then $f$ factors through the metric quotient to give a map $\ol f : \ol X \rta Y$ such that $\ol f \circ p = f$, where $p : X\rta \ol X$ is the quotient map.  
\end{remark} 
\vspace{6pt}

\noindent
For a curve $\gamma : [a,b] \rta X$, the \emph{$d_X$-length} of $\gamma$ is defined by 
\eqbn
\op{len}\left( \gamma ; d_X  \right) := \sup_P \sum_{i=1}^{\# P} d_X (\gamma(t_i) , \gamma(t_{i-1})) 
\eqen
where the supremum is over all partitions $P : a= t_0 < \dots < t_{\# P} = b$ of $[a,b]$. Note that the $d_X$-length of a curve may be infinite.
\vspace{6pt}

\noindent
For $Y\subset X$, the \emph{internal metric $d_Y$ of $d_X$ on $Y$} is defined by
\eqb \label{eqn-internal-def}
d_Y (x,y)  := \inf_{\gamma \subset Y} \op{len}\left(\gamma ; d_X \right) ,\quad \forall x,y\in Y 
\eqe 
where the infimum is over all curves in $Y$ from $x$ to $y$. 
The function $d_Y$ satisfies all of the properties of a metric on $Y$ except that it may take infinite values. 
\vspace{6pt}
 
\noindent
We say that $(X,d_X)$ is a \emph{length space} if for each $x,y\in X$ and each $\ep > 0$, there exists a curve of $d_X$-length at most $d_X(x,y) + \ep$ from $x$ to $y$. 
\vspace{6pt}

The end this subsection, we give an example which illustrates some of the subtleties involved when trying to take limits of metric gluings.

\begin{example}[Discontinuity of metric gluings] \label{example-gluing}
Let $(X_1,d_1)$ and $(X_2,d_2)$ be two copies of $[0,1]\times [0,1]$, each equipped with the Euclidean metric.
For an increasing function $f : [0,1] \rta [0,1]$, we can define the \emph{gluing of $(X_1,d_1)$ and $(X_2,d_2)$ according to $f$} to be the pseudometric space $(Y^f,d^f)$ obtained as the metric quotient of the disjoint union of $(X_1,d_1)$ and $(X_2,d_2)$ under the equivalence relation which identifies $(s,0) \in X_1$ with $(f (s) , 0) \in X_2$. We define the \emph{gluing interface} to be the subset of $Y^f$ which is the image of the two copies of $[0,1]\times\{0\}$ under the quotient map. 
As a simple example, if $f$ is the identity map then $(Y^f, d^f)$ is isometric to $[0,2] \times [0,1]$, equipped with the Euclidean metric. 

The above gluing operation is wildly discontinuous with respect to $f$. Indeed, we will give an example of a sequence $\{f^n\}_{n\in\BB N}$ of gluing maps for which $f^n$ converges uniformly to the identity map but all of the corresponding gluings $(Y^{f^n}  ,d^{f^n})$ are degenerate in the sense that the $d^{f^n}$-distance between any two points of the gluing interface is zero. 
 
To this end, let $\{\nu^n\}_{n\in\BB N}$ be a sequence of non-atomic Borel probability measures on $[0,1] $ which are each mutually singular with respect to Lebesgue measure, but which converge to Lebesgue measure with respect to the Prokhorov distance (for example, $\nu^n$ could be the $\gamma^n$-LQG boundary length measure, normalized to have total mass 1, for some sequence $\gamma^n \rta 0$). 
For $n\in\BB N$ and $s\in [0,1]$, let $f^n(s) := \nu^n([0,s])$. Then $f^n$ converges uniformly to the identity map $[0,1]\rta [0,1]$, but it is easy to see (see~\cite[Lemma 2.2]{gwynne-miller-gluing} for a proof) that the $d^{f^n}$-distance between any two points of the gluing interface for $Y^{f^n}$ is zero. 

One can modify the above example so that one still has $f^n\rta \op{Id}$ uniformly, each $(Y^{f^n}  ,d^{f^n})$ is homeomorphic to $[0,2]\times[0,1]$, but the $d^{f^n}$-diameter of the gluing interface converges to zero, so the glued spaces become degenerate in the limit. This can be arranged by replacing $\nu^n$ with a measure which is absolutely continuous with respect to Lebesgue measure but very close to $\nu^n$ in the Prokhorov distance (much closer than the Prokhorov distance between $\nu^n$ and Lebesgue measure). 
\end{example}

\subsection{The Gromov-Hausdorff-Prokhorov-uniform metric}
\label{sec-ghpu-prelim}

In this subsection we will review the definition of the local Gromov-Hausdorff-Prokhorov-uniform (GHPU) metric from~\cite{gwynne-miller-uihpq}, 
which is the metric with respect to which our scaling limit results hold. 

We start by defining the metric in the compact case. 
For a metric space $(X,d)$, we let $C_0(\BB R , X)$ be the space of continuous curves $\eta : \BB R\rta X$ which are ``constant at $\infty$," i.e.\ $\eta$ extends continuously to the extended real line $[-\infty,\infty]$. 
Each curve $\eta : [a,b] \rta X$ can be viewed as an element of $C_0(\BB R ,X)$ by defining $\eta(t) = \eta(a)$ for $t < a$ and $\eta(t) = \eta(b)$ for $t> b$. 
\begin{itemize}
\item Let $\BB d_d^{\op{H}}$ be the $d$-Hausdorff metric on compact subsets of $X$.
\item Let $\BB d_d^{\op{P}}$ be the $d$-Prokhorov metric on finite measures on $X$.
\item Let $\BB d_d^{\op{U}}$ be the $d$-uniform metric on $C_0(\BB R , X)$.
\end{itemize}
 
Let $\BB M^{\op{GHPU}}$ be the set of $4$-tuples $\frk X  = (X , d , \mu , \eta)$ where $(X,d)$ is a compact metric space, $d$ is a metric on $X$, $\mu$ is a finite Borel measure on $X$, and $\eta \in C_0(\BB R,X)$. 

Given elements $\frk X_1 = (X_1 , d_1, \mu_1 , \eta_1) $ and $\frk X_2 =  (X_2, d_2,\mu_2,\eta_2) $ of $ \BB M^{\op{GHPU}}$, a compact metric space $(W, D)$, and isometric embeddings $\iota_1 : X_1\rta W$ and $\iota_2 : X_2\rta W$, we define their \emph{GHPU distortion} by 
\begin{align}
\label{eqn-ghpu-var}
\op{Dis}_{\frk X_1,\frk X_2}^{\op{GHPU}}\left(W,D , \iota_1, \iota_2 \right)   
:=  \BB d^{\op{H}}_D \left(\iota_1(X_1) , \iota_2(X_2) \right) +   
\BB d^{\op{P}}_D \left(( (\iota_1)_*\mu_1 ,(\iota_2)_*\mu_2) \right) + 
\BB d_D^{\op{U}}\left( \iota_1 \circ \eta_1 , \iota_2 \circ\eta_2 \right) .
\end{align}
We define the \emph{Gromov-Hausdorff-Prokhorov-Uniform (GHPU) distance} by
\begin{align} \label{eqn-ghpu-def}
 \BB d^{\op{GHPU}}\left( \frk X_1 , \frk X_2 \right) 
 = \inf_{(W, D) , \iota_1,\iota_2}  \op{Dis}_{\frk X_1,\frk X_2}^{\op{GHPU}}\left(W,D , \iota_1, \iota_2 \right)      ,
\end{align}
where the infimum is over all compact metric spaces $(W,D)$ and isometric embeddings $\iota_1 : X_1 \rta W$ and $\iota_2 : X_2\rta W$.
It is shown in~\cite{gwynne-miller-uihpq} that this defines a complete separable metric on~$\BB M^{\op{GHPU}}$ provided we identify two elements of~$\BB M^{\op{GHPU}}$ which differ by a measure- and curve- preserving isometry.  

We now define the local version of the GHPU metric.
Following~\cite{gwynne-miller-uihpq}, we let $\BB M_\infty^{\op{GHPU}}$ be the set of $4$-tuples $\frk X = (X,d,\mu,\eta)$ where $(X,d)$ is a locally compact length space, $\mu$ is a measure on $X$ which assigns finite mass to each finite-radius metric ball in $X$, and $\eta$ is a curve in~$X$ which satisfies one of the following two conditions. Either $\eta : \BB R\rta X$ or $\eta : (a,b) \rta X$ for some open interval $(a,b) \subset \BB R$ (with possibly one of $a$ or $b$ equal to $\infty$) and $\eta$ extends to a continuous curve from the closure of $(a,b)$ to the one-point compactification $X\cup \{\infty\}$. In the latter case, we view $\eta$ as a continuous function $\BB R \rta X\cup \{\infty\}$ which is constant outside of $[a,b]$.  

Let $\ol{\BB M}_\infty^{\op{GHPU}}$ be the set of equivalence classes of elements of $\BB M_\infty^{\op{GHPU}}$ under the equivalence relation whereby $(X_1 , d_1, \mu_1 , \eta_1) \sim (X_2, d_2,\mu_2,\eta_2)$ if and only if there exists an isometry $f : X_1\rta X_2$ such that $f_*\mu_1=\mu_2$ and $f\circ \eta_1=\eta_2$.  
 
\begin{defn} \label{def-ghpu-truncate}
Let $\frk X = (X , d, \mu,\eta)$ be an element of $\BB M_\infty^{\op{GHPU}}$. For $r > 0$, let 
\eqb
\ul\tau_r^\eta := (-r) \vee \sup\left\{t < 0 \,:\, d(\eta(0) ,\eta(t)) = r\right\} \quad \op{and}\quad \ol\tau_r^\eta := r\wedge \inf\left\{t > 0 \,:\, d(\eta(0),\eta(t)) = r\right\} .
\eqe
The \emph{$r$-truncation} of $\eta$ is the curve $\frk B_r\eta \in C_0(\BB R , X)$ defined by
\eqbn
\frk B_r\eta(t) = 
\begin{cases}
\eta(\ul\tau_r^{\eta}) ,\quad &t\leq \ul\tau_r^\eta  \\
\eta(t) ,\quad &t\in (\ul\tau^\eta , \ol\tau_r^\eta) \\
\eta( \ol\tau_r^{\eta}) ,\quad &t\geq  \ol\tau_r^\eta  .
\end{cases}
\eqen
The \emph{$r$-truncation} of $\frk X$ is the curve-decorated metric measure space
\eqbn
\frk B_r \frk X  = \left( B_r(\eta(0) ;d) , d|_{B_r(\eta(0) ;d)} , \mu|_{B_r(\eta(0) ;d)} , \frk B_r\eta \right) .
\eqen
\end{defn}
 
The \emph{local GHPU metric} on $\BB M_\infty^{\op{GHPU}}$ is defined by
\eqb \label{eqn-ghpu-local-def}
\BB d^{\op{GHPU}}_\infty \left( \frk X_1,\frk X_2\right) = \int_0^\infty e^{-r} \left(1 \wedge \BB d^{\op{GHPU}}\left(\frk B_r\frk X_1, \frk B_r\frk X_2 \right)  \right)  dr
\eqe 
where $\BB d^{\op{GHPU}}$ is as in~\eqref{eqn-ghpu-def}. It is shown in~\cite{gwynne-miller-uihpq} that~$\BB d^{\op{GHPU}}_\infty$ defines a complete separable metric on~$ \BB M_\infty^{\op{GHPU}}$ provided we identify spaces which differ by a measure- and curve-preserving isometry.
 
\begin{remark}[Graphs as elements of $\BB M_\infty^{\op{GHPU}}$] \label{remark-ghpu-graph}
In this paper we will often be interested in a graph $G$ equipped with its graph distance $d_G$. 
In order to study continuous curves in $G$, we need to linearly interpolate~$G$. We do this by identifying each edge of $G$ with a copy of the unit interval $[0,1]$. 
We extend the graph metric on $G$ by requiring that this identification is an isometry. 
If $\lambda$ is a path in $G$, mapping some discrete interval $[a,b]_{\BB Z}$ to $\mcl E(G)$, we extend $\lambda$ from $[a,b]_{\BB Z}$ to $[a-1,b] $ by linear interpolation, so that for $i\in [a,b]_{\BB Z}$, $\lambda$ traces each edge $\lambda(i)$ at unit speed during the time interval $[i-1,i]$.
If we are given a measure $\mu$ on vertices of $G$ and we view $G$ as a connected metric space and $\lambda$ as a continuous curve as above, then $(G , d_G , \mu , \lambda)$ is an element of $\BB M_\infty^{\op{GHPU}}$. 
\end{remark} 

In the remainder of this subsection we explain how local GHPU convergence is equivalent to a closely related type of convergence which is often easier to work with, in which all of the curve-decorated metric measure spaces are subsets of a larger space. For this purpose we need to introduce the following definition, which we take from~\cite{gwynne-miller-uihpq}. 
 
\begin{defn}[Local HPU convergence] \label{def-hpu-local}
Let $(W ,D)$ be a metric space. Let $\frk X^n = (X^n , d^n , \mu^n , \eta^n)$ for $n\in\BB N$ and $\frk X = (X,d,\mu,\eta)$ be elements of $\BB M^{\op{GHPU}}_\infty$ such that $X$ and each $X^n$ is a subset of $W$ satisfying $D|_X = d$ and $D |_{X^n} = d^n$. We say that $\frk X^n\rta \frk X$ in the \emph{$D$-local Hausdorff-Prokhorov-uniform (HPU) sense} if the following are true. 
\begin{itemize}
\item For each $r>0$, we have $B_r(\eta^n(0) ; d^n) \rta B_r(\eta(0) ; d)$ in the $D$-Hausdorff metric, i.e.\ $\frk X^n \rta \frk X$ in the $D$-local Hausdorff metric.
\item For each $r > 0$ such that $\mu\left(\bdy B_r(\eta(0) ;d)\right) = 0$, we have $\mu^n|_{B_r(\eta^n(0) ;d^n)} \rta \mu|_{B_r(\eta(0) ;d)}$ in the $D$-Prokhorov metric. 
\item For each $a,b\in\BB R$ with $a<b$, we have $\eta^n|_{[a,b]}\rta \eta|_{[a,b]}$ in the $D$-uniform metric.  
\end{itemize} 
\end{defn}

The following result, which is~\cite[Proposition~1.9]{gwynne-miller-uihpq}, will play a key role in Section~\ref{sec-saw-conv}. 
 
\begin{prop} \label{prop-ghpu-embed-local}
Let $ \frk X^n = (X^n , d^n , \mu^n , \eta^n)$ for $ n\in\BB N $ and $\frk X = (X,d,\mu,\eta)$ be elements of $\BB M_\infty^{\op{GHPU}}$. Then $\frk X^n\rta \frk X$ in the local GHPU topology if and only if there exists a boundedly compact metric space $( Z, D )$ (i.e., one for which closed bounded sets are compact) and isometric embeddings $X^n \rta Z$ for $n\in\BB N$ and $X\rta Z$ such that the following is true. If we identify $X^n$ and $X$ with their embeddings into $Z$, then $\frk X^n \rta \frk X$ in the $D$-local HPU sense.
\end{prop}

\subsection{The Brownian half-plane}
\label{sec-bhp-prelim}

A \emph{Brownian surface} is a random metric measure space which locally looks like the Brownian map (see~\cite{miermont-survey,miermont-st-flour,legall-sphere-survey} and the references therein for more on the Brownian map).  Brownian surfaces arise as the scaling limits of uniformly random planar maps.  Several specific Brownian surfaces have been constructed via continuum analogs of the Schaeffer bijection~\cite{schaeffer-bijection}, including the Brownian map itself, which is the scaling limit of uniform quadrangulations of the sphere~\cite{miermont-brownian-map,legall-uniqueness}; the Brownian disk, which is the scaling limit of uniform quadrangulations with boundary~\cite{bet-mier-disk}; the Brownian plane, which is the scaling limit of uniform infinite quadrangulations without boundary~\cite{curien-legall-plane}; and the Brownian half-plane, which is the scaling limit of uniform infinite half-planar quadrangulations~\cite{caraceni-curien-uihpq,gwynne-miller-uihpq,bmr-uihpq}. See also~\cite{bmr-uihpq} for some additional Brownian surfaces which arise as scaling limits of certain quadrangulations with boundary.

The limiting objects in our main theorems are described by gluing together Brownian half-planes along their boundaries, so in this section we give a brief review of the definition of this object. 
We will not use most of the objects involved in this construction later in the paper, except for the definition of the area measure, boundary length measure, and boundary path. We review it only for the sake of making this work more self-contained.
We use the construction from~\cite{gwynne-miller-uihpq}. 
A different construction, which we expect to be equivalent, is given in~\cite[Section~5.3]{caraceni-curien-uihpq} but the construction we give here is the one which was been proven to be the scaling limit of the UIHPQ and UIHPQ$_{\op{S}}$ in~\cite{gwynne-miller-uihpq}. 

Let $W_\infty : \BB R\rta[0,\infty)$ be the process such that $\{W_\infty(t) \}_{t\geq 0}$ is a standard linear Brownian motion and $\{W_\infty(-t)\}_{t\geq 0}$ is an independent Brownian motion conditioned to stay positive (i.e., a $3$-dimensional Bessel process). 
For $r\in \BB R$, let
\eqbn
T_\infty(r) := \inf\left\{ t \in \BB R \,:\, W_\infty(t) = -r\right\} ,
\eqen
so that $r\mapsto T_\infty(r)$ is non-decreasing and for each $r\in \BB R$, 
\eqbn
\{ W_\infty(T_\infty(r) + t)  + r \}_{t\in\BB R} \eqD \{ W_\infty(t) \}_{t\in\BB R}  .
\eqen
Also let $T_\infty^{-1} : \BB R\rta \BB R$ be the right-continuous inverse of $T$.
 
For $s,t\in \BB R$, let
\eqb \label{eqn-W_infty-dist}
d_{W_\infty}(s,t) := W_\infty(s) + W_\infty(t)  - 2\inf_{u \in [s\wedge t , s\vee t]} W_\infty(u) .
\eqe 
Then $d_{W_\infty}$ defines a pseudometric on $\BB R$ and the quotient metric space $\BB R / \{d_{W_\infty} = 0\}$ is a forest of continuum random trees, indexed by the excursions of $W_\infty$ away from its running infimum. 

Conditioned on $W_\infty$, let $ Z_\infty^0$ be the centered Gaussian process with
\eqb \label{eqn-Z^0-def-infty}
\op{Cov}(Z_\infty^0(s) , Z_\infty^0(t) ) = \inf_{u\in [s\wedge t , s\vee t]} \left( W_\infty(u) - \inf_{v \leq u } W_\infty(v) \right) , \quad s,t\in \BB R .
\eqe 
By the Kolmogorov continuity criterion, $Z_\infty^0$ a.s.\ admits a continuous modification which satisfies $Z_\infty^0(s) = Z_\infty^0(t)$ whenever $d_{W_\infty}(s,t) = 0$.

Let $\frk b_\infty : \BB R\rta \BB R$ be $\sqrt 3$ times a two-sided standard linear Brownian motion. For $t\in \BB R$, define 
\eqbn
Z_\infty(t) := Z_\infty^0(t)  +  \frk b_\infty(T_\infty^{-1}(t)) .
\eqen  
For $s,t\in \BB R$, let
\eqb \label{eqn-d_Z-infty}
d_{Z_\infty} \left(s,t \right) = Z_\infty(s) + Z_\infty(t) - 2\inf_{u\in [s\wedge t , s\vee t]} Z_\infty(u) .
\eqe 
Also define the pseudometric
\eqb \label{eqn-dist0-def-infty}
d_\infty^0(s,t) = \inf \sum_{i=1}^k d_{Z_\infty}(s_i , t_i)
\eqe 
where the infimum is over all $k\in\BB N$ and all $(2k+2)$-tuples $( t_0, s_1 , t_1 , \dots , s_{k } , t_{k } , s_{k+1}) \in \BB R^{2k+2}$ with $t_0 = s$, $s_{k+1} = t$, and $d_{W_\infty}(t_{i-1} , s_i) = 0$ for each $i\in [1,k+1]_{\BB Z}$. 

The Brownian half-plane is the quotient space $X_\infty = \BB R/\{d_\infty^0 = 0\}$ equipped with the quotient metric~$d_\infty$.  
We write $\BB p_\infty : \BB R \rta X_\infty$ for the quotient map. 
The Brownian half-plane comes with a natural marked boundary point, namely $\BB p(0)$. 
The \emph{area measure} of $X_\infty$ is the pushforward of Lebesgue measure on $\BB R$ under $\BB p_\infty$, and is denoted by $\mu_\infty$. 
The \emph{boundary} of $X_\infty$ is the set $\bdy X_\infty = \BB p\!\left(\{T_\infty(r) \,:\, r \in \BB R\} \right)$. 
The \emph{boundary measure} of $X_\infty$ is the pushforward of Lebesgue measure on $\BB R$ under the map $r\mapsto \BB p_\infty(T_\infty(r))$. 
The \emph{boundary path} of $X_\infty$ is the path $\eta_\infty : \BB R\rta X_\infty$ defined by $\eta_\infty(r) = \BB p_\infty(T_\infty(r))$, which satisfies $\eta_\infty(0) = \BB p(0)$.  
Note that this path traverses one unit of boundary length in one unit of time. 
 
We observe that
\eqb \label{eqn-bhp-ghpu}
(X_\infty , d_\infty  ,\mu_\infty  ,\eta_\infty) \in \BB M_\infty^{\op{GHPU}} ,
\eqe 
where $\BB M_\infty^{\op{GHPU}}$ is as in Section~\ref{sec-ghpu-prelim}.

\subsection{Liouville quantum gravity}
\label{sec-lqg-prelim}

In this subsection we review the definition of Liouville quantum gravity (LQG) surfaces and explain their equivalence with Brownian surfaces in the case when $\gamma = \sqrt{8/3}$. 
We do not use LQG in our proofs, but LQG is important for motivating and interpreting our main results. 
In particular, we will explain in this subsection why the limiting objects in our main theorem are equivalent to SLE-decorated LQG surfaces.

For $\gamma \in (0,2)$, a \emph{Liouville quantum gravity} surface with $k\in\BB N_0$ marked points is an equivalence class of $(k+2)$-tuples $(D, h ,x_1,\dots ,x_k)$, where $D\subset \BB C$ is a domain; $h$ is a distribution on $D$, typically some variant of the Gaussian free field (GFF)~\cite{shef-kpz,shef-gff,ss-contour,shef-zipper,ig1,ig4}; and $x_1,\dots , x_k \in D\cup \bdy D$ are $k$ marked points. Two such $(k+2)$-tuples $(D, h ,x_1,\dots ,x_k)$ and $(\wt D, \wt h  ,\wt x_1,\dots , \wt x_k)$ are considered equivalent if there is a conformal map $f : \wt D \rta D$ such that
\eqb \label{eqn-lqg-coord}
f(\wt x_j) = x_j,\quad \forall j \in [1,k]_{\BB Z} \quad \op{and}\quad  \wt h = h\circ f + Q \log |f'|  \quad \op{where} \: Q = \frac{2}{\gamma} + \frac{\gamma}{2} .
\eqe  
Several specific types of $\gamma$-LQG surfaces (which correspond to particular choices of the GFF-like distribution~$h$) are studied in~\cite{wedges}, including quantum spheres, quantum disks, $\alpha$-quantum cones for $\alpha < Q$, and $\alpha$-quantum wedges for $\alpha < Q + \gamma/2$.

In this paper we will be particularly interested in $\alpha$-quantum wedges and $\alpha$-quantum cones for $\alpha < Q$, so we provide some additional detail on these surfaces. See~\cite[Section~4.2]{wedges} for a precise definition. 
Roughly speaking, an $\alpha$-quantum wedge for $\alpha < Q$ is the quantum surface $(\BB H , h , 0, \infty)$ obtained by starting with the distribution $\wt h - \alpha \log |\cdot|$, where $\wt h$ is a free-boundary GFF on $\BB H$, then zooming in near the origin and re-scaling to get a surface which describes the local behavior of this field when the additive constant is fixed appropriately.
An $\alpha$-quantum cone is the quantum surface $(\BB C , h , 0, \infty)$ which is defined in a similar manner but starting with a whole-plane GFF plus an $\alpha$-log singularity rather than a free-boundary GFF plus an $\alpha$-log singularity. 
 
Instead of the log-singularity parameter $\alpha$, one can also parameterize the spaces of quantum wedges and quantum cones by the \emph{weight} parameter $\frk w$, defined by 
\eqb \label{eqn-wedge-weight}
\frk w = \gamma\left( \frac{\gamma}{2} + Q - \alpha \right) ,\: \text{for wedges} \quad \op{and} \quad 
\frk w = 2\gamma\left(   Q - \alpha \right) ,\: \text{for cones} 
\eqe
with $Q$ as in~\eqref{eqn-lqg-coord}. The reason for using the parameter $\frk w$ is that it is invariant under the cutting and gluing operations, which we will describe below. 

It is shown in~\cite{shef-kpz} that a Liouville quantum gravity surface admits a natural area measure $\mu_h$, which can be interpreted as ``$e^{\gamma h(z)} \, dz$", where $dz$ is Lebesgue measure on $D$, and a length measure $\nu_h$ defined on certain curves in $D$, including $\bdy D$ and SLE$_\kappa$-type curves for $\kappa = \gamma^2$. 
It was recently proven by Miller and Sheffield that in the special case when $\gamma =\sqrt{8/3}$, a $\sqrt{8/3}$-LQG surface admits a natural metric $\frk d_h$~\cite{lqg-tbm1,lqg-tbm2,lqg-tbm3}, building on \cite{qle}. 
All three of these objects are invariant under coordinate changes of the form~\eqref{eqn-lqg-coord}. 

Several particular types of $\sqrt{8/3}$-LQG surfaces equipped with this metric structure are isometric to Brownian surfaces:
\begin{itemize}
\item The Brownian map is isometric to the quantum sphere;
\item The Brownian disk is isometric to the quantum disk; 
\item The Brownian plane is isometric to the weight-$4/3$ quantum cone;
\item The Brownian half-plane is isometric to the weight-$2$ quantum wedge.
\end{itemize}
It is shown in~\cite{lqg-tbm3} that the metric measure space structure a.s.\ determines the embedding of the quantum surface into (a subset of) $\BB C$.  Hence a Brownian surface possesses a canonical embedding into the complex plane. 

One can take the above isometry to push forward the $\sqrt{8/3}$-LQG area measure to the natural volume measure on the corresponding Brownian surface and (in the case of the disk or half-plane) one can take it to push forward the $\sqrt{8/3}$-LQG boundary length measure to the natural boundary length measure on the Brownian disk or half-plane.  
In particular, if we let $(\BB H , h , 0, \infty)$ be a $\sqrt{8/3}$-quantum wedge, equipped with its area measure $\mu_h$, boundary length measure $\nu_h$, and metric $\frk d_h$, and we let $\eta_h : \BB R\rta \BB R$ be the curve satisfying $\eta_h(0)= 0$ and $\nu_h(\eta_h([a,b])) = b-a$ for each $a < b$, then the curve-decorated metric measure spaces $(\BB H , \frk d_h , \mu_h,\eta_h)$ and $(X_\infty,d_\infty,\mu_\infty,\eta_\infty)$, the latter defined as in~\eqref{eqn-bhp-ghpu}, are equivalent as elements of $\BB M_\infty^{\op{GHPU}}$. 

It is shown in~\cite{shef-zipper,wedges} that one can conformally weld a weight-$\frk w_-$ quantum wedge and a weight-$\frk w_+$ quantum wedge together according to quantum length along their positive boundary rays (corresponding to $[0,\infty)$ in our parameterization of the quantum wedge) to obtain a weight-$(\frk w_- + \frk w_+)$ quantum wedge decorated by an independent chordal SLE$_{\gamma^2}(\frk w_--2; \frk w_+ -2)$ curve parameterized by quantum length with respect to the wedge. 
Similarly, one can conformally weld two such quantum wedges together according to quantum length along their entire boundary to obtain a weight-$(\frk w_- + \frk w_+)$ quantum cone decorated by a two-sided chordal SLE$_{\gamma^2}$-type curve parameterized by quantum length with respect to the wedge. One can also conformally weld the positive and negative boundary rays of single quantum wedge of weight $\frk w$ to each other according to quantum length to get a quantum cone of the same weight decorated by an independent whole-plane SLE$_{\gamma^2}(\frk w-2)$ curve. 

It is proven in~\cite{gwynne-miller-gluing} that in the case when $\gamma = \sqrt{8/3}$, when one performs these gluing operation the $\sqrt{8/3}$-LQG metric on the glued surface is the metric space quotient of the metrics on the wedges being glued. Due to the equivalence between the weight-$2$ quantum wedge and the Brownian half-plane, we find the following (recall Figure~\ref{fig-thick-gluing}):
\begin{itemize}
\item Gluing two independent Brownian half-planes together along their positive boundaries and embedding the resulting metric measure space into $\BB H$ produces a weight-$4$ quantum wedge decorated by an independent chordal SLE$_{8/3}$ curve. 
\item Gluing two independent Brownian half-planes together along their entire boundaries and embedding the resulting metric measure space into $\BB C$ produces a weight-$4$ quantum cone decorated by an independent two-sided SLE$_{8/3}$-type curve which can be sampled as follows. First sample a whole-plane SLE$_{8/3}(2)$ curve $\eta_1$ from $0$ to $\infty$; then, conditional on $\eta_1$, sample a chordal SLE$_{8/3}$ curve~$\eta_2$ from~$0$ to~$\infty$ in $\BB C\setminus \eta_1$. Then concatenate these two curves and parameterize the two-sided curve thus obtained by $\sqrt{8/3}$-LQG length.  (This SLE$_{8/3}$-type path can be described as a pair of GFF flow lines \cite{ig1,ig4}.)
\item Gluing the two boundary rays of a single Brownian half-plane together along their entire boundaries and embedding the resulting metric measure space into $\BB C$ produces a weight-2 quantum cone decorated by an independent whole-plane SLE$_{8/3}$ curve.
\end{itemize}
Thus the limiting objects in Theorems~\ref{thm-saw-conv-wedge},~\ref{thm-saw-conv-2side}, and~\ref{thm-saw-conv-cone} are $\sqrt{8/3}$-LQG surfaces decorated by independent SLE$_{8/3}$ curves.

\section{Peeling of the UIHPQ with simple boundary}
\label{sec-peeling-prelim}

In this section, we will study the peeling procedure for the UIHPQ$_{\op{S}}$ (also known as the spatial Markov property), which will be one of the key tools in the proofs of our main theorems. The idea of peeling was first used heuristically in the physics literature to study two-dimensional quantum gravity~\cite{adj-quantum-geometry}. The first rigorous use of peeling was in~\cite{angel-peeling}, in the context of the uniform infinite planar triangulation. The peeling procedure was later adapted to the case of the uniform infinite planar quadrangulation~\cite{benjamini-curien-uipq-walk}. In this paper, we will only be interested peeling on the UIHPQ$_{\op{S}}$, which is also studied, e.g., in~\cite{angel-curien-uihpq-perc,angel-ray-classification,richier-perc}.  

In Section~\ref{sec-peeling}, we will review the definition of peeling on the UIHPQ$_{\op{S}}$, introduce notation for the objects involved, and review some formulas for peeling probabilities from elsewhere in the literature.  Then, in Sections~\ref{sec-1vertex-peeling} and~\ref{sec-linear-peel} we will use peeling to prove some particular estimates for the UIHPQ$_{\op{S}}$.  The reader may wish to temporarily skip these last two subsections and refer back to them when the corresponding estimates are used.

\subsection{Peeling of quadrangulations with boundary}
\label{sec-peeling}

\subsubsection{General definitions} 
\label{sec-general-peeling}

Let $Q$ be an infinite quadrangulation with simple boundary. For an edge $e  \in \mcl E(\bdy Q)$, let $\frk f(Q,e)$ be the quadrilateral of $Q$ containing $e$ on its boundary. 
The quadrilateral $\frk f(Q,e)$ has either two, three, or four vertices in $\bdy Q$, so divides $Q$ into at most three connected components, whose union includes all of the vertices of $Q$ and all of the edges of $Q$ except for $e$ (if $\frk f(Q,e)$ has an edge other than $e$ in $\bdy Q$, this single edge counts as a connected component). Exactly one such component is infinite. These components have a natural cyclic ordering inherited from the cyclic ordering of their intersections with $\bdy Q$. 
We define the \emph{peeling indicator}
\eqb \label{eqn-peeling-indicator}
\frk P(Q,e) \in (\BB N_0 \cup \{\infty\}) \cup (\BB N_0 \cup \{\infty\})^2 \cup (\BB N_0 \cup \{\infty\})^3 ,
\eqe
to be the vector whose elements are the number of edges of each of these components shared by $\bdy Q$, listed in counterclockwise cyclic order started from $e$. 
Note that if $i\in \{1,2,3\}$ and the $i$th component of $\frk P(Q,e)$ is $k$, then the total boundary length of the $i$th connected component of $Q\setminus \frk f(Q,e)$ in counterclockwise cyclic order is $k+1$ (resp.\ $k+2$; $\infty$) if $k$ is odd (resp.\ even; $\infty$). 

We refer to $\frk P(Q,e)$ as the \emph{peeling indicator}.   
The procedure of extracting $\frk f(Q,e)$ and $\frk P(Q,e)$ from $(Q,e)$ will be referred to as \emph{peeling $Q$ at $e$}. See Figure~\ref{fig-peeling-cases} for an illustration of some of the possible cases that can arise when peeling $Q$ and $e$.

\begin{figure}[ht!]
 \begin{center}
\includegraphics[scale=1]{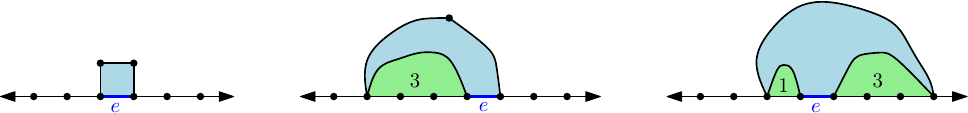} 
\caption[Peeling the UIHPQ$_{\op{S}}$]{An infinite quadrangulation $Q $ with simple boundary together with three different cases for the peeled quadrilateral $\frk f(Q,e)$ (shown in light blue). In the left panel $\frk P(Q,e) = \infty$. In the middle panel, $\frk P(Q,e) = ( \infty , 3)$. In the right panel, $\frk P(Q,e) = (3,\infty,1)$. }\label{fig-peeling-cases}
\end{center}
\end{figure}
  
We now introduce notation for some additional objects associated with peeling.  
\begin{itemize} 
\item Let $\op{Peel} (Q , e)$ be the infinite connected component of $Q\setminus \frk f(Q,e)$. 
\item Let $\frk F (Q,e)$ be the union of the components of $Q\setminus \frk f(Q,e)$ other than $\op{Peel} (Q,e)$.
\item Let $\op{Co} (Q,e)$ be the number of \emph{covered edges} of $\bdy Q$, i.e.\ the number of edges of $ \bdy Q $ which do not belong to $\op{Peel} (Q,e)$ (equivalently, one plus the number of such edges which belong to $\frk F (Q,e)$). 
\item Let $\op{Ex} (Q,e)$ be the number of \emph{exposed edges} of $\frk f(Q,e)$, i.e.\ the number of edges of $ \op{Peel} (Q,e)$ which do not belong to $\bdy Q$ (equivalently, those which are incident to $\frk f(Q,e)$). 
\end{itemize}

\subsubsection{Peeling the UIHPQ with simple boundary} 
\label{sec-uihpq-peeling}
 
In this subsection we will give explicit descriptions of the laws of the objects defined in Section~\ref{sec-general-peeling} when we peel the uniform infinite half-plane quadrangulation with simple boundary (UIHPQ$_{\op{S}}$). These laws will be described in terms of the \emph{free Boltzmann partition function} which is defined by  
\eqb \label{eqn-fb-partition}
\frk Z(2 l)   :=    \frac{8^l(3l-4)!}{(l-2)! (2l)!} , \quad \frk Z(2l+1) := 0 ,\quad \forall l \in \BB N ,
\eqe 
where we set $(-1)!=1$. 
The reason for the name is that $\frk Z$ is the partition function for the so-called \emph{free Boltzmann quadrangulation with simple boundary of perimeter $2l$}~\cite{bg-simple-quad}, but we will not need this model here.

Suppose now that $(Q_{\op{S}} , \BB e_{\op{S}} )$ is an instance of the UIHPQ$_{\op{S}}$.  
As explained in~\cite[Section~2.3.1]{angel-curien-uihpq-perc}, the distribution of the peeling indicator of Section~\ref{sec-general-peeling} when we peel at the root edge is described as follows.
\begin{align} \label{eqn-uihpq-peel-prob}
\BB P\left[  \frk P(Q_{\op{S}} , \BB e_{\op{S}}) = \infty  \right]  
&= \frac38 \notag \\
\BB P\left[ \frk P(Q_{\op{S}} , \BB e_{\op{S}}) = (k , \infty) \right]  
&=  \frac{1}{12} 54^{(1-k)/2}  \frk Z(k+1)  ,\quad \text{$\forall k\in \BB N$ odd} \notag \\
\BB P\left[ \frk P(Q_{\op{S}} , \BB e_{\op{S}}) = (k , \infty) \right] 
&=   \frac{1}{12}  54^{-k/2}   \frk Z(k+2)  ,\quad \text{$\forall k\in \BB N_0$ even} \notag \\
\BB P\left[ \frk P(Q_{\op{S}} , \BB e_{\op{S}}) = (k_1, k_2 , \infty)   \right]  
&=  54^{-(k_1+k_2)/2}   \frk Z(k_1+1) \frk Z(k_2+1)  ,   \quad \text{$\forall k_1,k_2\in  \BB N$ odd} . 
\end{align}
We get the same formulas if we replace $(k,\infty)$ with $(\infty,k)$ or $(k_1,k_2,\infty)$ with either $(\infty,k_1,k_2)$ or $(k_1,\infty,k_2)$ (which corresponds to changing which side of $\BB e_{\op{S}}$ the bounded complementary connected components of $\frk f(Q_{\op{S}} ,\BB e_{\op{S}})$ lie on). 
The probabilities~\eqref{eqn-uihpq-peel-prob} are computed in~\cite[Section~2.3.1]{angel-curien-uihpq-perc}. 

If we condition on $\frk P(Q_{\op{S}} , \BB e_{\op{S}})$, then the connected components of $Q\setminus \frk f(Q_{\op{S}} , \BB e_{\op{S}})$ are conditionally independent.  The conditional law of the unbounded connected component $\op{Peel}(Q_{\op{S}} , \BB e_{\op{S}})$, rooted at one of the boundary edges it shares with $\frk f(Q_{\op{S}} , \BB e_{\op{S}})$ (chosen by some deterministic convention in the case when there is more than one such edge) is again that of a UIHPQ$_{\op{S}}$. This fact is referred to as the \emph{Markov property of peeling}.

\subsubsection{Peeling processes}
\label{sec-peeling-procedure}

Due to the Markov property of peeling, one can iteratively peel a UIHPQ$_{\op{S}}$ to obtain a sequence of quadrangulations which each has the law of a UIHPQ$_{\op{S}}$. 
To make this notion precise, let $(Q_{\op{S}} , \BB e_{\op{S}} )$ be a UIHPQ$_{\op{S}}$. 
A \emph{peeling process} on $Q_{\op{S}}$ is a sequence of quadrangulation-edge pairs $\{(Q^{i-1} , e^i) \}_{i \in  [1, \mcl I]_{\BB Z} }$ with $\mcl I \in \BB N$ a possibly infinite random time, called the \emph{terminal time}, such that the following is true. 
\begin{enumerate}
\item $Q^0 = Q_{\op{S}}$ and for each $i \in [1,\mcl I]_{\BB Z}$, we have $e^i \in \mcl E(\bdy Q^{i-1})$ and $Q^i = \op{Peel} (Q^{i-1} , e^i)$.  \label{item-peeling-iterate}
\item Each edge $e^i$ is chosen in a manner which is measurable with respect to the $\sigma$-algebra $\mcl G^{i-1}$ generated by the peeling indicator variables $\frk P(Q^{j -1} , e^j)$ for $j\in [1,i-1]_{\BB Z}$ and the planar map $ Q^{i-1}$. Furthermore, $\{\mcl I \leq i\} \in \mcl G^i$ for each $i\in \BB N_0$. \label{item-peeling-msrble}
\end{enumerate}
It follows from the Markov property of peeling that for each $i\in \BB N $, the conditional law of $(Q^{i-1} , e^i)$ given the $\sigma$-algebra $\mcl G^{i-1}$ of condition~\ref{item-peeling-msrble} on the event $\{\mcl I \geq i-1\}$ is that of a UIHPQ$_{\op{S}}$. 

We will have occasion to consider several different peeling processes in this paper.

\subsubsection{Estimates for peeling probabilities}
\label{sec-peeling-estimate}

In this subsection we will write down some estimates for the probabilities appearing in Section~\ref{sec-uihpq-peeling}. Throughout, we let $(Q_{\op{S}}, \BB e_{\op{S}})$ be a UIHPQ$_{\op{S}}$.   

Stirling's formula implies that for each even $k\in\BB N$, the free Boltzmann partition function~\eqref{eqn-fb-partition} satisfies
\eqb \label{eqn-stirling-asymp}
\frk Z(k ) \asymp 54^{ k/2}  k^{-5/2} ,\quad \text{for $k$ even}   
\eqe 
with universal implicit constant. 
From this we infer the following approximate versions of the probabilities~\eqref{eqn-uihpq-peel-prob}.   
\begin{align} \label{eqn-uihpq-peel-prob-approx} 
\BB P\left[ \frk P(Q_{\op{S}} , \BB e_{\op{S}}) = (k , \infty) \right]  
&\asymp k^{-5/2}  ,\quad \text{$\forall k\in \BB N$ odd} \notag \\
\BB P\left[ \frk P(Q_{\op{S}} , \BB e_{\op{S}}) = (k , \infty) \right]   
&\asymp k^{-5/2},\quad \text{$\forall k\in \BB N$ even} \notag \\
\BB P\left[ \frk P(Q_{\op{S}} , \BB e_{\op{S}}) = (k_1, k_2 , \infty)   \right]  
&\asymp  k_1^{-5/2} k_2^{-5/2}  ,  \quad \text{$\forall k_1,k_2\in  \BB N$ odd} . 
\end{align}
We get the same approximate formulas if we replace $(k,\infty)$ with $(\infty,k)$ or $(k_1,k_2,\infty)$ with either $(\infty,k_1,k_2)$ or $(k_1,\infty,k_2)$.   
  
Let $(Q_{\op{S}}, \BB e_{\op{S}})$ be a UIHPQ$_{\op{S}}$ and recall the definitions of the number of exposed edges $\op{Ex}(Q_{\op{S}}, \BB e_{\op{S}})$ and the number of covered edges $\op{Co} (Q_{\op{S}}, \BB e_{\op{S}})$ from Section~\ref{sec-general-peeling}
As explained in~\cite[Section~2.3.2]{angel-curien-uihpq-perc}, one has the following facts about the joint law of these random variables. We have the equality of means
\eqb \label{eqn-net-bdy-mean}
\BB E\left[  \op{Co} (Q_{\op{S}}  , \BB e_{\op{S}} ) \right] = \BB E\left[  \op{Ex} (Q_{\op{S}}  , \BB e_{\op{S}} ) \right]  ,
\eqe 
i.e.\ the expected net change in the boundary length of $Q_{\op{S}} $ under the peeling operation is 0. We always have $\op{Ex} (Q_{\op{S}}  , \BB e_{\op{S}} ) \in \{1,2,3\}$, but $\op{Co} (Q_{\op{S}}  , \BB e_{\op{S}} )$ can be arbitrarily large. In fact, there is a constant $c_* > 0$ such that for $k\in\BB N$, 
\eqb \label{eqn-cover-tail}
\BB P\left[ \op{Co}(Q_{\op{S}}  , \BB e_{\op{S}} ) = k \right] = (c_* + o_k(1)) k^{-5/2} .
\eqe 
 
\subsection{Peeling all quadrilaterals incident to a vertex}
\label{sec-1vertex-peeling}

Suppose we want to use peeling to approximate the graph-distance ball centered at a vertex $v \in \mcl V(\bdy Q)$ for a given quadrangulation $Q$ with boundary $\partial Q$. If $v$ has high degree, it is a priori possible, e.g., that we need to peel a large number of edges in order to cover the graph metric ball $B_1(v ; Q)$. Similar issues arise when trying to use peeling to approximate metric balls with bigger radius. The purpose of this subsection is to show that versions of the estimates of Section~\ref{sec-peeling-estimate} are still valid if instead of peeling a single edge incident to $v$ we continue peeling edges incident to $v$ until we disconnect $v$ from the target edge. In particular, we will prove the following lemma.

\begin{lem} \label{prop-peel-degree} 
Let $(Q_{\op{S}} , \BB e_{\op{S}})$ be an instance of the UIHPQ$_{\op{S}}$ and let $v\in \mcl V(\bdy Q_{\op{S}})$ be one of the endpoints of $\BB e_{\op{S}}$.  
Let $\mcl Q_v$ be the set of quadrilaterals $q \in \mcl F(Q_{\op{S}})$ which are incident to $v$ and let $E_v^L$ (resp.\ $E_v^R$) be the set of edges of $\bdy Q_{\op{S}}$ lying to the left (resp.\ right) of $v$ which are disconnected from $\infty$ in $Q_{\op{S}}$ by some quadrilateral $q\in \mcl Q_v$. Then for $k \in \BB N$, 
\eqb \label{eqn-peel-degree-tail}
\BB P\left[ \# E_v^L = k \right] \asymp k^{-5/2} 
\eqe 
with universal implicit constant. The same is true with ``$R$'' in place of ``$L$.''
\end{lem}

The proof of Lemma~\ref{prop-peel-degree} will be a straightforward application of the following peeling process, which is illustrated in Figure~\ref{fig-peeling-1vertex}.

\begin{defn}[One-vertex peeling process] \label{def-1vertex-peeling}
Let $(Q_{\op{S}} , \BB e_{\op{S}})$ be an instance of the UIHPQ$_{\op{S}}$ and let $v\in \mcl V(\bdy Q_{\op{S}})$. The \emph{left one-vertex peeling process} of $Q_{\op{S}}$ at $v$ is the sequence of infinite planar maps $\{Q_{\op{V}^L }^i\}_{i\in [0,\mcl I_{\op{V}^L}]_{\BB Z} }$ and edges $\{e_{\op{V}^L}^i\}_{i\in [1,\mcl I_{\op{V}^L} ]_{\BB Z} }$ defined as follows.  

Let $Q_{\op{V}^L}^0 = Q_{\op{S}}$. Inductively, if $i \in \BB N$ and an infinite quadrangulation $Q_{\op{V}^L}^{i-1}$ with simple boundary has been defined, we define $Q_{\op{V}^L}^i$ as follows. If $v \notin \mcl V(\bdy Q^{i-1})$, we set $Q_{\op{V}^L}^i = Q_{\op{V}^L}^{i-1}$. Otherwise, we let $e_{\op{V}^L}^i $ be the edge of $\bdy Q_{\op{V}^L}^{i-1}$ immediately to the left of $v$ and we set $Q_{\op{V}^L}^i := \op{Peel}(Q_{\op{V}^L}^{i-1} , e_{\op{V}^L}^i)$. 
We define the \emph{terminal time} $\mcl I_{\op{V}^L}  $ to be the smallest $i\in\BB N$ for which $v \notin \mcl E(\bdy Q_{\op{V}^L}^i)$.
The edge $e_{\op{V}^L}^{\mcl I_{\op{V}^L}}$ is chosen in an arbitrary manner.

We define the \emph{right one-vertex peeling process} in the same manner as above but with ``left'' in place of ``right,'' and denote the objects involved by replacing the superscript ``$L$'' with a superscript ``$R$.''  
\end{defn}

\begin{figure}[ht!]
 \begin{center}
\includegraphics[scale=1]{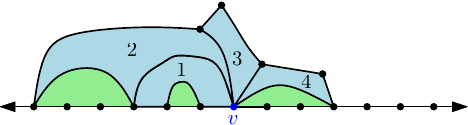} 
\caption[One-vertex peeling process]{An illustration of the left one-vertex peeling process at $v$. The blue quadrilaterals incident to $v$ are enumerated by the order in which they are peeled.  }\label{fig-peeling-1vertex}
\end{center}
\end{figure}

The one-vertex peeling process is also studied in~\cite{richier-perc,caraceni-curien-saw}. 

\begin{lem} \label{prop-peel-degree-pos}
Let $(Q_{\op{S}}, \BB e_{\op{S}})$ be an instance of the UIHPQ$_{\op{S}}$ and let $v$ be one of the endpoints of $\BB e_{\op{S}}$. 
If $\mcl I_{\op{V}^L}$ is the terminal time of the left one-vertex peeling process of $Q_{\op{S}}$ at $v$ as in Definition~\ref{def-1vertex-peeling}, then $\mcl I_{\op{V}^L} $ has a geometric distribution with some parameter $b\in (0,1)$. 
\end{lem}
\begin{proof}
The time $\mcl I_{\op{V}^L} $ is the smallest $i\in\BB N$ for which the peeled quadrilateral $\frk f(Q_{\op{V}^L}^{i-1} , e_{\op{V}^L}^i)$ is incident to an edge of $\bdy Q_{\op{V}^L}^{i-1}$ to the right of $e_{\op{V}^L}^i$. Hence the statement of the lemma follows from~\eqref{eqn-uihpq-peel-prob} and the Markov property of peeling. 
\end{proof}
 
\begin{proof}[Proof of Lemma~\ref{prop-peel-degree}]
We consider the left one-vertex peeling process of $Q_{\op{S}}$ at $v$ as in Definition~\ref{def-1vertex-peeling} and use the notation of that definition.  The final (time-$\mcl I_{\op{V}^L}$) one-vertex peeling cluster disconnects $v$ from $\infty$ in $Q_{\op{S}}$, so must disconnect each edge in $ E_v^R$ from $\infty$.  Since the time-$(\mcl I_{\op{V}^L}-1)$ cluster does not disconnect any edge in $E_v^R$ from $\infty$, it follows that each edge in $ E_v^R$ is disconnected from $\infty$ in $Q_{\op{V}^L}^{\mcl I_{\op{V}^L}-1}$ by the last peeled quadrilateral $\frk f\left(Q_{\op{V}^L}^{\mcl I_{\op{V}^L}-1} , e_{\op{V}^L}^{\mcl I_{\op{V}^L}} \right)$.  Hence $\# E_v^R$ is the same as the number of quadrilaterals of $\bdy Q_{\op{V}^L}^{\mcl I_{\op{V}^L}-1}$ lying to the right of $e_{\op{V}^L}^{\mcl I_{\op{V}^L}}$ which are disconnected from $\infty$ by this peeled quadrilateral.  

If we condition on $\{ \mcl I_{\op{V}^L} = i\}$ for some $i\in\BB N$, then the conditional law of the quadrilateral $\frk f(Q_{\op{V}^L}^{i-1} , e_{\op{V}^L}^{i})$ is the same as its conditional law given that it covers up at least one edge of $\bdy Q_{\op{V}^L}^{i-1}$ to the right of $e_{\op{V}^L}^i$. By~\eqref{eqn-uihpq-peel-prob} the probability of this event is a universal constant, so the estimate~\eqref{eqn-peel-degree-tail} (with ``$R$" in place of ``$L$'') follows by taking an appropriate sum of the probabilities in~\eqref{eqn-uihpq-peel-prob-approx}. The analogous estimate for $E_v^L$ follows by symmetry. 
\end{proof}

\subsection{Peeling all quadrilaterals incident to a boundary arc}
\label{sec-linear-peel}

Let $(Q_{\op{S}} , \BB e_{\op{S}})$ be an instance of the UIHPQ$_{\op{S}}$. Let $A^L$ and $A^R$ be the infinite rays of $\bdy Q_{\op{S}}$ lying to the left and right of $\BB e_{\op{S}}$, respectively, defined in such a way that $\BB e_{\op{S}} \in \mcl E(A^L)\setminus \mcl E(A^R)$, the left endpoint of $\BB e_{\op{S}}$ belongs to~$A^L$, and the right endpoint of $\BB e_{\op{S}}$ belongs to $A^R$. 
 
The goal of this subsection is to estimate the number of edges of $A^R$ which are disconnected from $\infty$ by quadrilaterals incident to $A^L$ if we disregard the ``big" jumps made by the peeling process.  
In particular, we will prove the following lemma.

\begin{lem} \label{prop-general-peel-moment} 
Let $\{(\wh Q^{i-1} , \wh e^i)\}_{i\in [1,\wh{\mcl I}]_{\BB Z} }$ be a peeling process of $Q_{\op{S}}$ such that each edge $\wh e^i$ is incident to some vertex in $A^L$.
 For $i\in\BB N_0$, let
\eqbn
\wh O^i  := \# \left( \mcl E(A^R) \setminus \mcl E\left(\bdy \wh Q^i  \right) \right)   
\eqen
be the number of edges of $A^R$ which have been disconnected from $\infty$ by step $i$. For $n\in \BB N$, let
\eqbn
\wh X(n) := \sum_{i=1}^{\wh{\mcl I}} (\wh O^i - \wh O^{i-1})\wedge n.
\eqen 
 For each $n\in\BB N$ and each $p \geq 1$,  
\eqbn
\BB E\left[ \wh X(n)^p \right] \preceq   n^{p-1/2} 
\eqen
with implicit constant depending only on $p$, \emph{not} on the particular choice of peeling process.
\end{lem}

We will first prove Lemma~\ref{prop-general-peel-moment} for a particular peeling process which is in some sense maximal, which we now define. 
  
\begin{defn}[Linear peeling process]
\label{def-linear-peel}
 The \emph{linear peeling process} of $Q_{\op{S}}$ started from $\BB e_{\op{S}}$ is the sequence of infinite planar maps $\{Q_{\op{L}}^i\}_{i\in \BB N_0}$ and edges $\{e_{\op{L}}^i\}_{i\in\BB N}$ defined as follows. Let $Q_{\op{L}}^0 := Q_{\op{S}}$. Inductively, if $i\in\BB N$ and $Q_{\op{L}}^{i-1}$ has been defined, let $v_{\op{L}}^i$ be the rightmost vertex of $\bdy  Q_{\op{L}}^{i-1}$ which also belongs to $A^L$. 
Let $e_{\op{L}}^i$ be the edge of $\bdy Q_{\op{L}}^{i-1}$ lying immediately to the right of $v_{\op{L}}^i$ and let $Q_{\op{L}}^i := \op{Peel}(Q_{\op{L}}^{i-1} , e_{\op{L}}^i)$.
\end{defn}

\begin{figure}[ht!]
\begin{center}
\includegraphics[scale=1]{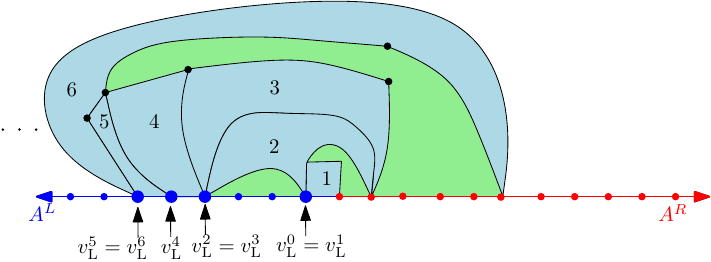} 
\caption{Illustration of the left linear peeling process run for 6 units of time. Quadrilaterals are numbered in the order in which the are peeled. Each vertex $v_{\op{L}}^i$ lies in the intersection of $A^L$ with the boundary of the peeled quadrilateral at time $i-1$. Consequently, if $i = I_m+1$ for some $m$, i.e., the previous peeled quadrilateral intersects $A^R$, then the $Q_{\op{L}}^{I_m }$-graph distance from $e_{\op{L}}^{I_m+1}$ to $A^R$ is at most 2. Here this is the case for $i= 1,2,3,4,7$.
}\label{fig-linear-peel}
\end{center}
\end{figure}

See Figure~\ref{fig-linear-peel} for an illustration of the above definition. We now devote our attention to proving Lemma~\ref{prop-general-peel-moment} in the special case of the left linear peeling process.  
 
\begin{lem}  \label{prop-linear-peel-moment}
Suppose we are in the setting of Definition~\ref{def-linear-peel}.   
For $i\in\BB N$, let
\eqbn
O_{\op{L}}^i  := \# \left( \mcl E(A^R) \setminus \mcl E\left(\bdy Q_{\op{L}}^i  \right) \right)    
\eqen
and for $n\in\BB N $, let
\eqbn
X_{\op{L}}(n)  := \sum_{i=1}^\infty (O_{\op{L}}^i - O_{\op{L}}^{i-1}) \wedge n.
\eqen 
For each $n\in\BB N$ and each $p \geq 1$,  
\eqb \label{eqn-linear-peel-moment}
\BB E\left[ X_{\op{L}}(n)^p \right] \preceq   n^{p-1/2} 
\eqe 
with implicit constant depending only on $p$. 
\end{lem}

For the proof of Lemma~\ref{prop-linear-peel-moment}, we will use the following notation. 
For $i \in\BB N_0$, let 
\eqb  \label{eqn-linear-peel-filtration}
\mcl G_{\op{L}}^i  := \sigma\left( \frk P\left( Q^{j-1}_{\op{L}} , e_{\op{L}}^j \right) \,:\, j \in [1,i]_{\BB Z} \right) 
\eqe 
be the $\sigma$-algebra generated by the first $i$ peeling steps of the left linear peeling process. 

Let $I_0 = 0$ and for $m\in\BB N$, let $I_m$ be the $m$th smallest $i\in\BB N$ for which $O_{\op{L}}^i - O_{\op{L}}^{i-1} \not=0$, or $m =\infty$ if there are fewer than $m$ such times $i$. Observe that each $I_m$ is a stopping time for the filtration~\eqref{eqn-linear-peel-filtration}. Let $M$ be the smallest $m\in\BB N$ for which $I_m =\infty$.   

Let $\{v_y\}_{y\in \BB N_0}$ be the vertices of $A^L$, ordered from right to left. 
For $y \in \BB N_0$, let $E_y$ be the set of edges of $A^R$ which are disconnected from $ \infty$ by some quadrilateral of $Q_{\op{S}}$ which is incident to $v_y$, so that 
\eqb \label{eqn-linear-vertex-union}
\mcl E\left( Q_{\op{L}}^i \cap A^R \right) \subset \bigcup_{y=0}^\infty E_y ,\quad \forall i \in \BB N. 
\eqe 
 
The key observation in the proof of Lemma~\ref{prop-linear-peel-moment} is the following Markov property; see Figure~\ref{fig-linear-peel} for an illustration. For each $i \geq 2$, the vertex $v_{\op{L}}^i$ is incident to the peeled quadrilateral $\frk f(Q_{\op{L}}^{i-2} , e_{\op{L}}^{i-1})$ at the previous step. Hence for each $m\in\BB N$ for which $I_m< \infty$, the $\bdy Q_{\op{L}}^{I_m }$-graph distance from $e_{\op{L}}^{I_m+1}$ to $\mcl E(A^R) \cap \mcl E(\bdy Q_{\op{L}}^{I_m })$ is either $0, 1,$ or $2$ (since there is a path of length at most 2 around $\bdy \frk f(Q_{\op{L}}^{I_m-1} , e_{\op{L}}^{I_m})$ from $e_{\op{L}}^{I_m+1}$ to $\mcl E(A^R) \cap \mcl E(\bdy Q_{\op{L}}^{I_m })$). By the Markov property of the peeling process, we find that the random variables $O_{\op{L}}^{I_m} - O_{\op{L}}^{I_{m-1}+1}$ are almost i.i.d., except that, depending on $\mcl G_{\op{L}}^{I_m}$, we could peel started at distance either 0, 1, or 2 from $\mcl E(A^R) \cap \mcl E(\bdy Q_{\op{L}}^{I_{m-1}})$.
  
\begin{lem} \label{prop-linear-peel-infty}
In the setting described just above, there is a universal constant $b>0$ such that for each $m \in \BB N$, 
\eqb \label{eqn-linear-peel-infty}
  \BB P\left[ M  = m+1 \,|\, \mcl G_{\op{L}}^{I_m} \right] \BB 1_{(M > m)}  \geq b \BB 1_{(M > m)}.
\eqe 
\end{lem}
\begin{proof}
By the Markov property noted just above the statement of the lemma, it suffices to show that
\eqb \label{eqn-infty-pos}
  \BB P\left[ M =1 \right]  > 0 .
\eqe 
By Lemma~\ref{prop-peel-degree}, with $E_y$ defined just above~\eqref{eqn-linear-vertex-union}, for each $y\in\BB N$ and each $k\in\BB N$ we have
\eqb \label{eqn-linear-vertex-tail}
\BB P\left[ \# E_y \geq k \right] \preceq (k +y)^{-3/2}  
\eqe 
with universal implicit constant. 
Taking $k = 1$ and summing over all $y\geq y_0$, we see that there exists some $y_0 \in \BB N$ such that 
\eqb \label{eqn-late-infty}
\BB P\left[\text{$ \frk f(Q_{\op{S}} , \{v_y , v_{y+1}\})$ does not share a vertex with $A^R$, $\forall y \geq y_0$} \right] \geq \frac12 ,
\eqe 
where here $\{v_y , v_{y+1}\}$ is the edge from $v_y$ to $v_{y+1}$.
Furthermore, by~\eqref{eqn-uihpq-peel-prob} and the Markov property of peeling there exists $k_0 \in \BB N$ such that with positive probability $\bdy Q_{\op{L}}^k$ contains no edges of $A^R$ and the edge $e_{\op{L}}^{k_0}$ lies at $\bdy Q_{\op{L}}^k$-graph distance at least $y_0$ from $A^R$. By~\eqref{eqn-late-infty} and another application of the Markov property of peeling, we obtain~\eqref{eqn-infty-pos}. 
\end{proof}
 
\begin{proof}[Proof of Lemma~\ref{prop-linear-peel-moment}]  
Fix $p \geq 1$. 
We first prove a $p$th moment bound for $O^{I_1} \wedge n  $.
If $O^{I_1} \geq k$ for some $k\in\BB N$, then by~\eqref{eqn-linear-vertex-union} there exists $y \in\BB N$ such that $\# E_y \geq k$. 
By~\eqref{eqn-linear-vertex-tail}, 
\alb
\BB P \left[ O_{\op{L}}^{I_1} \geq k   \right] 
\preceq  \sum_{y=0}^\infty ( k+y)^{-3/2} \preceq k^{-1/2} .
\ale 
Therefore, for $n\in\BB N$,  
\eqbn
\BB E\left[ (O_{\op{L}}^{I_1} \wedge n )^p  \right] \preceq \sum_{k=1}^n k^{p-1} \BB P\left[ O_{\op{L}}^{I_1} \geq k \right] \preceq \sum_{k=1}^n k^{p-3/2} \preceq n^{p-1/2}   .
\eqen
By the Markov property described just above Lemma~\ref{prop-linear-peel-infty} (and a trivial modification to the above argument to treat the case when we start at distance $1$ or $2$, rather than $0$, from $A^R$) we also have 
\eqb \label{eqn-linear-peel-step}
\BB E\left[  \left( ( O_{\op{L}}^{I_m} - O_{\op{L}}^{I_{m-1}} ) \wedge n \right)^p   \,|\, \mcl G_{\op{L}}^{I_{m-1}} \right] \preceq n^{p-1/2} , \quad \forall m \in \BB N.  
\eqe
 
By Lemma~\ref{prop-linear-peel-infty}, for each $m \in \BB N$ the conditional law of $M$ given $\mcl G_{\op{L} }^{I_m}$ is stochastically dominated by $m$ plus a geometric random variable with parameter $b$ (where $b$ is as in the statement of Lemma~\ref{prop-linear-peel-infty}). In particular,
\eqbn
\BB E\left[ M^{p-1} \,|\, \mcl G_{\op{L} }^{I_m} \right] \preceq \sum_{t=1}^\infty (m + t)^{p-1} (1-b)^t \preceq m^{p-1}  .
\eqen 
Since $O^{I_m} \in \mcl G_{\op{L}}^{I_m}$ and $\{M \geq m\} = \{M \leq m-1\}^c \in \mcl G_{\op{L}}^{I_{m-1}}$, we infer from this and~\eqref{eqn-linear-peel-step} that
\begin{align}
\BB E\left[ X_{\op{L} }(n)^p \right] 
&= \BB E\left[ \left( \sum_{m=1}^M   (O_{\op{L}}^{I_m} - O_{\op{L}}^{I_{m-1}}) \wedge n  \right)^p       \right]  \notag\\
&\leq  \sum_{m=1}^\infty \BB E\left[  \left( ( O_{\op{L}}^{I_m} - O_{\op{L}}^{I_{m-1}} ) \wedge n \right)^p  M^{p-1}  \BB 1_{M\geq m}  \right] \quad \text{(by Jensen's inequality)} \notag\\
&= \sum_{m=1}^\infty \BB E\left[ \BB E\left[ \left( ( O_{\op{L}}^{I_m} - O_{\op{L}}^{I_{m-1}} ) \wedge n \right)^p \BB E\left[  M^{p-1}  \,|\, \mcl G_{\op{L} }^{I_m}  \right]  \,|\, \mcl G_{\op{L}}^{I_{m-1}} \right]  \BB 1_{M\geq m}  \right]  \notag\\
&\preceq n^{p-1/2} \sum_{m=1}^\infty m^{p-1} (1-b)^m \preceq n^{p-1/2}  \notag
\end{align} 
which is~\eqref{eqn-linear-peel-moment}.
\end{proof} 

We will now extend Lemma~\ref{prop-linear-peel-moment} to get our desired estimate for general peeling processes. 
 
\begin{proof}[Proof of Lemma~\ref{prop-general-peel-moment}]
Let us first observe that the left linear peeling process must peel (not just disconnect from $\infty$) every quadrilateral of $Q_{\op{S}}$ which is incident to both $A^L$ and $A^R$.
Indeed, suppose that $q$ is such a quadrilateral, let $v$ be the rightmost vertex of $A^L$ lying on the boundary of $q$, and let $e$ be the edge of $q$ which is immediately to the right of $v$ when we stand on $\bdy A^L$ and look inward.  
Then $q$ disconnects from $\infty$ in $Q_{\op{S}}$ each quadrilateral $q'$ of $Q_{\op{S}}$ such that either (a) $q'$ is incident to a vertex of $A^L$ which lies to the right of $v$ or (b) $q'$ is incident to an edge which has $v$ as an endpoint and which lies to the right of $e$.
Hence the union of all such quadrilaterals $q'$ does not disconnect $q$ from $\infty$. 
On the other hand, the definition of the linear peeling process shows that each quadrilateral peeled by this process before it either peels $q$ or disconnects $q$ from $\infty$ must satisfy either (a) or (b) above. Since the linear peeling process eventually peels or disconnects every quadrilateral incident to $A^L$, it follows that this process must peel $q$.  
 
Now consider $i\in [1,\wh{\mcl I}]_{\BB Z}$ and let $q := \frk f(\wh Q^{i-1} , \wh e^i)$. 
If $\wh O^i - \wh O^{i-1} \not=0$, then $q$ must be incident to both $A^L$ and $A^R$. 
Consequently, the preceding paragraph shows that exists $j\in\BB N$ such that $q$ is equal to the $j$th peeled quadrilateral $\frk f(Q_{\op{L}}^{j-1} , e_{\op{L}}^j)$ in the left linear peeling process. 

Continuing to assume that $\wh O^i - \wh O^{i-1} \not=0$, we now argue that, in the notation of Lemma~\ref{prop-linear-peel-moment}, 
\eqb \label{eqn-peel-moment-contain}
O_{\op{L}}^j - O_{\op{L}}^{j-1} \leq \wh O^i - \wh O^{i-1} .
\eqe
We first observe that every quadrilateral $q'$ of $Q_{\op{S}}\setminus \wh Q^{j-1}$ which intersects both $A^L$ and $A^R$ must be disconnected from $\infty$ in $Q_{\op{S}}$ by $q$, since otherwise $q'$ would disconnect $q$ from $\infty$, contrary to the fact that $q\in \mcl F(\wh Q^{j-1})$. 
Every quadrilateral of $Q_{\op{S}}$ which is incident to both $A^L$ and $A^R$ and which is disconnected from $\infty$ by $q$ is peeled by the linear peeling process at or before time $j$.  
It therefore follows that $\mcl E(\wh Q^{j-1})\setminus \mcl E(A^R) \subset  \mcl E(Q_{\op{L}}^{j-1})\setminus \mcl E( A^R)$, i.e., $O_{\op{L}}^{j-1} \geq \wh O^{i-1}$.
Furthermore, $\mcl E(\wh Q^{i})\setminus \mcl E(A^R) $ and $\mcl E(Q_{\op{L}}^{j})\setminus \mcl E( A^R)$ are each equal to the set of edges of $A^R$ which are disconnected from $\infty$ by $q$, so $O_{\op{L}}^j = \wh O^i$. Thus~\eqref{eqn-peel-moment-contain} holds.

Since every quadrilateral which is peeled or disconnected from $\infty$ by our given peeling process is also peeled or disconnected from $\infty$ by the left linear peeling process, we have
\eqbn
 \sum_{i=1}^{\wh{\mcl I}} (\wh O^i - \wh O^{i-1 })  \leq \sum_{j=1}^{\infty } (O_{\op{L}}^{j} - O_{\op{L}}^{j-1})  .
\eqen 
By~\eqref{eqn-peel-moment-contain}, every non-zero term in the sum on the left is greater than or equal to a unique corresponding non-zero term in the sum on the right. Hence the inequality continues to hold if we truncate each of the terms in each of the sums at level $n$. That is, $\wh X(n) \leq X_{\op{L}}(n)$ so the statement of the lemma follows from Lemma~\ref{prop-linear-peel-moment}.  
\end{proof}

\section{Peeling the glued map}
\label{sec-peeling-glued}

In this section we will introduce a two-sided peeling process for a pair of UIHPQ$_{\op{S}}$'s glued together along their boundaries, which we call the \emph{glued peeling process} and which will be an important tool in the proofs of our main theorems.  The main reason for our interest in this peeling process is that it satisfies a simple Markov property (Lemma~\ref{prop-peel-law}) and provides an upper bound for metric balls in the glued map (Lemma~\ref{prop-peel-ball}). We will also prove in Section~\ref{sec-peel-jump} some basic estimates for how many edges of the boundary of our original pair of UIHPQ$_{\op{S}}$'s are swallowed by this peeling process. These bounds will later be used to deduce moment estimates in Section~\ref{sec-peeling-moment}. 

\subsection{Glued peeling process}
\label{sec-glued-peeling}
 
Let $(Q_- , \BB e_-)$ and $(Q_+ , \BB e_+)$ be two independent samples of the UIHPQ$_{\op{S}}$.  Let $\lambda_- : \BB Z \rta \mcl E(\bdy Q_-)$ (resp.\ $\lambda_+ : \BB Z \rta \mcl E(\bdy Q_+)$) be the boundary path for $Q_-$ (resp.\ $Q_+$) started from $\BB e_-$ (resp.\ $\BB e_+$) and traveling to the right.

Fix \emph{gluing times} $\ul{\BB x} , \BB x_- , \BB x_+ \in \BB N$ with $\ul{\BB x} \leq \BB x_- \wedge \BB x_+$ and let $Q_{\op{zip}}$ be the planar map obtained from $Q_-$ and $Q_+$ by identifying $\lambda_-(x)$ with $\lambda_+(x)$ for each $x\in [0,\ul{\BB x} ]_{\BB Z}$ and $\lambda_-(\BB x_- + y)$ with $\lambda_+(\BB x_+ + y)$ for each $y \in \BB N_0$. See Figure~\ref{fig-glued-peel}, left, for an illustration. Taking $\ul{\BB x}  = \BB x_- = \BB x_+  $ corresponds to gluing $Q_\pm$ together along their positive boundaries, which is the setting of Theorem~\ref{thm-saw-conv-wedge} and the main case we are interested in. Other choices of $\ul{\BB x}$ and $\BB x_\pm$ result in a ``hole" in $Q_{\op{zip}}$ with left/right boundary lengths $\BB x_- - \ul{\BB x}$ and $\BB x_+ - \ul{\BB x}$. We need to consider the case when there is such a hole due to the Markov property of our peeling process (Lemma~\ref{prop-peel-law} below). 

We slightly abuse notation by identifying $Q_-$ and $Q_+$ with the corresponding subsets of $Q_{\op{zip}}$, so we write $\lambda_-(\ul{\BB x}) = \lambda_+(\ul{\BB x})$, etc. 
 
Choose a finite, non-empty, connected initial edge set $\BB A \subset \bdy Q_- \cup \bdy Q_+$ (which is where we will start our peeling process). Note that since we are identifying $Q_-$ and $Q_+$ with the corresponding subsets of $Q_{\op{zip}}$, the set $\BB A$ can be connected even if it intersects both $\bdy Q_-$ and $\bdy Q_+$ (which is the main case we are interested in). In the case when either $\BB x_+$ or $\BB x_-$ is not equal to $\ul{\BB x}$, we require that 
\eqb \label{eqn-initial-edge-condition}
\lambda_-\left( [\ul{\BB x} , \BB x_- ] \right) \cup \lambda_+\left( [\ul{\BB x} , \BB x_+ ] \right) \subset \BB A 
\eqe 
so that $\BB A$ contains every edge along the boundary of the hole in $Q_{\op{zip}}$. 
  
We will define a joint peeling process for $Q_-$ and $Q_+$ (i.e., edges of both $Q_-$ and $Q_+$ will be peeled), called the \emph{glued peeling process} started from $\BB A$, whose clusters at certain special times $J_r$ contain the radius-$r$ graph metric ball centered at $\BB A$ in $Q_{\op{zip}}$. See Figure~\ref{fig-glued-peel}, right, for an illustration. 
The glued peeling process will be described by a sequence of finite planar maps $\{ \dot Q^j  \}_{j\in \BB N_0}$ contained in $Q_{\op{zip}}$, a sequence of infinite quadrangulations with boundary $\{Q^j_\pm\}_{j\in\BB N_0}$ contained in $Q_\pm$ which intersect $\dot Q^j$ only along their boundaries with the property that $Q_{\op{zip}} = \dot Q^j \cup Q_-^j \cup Q_+^j$ for each $j\in\BB N$, and an increasing sequence of non-negative integer stopping times $\{J_r\}_{r\in \BB N_0}$.   
We define $\bdy \dot Q^j =  \dot Q^j \cap (\bdy Q_-^j \cup \bdy Q_+^j)$. Note that in the case when the map $Q_{\op{zip}}$ has a hole, the outer boundary of this hole need not be part of $\bdy \dot Q^j$.

\begin{remark} \label{remark-peeling-by-layers}
The glued peeling process described just below is similar to the so-called peeling by layers algorithm for infinite planar quadrangulations or triangulations without boundary which is studied in~\cite{curien-legall-peeling}. However, unlike the clusters produced by the peeling by layers algorithm, our glued peeling clusters do not closely approximate filled metric balls (instead they are just larger than metric balls) since we peel edges which are disconnected from $\infty$ on one side of the gluing interface but not the other.  Furthermore, the glued peeling process is equivalent to the peeling process introduced and studied independently of the present work in~\cite[Section~2]{caraceni-curien-saw}, but the estimates proven for this process in the present paper are stronger than those in~\cite{caraceni-curien-saw}.
\end{remark}

\begin{figure}[ht!]
 \begin{center}
\includegraphics[scale=1]{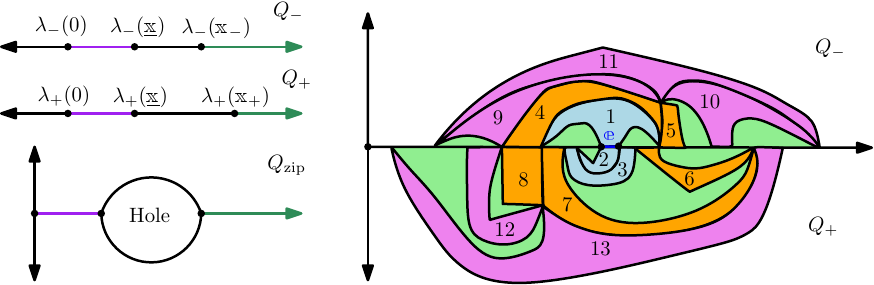} 
\caption[Glued peeling process]{\textbf{Left:}  Illustration of the gluing procedure for $Q_-$ and $Q_+$ which produces a glued quadrangulation $Q_{\op{zip}}$ with possibly a single hole. The purple (resp.\ green) boundary arcs of $Q_-$ and $Q_+$ are identified, but the black arcs are not identified. 
\textbf{Right:} 
Illustration of the glued peeling process run for several peeling steps in the case when $\ul{\BB x } = \BB x_- = \BB x_+$ (so there is no hole) and $\BB A$ is a single edge $\BB e$. Quadrilaterals are numbered by the order in which they are peeled. Quadrilaterals peeled during the first (resp.\ second, third) layer are colored blue (resp.\ orange, purple). Disconnected regions are colored light green. Here $J_1 = 3$, $J_2 = 8$, and $J_3 = 13$. The map $\dot Q^{13}$ is the union of the 13 colored quadrilaterals and the light green regions which they disconnect from $\infty$ in either $Q_-$ or $Q_+$. The unexplored quadrangulations $Q_-^{13}$ and $Q_+^{13}$ are glued together in the manner described in this section with a ``hole" which is filled in by $\dot Q^{13}$. }\label{fig-glued-peel}
\end{center}
\end{figure}
 
To start the definition, let $\dot Q^0$ be the smallest subgraph of $\bdy Q_-\cup\bdy Q_+$ containing $\BB A$. We can view $\dot Q^0$ as a planar map with at most two faces, one of which is unbounded and the other of which is the ``hole" surrounded by $\BB A$, in case such a hole exists. Let $Q^0_\pm = Q_\pm$. Also let $J_0 = 0$. 
 
Inductively, suppose $j \in \BB N$, $\dot Q^i$, $Q_-^i$, and $Q_+^i$ have been defined for $i\leq j-1$, and $J_r$ for $r\in \BB N_0$ has been defined on the event $\{J_r \leq j-1\}$. 
Let $r_{j-1}$ be the largest $r\in \BB N_0$ for which $J_r \leq j-1$, and suppose that $\bdy \dot Q^{J_{r_{j-1}}}$ shares a vertex with $\bdy \dot Q^{j-1}$ 

Let $\dot e^j$ be an edge in $\mcl E(\bdy\dot Q^{j-1} )$ (which we recall is contained in $\mcl E(\bdy Q^{j-1}_- \cup \bdy Q_+^{j-1}) $) which has at least one endpoint in $\mcl V(\dot Q^{J_{r_{j-1}}})$, chosen in a manner which depends only on $\bdy \dot Q^{j-1}$ and $\mcl V(\dot Q^{J_{r_{j-1}}})$ (the precise manner in which the edge is chosen is not important for our purposes).  Such an edge exists by our inductive hypothesis.  If $\dot e^j \in \bdy Q_-^{j-1}$ we set $\xi^j = -$ and otherwise (in which case $\dot e^j \in \bdy Q_+^{j-1}$) we set $\xi^j = +$.  

Recalling the notation of Section~\ref{sec-general-peeling}, we peel $ Q^{j-1}_{\xi^j}$ at $\dot e^j$ to obtain the quadrilateral $\frk f( Q^{j-1}_{\xi^j} , \dot e^j)$ and the planar map $\frk F( Q^{j-1}_{\xi^j} , \dot e^j)$ which it disconnects from $ \infty$ in $Q^{j-1}_{\xi^j}$. We let 
\alb
& \qquad \dot Q^j := \dot Q^{j-1} \cup  \frk f( Q^{j-1}_{\xi^j} , \dot e^j) \cup \frk F ( Q^{j-1}_{\xi^j} , \dot e^j)  , \\
&Q_{\xi^j}^j := \op{Peel} \left( Q^{j-1}_{\xi^j} , \dot e^j  \right) ,\quad \op{and}\quad Q_{-\xi^j}^j := Q_{-\xi^j}^{j-1} .
\ale
By induction $Q_\pm^j$ are infinite quadrangulations with boundary, $\dot Q^j$ is a finite quadrangulation with boundary (possibly with a single hole corresponding to the hole in $Q_{\op{zip}}$) and $Q_{\op{zip}} = Q_-^j \cup Q_+^j \cup \dot Q^j$.
If $\bdy \dot Q^j$ shares a vertex with $\bdy \dot Q^{J_{r_{j-1}}}$, we declare that $J_{r_{j-1}+1} > j$, and otherwise we declare that $ J_{r_{j-1}+1} =  j$.  These definitions imply that $\bdy \dot Q^{J_{r_{j }}}$ shares a vertex with $\bdy \dot Q^j$, which completes the induction.

Define the filtration
\eqb \label{eqn-peel-filtration}
\mcl F^j := \sigma\left( \dot Q^i,\, \frk P(Q_{\xi^i}^{i-1} , e^i)  \,:\, i \in [1,j]_{\BB Z} \right) ,\quad \forall j\in\BB N_0 ,
\eqe 
where here $\frk P(\cdot,\cdot)$ is the peeling indicator variable from Section~\ref{sec-general-peeling}
Note that $\dot Q^j$ and $\dot e^{j+1}$ are $\mcl F^j$-measurable for $j\in\BB N_0$, and $J_r$ for $r\in\BB N_0$ is a stopping time for $\{\mcl F^j\}_{j\in\BB N_0}$.

The glued peeling process satisfies a Markov property, described as follows. 

\begin{lem} \label{prop-peel-law}
With the above definitions, the following is true for each $\mcl F^j$-stopping time $\iota$. The quadrangulations $Q_-^\iota$ and $Q_+^\iota$ are conditionally independent given $\mcl F^\iota$, and the conditional law of each is that of a UIHPQ$_{\op{S}}$.  
Furthermore, if $\iota = J_r$ for some $r\in \BB N_0$, then there exists $\mcl F^{J_r}$-measurable $\ul{\BB x}^{J_r} , \BB x_-^{J_r} , \BB x_+^{J_r} \in \BB N_0$ with $\ul{\BB x}^{J_r} \leq  \BB x_-^{J_r}  \wedge  \BB x_+^{J_r}$ such that $Q_\pm^{J_r}$ are glued together in the manner described at the beginning of this subsection with this choice of gluing times and $\{\dot Q^{j + J_r}\}_{j\in\BB N_0}$ is the set of clusters of the glued peeling process of $Q_-^{J_r} \cup Q_+^{J_r}$ started from $\BB A = \mcl E(\bdy \dot Q^{J_r})$.  
\end{lem}

The second statement of Lemma~\ref{prop-peel-law} is the main reason why we allow general choices of $\ul{\BB x}$, $\BB x_-$, and $\BB x_+$ in the above construction---cutting out the cluster $\dot Q^{J_r}$ produces a hole in $Q_{\op{zip}}$. 

\begin{proof}[Proof of Lemma~\ref{prop-peel-law}]
This is immediate from the above inductive construction of the glued peeling process and the Markov property of peeling (recall Section~\ref{sec-uihpq-peeling}). 
\end{proof}

The following lemma is the main reason for our interest in the planar maps $\dot Q^j$. 

\begin{lem} \label{prop-peel-ball}
For each $r\in\BB N_0$,
\eqb \label{eqn-peel-ball}
B_r\left( \BB A ; Q_{\op{zip}} \right) \subset  \dot Q^{J_r}  .
\eqe 
\end{lem}
\begin{proof} 
It suffices to show inclusion of the vertex sets of the graphs in~\eqref{eqn-peel-ball}, since an edge in either of these graphs is the same as an edge of $Q_{\op{zip}}$ whose endpoints are both in the vertex set of the graph.
We proceed by induction on $r$. The base case $r = 0$ (in which case $J_r = 0$) is true by definition. Now suppose $r\in\BB N$ and~\eqref{eqn-peel-ball} holds with $r-1$ in place of $r$. If we are given a vertex $v$ of $B_r\left( \BB A ; Q_{\op{zip}} \right) \setminus \mcl V(\dot Q^{J_{r-1}})$, then there is a $w\in B_{r-1}\left( \BB A ; Q_{\op{zip}} \right)$ with $\op{dist}\left(w , \BB A ; Q_{\op{zip}} \right) = r-1$. By the inductive hypothesis, $w$ belongs to $\mcl V(\bdy \dot Q^{J_{r-1}})$. By definition of $J_r$, we have $w\notin \mcl V(\bdy \dot Q^{J_r})$ so we must have $v\in \mcl V(\dot Q^{J_r})$. 
\end{proof}

It is not clear a priori that $\dot Q^{J_r}$ is typically contained in a graph metric ball of radius comparable to $r$. Indeed, at each step of the glued peeling process we are allowed to peel at an edge which is only disconnected from $\infty$ on one side of the gluing interface, so $\dot Q^{J_r}$ can potentially be \emph{much} larger than $B_r(\BB A ; Q_{\op{zip}})$. However, we have the following lemma which gives an upper bound for the size of $\dot Q^{J_r}$ in terms of the size of its intersection with $\bdy Q_- \cup \bdy Q_+$.

\begin{lem} \label{prop-peel-ball-upper}
For each $r\in\BB N_0$,
\eqb \label{eqn-peel-ball-upper}
  \bdy \dot Q^{J_r} \cap Q_-   \subset  B_{2r} \left( \dot Q^{J_r} \cap \bdy Q_- ; \dot Q^{J_r} \cap Q_- \right) ,
\eqe 
and similarly with ``$+$" in place of ``$-$." 
\end{lem}
\begin{proof} 
As in the proof of Lemma~\ref{prop-peel-ball-upper}, it suffices to show an inclusion of vertex sets.
We proceed by induction on $r$, noting that the base case $r=0$ is trivial. Suppose $r\in\BB N$ and~\eqref{eqn-peel-ball-upper} holds with $r-1$ in place of $r$. Let $v\in \mcl V(\bdy \dot Q^{J_r} \cap Q_-)$. If $v\in \mcl V(\bdy Q_-)$, then $v\in \mcl V\left( B_0 \left( \dot Q^{J_r} \cap \bdy Q_- ; \dot Q^{J_r} \cap Q_- \right) \right)$, so we can assume that $v\notin \mcl V(\bdy Q_-)$. 

Every vertex of $\dot Q^{J_r} \cap (Q_-\setminus \bdy Q_-)$ which does not belong to one of the peeled quadrilaterals $ \frk f( Q^{j-1}_- , \dot e^j)$ for $j \in [J_{r-1}+1, J_r]_{\BB Z}$ with $\xi^j = -$ is disconnected from $\infty$ in $Q_-$ by some such quadrilateral, so does not belong to $\bdy \dot Q^{J_r}$. Therefore $v$ must be a vertex of one of these peeled quadrilaterals. By definition of $J_r$, this quadrilateral has a vertex in $\bdy \dot Q^{J_{r-1}}$. 
Hence $v$ lies at $Q_-$-graph distance at most 2 from $\bdy \dot Q^{J_{r-1}}$, so by the inductive hypothesis $v\in \mcl V\left( B_{2r} \left( \dot Q^{J_r} \cap \bdy Q_- ; \dot Q^{J_r} \cap Q_- \right) \right)$. 
\end{proof}

\subsection{Bounds for the size of jumps}
\label{sec-peel-jump} 

Suppose we are in the setting of Section~\ref{sec-glued-peeling}. In light of Lemma~\ref{prop-peel-ball-upper}, in order to prove an upper bound for the size of $\dot Q^{J_r}$ we need estimates for the number of edges of $\bdy Q_- \cup \bdy Q_+$ (which includes the gluing interface) contained in the glued peeling clusters $\dot Q^j$.  To this end, for $j\in \BB N_0$, let  
\eqb \label{eqn-saw-edge-count'} 
\wh Y^j := \#\mcl E\left( \dot Q^j  \cap ( \bdy Q_-  \cup \bdy Q_+) \right)    
\eqe  
so that $\wh Y^0 = \#\BB A$.
For $n\in\BB N$, also define
\eqb \label{eqn-jump-truncated}
\wh Y^j_n := \sum_{i=1}^j \left( \wh Y^i - \wh Y^{i-1} \right)\wedge n 
\eqe
so that $\wh Y^j_n$ is the sum of the upward jumps made by $\wh Y $ before time $j$ truncated at level $n$.

The goal of this subsection is to prove an upper bound for $\wh Y_n^{J_r}$ (which implies an upper bound for the total length of the small jumps made by $\wh Y$ before time $J_r$) and an upper bound for the number of big jumps made by $\wh Y$ before time $J_r$.  These bounds will be used in Section~\ref{sec-peeling-moment} to prove various moment bounds for the glued peeling procedure. 
We first state our bounds and give a rough idea of how they are proven, then give the details.

\begin{lem}[Upper bound for total length of small jumps] \label{prop-small-bubble-moment} 
In the notation of~\eqref{eqn-jump-truncated}, for each $r , n   \in \BB N$ and each $p \geq 1$, we have
\eqb  \label{eqn-small-bubble-moment}
\BB E\left[ \left(\wh Y_n^{J_r}  \right)^p \right] \preceq    (r^2 \vee n)^p
\eqe 
with implicit constant depending only on $p$. 
\end{lem}

Lemma~\ref{prop-small-bubble-moment} is a consequence of Lemma~\ref{prop-general-peel-moment}. We consider for $r\in\BB N$ the one-sided peeling processes obtained by restricting the glued peeling process to each of the UIHPQ$_{\op{S}}$'s $Q_-^{J_s}$ and $Q_+^{J_s}$, restricted to the time interval $[J_s+1,J_{s+1}]_{\BB Z}$. 
Lemma~\ref{prop-general-peel-moment} gives us a bound for the sum of the truncated jump sizes for each of these peeling processes, which we then sum over all $s\leq r-1$ to get Lemma~\ref{prop-small-bubble-moment}. 

The other main result of this subsection bounds the number of $j\in[1,J_r]_{\BB Z}$ for which $\wh Y^j - \wh Y^{j-1}$ is unusually large.  
 
\begin{lem}[Upper bound for the number of big jumps] \label{prop-big-jump}
For $r >0$ and $n\in\BB N$, let $K_r(n)$ be the number of $j\in [1,J_r  ]_{\BB Z}$ for which $\wh Y^{j} - \wh Y^{j-1} \geq n$. There is a universal constant $a >0$ such that for each $k\in\BB N$,  
\eqbn
\BB P\left[ K_r(n) > k \right] \leq (a  n^{-1/2} r)^k  .
\eqen 
\end{lem}

The proof of Lemma~\ref{prop-big-jump} is based on the fact that the times of the big jumps (i.e., the values of $j$ for which $\wh Y^j - \wh Y^{j-1} \geq n$) are stopping times for the filtration $\{\mcl F_j\}_{j\in\BB N}$ of~\eqref{eqn-peel-filtration}. Using basic peeling estimates, we will show in Lemma~\ref{prop-layer-cover} that if we condition on everything which happens before a stopping time $\iota$ for $\{\mcl F_j\}_{j\in\BB N}$, then with conditional probability at least $1-O_n(n^{-1/2})$ we do \emph{not} see any more big jumps before reaching the next time of the form $J_r$ which comes after $\iota$. Applying this lemma with $\iota$ equal to either the time of the $k$th big jump or one of the times $J_r$ and combining the resulting estimates in the appropriate manner will yield Lemma~\ref{prop-big-jump}.

We now prove our upper bound for the total length of the small jumps. 
 
\begin{proof}[Proof of Lemma~\ref{prop-small-bubble-moment}]
By H\"older's inequality it suffices to prove~\eqref{eqn-small-bubble-moment} for $p\in\BB N$.
For $r\in \BB N_0$, let $A_{r,\pm}^L$ (resp.\ $A_{r,\pm}^R$) be the arc of $\bdy Q_\pm^{J_r} \cap \bdy Q_\pm$ lying to the left (resp.\ right) of $\bdy \dot Q^{J_r}$. Then $Q_\pm^{J_{r+1}}$ is obtained from $Q_\pm^{J_r}$ by removing some of the quadrilaterals of $Q_\pm^{J_r}$ which are incident to vertices of $\bdy Q_\pm^{J_r} \setminus  A_{r,\pm}^L  $ together with the vertices and edges which they disconnect from $\infty$ in $Q_\pm^{J_r}$.  
It therefore follows from Lemma~\ref{prop-general-peel-moment} and Lemma~\ref{prop-peel-law} (together with left/right symmetry) that for each $p \in \BB N$,  
\eqb \label{eqn-one-layer-moment}
\BB E\left[ (\wh Y_n^{J_{r+1}  } - \wh Y_n^{J_r })^p \,|\, \mcl F^{J_{r-1}} \right] \preceq n^{p-1/2}
\eqe 
with implicit constants depending only on $p$.  

Now let $r\in \BB N$ and for $s\in [1,r]_{\BB Z}$, let $X_s := \wh Y_n^{J_s} - \wh Y_n^{J_{s-1}}$. Then for $p\in\BB N$ we have
\begin{align} \label{eqn-small-bubble-moment-sum}
\BB E\left[ (\wh Y_n^{J_r} )^p \right] 
&= \BB E\left[  \left( \sum_{s=1}^r X_s \right)^p       \right] 
\preceq   \sum_{(s_1, \dots , s_p) \in [1, r]_{\BB Z}^p} \BB E\left[  X_{s_1}  \cdots X_{s_p}  \right].
\end{align}
For $q \in [1,p]_{\BB Z}$, let $S_q$ be the set of $p$-tuples $(s_1, \dots , s_p) \in [1, r]_{\BB Z}^p$ with exactly $q$ distinct indices. By~\eqref{eqn-one-layer-moment}, for $(s_1, \dots , s_p) \in S_q$ the corresponding term in the sum on the right side of~\eqref{eqn-small-bubble-moment-sum} is bounded above by $n^{p-q/2}$. We have $\# S_q \preceq r^q$ (implicit constant depending on $p$) since we need to choose $q$ of the $ r$ possible indices. Therefore,~\eqref{eqn-small-bubble-moment-sum} is bounded above by a $p$-dependent constant times
\eqbn
 \sum_{q=1}^p n^{p-q/2} r^q   \preceq (r^2 \vee n)^p \qedhere
\eqen 
\end{proof}

For the proof of Lemma~\ref{prop-big-jump} we will need the following notation. 

\begin{defn} \label{def-min-vertex-set}
For $j \in \BB N_0$, let $\rho(j)$ be the smallest $r\in \BB N$ for which $J_r \geq j+1$.  
For a vertex $v\in \mcl V\left(\bdy \dot Q^j  \cap \bdy Q_\pm^j\right)$, let $\ell_{v,\pm}^j$ be graph distance from $v$ to $  \bdy Q_\pm^j \setminus \bdy \dot Q^j $ in $  \bdy Q_\pm^j $, i.e., the number of edges in the boundary arc of  $\bdy Q_\pm^j$  between $v$ and $\bdy Q_\pm^j \setminus \bdy \dot Q^j$.
\end{defn}

Note that for $r\in \BB N_0$, we have $\rho(J_r) = r+1$.

\begin{lem} \label{prop-layer-cover}
Let $\iota$ be a stopping time for the filtration $\{\mcl F^j\}_{j\in\BB N_0}$ of~\eqref{eqn-peel-filtration}. Then for $n\in\BB N$,  
\eqbn 
\BB P\left[  \text{$\exists j \in [\iota+1 , J_{\rho(\iota)}  ]_{\BB Z}$ with $\wh Y^j - \wh Y^{j-1}  \geq n$} \,|\, \mcl F^\iota \right]  
\preceq n^{-1/2} 
\eqen
with universal implicit constant.
\end{lem}
\begin{proof} 
Define $\ell_{v,\pm}^\iota$ for $ v\in \mcl V\left(\bdy \dot Q^\iota  \cap \bdy Q_\pm^\iota \right)$ as in Definition~\ref{def-min-vertex-set}. 
By the construction in Section~\ref{sec-glued-peeling}, the time $J_{\rho(\iota)}$ is the smallest $j \geq \iota+1$ for which no element of $\mcl V\left(\bdy \dot Q^\iota \cap \bdy \dot Q^{J_{\rho(\iota)-1}} \right)$ belongs to $\bdy \dot Q^{j }$. 
Furthermore, every edge $\dot e^i$ for $i\in [\iota+1, J_{\rho(\iota)}]_{\BB Z}$ is incident to some vertex in $\mcl V\left(\bdy \dot Q^\iota \cap \bdy \dot Q^{J_{\rho(\iota)-1}} \right)$. 
Hence if there is a $j\in [\iota+1 , J_{\rho(\iota)}]_{\BB Z}$ for which $\wh Y^j - \wh Y^{j-1} \geq n$, then either there is a $v \in  \mcl V\left(\bdy \dot Q^\iota  \cap \bdy Q_-^\iota \right)  $ and a quadrilateral of $Q_-^\iota$ incident to $v$ which disconnects at least $\ell_{v,-}^\iota + n$ edges of $\bdy Q_-^\iota$ from $\infty$ in $\bdy Q_-^\iota$; or the same holds with ``$+$" in place of ``$-$." 
By Lemmas~\ref{prop-peel-degree} and~\ref{prop-peel-law}, if $v\in \mcl V(\bdy \dot Q^\iota \cap \bdy Q_\pm^\iota$ is chosen in an $\mcl F^\iota$-measurable manner, then the conditional probability given $\mcl F^\iota$ that such a quadrilateral exists is at most a universal constant times $(\ell_{v,-}^\iota + n)^{-3/2}$

For each $k\in\BB N$, there are at most 2 vertices $v \in \mcl V\left(\bdy \dot Q^\iota   \cap \bdy Q_\pm^\iota \right)$ with $\ell_{v,-}^\iota = k$; and the same holds with ``$+$" in place of ``$-$." 
Consequently, we can sum the estimate at the end of the preceding paragraph over all $v\in \mcl V(\bdy\dot Q^j)$ to get that
\alb
\BB P\left[  \text{$\exists j \in [\iota+1 , J_{\rho(\iota)}  ]_{\BB Z}$ with $\wh Y^j - \wh Y^{j-1}  \geq n$} \,|\, \mcl F^\iota \right]   
&\preceq  \sum_{\xi\in \{-,+\}} \sum_{v\in  \mcl V\left(\bdy \dot Q^j  \cap \bdy Q_\xi^\iota \right)  }  (\ell_{v,\xi}^\iota + n)^{-3/2}  \\
&\preceq  \sum_{k = 1}^\infty (n + k)^{-3/2}  \preceq n^{-1/2}
\ale 
with universal implicit constant.  
\end{proof}

\begin{proof}[Proof of Lemma~\ref{prop-big-jump}]
Let $T_0 = 0$ and for $k\in\BB N$ let $T_k$ be the $k$th smallest $j\in\BB N$ for which $\wh Y^j - \wh Y^{j-1} \geq n$.  If $T_k \leq J_r$ then there exists 
\eqbn
j \in  [T_{k-1}+1,  J_r]_{\BB Z}   \subset   [T_{k-1}+1,  J_{\rho(T_{k-1}) + r}]_{\BB Z}  
\eqen 
such that $\wh Y^j - \wh Y^{j-1} \geq n$.  By applying Lemma~\ref{prop-layer-cover} for each of the stopping times $T_{k-1}, J_{\rho(T_{k-1})} , J_{\rho(T_{k-1})+1} , \dots , J_{\rho(T_{k-1})+r-1}$ and taking a union bound, we find that
\eqbn
\BB P\left[  T_k \leq   J_r    \,|\, \mcl F^{T_{k-1}} \right]   \preceq n^{-1/2} r .
\eqen 
Iterating this estimate $k$ times yields the statement of the lemma.
\end{proof}

\section{Moment bounds for the glued peeling process}
\label{sec-peeling-moment}

Suppose we are in the setting of Section~\ref{sec-glued-peeling} for some choice of gluing times $\ul{\BB x} , \BB x_-$, and $\BB x_+$ and initial edge set $\BB A$ satisfying the conditions of that section. Define the clusters $\{\dot Q^j\}_{j\in\BB N_0}$, the stopping times $\{J_r\}_{r\in \BB N_0}$, and the complementary UIHPQ$_{\op{S}}$'s $\{Q_\pm^j\}_{j\in\BB N_0}$ for the glued peeling process of $Q_{\op{zip}}$ started from $\BB A$. 
 
The main goal of this section is to prove the following upper bound for the boundary length of the clusters $\dot Q^{J_r}$ and the number of edges of $\bdy Q_- \cup \bdy Q_+$ which they intersect. It is natural to expect these quantities to be of order~$r^2$ because it is natural to expect that~$\dot{Q}^{J_r}$ is roughly comparable to a graph metric ball of radius~$r$.

\begin{prop} \label{prop-hull-moment}
Suppose we are in the setting described just above. For each $p\in [1,3/2)$, we have
\eqb \label{eqn-cluster-moment} 
\BB E\left[ \# \mcl E \left( \dot Q^{J_r} \cap (\bdy Q_- \cup \bdy Q_+) \right)^p \right] \preceq \left(r + (\#\BB A)^{1/2} \right)^{2p}
\eqe 
and
\eqb  \label{eqn-cluster-bdy-moment}
\BB E\left[  \left( \max_{j\in [1,J_r]_{\BB Z}} \#\mcl E\left( \bdy \dot Q^j \right) \right)^p \right] \preceq \left(r + (\#\BB A)^{1/2} \right)^{2p}
\eqe  
with implicit constant depending only on $p$.   
\end{prop}

Proposition~\ref{prop-hull-moment} is our most important estimate for the glued peeling clusters. The reason why we get moments up to order $3/2$ is related to the $5/2$ exponent appearing in~\eqref{eqn-cover-tail}.

We will deduce several consequences of Proposition~\ref{prop-hull-moment} in Section~\ref{sec-hull-moment-misc} below.  Since $\dot{Q}^{J_r}$ dominates from above a $Q_{\op{zip}}$-graph metric ball of radius $r$, an upper bound on $\dot{Q}^{J_r}$ leads to \emph{lower} bounds on $Q_{\op{zip}}$-graph distances.  This will lead to the fact that the gluing interface does not form loops at large scales, which will be quantified by a reverse H\"older continuity estimate for the gluing interface with respect to the $Q_{\op{zip}}$-graph metric (Lemma~\ref{prop-reverse-holder}).  We will also use~\eqref{eqn-cluster-moment} to obtain an upper bound for the size of a $Q_{\op{zip}}$-metric ball in terms of $Q_\pm$-metric balls (Lemma~\ref{prop-hull-diam}).  Finally, we will obtain a lower bound for the length of a path in $Q_{\op{zip}}$ which stays near $\bdy Q_-\cup \bdy Q_+$ (Lemma~\ref{prop-saw-neighborhood-dist}). We will use these consequences as well as Proposition~\ref{prop-hull-moment} itself several times in the later sections.

The proof of Proposition~\ref{prop-hull-moment} is carried out in Sections~\ref{sec-first-moment}--\ref{sec-hull-moment-proof} below.  One may skip the details of the proof on a first reading of the article, since this section is connected to the rest of the paper only through the statement of Proposition~\ref{prop-hull-moment} and its consequences deduced in Section~\ref{sec-hull-moment-misc}.

The proof of Proposition~\ref{prop-hull-moment} is based on an analysis of certain discrete processes associated with the boundary lengths of the clusters $\dot Q^j$, which are illustrated in Figure~\ref{fig-glued-peel-bdy}. For $j\in \BB N_0$, let $\wh Y^j$ be the number of edges in $\partial Q_+ \cup \partial Q_-$ which are contained in the glued peeling cluster $\dot{Q}^j$ at time $j$, as in~\eqref{eqn-saw-edge-count'}.  Then $\wh Y^0 = \#\BB A$ and~\eqref{eqn-cluster-moment} of Proposition~\ref{prop-hull-moment} is equivalent to a $p$th moment bound for $\wh Y^{J_r}$.

\begin{figure}[ht!]
 \begin{center}
\includegraphics[scale=1]{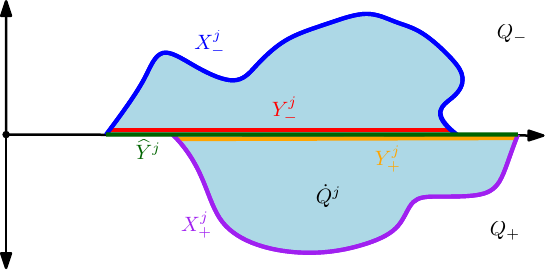} 
\caption[Boundary length processes for the glued peeling process]{Illustration of the boundary length processes for the glued peeling process used in Section~\ref{sec-peeling-moment}. The cluster $\dot Q^j$ is shown in light blue and the notation for each process is shown in the same color as the arc whose length it represents. We also have $X^j = X_-^j + X_+^j$, $Y^j = Y_-^j + Y_+^j$, and $Z^j = X^j - Y^j$.}\label{fig-glued-peel-bdy}
\end{center}
\end{figure}

Let
\eqb \label{eqn-interior-count}
X^j_\pm := \#\left(   \mcl E\left(\bdy \dot Q^j \cap \bdy Q^j_\pm   \right) \setminus \mcl E( \bdy Q_\pm)  \right) 
\eqe  
be the number of edges of $\bdy \dot Q^j$ which belong to the interior of $Q_\pm$. 
Also let
\eqb \label{eqn-bdy-count}
Y^j_\pm :=  \# \left( \mcl E\left( \bdy Q_\pm  \right) \setminus \mcl E\left(  \bdy Q_\pm^j \right) \right)  + \#\BB A 
\eqe 
be the number of edges of $\bdy Q_\pm$ which are disconnected from $\infty$ in $Q^j_\pm$ by $\bdy \dot Q^j$, plus the number of edges in the initial edge set.  
Define
\eqb \label{eqn-2side-count}
X^j := X^j_- + X^j_+   , \quad   Y^j := Y^j_-  +Y^j_+ ,\quad \op{and} \quad Z^j := X^j - Y^j .
\eqe 
Note that $X^0 = 0 $, $ Y^0 = 2\# \BB A$, $Z^0 = -2\#\BB A$, and
\eqb \label{eqn-count-compare}
\wh Y^j   \leq Y^j \leq 2 \wh Y^j .
\eqe 
Furthermore, in the notation of Section~\ref{sec-peeling} (recall also the signs $\xi^j$ from Section~\ref{sec-glued-peeling}), we have
\eqb \label{eqn-co-ex-mart}
Z^j - Z^{j-1}  =   (X^j_{\xi^j} - Y^j_{\xi^j} ) - ( X^{j-1}_{\xi^j} - Y^{j-1}_{\xi^j})   =  \op{Ex}\left( Q^{j-1}_{\xi^j} , \dot e^j \right)  - \op{Co}\left( Q^{j-1}_{\xi^j} , \dot e^j \right) 
\eqe  
Hence Lemma~\ref{prop-peel-law} implies that the increments $\{Z^j - Z^{j-1}\}_{j\in\BB N}$ are i.i.d.\  and adapted to the filtration~\eqref{eqn-peel-filtration}. Furthermore, since the number of covered and exposed edges for a peeling step have the same expectation, we find that $Z$ is an $\mcl F^j$-martingale.

The proof of Proposition~\ref{prop-hull-moment} consists of three main steps, which we outline below.
\begin{itemize}
	\item[Step 1:] First moment bound for the total number $J_r$ of quadrilaterals revealed by the time that $r$ layers have been peeled (Section~\ref{sec-first-moment}).  The idea is first to bound the conditional expectation of $J_{r+1} - J_r$ given $\mcl F^{J_{r }}$ in terms of $X^{J_{r }}$ and $Y^{J_{r }}$.  This is a natural bound since the number of quadrilaterals revealed in the $(r+1)$st layer should only depend on the total boundary length of the cluster when $r$ layers have been peeled.  To complete the proof, we then deduce a recursive bound for $X^{J_r} + Y^{J_r}$, which leads to a first moment bound for $X^{J_r} + Y^{J_r}$ and thereby our desired first moment bound for $J_r$.
\item[Step 2:] Establish a bound for the $p$th moment of $\max_{j\in [1,J_r]_{\BB Z}} (Z^j-Z^0)$ for $p \in [1,3/2)$ (Section~\ref{sec-mart-moment}).  This will follow from the first moment bound for $J_r$ from the previous step together with a standard estimate for sums of i.i.d.\ heavy-tailed random variables.  In particular, since $Z^j$ is a sum of i.i.d.\ random variables which have probability of order $k^{-5/2}$ of being equal to $k$ (recall~\eqref{eqn-cover-tail}), it follows that $|Z^j|$ (at a deterministic time $j$) is of order $j^{2/3}$.  Therefore the first moment bound for $J_r$ will indeed suffice to control the $p \in [1,3/2)$ moments of the maximum of $Z^j - Z^0$ up to time $J_r$.
\item[Step 3:] Complete the proof of Proposition~\ref{prop-hull-moment} (Section~\ref{sec-hull-moment-proof}). The bound~\eqref{eqn-cluster-moment} is equivalent to a moment bound for $\wh Y^{J_r}$. We will prove such a bound by first bounding the moments of $\wh Y^{T_k \wedge J_r}$, where $T_k$ is the $k$th time at which $\wh Y^j$ has a macroscopic jump (i.e., at least $r^2$ edges of $\bdy Q_- \cup \bdy Q_+$ are disconnected from $\infty$ simultaneously). This is done in Lemma~\ref{prop-big-jump-recursion} using a recursive argument together with our upper bound for $Z^j$ and our bound for the moments of $\wh Y^{J_r}$ when we skip the big jumps (Lemma~\ref{prop-small-bubble-moment}). We then conclude~\eqref{eqn-cluster-moment} using Lemma~\ref{prop-big-jump}, which gives an upper bound for the number of big jumps. The bound~\eqref{eqn-cluster-bdy-moment} follows easily by writing $X^j \leq Z^j + 2\wh{Y}^j$ and using~\eqref{eqn-cluster-moment} and our bound for the maximum of $Z^j$.
\end{itemize}
As we mentioned earlier, in Section~\ref{sec-hull-moment-misc} we deduce a number of consequences of Proposition~\ref{prop-hull-moment}. 

When reading the estimates in this section, it will be helpful to keep in mind that a radius $r$ metric ball in a uniformly random quadrangulation typically has outer boundary length of order $r^2$, the glued peeling process up to radius $r$ typically reveals of order $r^3$ quadrilaterals, and the total number of quadrilaterals cut off from $\infty$ is typically of order $r^4$.

\subsection{First moment bounds} 
\label{sec-first-moment}

In this subsection we will prove recursive bounds for the number $J_r$ of quadrilaterals revealed in the glued peeling cluster $\dot{Q}^j$ when $r$ layers have been peeled and the number $\wh Y^{J_r}$ of edges in $\partial Q_- \cup \partial Q_+$ which are contained in the glued peeling cluster, also when $r$ layers have been peeled.  These bounds will eventually lead to the following first moment bound for $J_r$.

\begin{lem} \label{prop-time-mean}
For each $r\in \BB N$, 
\eqb \label{eqn-time-mean}
\BB E\left[ J_r \right] \preceq  \left( r + (\#\BB A)^{1/2} \right)^3
\eqe 
with universal implicit constant.
\end{lem}

We emphasize that the exponent 3 on the right side of~\eqref{eqn-time-mean} is natural because the same power arises for the number of quadrilaterals revealed when one peels a radius-$r$ metric ball rooted at a vertex on the boundary of the UIHPQ$_{\op{S}}$ or at the root vertex of the UIPQ.

To prove Lemma~\ref{prop-time-mean}, we first prove a recursive bound for the conditional expectation of $J_{r+1}-J_r$ given $\mcl F^{J_r}$ in terms of $X^{J_r} + Y^{J_r}$.  This comes because the number of peeling steps necessary to cover a vertex on the boundary has a geometric distribution (Lemma~\ref{prop-time-induct}), and in particular has finite expectation.  We then prove a first moment bound for $X^{J_r}  + Y^{J_r}$ using another recursive argument (Lemma~\ref{prop-saw-edge-mean}). Combining these two lemmas and summing over $s \leq r$ will imply Lemma~\ref{prop-time-mean}.

\begin{lem} \label{prop-time-induct}
There is a universal constant $c_1>0$ such that each $r \in \BB N_0$,  
\eqbn
\BB E\left[ J_{r+1} \,|\, \mcl F^{J_r} \right] \leq J_r + c_1(X^{J_r} + Y^{J_r}   ) .
\eqen 
\end{lem} 
\begin{proof} 
For $v\in \mcl V\left(\bdy \dot Q^{J_r} \cap Q_\pm^{J_r} \right)$, let $\mcl I_{v,\pm}^L$ (resp.\ $\mcl I_{v,\pm}^R$) be the terminal time of the left (resp.\ right) one-vertex peeling process of $Q_\pm^{J_r}$ at $v$ (Definition~\ref{def-1vertex-peeling}). 
If $v\in \mcl V\left(\bdy \dot Q^{J_r} \cap Q_\pm^{J_r} \right)$, then every quadrilateral of $Q_\pm^{J_r}$ incident to $v$ which is peeled by the glued peeling process between times $J_r+1$ and $J_{r+1}$ is peeled by either the left or right one-vertex peeling process of $Q_\pm^{J_r}$ at $v$. Furthermore, by definition every quadrilateral which is peeled by the glued peeling process started from $\BB A$ and targeted at $\infty$ between times $J_r$ and $J_{r+1}$ is incident to some $v\in  \mcl V(\bdy \dot Q^{J_r})$. 
Therefore,
\eqbn
J_{r+1} - J_r \leq \sum_{\xi\in \{-,+\}} \sum_{ v\in \mcl V\left(\bdy \dot Q^{J_r} \cap Q_\xi^{J_r} \right)  } (\mcl I_{v,\xi}^L + \mcl I_{v,\xi}^R)  .
\eqen
By Lemmas~\ref{prop-peel-degree-pos} and~\ref{prop-peel-law}, $\BB E[\mcl I_{v,\pm}^L + \mcl I_{v,\pm}^R \,|\, \mcl F^{J_r}] $ is bounded above by a universal constant. Hence
\eqbn
\BB E\left[ J_{r+1}  - J_r \,|\, \mcl F^{J_r} \right] \preceq \# \mcl V(\bdy \dot Q^{J_r}) .
\eqen
On the other hand, 
\eqb  \label{eqn-bdy-edge-count}
\# \mcl V(\bdy \dot Q^{J_r}) \leq \#\mcl E\left(\bdy \dot Q^{J_r} \right)    + 2 \preceq X^{J_r} + \wh Y^{J_r} 
\eqe 
where in the last inequality we recall that $\wh Y^{J_r} \geq \#\BB A \geq 1$. 
The statement of the lemma follows since $\wh Y^{J_r} \asymp Y^{J_r}$ (recall~\eqref{eqn-count-compare}). 
\end{proof}

Lemma~\ref{prop-time-induct} gives us the necessary integrability in order to deduce that the stopped process $Z^{J_r}$ is a martingale with respect to the filtration $\mcl F^{J_r}$.
 
\begin{lem}  \label{prop-peel-optional}
For $r\in \BB N_0$ we have $\BB E\left[ Z^{J_{r+1} }   \,|\, \mcl F^{J_r} \right] = Z^{J_r} $. 
\end{lem}
\begin{proof}  
By Lemma~\ref{prop-time-induct}, we have $\BB E[ J_{r+1} - J_r \,|\, \mcl F^{J_r}] < \infty$ for each $r \in \BB N$. 
The discussion just after~\eqref{eqn-co-ex-mart} tells us that $Z$ is a $\mcl F^j$-martingale. By~\eqref{eqn-cover-tail},
\eqbn
\BB E\left[|Z^j - Z^{j-1} | \,|\, \mcl F^{j-1} \right] = \BB E\left[ |Z^1 - Z^0  | \right]   < \infty
\eqen
 for each $j\in\BB N$. 
Therefore, the statement of the lemma follows from Lemma~\ref{prop-time-induct} and the optional stopping theorem applied to the martingale $\{Z^j \}_{j \geq J^r}$ (see~\cite[Theorem~5.75]{durrett} for the precise statement we use here). 
\end{proof}
  
The following lemma gives us a recursive bound for $\wh Y^{J_r}$, which will be used to prove a moment bound for $X^{J_r} + Y^{J_r}$.  The $(X^{J_r} + Y^{J_r})^{1/2}$ term which appears on the right side below comes from peeling the quadrilaterals on the boundary of $\dot Q^{J_r}$.  The specific power $1/2$ arises because we are taking the mean of a distribution which equals $k$ with probability of order $(\ell +k)^{-5/2}$, where $\ell$ is the distance from the peeled edge to $\bdy Q_-\cup \bdy Q_+$ along $\bdy \dot Q^{J_r}$, and then summing $\ell$ over the boundary length of the glued peeling cluster (recall~\eqref{eqn-cover-tail}).

\begin{lem} \label{prop-peel-induct}
There is a universal constant $c_2 >0$ such that for $r\in \BB N $,   
\eqb \label{eqn-peel-induct}
 \BB E\left[ \wh Y^{J_{r+1}} \,|\, \mcl F^{J_r} \right]  \leq \wh Y^{J_r} + c_2 (X^{J_r} + \wh Y^{J_r} )^{1/2} . 
\eqe 
\end{lem}
\begin{proof}
Define $\ell_{v,\pm}^{J_r}$ for $v\in \mcl V\left(\bdy \dot Q^{J_r} \cap \bdy Q_\pm^{J_r} \right)$ as in Definition~\ref{def-min-vertex-set}.
For $v \in   \mcl V\left(\bdy \dot Q^{J_r} \cap \bdy Q_\pm^{J_r} \right)$ let $E_{v,\pm}^{J_r}$ be the set of edges of $\bdy Q_\pm^{J_r} \cap \bdy Q_\pm$ which are disconnected from $\infty$ in $Q_\pm^{J_r}$ by the union of the quadrilaterals of $Q_\pm^{J_r}$ incident to $v$.  

Every edge of $\dot Q^{J_{r+1}}\cap (\bdy Q_-\cup \bdy Q_+)$ which does not belong to $\dot Q^{J_{r }}\cap (\bdy Q_-\cup \bdy Q_+)$ belongs to $E_{v,-}^{J_r}$ or $E_{v,+}^{J_r}$ for some $v\in \mcl V(\bdy \dot Q^{J_r})$. Therefore,  
\eqb \label{eqn-peel-induct-sum}
\wh Y^{J_{r+1}} - \wh Y^{J_{r}} \leq  \sum_{\xi \in \{-,+\}}  \sum_{v \in  \mcl V\left(\bdy \dot Q^{J_r} \cap \bdy Q_\xi^{J_r} \right) } \# E_{v,\xi}^{J_r}   .
\eqe 
We will now bound the right side of~\eqref{eqn-peel-induct-sum} using peeling estimates.

If $v\in  \mcl V\left(\bdy \dot Q^{J_r} \cap \bdy Q_\pm^{J_r} \right)$ and $\# E_{v,\pm}^{J_r} \geq n$ for some $n\in\BB N$, then there are at least $n + \ell_{v,\pm}^{J_r}$ edges of $\bdy Q_\pm^{J_r}$ which are disconnected from $\infty$ in $Q_\pm^{J_r}$ by the union of the quadrilaterals of $Q_\pm^{J_r}$ incident to $v$. Therefore, Lemma~\ref{prop-peel-degree} implies that
\eqbn
\BB E\left[ \# E_{v,\pm}^{J_r} \,|\, \mcl F^{J_r} \right] \preceq \sum_{n=1}^\infty (n + \ell_{v,\pm}^{J_r} )^{-3/2} \preceq (\ell_{v,\pm}^{J_r})^{-1/2} .
\eqen 
For each $m \in \BB N$, there are at most two elements of $  \mcl V(\bdy \dot Q^{J_r} \cap Q_\pm) $ with $\ell_{v,\pm}^{J_r} = m$. Hence
\alb
\BB E\left[ \wh Y^{J_{r+1}}  - \wh Y^{J_r}  \,|\, \mcl F^{J_r} \right]  
&\leq    \sum_{\xi \in \{-,+\}} \sum_{v \in  \mcl V\left(\bdy \dot Q^{J_r} \cap \bdy Q_\xi^{J_r} \right) } \BB E\left[ \# E_{v,\xi}^{J_r} \,|\, \mcl F^{J_r} \right] \\  
&\preceq \sum_{\xi \in \{-,+\}} \sum_{v \in  \mcl V\left(\bdy \dot Q^{J_r} \cap \bdy Q_\xi^{J_r} \right) } (\ell_{v,\xi}^{J_r})^{-1/2}  
\preceq  \#  \mcl V(\bdy \dot Q^{J_r})^{1/2}     ,
\ale
where the implicit constants in $\preceq$ are universal. By combining this estimate with~\eqref{eqn-bdy-edge-count} we conclude.
\end{proof}
  
From Lemmas~\ref{prop-peel-optional} and~\ref{prop-peel-induct} we obtain a first moment bound for $X^{J_r} + Y^{J_r}$.  As we mentioned earlier, it is natural to expect that $\dot{Q}^{J_r}$ is a good approximation for a filled $Q_{\op{zip}}$-metric ball of radius $r$ hence it is natural to expect that its boundary length should be of order $r^2$.

\begin{lem} \label{prop-saw-edge-mean}
For each $r\in \BB N_0$,   
\eqbn
\BB E\left[ X^{J_r} + Y^{J_r} \right] \preceq \left( r + (\#\BB A)^{1/2} \right)^2
\eqen
with universal implicit constant.
\end{lem}
\begin{proof}%[Proof of Lemma~\ref{prop-saw-edge-mean}]
For $j\in \BB N_0$, let $W^j := 4 \wh Y^j + Z^j$. Since $Y^j \leq 2 \wh Y^j$ (recall~\eqref{eqn-count-compare}) and $Z^j = X^j - Y^j$, we have $W^j \geq X^j + Y^j \geq 0$. By Lemmas~\ref{prop-peel-optional} and~\ref{prop-peel-induct}, for $r\in \BB N$,
\eqbn
\BB E\left[ W^{J_{r+1}}  \,|\, \mcl F^{J_r} \right] \leq 4 \wh Y^{J_r} + Z^{J_r} + 4 c_2 (\wh Y^{J_r} + X^{J_r}  )^{1/2} \leq  W^{J_r} +  c_3 (W^{J_r})^{1/2}    
\eqen
for $c_3 >0$ a universal constant. 
Therefore
\eqbn
\BB E\left[ W^{J_{r+1}} \right] \leq \BB E\left[ W^{J_r} \right] + c_3 \BB E\left[ W^{J_r} \right]^{1/2} 
\eqen
where here we have used H\"older's inequality to move the square root outside the expectation. Iterating this estimate yields
\eqb \label{eqn-peel-induct-end}
\BB E\left[ W^{J_r} \right] \leq c_3 \sum_{s=0}^{r-1} \BB E\left[W^{J_s} \right]^{1/2} .
\eqe 
Since $W^{J_0} = 2 \#\BB A $, we infer from~\eqref{eqn-peel-induct-end} and induction that $\BB E[W^{J_r}] < \infty$ for each $r\in \BB N$. 

In fact, one has the following elementary inequality, which can be proven by induction: if $\{x_r\}_{r \in \BB N_0}$ are real numbers and $c>0$ such that $x_r \leq c \sum_{s=0}^{r-1} x_s^{1/2}$ for each $s\geq 0$, then $x_r \preceq (x_0^{1/2} +r)^2$, with an implicit constant depending only on $c$.
By~\eqref{eqn-peel-induct-end}, we can apply this with $x_r = \BB E[W^{J_r}]$ to get that
\eqbn
\BB E\left[ X^{J_r} + Y^{J_r}  \right] \leq \BB E\left[W^{J_r} \right] \preceq \left( r + (\#\BB A)^{1/2} \right)^2  .
\eqen
\end{proof}

\begin{comment}

\begin{lem}
If $\{x_r\}_{r \in \BB N_0}$ are real numbers and $c>0$ such that $x_r \leq c \sum_{s=0}^{r-1} x_s^{1/2}$ for each $s\geq 0$, then $x_r \preceq (x_0^{1/2} +r)^2$, with an implicit constant depending only on $c$.
\end{lem}
\begin{proof}
Let $y_r := \sum_{s=0}^r x_s^{1/2}$. 
Then $(y_r-y_{r-1})^2 \leq c y_{r-1}$, i.e., $y_r-y_{r-1} \leq c^{1/2} y_{r-1}^{1/2}$.
Fix $b \geq 1$. We claim that if $b$ is chosen sufficiently small, in a manner depending only on $c$, then $y_r \leq b (x_0^{1/2} + r)^2$.
The base case $r = 0$ is trivial, for any $b\geq 1$.
If $r\in\BB N$ and the estimate is true with $r-1$ in place of $r$, then
\alb
y_r &\leq y_{r-1} + c^{1/2} y_{r-1}^{1/2} \notag \\ 
&\leq b (x_0^{1/2} + r-1)^2 + c^{1/2} (x_0^{1/2} + r-1) \notag \\ 
&= b (x_0^{1/2} + r)^2  - 2b (x_0^{1/2} + r)     + c^{1/2} (x_0^{1/2} + r )  + b - c^{1/2} \notag \\
&= b (x_0^{1/2} + r)^2  -  b (x_0^{1/2} + r)     + c^{1/2} (x_0^{1/2} + r )   \quad \text{since $x_0^{1/2} + r \geq r\geq 1$} \notag \\ 
&\leq b (x_0^{1/2} + r)^2 .
\ale
provided $b \geq c^{1/2}$. We have $x_r \leq y_{r-1}$, so this concludes the proof. 
\end{proof}

\end{comment}

Finally, we deduce our expectation bound for $J_r$. 
 
\begin{proof}[Proof of Lemma~\ref{prop-time-mean}]
By Lemmas~\ref{prop-time-induct} and~\ref{prop-saw-edge-mean}, for $s\in \BB N_0$ we have 
\eqbn
\BB E\left[ J_{s+1} - J_s  \,|\, \mcl F^{J_s} \right] \leq c_0 \BB E\left[ X^{J_s} + Y^{J_s} \right] \preceq \left( s + (\#\BB A)^{1/2} \right)^2 .
\eqen
Summing from $s = 0$ to $s = r-1$ yields the statement of the lemma. 
\end{proof}

\subsection{Upper bound for the martingale}
\label{sec-mart-moment}

We next deduce from Lemma~\ref{prop-time-mean} a tail bound for $Z^{J_r}$ which improves on the tail bound implied by Lemma~\ref{prop-saw-edge-mean}. 

\begin{lem} \label{prop-outer-tail}
For each $C > 1$ and $r\geq 1$, 
\eqb  \label{eqn-outer-tail}
\BB P \left[ \max_{j\in [0,J_r]_{\BB Z}} (Z^j - Z^0) > C r^2 \right] \preceq (\log C)^2 C^{-3/2} 
\eqe
with universal implicit constant. In particular, for each $p \in [1,3/2)$, 
\eqb \label{eqn-outer-moment}
\BB E\left[ \left(  \max_{j\in [0,J_r]_{\BB Z}} (Z^j-Z^0) \right)^p \right] \preceq r^{2p} 
\eqe 
with implicit constant depending only on $p$. 
\end{lem}

In the statement of Lemma~\ref{prop-outer-tail}, we recall that $Z^0 = -2\#\BB A$.  The tail bound in~\eqref{eqn-outer-tail} is natural because $|Z^j|$ is of order $j^{2/3}$ for a deterministic value of $j$ and $J_r$ is typically of order $r^3$.  For the proof of the lemma, we will need the following basic tail bound for sums of i.i.d.\  random variables with heavy tails.  (One can skip the proof of Lemma~\ref{prop-levy-upper} on a first reading.)

\begin{lem} \label{prop-levy-upper}
Let $\alpha \in (1,2)$ and $b >0$. Let $\{X_j\}_{j\in\BB N}$ be a sequence of i.i.d.\  mean-zero random variables such that $X_j \leq b$ a.s.\ and for $r > 0$, we have $\BB P\left[X_j < -r \right] \sim r^{-\alpha}$. Let $S_0=0$ and for $n\in\BB N$, let $S_n := \sum_{j=1}^n X_j$. For $C>0$ and $n\in\BB N$, 
\eqbn
\BB P\left[\max_{m\in [1,n]_{\BB Z}} S_m \geq C n^{1/\alpha} \right] \leq a_0 e^{-a_1 C}
\eqen
where $a_0 , a_1 > 0$ are constants which do not depend on $n$ or $C$. 
\end{lem}
\begin{proof}
Let $I_0 = 0$ and for $k\in\BB N$, inductively let 
\eqbn
I_k := \min \left\{ j\geq I_{k-1} +1 \,:\, S_j > S_{I_{k-1}}  \right\} .
\eqen
Note that the vectors of random variables $ (X_{I_{k-1}+1} , \dots , X_{I_k}) $ for $k\in\BB N$ are i.i.d.\  and we always have $S_{I_k} - S_{I_{k-1}} \in (0,b]$.  Furthermore, $S_{I_k} = \max_{m\in [0,I_k]_{\BB Z}} S_m$.
For $n\in\BB N$ and $t>0$, let
$H^n_t := n^{-1} S_{\lfloor n^\alpha t \rfloor} $. 
By the classical scaling limit theorem for stable processes, $H^n \rta H$ in law in the local Skorokhod topology, where $H$ is an $\alpha$-stable L\'evy process with only downward jumps. 

Let $\sigma^n := n^{-\alpha} I_n  $
and for $s\geq 0$, let $\tau_s := \inf\{t \geq 0\,:\, H_t = s \}$.
Then
\eqbn
H^n_{\sigma^n} =  \frac{1}{n} \sum_{k=1}^n  (S_{I_k} - S_{I_{k-1}})
\eqen
so by the law of large numbers $H^n_{\sigma^n} \rta \beta$ in probability, where 
$\beta := \BB E\left[ S_{I_1}  \right] \in (0,b]$.
The time $ \sigma^n$ is equal to the first time that $H^n$ hits $s^n$ for some random $s^n > 0$. Since the upward jumps of $H^n$ have size at most $\beta n^{-1}$, necessarily satisfies $s^n\rta \beta$ in probability. 
Since $H$ has no upward jumps, we infer that $\sigma^n \rta \tau_\beta$ in law. 

By the converse to the heavy-tailed central limit theorem,   
\eqbn
\BB P\left[ I_1 > s \right] \sim s^{-1/\alpha} \quad \op{as}\quad s \rta \infty .
\eqen
If $ \max_{m\in [1,n]_{\BB Z}} S_m \geq C n^{ 1/\alpha} $, then $I_{\lfloor (C-b) n^{1/\alpha} \rfloor} \leq n$. We therefore have (for an appropriate $n,C$-independent constant $\wt a >0$)
\alb
\BB P\left[\max_{m\in [1,n]_{\BB Z}} S_m \geq C n^{1/\alpha} \right]
 \leq \BB P\left[\max_{k\in [1, (C-b) n^{1/\alpha}]_{\BB Z}} (I_k - I_{k-1})  \leq  n\right] 
\leq \left( 1 - \wt a n^{-1/\alpha}  \right)^{-(C-b) n^{1/\alpha}} \wedge 1 
\leq a_0 e^{-a_1 C} ,
\ale
 with $a_0 , a_1>0$ as in the statement of the lemma.
\end{proof}

\begin{proof}[Proof of Lemma~\ref{prop-outer-tail}]
Recall from the discussion just after~\eqref{eqn-cover-tail} that the increments $Z^j -Z^{j-1}$ for $j\in\BB N$ are i.i.d.\  with zero mean.  
Furthermore, $Z^j - Z^{j-1} \leq 2$ a.s.\ and by~\eqref{eqn-cover-tail}, for $s \in\BB N$ 
\eqbn
\BB P\left[ Z^j - Z^{j-1}  < -s \right] \sim s^{-3/2} .
\eqen
By Lemma~\ref{prop-levy-upper}, for $n\in\BB N$ and $A>0$ we have
\eqb \label{eqn-outer-sup}
\BB P\left[ \max_{j \in [0,n]_{\BB Z}} (Z^j -Z^0) > A n^{2/3} \right] \preceq e^{-a_1 A}
\eqe 
for $a_1>0$ a universal constant. 
  
If we are given $C>1$ and we set $N =\lfloor (\log C)^{-2} C^{3/2} r^3 \rfloor$, then by~\eqref{eqn-outer-sup}  
\eqbn
\BB P\left[ \max_{j\in [0,J_r]_{\BB Z}} (Z^j - Z^0) > C r^2 ,\, J_r \leq N \right] 
\leq \BB P\left[ \max_{j\in [1, N]_{\BB Z}} (Z^j - Z^0) > (\log C)^{4/3} N^{2/3} \right] \preceq e^{-a_1 (\log C)^{4/3}} \preceq C^{-3/2}   .
\eqen 
On the other hand, Lemma~\ref{prop-time-mean} and the Chebyshev inequality imply that
\eqbn
\BB P\left[ J_r > N \right] \preceq (\log C)^2 C^{-3/2} .
\eqen
The estimate~\eqref{eqn-outer-tail} now follows from a union bound.

The moment bound~\eqref{eqn-outer-moment} follows from~\eqref{eqn-outer-tail} and the formula
\eqbn
\BB E\left[ W^p \right] = \int_0^\infty  p t^{p-1} \BB P[ W \geq t] \, dt 
\eqen
applied to the non-negative random variable $W = r^{-2} \max_{j\in [0,J_r]_{\BB Z}} (Z^j - Z^0)$.
\end{proof}

\subsection{Proof of Proposition~\ref{prop-hull-moment}}
\label{sec-hull-moment-proof}

We now turn to complete the proof of Proposition~\ref{prop-hull-moment}.  In view of the results of Section~\ref{sec-peel-jump}, in which we bounded the moments of the number $\wh{Y}^j$ of edges of $\partial Q_- \cup \partial Q_+$ cut off from $\infty$ by the glued peeling cluster $\dot{Q}^j$ after truncating away the macroscopic jumps, the main input in the proof of~\eqref{eqn-cluster-moment} in Proposition~\ref{prop-hull-moment} is the following bound for the $p$th moments of $\wh Y$ stopped at the times when it makes a macroscopic jump.  These macroscopic jumps occur whenever the glued peeling cluster cuts off a macroscopic region from $\infty$ upon revealing a quadrilateral which is adjacent to the gluing interface.
 
\begin{lem} \label{prop-big-jump-recursion}
Suppose $c>1$ and $r\in\BB N$. Let $T_0 = T_0(cr^2) = 0$ and for $k\in\BB N$ let $T_k = T_k(cr^2)$ be the $k$th smallest $j \in \BB N$ for which $\wh Y^j - \wh Y^{j-1} \geq c r^2$. 
For each $p \in [1,3/2)$, there exists a constant $A_p \geq 1$, depending only on $p$, such that for each $r\in\BB N$, each $c>1$ and each $k\in\BB N$,
\eqb \label{eqn-big-jump-recursion} 
\BB E\left[ \left( \wh Y^{T_k\wedge J_r } \right)^p  \right] \preceq  A_p^k   c^p \left(r  + (\#\BB A)^{1/2} \right)^{2p}
\eqe  
with implicit constant depending only on $p$. 
\end{lem} 

The key point of Lemma~\ref{prop-big-jump-recursion} is that $A_p$ and the implicit constant in~\eqref{eqn-big-jump-recursion} do not depend on $c$. As we will see below, choosing $c>1$ sufficiently large and applying Lemma~\ref{prop-big-jump} will allow us to cancel out the exponential factor $A_p^k$ in~\eqref{eqn-big-jump-recursion} using the fact that the largest $k$ for which $T_k \leq J_r$ has an exponential tail (Lemma~\ref{prop-big-jump}).

\begin{proof}[Proof of Lemma~\ref{prop-big-jump-recursion}]
We will prove the lemma by deriving a recursive bound for $ \BB E\left[ \left( \wh Y^{ T_k \wedge J_r } \right)^p \right]$ in terms of $ \BB E\left[ \left( \wh Y^{ T_{k-1} \wedge J_r } \right)^p \right]$. 
For $k\in\BB N$ let $\ell^{T_k}$ be the $\bdy Q_{\xi^{T_k}}^{T_k-1}$-graph distance from the $T_k$th peeled edge $\dot e^{T_k}$ to $\mcl E(\bdy Q_{\xi^{T_k}}^{T_k}) \setminus  \mcl E( \dot Q^{T_k} )$.
Note that
\eqbn
\ell^{T_k} \leq X^{T_k-1} + Y^{T_k-1}  \leq Z^{T_k-1} + 4 \wh Y^{T_k-1}. 
\eqen 
If $k \in \BB N$ and we condition on $\sigma(T_k)\vee \mcl F^{T_k-1}$, then the conditional law of the $T_k$th peeling step is the same as its conditional law given that the peeled quadrilateral $\frk f(Q_{\xi^{T_k}}^{T_k} , \dot e^{T_k})$ disconnects at least $c r^2$ edges in $\mcl E(\bdy Q_{\xi^{T_k}}^{T_k})\setminus  \mcl E( \dot Q^{T_k} )$ from $\infty$ in $ Q_{\xi^{T_k}}^{T_k}$. This is the case provided $\frk f(Q_{\xi^{T_k}}^{T_k} , \dot e^{T_k})$ disconnects at least $\ell^{T_k} + cr^2$ edges of $\bdy Q_{\xi^{T_k}}^{T_k}$ lying either to the left or to the right of $\dot e^{T_k}$ (where the choice is $\mcl F^{T_k-1}$-measurable) from $\infty$ in $ Q_{\xi^{T_k}}^{T_k}$. 
By~\eqref{eqn-cover-tail}, for $m\in\BB N$ with $m\geq  cr^2$ we have 
\eqbn
\BB P\left[ \wh Y^{T_k}  - \wh Y^{T_{k-1}} \geq m \,|\, \sigma(T_k)\vee \mcl F^{T_k-1} \right] \preceq \left(\ell^{T_k} + c r^2 \right)^{3/2} (m + \ell^{T_k} )^{-3/2}  .
\eqen 
Therefore, 
\begin{align} \label{eqn-cond-jump-moment}
\BB E\left[ \left(\wh Y^{T_k  }  - \wh Y^{T_k -1  } \right)^p \,|\, \sigma(T_k)\vee \mcl F^{T_k-1} \right] 
&\preceq \left(\ell^{T_k } + c r^2 \right)^{3/2} \sum_{m= \lfloor cr^2 \rfloor}^\infty m^{p-1} (m + \ell^{T_k} )^{-3/2} \notag \\
&\preceq \left(\ell^{T_k} + c r^2 \right)^{3/2} \sum_{m= \lfloor cr^2 \rfloor}^\infty  (m + \ell^{T_k} )^{p -5/2} \notag\\
&\preceq \left(\ell^{T_k} + c r^2 \right)^{p} 
\preceq \left( Z^{T_k-1} \vee 0 \right)^p +  \left( \wh Y^{T_k-1} \right)^p  + c^p  r^{2p }  .
\end{align} 

If $T_k > J_r$, then $\wh Y^{T_k \wedge J_r } - \wh Y^{(T_k - 1) \wedge J_r} = 0$. Hence~\eqref{eqn-cond-jump-moment} implies that 
\begin{align} \label{eqn-cond-jump-moment'}
& \BB E\left[ \left( \wh Y^{T_k\wedge J_r } \right)^p \,|\,  \sigma(T_k)\vee \mcl F^{T_k-1}  \right] \notag\\
\preceq & \BB E\left[ \left( \wh Y^{ (T_k-1)  \wedge J_r} \right)^p  + \left(\wh Y^{T_k \wedge J_r  }  - \wh Y^{ (T_k-1)  \wedge J_r}  \right)^p      \,|\,  \sigma(T_k)\vee \mcl F^{T_k-1}  \right] \notag \\
\preceq & \left( \wh Y^{(T_k-1) \wedge J_r} \right)^p  +  \left[ \left( Z^{(T_k-1)\wedge J_r} \vee 0\right)^p  +  \left( \wh Y^{(T_k-1)\wedge J_r} \right)^p  + c^p  r^{2p } \right] \BB 1_{(T_k \leq J_r)}  .
\end{align}  
By Lemma~\ref{prop-outer-tail},
\eqb  \label{eqn-cond-jump-moment-term1}
\BB E\left[\left( Z^{(T_k-1)\wedge J_r} \vee 0\right)^p  \right] \preceq r^{2p} .
\eqe  
In the notation of~\eqref{eqn-jump-truncated}, $ Y^{(T_k-1)\wedge J_r} -  Y^{T_{k-1} \wedge J_r}   \leq \wh Y_{c r^2}^{J_r}$ so by Lemma~\ref{prop-small-bubble-moment},
\begin{align} \label{eqn-jump-gap}
\BB E\left[ \left( \wh Y^{(T_k-1) \wedge J_r} \right)^p  \right]
\preceq&  \BB E\left[  \left( \wh Y^{T_{k-1} \wedge J_r} \right)^p   \right] + \BB E\left[  \left( \wh Y^{T_{k-1} \wedge J_r}  - \wh Y^{(T_k-1)\wedge J_r} \right)^p   \right] \notag\\
\preceq&  \BB E\left[  \left( \wh Y^{T_{k-1} \wedge J_r} \right)^p   \right] + c^p r^{2p }   .
\end{align} 
Taking expectations of both sides of~\eqref{eqn-cond-jump-moment'} (ignoring the indicator function) and plugging in the estimates~\eqref{eqn-cond-jump-moment-term1} and~\eqref{eqn-jump-gap} gives
\eqb \label{eqn-p-moment-recursion} 
\BB E\left[ \left( \wh Y^{T_k\wedge J_r } \right)^p  \right] \preceq  \BB E\left[ \left( \wh Y^{T_{k-1} \wedge J_r} \right)^p \right] + c^p r^{2p}, 
\eqe 
implicit constants depending only on $p$. 

We now use the following elementary inequality, which is easily checked by induction: if $\{x_k\}_{k\in\BB N_0}$ are real numbers and $C_1,C_2 \geq 1$ such that $x_k \leq C_1 x_{k-1} + C_2$ for each $k\in\BB N_0$, then
\eqb \label{eqn-p-moment-elem}
x_k \leq  \frac{  C_1^k(     C_2   + (C_1 - 1) x_0   )   -C_2   }{C_1-1} .
\eqe 
We have $\wh Y^{T_0 \wedge J_r} = \wh Y^0 = \#\BB A$. 
By~\eqref{eqn-p-moment-recursion}, we can apply~\eqref{eqn-p-moment-elem} with $x_k = \BB E\left[ \left( \wh Y^{T_k\wedge J_r } \right)^p  \right]  $, $C_1 \preceq 1$, and $C_2\preceq c^p r^{2p}$ (implicit constants depending only on $p$), drop the $-C_2$ term, and absorb some $C_1$-dependent constants into the implicit constant in $\preceq$ to get
\eqb \label{eqn-big-jump-recursion0}
\BB E\left[ \left( \wh Y^{T_k\wedge J_r } \right)^p  \right] \preceq   A_p^k(   c^p r^{2p}    +   \#\BB A  ) \leq    A_p^k c^p \left(    r^{ p}    +   (\#\BB A)^{1/2}  \right)^2 ,
\eqe
for an appropriate constant $A_p\geq 1$ as in the statement of the lemma.
\end{proof}

\begin{comment}
%RSolve[{x[k] == c1 x[k - 1] + c2 , x[0] == y}, x[k], k]
\begin{lem}
Suppose $\{x_k\}_{k\in\BB N_0}$ are real numbers and $C_1,C_2 \geq 1$ such that $x_k \leq C_1 x_{k-1} + C_2$ for each $k\in\BB N_0$. 
Then 
\eqbn
x_k \leq  \frac{  C_1^k(     C_2   + (C_1 - 1) x_0   )   -C_2   }{C_1-1}
\eqen
\end{lem}
\begin{proof}
The inequality for $k = 0$ is obvious. Now assume $k\in\BB N$ and the inequality holds with $k-1$ in place of $k$.
Then 
\eqbn
x_k \leq C_1 \frac{  C_1^{k-1}(     C_2   + (C_1 - 1) x_0   )   -C_2   }{C_1-1} + C_2 
=   \frac{  C_1^{k  }(     C_2   + (C_1 - 1) x_0   )   - C_1 C_2   }{C_1-1} + \frac{C_2(C_1-1)}{C_1-1}
=   \frac{  C_1^{k }(     C_2   + (C_1 - 1) x_0   )   - C_2   }{C_1-1}  .
\eqen
\end{proof}

\end{comment}

\begin{proof}[Proof of Proposition~\ref{prop-hull-moment}]
Fix $p\in [1,3/2)$ and $r\in \BB N$. Let $c>1$ to be chosen later, depending on $p$ and for $k\in \BB N$ let $T_k = T_k(cr^2)$ be the $k$th largest $j\in\BB N$ for which $\wh Y^j - \wh Y^{j-1} \geq c r^2$, as in Lemma~\ref{prop-big-jump-recursion}. Also let $K_r = K_r(cr^2)$ be the largest $k\in\BB N$ for which $T_k \leq J_r$ and let $\wh Y_{c r^2}^{J_r}$ be the sum of the truncated jumps be as in~\eqref{eqn-jump-truncated}.
For each $p\in [1, 3/2)$, we have the crude bound
\eqb \label{eqn-hull-moment-split}
\left( \wh Y^{J_r} \right)^p \preceq \left( \wh Y^{J_r}_{cr^2}   \right)^p + \left( \wh Y^{T_{K_r} } \right)^p
\eqe 
with implicit constant depending only on $p$. 
There is a large amount of over-counting in~\eqref{eqn-hull-moment-split} since $\wh Y^{T_{K_r}}$ includes \emph{all} of the jumps of $\wh Y$ up to time $T_{K_r}$. The term $\wh Y^{J_r}_{c r^2}$ is only needed to deal with the jumps between times $T_{K_r}$ and $J_r$. Nevertheless, it turns out that this rather crude estimate is sufficient for our purposes.

By~\eqref{eqn-hull-moment-split} and Lemma~\ref{prop-small-bubble-moment},
\eqb \label{eqn-small-jump-p-moment}
\BB E\left[\left( \wh Y^{J_r} \right)^p \right] \preceq c^p r^{2p}  + \BB E\left[ \left( \wh Y^{T_{K_r} } \right)^p \right] ,
\eqe 
so it remains to bound $ \BB E\left[ \left( \wh Y^{T_{K_r} } \right)^p \right]$. 

To this end, let $q  \in (1, \frac32 p^{-1})$. By H\"older's inequality,
\begin{align} \label{eqn-hull-moment-holder}
\BB E\left[ \left( \wh Y^{T_{K_r} } \right)^p \right] 
\leq \BB E\left[ \left( \sum_{k=1}^{K_r} \wh Y^{T_k \wedge J_r}  \right)^p \right]  
&\leq  \sum_{k=1}^\infty \BB E\left[  K_r^{p-1}  \BB 1_{(K_r \geq k )} \left( \wh Y^{T_k \wedge J_r}  \right)^p   \right]  \notag\\
&\leq  \sum_{k=1}^\infty \BB E\left[  K_r^{\frac{q(p-1)}{q-1} }  \BB 1_{(K_r \geq k )} \right]^{1-\frac{1}{q}}  \BB E\left[ \left( \wh Y^{T_k \wedge J_r}  \right)^{qp}   \right]^{\frac{1}{q}} .
\end{align}
By Lemma~\ref{prop-big-jump}, there is a universal constant $a>0$ such that the law of $K_r$ is stochastically dominated by that of a geometric random variable with parameter $a c^{-1/2}$. Consequently, if we take $c > a^3$, say, then 
\eqbn
\BB E\left[  K_r^{\frac{q(p-1)}{q-1} }  \BB 1_{(K_r \geq k )} \right]^{1-\frac{1}{q}} \preceq c^{-b k} 
\eqen
with $b >0$ and the implicit constant depending only on $p$ and $q$. By Lemma~\ref{prop-big-jump-recursion},
\eqbn
\BB E\left[ \left( \wh Y^{T_k \wedge J_r}  \right)^{qp}   \right]^{\frac{1}{q}} \preceq A_{qp}^{k/q}  c^p \left(r +(\#\BB A)^{1/2} \right)^{2p} 
\eqen
with $A_{qp} > 1$ and the implicit constant depending only on $p$ and $q$. If we choose $c$ sufficiently large that $c^b > A_{qp}^{1/q}$, then~\eqref{eqn-hull-moment-holder} implies that
\eqbn
\BB E\left[ \left( \wh Y^{T_{K_r} } \right)^p \right]  \preceq \left(r +(\#\BB A)^{1/2} \right)^{2p}  ,
\eqen
where now the implicit constant is also allowed to depend on $c$. By combining this with~\eqref{eqn-small-jump-p-moment} we obtain~\eqref{eqn-cluster-moment}. 

Next we deduce the boundary length estimate~\eqref{eqn-cluster-bdy-moment} from~\eqref{eqn-cluster-moment}. For $j\in [1,J_r]_{\BB Z}$,  
\eqbn
\# \mcl E\left( \bdy \dot Q^j   \right) \leq X^j  +\wh Y^j \leq \max_{j\in [1,J_r]_{\BB Z}} (Z^j-Z^0) + 3\wh Y^{J_r} 
\eqen
where here we have used that $X^j = Z^j + Y^j \leq Z^j + 2\wh Y^j$, that $Z^0 =-2\#\BB A  <0$, and that $\wh Y^j$ is monotone non-decreasing. We have a $p$th moment bound for $\wh Y^{J_r}$ by~\eqref{eqn-cluster-moment} and a $p$th moment bound for $\max_{j\in [1,J_r]_{\BB Z}} (Z^j-Z^0) $ by Lemma~\ref{prop-outer-tail}.  
\end{proof}

\subsection{Consequences of the moment bound}
\label{sec-hull-moment-misc}

In this subsection we will deduce some consequences of Proposition~\ref{prop-hull-moment} which will play an important role in later sections for controlling the large scale geometry of the gluing interface and are also of independent interest.
Throughout this section, one should think of $r$ as being large --- it will eventually be sent to $\infty$ when we pass to the scaling limit. 
%The estimates in this section are required to be uniform in $r$ provided $r$ is large enough. 

\subsubsection{Reverse H\"older continuity estimate for the curve}
\label{sec-reverse-holder}

Here we prove a reverse H\"older continuity estimate for the boundary path $\lambda_-$ of $Q_-$ with respect to the graph metric on $Q_{\op{zip}}$, which will eventually be used to show that the gluing interface for any subsequential scaling limit of the maps $Q_{\op{zip}}$ is a simple curve. Note that $\lambda_-|_{\BB N_0}$ coincides with the SAW $\lambda_{\op{zip}}$ in the case when there is no ``hole" in $Q_{\op{zip}}$ and that (by symmetry) the same estimate is true with $\lambda_+$ in place of $\lambda_-$.

\begin{lem} \label{prop-reverse-holder}
Fix $L>0$. For $\delta \in (0,1)$, $r \in \BB N$, and $\beta \in (0,2/3)$, 
\eqb \label{eqn-reverse-holder}
\BB P\left[ \op{dist}\left( \lambda_-(x) ,\lambda_-(y) ; Q_{\op{zip}} \right) \geq \lfloor \delta  r \rfloor ,\, \forall x,y\in [-Lr^2 ,L r^2]_{\BB Z} \: \text{with $|x-y| \geq \delta^\beta r^2$} \right] \geq 1 - \delta^{\frac32 (2-\beta) - 2 + o_\delta(1)}
\eqe 
with the rate of the $o_\delta(1)$ depending only on $L$ and $\beta$ (not on $r$).  
\end{lem}
\begin{proof}
The idea of the proof is to use Proposition~\ref{prop-hull-moment} and a union bound to cover $\lambda_-([-Lr^2 , Lr^2]_{\BB Z})$ by graph metric balls which do not contain any points of $\bdy Q_-\cup \bdy Q_+$ which are unusually far apart. For this purpose, the fact that we get a moment of order $>1$ in Proposition~\ref{prop-hull-moment} is essential.

For $\delta \in (0,1)$ and $x\in [-L r^2 , L r^2]_{\BB Z} \cap (\lfloor \delta r^2 \rfloor \BB Z)$, let $E_\delta^r(x)$ be the event that the $Q_{\op{zip}}$-graph metric neighborhood $B_{ \delta r} \left( \lambda_-([x-\delta  r^2 , x]_{\BB Z}) ; Q_{\op{zip}} \right)$ does \emph{not} contain $\lambda_-(y)$ for any $y\in \BB Z$ with $|x-y| \geq (\delta^\beta - \delta) r^2$. Also let 
\eqbn
E_\delta^r := \bigcap_{x \in [-L r^2 , L r^2]_{\BB Z} \cap (\lfloor \delta^{2} r^2 \rfloor \BB Z)} E_\delta^r(x)   .
\eqen 
If $E_\delta^r(x)^c$ occurs, then by Lemma~\ref{prop-peel-ball} we can find $y\in \BB Z$ such that $|x-y| \geq \delta^\beta r^2$ and $\lambda_-(y)$ belongs to the glued peeling cluster started from $\BB A = \lambda_-\left([x-\delta r^2 , x ]_{\BB Z} \right)$ grown up to time $J_{\lfloor \delta r \rfloor}$. Since a glued peeling cluster contains every edge of $Q_-$ which it disconnects from $\infty$ in $Q_-$, it follows that this cluster contains at least $(\delta^\beta -\delta ) r^2$ edges of $\bdy Q_-$. By Proposition~\ref{prop-hull-moment} and the Chebyshev inequality, for each $p \in [1,3/2)$  
\eqbn
\BB P\left[ E_\delta^r(x)^c \right] \preceq \delta^{(2-\beta) p}  ,
\eqen
implicit constant depending only on $p$. 
By the union bound, 
\eqbn
\BB P\left[ (E_\delta^r)^c \right] \preceq \delta^{(2-\beta) p - 2} 
\eqen
with the implicit constant depending only on $p$ and $L$. Sending $p\rta 3/2$ gives $\BB P[(E_\delta^r)^c] \preceq \delta^{\frac32 (2-\beta) - 2 + o_\delta(1)}$, which tends to 0 as $\delta \rta 0$ provided $\beta < 2/3$. 

On the other hand, suppose $E_\delta^r $ occurs and we are given $x \in [-L r^2 , L r^2]_{\BB Z}$. Choose $x' \in [-L r^2 , L r^2]_{\BB Z} \cap (\lfloor \delta^2 r^2 \rfloor \BB Z)$ for which $x\in [x-\delta^2 r^2 , x]_{\BB Z}$. Then 
\eqbn
B_{\delta r}\left(\lambda_-(x) ; Q_{\op{zip}}   \right)  \subset  B_{ \delta r} \left( \lambda_-([x'-\delta^2 r^2 , x']_{\BB Z}) ; Q_{\op{zip}} \right)
\eqen
does not contain $\lambda_-(y)$ for any $y\in \BB Z$ with $|x-y| \geq  \delta^\beta  r^2$. 
\end{proof}

\subsubsection{H\"older continuity for distances along the boundary}

For our next two results (and at several later points in the paper) we need the following bound for the modulus of continuity of distances along the boundary of the UIHPQ$_{\op{S}}$, which follows from the scaling limit result for the UIHPQ$_{\op{S}}$ in~\cite{gwynne-miller-uihpq}. Note that the natural scaling for distances is $r^{-1}$ while the natural scaling of boundary lengths is $r^{-2}$.

\begin{lem} \label{prop-uihpq-bdy-holder}
Let $(Q_{\op{S}} , \BB e_{\op{S}})$ be an instance of the UIHPQ$_{\op{S}}$ and let $\lambda_{\op{S}} : \BB Z \rta \mcl E(\bdy Q)$ be its boundary path.  
For each $\alpha \in (0,1)$ and each $L>0$, there exists $C = C(\alpha , L ) > 0$ such that the following is true. For each $\ep >0$, there exists $r_* = r_*(\alpha , L , \ep) > 0$ such that for $r\geq r_*$, 
\eqb \label{eqn-uihpq-bdy-holder}
 \BB P\left[\frac{1}{r}  \op{dist}\left( \lambda_{\op{S}}(x ) , \lambda_{\op{S}}(y) ; Q_{\op{S}} \right) \leq C \left|\frac{x-y}{r^2} \right|^{1/2} \left( \log \left( \frac{r^2}{|x-y|} \right) \right)^2  +\ep ,\, \forall x,y \in [-Lr^2   , L r^2  ]_{\BB Z} \right] \geq 1- \alpha  .
\eqe  
The same holds (with a larger constant $C$) if we replace graph distances in $Q_{\op{S}} $ with (internal) graph distances in $B_r\left( \lambda_{\op{S}}([x ,y]_{\BB Z}) ; Q_{\op{S}} \right)$. 
\end{lem}
\begin{proof}
Since the UIHPQ$_{\op{S}}$ converges to the Brownian half-plane in the local GHPU topology~\cite[Theorem~1.12]{gwynne-miller-uihpq}, the first statement follows from~the bound \cite[Lemma~3.2]{gwynne-miller-gluing} for distances along the boundary of the Brownian disk and local absolute continuity of the Brownian half-plane with respect to the Brownian disk \cite[Proposition~4.2]{gwynne-miller-uihpq}. The second statement follows from the first by concatenating at most $C L^{1/2} $ paths of length at most $r$ between elements of $\lambda_{\op{S}}([x,y]_{\BB Z})$ to get a path from $x$ to $y$ which stays in $B_r\left( \lambda_{\op{S}}([x,y]_{\BB Z}) ; Q_{\op{S}}\right)$. 
\end{proof}

\subsubsection{Comparison of two-sided and one-sided metric balls}
\label{sec-hull-diam}
 
In this subsection we will prove an estimate for $Q_{\op{zip}}$-metric balls in term of $Q_\pm$-metric balls. 

\begin{lem} \label{prop-hull-diam}
For each $\ep > 0$, there exists $R = R(\ep) >0$ such that for each $r\in\BB N$ and each edge $e\in \mcl E(\bdy Q_-) \cap \mcl E(\bdy Q_+)$ chosen in a manner which depends only on $\bdy Q_- \cup \bdy Q_+$, 
\eqb \label{eqn-hull-diam}
\BB P\left[ B_r\left(e ; Q_{\op{zip}} \right) \subset B_{Rr}\left( e ; Q_- \right) \cup B_{Rr}\left( e ; Q_+ \right) \right] \geq 1-\ep .
\eqe 
\end{lem} 

In the statement of Proposition~\ref{prop-hull-diam}, the edge $e$ is allowed to be random so long as it is a measurable function of $\bdy Q_- \cup \bdy Q_+$.

\begin{proof}[Proof of Proposition~\ref{prop-hull-diam}]
Let $\dot Q^{J_r}$ be the radius-$r$ glued peeling cluster with initial edge set $\BB A = \{e\}$. By Lemma~\ref{prop-peel-ball},
$ B_r\left(e ; Q_{\op{zip}} \right) \subset \dot Q^{J_r} $.
Choose $x_\pm \in \BB N_0$ such that $\lambda_\pm(x_\pm) = e$.
By Proposition~\ref{prop-hull-moment}, there exists $L = L(\ep ) > 0$ such that with probability at least $1-\ep/2$,
\eqb \label{eqn-hull-diam-length}
\dot Q^{J_r} \cap \bdy Q_-\subset \lambda_-\left([x_- - L r^2 , x_- + L r^2 ]_{\BB Z} \right) 
\eqe 
and the same is true with ``$+$" in place of ``$-$." 
By Lemma~\ref{prop-uihpq-bdy-holder}, there exists $\rho  =\rho(\ep) > 0$ such that with probability at least $1-\ep/2$,  
\eqb \label{eqn-hull-diam-bdy}
\op{diam}\left( \lambda_-\left([x_- - L r^2 , x_- + L r^2 ]_{\BB Z} \right) ; Q_- \right) \leq \rho r 
\eqe 
and the same is true with ``$+$" in place of ``$-$."

Any vertex or edge in $B_r\left(e ; Q_{\op{zip}} \right) $ can be connected to $e$ by a path in $B_r\left(e ; Q_{\op{zip}} \right) $ of length at most $r$. By considering the segment of this path run until it first hits $\bdy Q_- $ or $\bdy Q_+$, we see that every vertex or edge in $B_r\left(e ; Q_{\op{zip}} \right) \cap Q_\pm$ lies at $Q_\pm$-graph distance at most $r$ from $ B_r\left(e ; Q_{\op{zip}} \right) \cap \bdy Q_\pm$. 
Hence if~\eqref{eqn-hull-diam-length} and~\eqref{eqn-hull-diam-bdy} hold, then
\eqbn
 B_r\left(e ; Q_{\op{zip}} \right)
 \subset B_r\left(  B_r\left(e ; Q_{\op{zip}} \right) \cap \bdy Q_-        ; Q_-   \right)  \cup B_r\left(  B_r\left(e ; Q_{\op{zip}} \right) \cap \bdy Q_+ ; Q_+   \right) 
 \subset B_{(\rho +1) r} \left( e ; Q_- \right) \cup  B_{(\rho +1) r} \left( e ; Q_+ \right) .
\eqen
This happens with probability at least $1-\ep$, so the statement of the lemma is satisfied with $R = \rho+1$.  
\end{proof}

\subsubsection{Lower bound for distances in a small neighborhood of the SAW}
\label{sec-saw-neighborhood-dist}

The last result of this section is a lower bound for the length of a path in $Q_{\op{zip}}$ which stays in a small neighborhood of $\bdy Q_-\cup \bdy Q_+$ (which we recall contains the SAW $\lambda_{\op{zip}}$ in the case when $\ul{\BB x} = \BB x_- = \BB x_+$, so that $Q_{\op{zip}}$ has no hole). This statement will be used in the proof of Proposition~\ref{prop-iterate-cluster-good'} to show that a $Q_{\op{zip}}$-geodesic is unlikely to spend too much time near $\bdy Q_-\cup \bdy Q_+$.

\begin{lem} \label{prop-saw-neighborhood-dist} 
Fix $L  >0$. 
For $\rho > 0$ and $r\in \BB N$, let $d_\rho^r$ be the (internal) graph metric on $B_{\rho r}\left( \lambda_-([-L r^2 , L r^2]_{\BB Z}) ; Q_{\op{zip}} \right)  $. 
For each $\alpha, \zeta \in (0,1)$, there exists $\rho_* = \rho_*( \alpha,\zeta) \in (0,1)$ such that for each $\rho \in (0,\rho_*)$ and each sufficiently large $r\in \BB N$, it holds with probability at least $1-\alpha$ that 
\eqb \label{eqn-saw-neighborhood-dist}
r^{-1}d_\rho^r\left(\lambda_-(x) , \lambda_-(y)  \right) \geq \rho^{-1+ \zeta  } \left| \frac{x-y}{r^2}   \right|^{3  + \zeta},\quad  \forall x ,y\in [-L r^2, L r^2]_{\BB Z}. 
\eqe 
\end{lem}

The basic idea of the proof of Lemma~\ref{prop-saw-neighborhood-dist} is to use Proposition~\ref{prop-hull-moment} to bound for each $k\in\BB N$ the number of $Q_{\op{zip}}$-metric balls of radius $\rho r$ which contain edges of $\lambda_-([-L r^2 , L r^2]_{\BB Z})$ separated by a boundary arc of $Q_-$ of length of order $2^k\rho^2 r^2$ (Lemma~\ref{prop-saw-nbd-count}). This will eventually allow us to show that when $\rho$ is small, $d_\rho^r$-graph distances can be bounded below in terms of the lengths of boundary arcs of $Q_-$.  We need to consider dyadic scales rather than just taking a union bound and looking at the largest possible separation between edges of $\lambda_-([-L r^2 , L r^2]_{\BB Z})$ which lie at $Q_{\op{zip}}$-distance at most $\rho r$ from each other (as in the proof of Lemma~\ref{prop-reverse-holder}) since the latter approach does not yield a lower bound for distances which can be made arbitrarily large for points $x,y$ with $|x-y| \asymp r^2$ by making $\rho$ small enough.

\begin{lem} \label{prop-saw-nbd-count}
Fix $L > 0$, $p\in (1,3/2 )$, and $\wt\zeta \in (0,p-1)$. 
For $r\in \BB N$, $\rho \in (0,1)$, and $k\in \BB Z$, let $A^r_\rho(k)$ be the set of $x \in [-L r^2, L r^2]_{\BB Z} $ for which 
\eqb \label{eqn-saw-size-set}
  \max\left\{ |x-y| \,:\, \lambda_-(y) \in B_{100 \rho r } \left(\lambda_-(x)   ; Q_{\op{zip}} \right)  \right\}  \geq  2^{k -1 } \rho^2 r^2 .
\eqe 
Also let
\eqb \label{eqn-saw-size-event}
E_\rho^r := \left\{\# A^r_\rho(k)  \leq   2^{-(p-\wt\zeta) k} r^2 ,\, \text{$\forall k \in \BB N$ with $  2^{-(p-1-\wt\zeta) k} \leq 2L \rho^{\wt\zeta  } $ } \right\} .
\eqe
For each $\alpha \in (0,1)$, there exists $\rho_0 \in (0,1)$ such that 
\eqb \label{eqn-saw-lower-prob}
\BB P\left[ E_\rho^r \right] \geq 1- \alpha ,\quad \forall \rho \in (0,\rho_0] .
\eqe  
\end{lem} 
\begin{proof}
By Lemma~\ref{prop-peel-ball}, Proposition~\ref{prop-hull-moment} (applied with $ \lfloor \rho r \rfloor$ in place of $r$), and the Chebyshev inequality, for each fixed $x \in [-L r^2, L r^2]_{\BB Z}   $, 
\eqbn
\BB P\left[ x \in A^r_\rho(k ) \right] \preceq \frac{\rho^{2p} r^{2p}}{2^{p k } \rho^{2p} r^{2p}}  = 2^{-p k}   
\eqen
with implicit constant depending only on $p$.   
Therefore, 
\eqb  
\BB E\left[ \# A^r_\rho(k)    \right] \preceq 2^{-pk} r^2 ,
\eqe 
with the implicit constant depending only on $p$ and $L$. 
We obtain~\eqref{eqn-saw-lower-prob} for small enough $\rho_0$ by applying the Chebyshev inequality to $\# A_\rho^r(k)$ for each $ k \in \BB N$ with $  2^{-(p-1-\wt\zeta) k} \leq 2L \rho^{\wt\zeta  } $ then taking a union bound.
\end{proof}

\begin{figure}[ht!]
 \begin{center}
\includegraphics[scale=1]{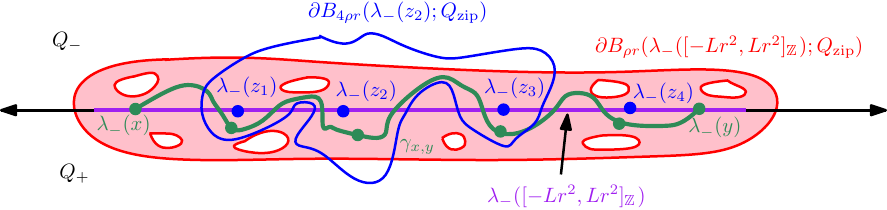} 
\caption{Illustration of the proof of Lemma~\ref{prop-saw-neighborhood-dist}. We have shown the case when $\bdy Q_-$ and $\bdy Q_+$ are glued along all of $\lambda_-([-L r^2 , L r^2]_{\BB Z})$ (purple line), but we also allow them to be glued along only part of this boundary arc, as in Figure~\ref{fig-glued-peel}, left. Given a $d_\rho^r$-geodesic $\gamma_{x,y}$, we set $N := \left\lfloor   \frac{1}{\rho r} d_\rho^r( \lambda_-(x) , \lambda_-(y) )  \right\rfloor$ and we choose for each $j=1,\dots,N$ a number $z_j \in [-Lr^2 ,Lr^2]_{\BB Z}$ such that the edge $\lambda_-(z_j)$ is close to $ \gamma_{x,y} ( \lfloor \rho r j  \rfloor )$. We get an upper bound for the quantities $|z_j - z_{j-1}|$ since we know that the $Q_{\op{zip}}$-distance between $\lambda_-(z_{j-1})$ and $\lambda_-(z_j)$ is at most $4\rho r$ and Proposition~\ref{prop-hull-moment} allows us to upper-bound how far away two points of $\bdy Q_-\cup \bdy Q_+$ which are contained in a small $Q_{\op{zip}}$-metric ball can be (see Lemma~\ref{prop-saw-nbd-count}). On the other hand, the sum of $|z_{j-1} -z_j|$ over all $j =1,\dots,N$ is at least $|x-y|$. This leads to a lower bound for $N$ and hence a lower bound for $d_\rho^r( \lambda_-(x) , \lambda_-(y) )$. 
}\label{fig-saw-nbd}
\end{center}
\end{figure}

\begin{proof}[Proof of Lemma~\ref{prop-saw-neighborhood-dist}]
See Figure~\ref{fig-saw-nbd} for an illustration of the proof.
\medskip

\noindent\textit{Step 1: defining a regularity event.}
Fix $L > 0$, $p\in (1,3/2 )$, and $\wt\zeta \in (0,p-1)$.  
For $\rho \in (0,1)$ and $r \in \BB N$, let 
\eqb \label{eqn-saw-upper-event}
G_\rho^r  := \left\{   \op{dist}\left( \lambda_-( x ) , \lambda_-(y) ; Q_- \right) \leq   \rho r   ,\, \forall x,y \in [-L r^2, L r^2 ]_{\BB Z} \: \text{with $|x-y| \leq \rho^{2+\wt\zeta} r^2$}  \right\}  
\eqe  
and let $E_\rho^r$ be the event of Lemma~\ref{prop-saw-nbd-count}. 
By Lemmas~\ref{prop-uihpq-bdy-holder} and~\ref{prop-saw-nbd-count}, there exists $\rho_1 = \rho_1(\alpha,p,\wt\zeta) \in (0,1)$ such that 
\eqb \label{eqn-saw-upper-prob}
\liminf_{r\rta\infty} \BB P\left[E_\rho^r \cap  G_\rho^r \right] \geq 1- \alpha ,\quad \forall \rho \in (0,\rho_1] .
\eqe 
It therefore suffices to show that for an appropriate choice of $p$ and $\wt\zeta$ depending only on $\zeta$,~\eqref{eqn-saw-neighborhood-dist} holds on $E_\rho^r\cap G_\rho^r$ for small enough $\rho  \in (0,\rho_1]$ (depending only on $p,\wt\zeta, L $, and $\alpha$). Henceforth assume that $E_\rho^r\cap G_\rho^r$ occurs. 
\medskip

\noindent\textit{Step 2: approximating a $d_\rho^r$-geodesic by points on $\bdy Q_-$.}
Since the estimates in Lemma~\ref{prop-saw-nbd-count} concern distances between edges on $\bdy Q_-$ (i.e., those of the form $\lambda_-(x)$) we will need to approximate distances between pairs of points on a $d_\rho^r$-geodesic by distances between pairs of such edges.
Let $x,y\in [- L r^2 , L r^2]_{\BB Z}$, and let $\gamma_{x,y}  : [0,d_\rho^r( \lambda_-(x) , \lambda_-(y) )] \rta \mcl E( Q_{\op{zip}} ) $ be a $d_\rho^r$-geodesic from $\lambda_-(x)$ to $\lambda_-(y)$. We will prove a lower bound for the length of $\gamma_{x,y}$. 
Write 
\eqb \label{eqn-saw-nbd-count}
N := \left\lfloor   \frac{1}{\rho r} d_\rho^r( \lambda_-(x) , \lambda_-(y) )  \right\rfloor . 
\eqe
By the definition of $d_\rho^r$, for each $j \in [1 ,  N-1]_{\BB Z}$ there exists $z_j \in [-L r^2,L r^2]_{\BB Z}$ such that  
\eqb \label{eqn-saw-nbd-approx}
\op{dist}\left( \lambda_-(z_j) , \gamma_{x,y} ( \lfloor \rho r j  \rfloor ) ; Q_{\op{zip}} \right) \leq    \rho r .
\eqe 
Set $z_0 = x$ and $z_N = y$, so that~\eqref{eqn-saw-nbd-approx} holds for all $j\in [0,N]_{\BB Z}$.
By the triangle inequality and since $\gamma_{x,y} $ is a geodesic for the metric $d_\rho^r$ (which is dominates the graph metric on $Q_{\op{zip}}$), for $j \in [1,  N]_{\BB Z}$ we have
\eqb \label{eqn-saw-inc-dist}
\op{dist} \left(\lambda_- (z_{j-1})  , \lambda_-(z_j)  ; Q_{\op{zip}} \right) \leq  2 \rho r + d_\rho^r\left(  \gamma_{x,y} ( \lfloor \rho r (j-1) \rfloor ) , \gamma_{x,y} ( \lfloor \rho r j \rfloor )  \right)  \leq  4 \rho r .
\eqe  
\medskip

\noindent\textit{Step 3: bounding the number of $z_j$'s which are separated by long arcs of $\bdy Q_-$.}
We now split up our points $z_j$ based on the lengths $|z_j-z_{j-1}|$ of the arcs of $\bdy Q_-$ separating $z_{j-1}$ and $z_j$.
We will use~\eqref{eqn-saw-size-event} to upper-bound the number of $j$'s for which $|z_j-z_{j-1}|$ is large. Since we know the sum of $|z_j-z_{j-1}|$ over all $j\in [1,N]_{\BB Z}$ has to be at least $|x-y|$, this will provide a lower bound for the number $N$ from~\eqref{eqn-saw-nbd-count} and hence a lower bound for $ d_\rho^r( \lambda_-(x) , \lambda_-(y) ) $. 
For $k\in\BB Z$, let 
\eqb \label{eqn-X^k-def}
X^k := \left\{j\in [1,N]_{\BB Z} \,:\,  2^{k } (4 \rho r)^2  \leq  |z_j - z_{j-1}| \leq 2^{k +1} (4\rho r)^2  \right\} .
\eqe  
If $z \in [-Lr^2,Lr^2]_{\BB Z}$ with $\lambda_-(z) \in B_{\rho r} \left(\lambda_-(z_j)  ; Q_{\op{zip}} \right)$, then by~\eqref{eqn-saw-inc-dist} and the triangle inequality, $\lambda_-(z)$ lies at $Q_{\op{zip}}$-graph distance at most $7\rho r$ from each of $z_j$ and $z_{j-1}$. By the definition of $X^k$, if $j\in X^k$ then either $|z-z_j|$ or $|z - z_{j-1}|$ is at least $2^{k-1} (4\rho r)^2$, so $z\in A_\rho^r(k)$ (defined as in~\eqref{eqn-saw-size-set}).

Since we have assumed that $G_\rho^r$ occurs, for each $j\in [1,N]_{\BB Z}$ either 
\eqbn
\lambda_-\left( \left[z_j  , z_j + \rho^{2+\wt\zeta} r^2  \right]_{ \BB Z } \right) 
\quad \op{or} \quad \lambda_-\left( \left[  z_j - \rho^{2+\wt\zeta} r^2  , z_j \right]_{ \BB Z } \right) 
\eqen
is contained in $ B_{ \rho r}\left( \lambda_- (z_j )    ; Q_{\op{zip}} \right)$, whence 
\eqb \label{eqn-ball-interval-set}
\# \left\{ B_{ \rho r}\left( \lambda_- (z_j )    ; Q_{\op{zip}} \right)  \cap \lambda_-([-L r^2 , L r^2]_{\BB Z}) \right\} \geq \rho^{2+\wt\zeta} r^2 .
\eqe  
By the discussion just after~\eqref{eqn-X^k-def}, whenever $j\in X^k$ and $z$ is such that $\lambda_-(z)$ belongs to the set in~\eqref{eqn-ball-interval-set}, it holds that $z\in A_\rho^r(k)$.
Furthermore, by the triangle inequality and since $\gamma_{x,y}$ is a geodesic, each of the sets in~\eqref{eqn-ball-interval-set} for $j\in [1,N]_{\BB Z}$ intersects at most a universal constant number of other such sets. Therefore,
\eqbn
\# A_\rho^r(k) \succeq \rho^{2+\wt\zeta} r^2 \# X^k  
\eqen  
with a universal implicit constant. 
Recalling the definition~\eqref{eqn-saw-size-event} of $E_\rho^r$, we find that for each $k\in \BB N$ with $2^{-(p-1-\wt\zeta) k} \leq 2 L \rho^{\wt\zeta} $,
\eqb \label{eqn-jump-set-bd}
 \# X^k \preceq 2^{ -(p -\wt\zeta) k } \rho^{-2-\wt\zeta}        .
\eqe 
\medskip

\noindent\textit{Step 4: upper bound for $N$.}
We will now upper-bound $|x-y| $ by $\sum_{ j\in [1,N]_{\BB Z} } |z_j - z_{j-1}|$ and use~\eqref{eqn-jump-set-bd} to say that $N$ cannot be too small relative to $|x-y|$. It turns out to be convenient to treat the values of $k$ for which $k\geq k_0$ and those for which $k < k_0$ separately, where $k_0$ is defined in~\eqref{eqn-saw-nbd-optimal} just below.
To lighten notation, set $s := p-1-\wt\zeta$. 
Fix a small constant $c \in (0,1)$, to be chosen later, and let $k_0 \in\BB N$ be chosen so that 
\eqb \label{eqn-saw-nbd-optimal}
2^{-s k_0 } r^2 \leq c \rho^{\wt\zeta} |x - y| \leq 2^{-s (k_0-1) } r^2 .
\eqe 
Since $|x-y| \leq 2 L r^2$, we have $2^{- s k_0} \leq 2 L \rho^{\wt\zeta} $. 
By breaking up the sum based on the value of $k$ for which $j\in X^k$, using that $|z_j - z_{j-1}| \preceq 2^k \rho^2 r^2$ for $j\in X^k$ (by~\eqref{eqn-X^k-def}) and applying~\eqref{eqn-jump-set-bd}, we get
\alb
|x - y| 
&\leq \sum_{ j\in [1,N]_{\BB Z} } |z_j - z_{j-1}|
\preceq \sum_{k  = k_0}^\infty 2^k \rho^2 r^2 \# X^k + 2^{ k_0} \rho^2 r^2 N   
\preceq  \rho^{-\wt\zeta} \sum_{k  = k_0}^\infty   2^{ - s k} r^2   + 2^{ k_0} \rho^2 r^2 N \\
&\preceq  \rho^{-\wt\zeta} 2^{ -s k_0} r^2   + 2^{ k_0} \rho^2 r^2 N  
\preceq  c |x-y|  + 2^{ k_0} \rho^2 r^2 N  ,
\ale
with implicit constant depending only on $p$, $\wt\zeta$, and $L$, where here in the last inequality we used~\eqref{eqn-saw-nbd-optimal}. 
If we choose $c$ sufficiently small, depending only on $p$, $\wt\zeta$, and $L$, then we can re-arrange to get $|x - y| \preceq 2^{k_0} \rho^2 r^2 N$, with the implicit constant depending on $p$, $\wt\zeta$, and $L$. Recalling the definitions of $k_0$ and $N$ from~\eqref{eqn-saw-nbd-count} and~\eqref{eqn-saw-nbd-optimal} we see that this implies that
\eqbn
\frac{|x - y|}{r^2} \preceq \rho^{1-\frac{\wt \zeta}{s}}  \left| \frac{x-y}{r^2}   \right|^{-\frac{1}{s}}  \left( r^{-1}  d_\rho^r\left(\lambda_-(x) , \lambda_-(y)  \right)  \right)
\eqen
and hence
\eqb \label{eqn-saw-nbd-s-bound}
r^{-1} d_\rho^r\left(\lambda_-(x) , \lambda_-(y)  \right)\succeq \rho^{-1+\frac{\wt \zeta}{s}}   \left| \frac{x-y}{r^2}   \right|^{ 1 + \frac{1}{s}}.
\eqe 
Recalling that $s = p-1-\wt \zeta$, if we choose $p$ sufficiently close to $3/2$ and $\wt \zeta$ sufficiently close to 0 we can arrange that $\wt\zeta/s \leq \zeta/2$ and $1 + 1/s \leq 3 + \zeta$. Then~\eqref{eqn-saw-nbd-s-bound} gives
\eqbn
r^{-1} d_\rho^r\left(\lambda_-(x) , \lambda_-(y)  \right) \succeq \rho^{-1+ \zeta/2 } \left| \frac{x-y}{r^2}   \right|^{3  + \zeta }.
\eqen
Hence~\eqref{eqn-saw-neighborhood-dist} holds on $  E^r_\rho \cap G^r_\rho$ for small enough $\rho $. 
\end{proof}

\section{Properties of geodesics in the glued map}
\label{sec-geodesic-properties}

Throughout this section we assume that $Q_{\op{zip}} = Q_- \cup Q_+$ is as in Theorem~\ref{thm-saw-conv-wedge} (equivalently, as in Section~\ref{sec-peeling-glued} with $\ul{\BB x} = \BB x_- = \BB x_+$).
We will use Proposition~\ref{prop-hull-moment} to prove two qualitative properties of the graph metric on $Q_{\op{zip}}$ which will be used in Section~\ref{sec-saw-proof} to identify the law of a subsequential limit (in the local GHPU topology) of the curve-decorated metric measure spaces in Theorem~\ref{thm-saw-conv-wedge} as the metric gluing of two Brownian half-planes. Propositions~\ref{prop-lipschitz-path} and~\ref{prop-geodesic-away} are the only results from this section which are needed in Section~\ref{sec-saw-conv}, so the latter section can be fully understood without reading the rest of the present section.

Our first result will eventually be used to show that any such subsequential limit can be mapped to the metric gluing of two Brownian half-planes via a bi-Lipschitz function.

\begin{prop}  \label{prop-lipschitz-path}
There is a universal constant $C\geq 1$ with the following property. 
For each $\alpha \in (0,1)$ and each $L>0$, there exists $\delta_* = \delta_*(\alpha,L ) > 0$ such that for each $\delta\in (0,\delta_*)$ there exists $n_* = n_*(\alpha,L ,\delta) \in \BB N$ such that the following holds for each $n\geq n_*$. Let $z_0 , z_1 \in [-L n^{1/2} , L n^{1/2}]_{\BB Z}$.
With probability at least $1-\alpha$, there exists a path $\wt\gamma$ in $Q_{\op{zip}}$ from $\lambda_-(z_0)$ to $\lambda_-(z_1)$ which crosses $\lambda_-([-Ln^{1/2} , L n^{1/2}]_{\BB Z})$ at most $2 L \delta^{-2}$ times and which has length
\eqb \label{eqn-lipschitz-path}
|\wt\gamma| \leq C \op{dist}\left( \lambda_-(z_0) , \lambda_-(z_1) ; Q_{\op{zip}} \right) + \delta^{1/2} n^{1/4} .
\eqe 
\end{prop} 

The quantity $2L\delta^{-2}$ in the proposition statement comes from the fact that in the proof, we will divide $[-Ln^{1/2} , Ln^{1/2}]_{\BB Z}$ into $2L\delta^{-2}$ intervals of length $\delta^2 n^{1/2}$ and consider a glued peeling cluster centered at each such interval. 

When we apply Proposition~\ref{prop-lipschitz-path}, we will first rescale both sides by $n^{-1/4}$, take a (subsequential) limit as $n \to \infty$, and then finally let $\delta \to 0$.  We emphasize that when we take limits in this order, we do \emph{not} have to send $C$ to $\infty$ to get an event which occurs with probability close to $1$. This is important because it will allow us to get a uniform Lipschitz constant for a map from a subsequential scaling limit of $Q_{\op{zip}}$ to the metric gluing of the scaling limits of $Q_-$ and $Q_+$. 

Our next result gives a uniform lower bound for the amount of time a $Q_{\op{zip}}$-geodesic spends away from $\bdy Q_- \cup \bdy Q_+$.

\begin{prop} \label{prop-geodesic-away}
There is a universal constant $\beta \in (0,1)$ such that the following is true.  
For each $\alpha \in (0,1)$ and each $L>0$, there exists $\delta_* = \delta_*(\alpha,L ) > 0$ such that for each $\delta\in (0,\delta_*]$, there exists $n_* = n_*(\alpha,L ,\delta) \in \BB N$ such that the following holds for each $n\geq n_*$. Let $z_0 , z_1 \in [-L n^{1/2} , L n^{1/2}]_{\BB Z}$. For each $Q_{\op{zip}}$-geodesic $\gamma$ from $\lambda_-(z_0)$ to $\lambda_-(z_1)$, let $T_\gamma^\beta(\delta)$ be the set of times $t\in [1,|\gamma|]_{\BB Z}$ such that $\gamma(t)$ lies at $Q_{\op{zip}}$-distance at least $\beta \delta n^{1/4}$ from $\lambda_-([-L n^{1/2}   , L n^{1/2}  ]_{\BB Z})$. 
With probability at least $1-\alpha$, for each such geodesic $\gamma$ it holds that 
\eqb \label{eqn-geodesic-away}
\# T_\gamma^\beta(\delta) \geq     \beta |\gamma|  - \delta^{1/2} n^{1/4} .
\eqe 
\end{prop} 

Note that we do not prove that the fraction of time that a $Q_{\op{zip}}$ geodesic spends in $\bdy Q_- \cup \bdy Q_+$ tends to~$0$ as $n\rta\infty$.  Rather, Proposition~\ref{prop-geodesic-away} only implies that the fraction of time that a $Q_{\op{zip}}$ geodesic spends in $\partial Q_- \cup \partial Q_+$ is uniformly bounded away from $1$.  In our application of Proposition~\ref{prop-geodesic-away}, we will take limits in the same order as in the case of Proposition~\ref{prop-lipschitz-path}.  Thus, as in the case of Proposition~\ref{prop-lipschitz-path}, we do not have to send $\beta \rta 0$ in order to get an event which occurs with probability close to $1$.

The proofs of Propositions~\ref{prop-lipschitz-path} and~\ref{prop-geodesic-away} proceed via similar arguments. We will show in Section~\ref{sec-geo-iterate} that, roughly speaking, the following is true. If we grow the glued peeling clusters $\{\dot Q^j\}_{j\in\BB N}$ started from a given arc $\BB A \subset \mcl E(\bdy Q_-\cup \bdy Q_+)$, then with high probability there exists a radius $r \in \BB N$ which is not too much bigger than $(\#\BB A)^{1/2}$ such that a certain ``good" event occurs. In the case of Proposition~\ref{prop-lipschitz-path}, this event corresponds to the existence of a path in $Q_{\op{zip}}$ between any two points of $\bdy \dot Q^{J_r}$ which crosses $\bdy Q_-\cup \bdy Q_+$ at most once and whose length is at most a constant times $r$. In the case of Proposition~\ref{prop-geodesic-away}, this event amounts to the requirement that every $Q_{\op{zip}}$-geodesic from a point of $\bdy \dot Q^{J_r}$ to a point near $\BB A$ must make an excursion away from $\bdy Q_- \cup \bdy Q_+$ of time length at least a small constant times $r$. 
In Section~\ref{sec-geo-proof}, we will prove Propositions~\ref{prop-lipschitz-path} and~\ref{prop-geodesic-away} by arguing that most of the intersection of the geodesic with the SAW can be covered by the good scales of Section~\ref{sec-geo-iterate}.

\subsection{Existence of a good scale}
\label{sec-geo-iterate}

Fix a finite connected arc $\BB A\subset \mcl E(\bdy Q_-\cup \bdy Q_+)$. Define the glued peeling clusters $\{\dot Q^j\}_{j\in\BB N_0}$ started from $\BB A$, the stopping times $\{J_r\}_{r\in\BB N_0}$, the complementary UIHPQ$_{\op{S}}$'s $\{Q_-^j\}_{j\in\BB N_0}$ and $\{Q_+^j\}_{j\in\BB N_0}$, and the $\sigma$-algebras $\{\mcl F^j\}_{j \in \BB N_0}$ as in Section~\ref{sec-glued-peeling}.   
In this subsection we will prove two lemmas to the effect that there typically exists a radius $r \in \BB N$ for which a certain good condition is satisfied. Our first lemma is needed for the proof of Proposition~\ref{prop-lipschitz-path}. 
 
\begin{lem} \label{prop-good-radius} 
For $C> 1$, let $ R(C)$ be the smallest $r \geq (\#\BB A)^{1/2}$ for which the following are true. 
\begin{enumerate}
\item $ \op{diam}\left( \bdy \dot Q^{J_r} \cap Q_{\xi} ; Q_{\xi}\right) \leq Cr$ for each $\xi \in \{\pm\}$. \label{item-good-radius-diam}
\item $\#\mcl E\left(  \dot Q^{J_r} \cap \left( \bdy Q_- \cup \bdy Q_+ \right) \right) \leq C^2 r^2$.  \label{item-good-radius-length}
\end{enumerate}
For each $p \in [1,3/2)$ there exists $C = C(p)  > 1$ such that for each $S > 0$, 
\eqb \label{eqn-good-radius}
\BB P\left[ R(C)  > (\#\BB A)^{1/2} S \right] \preceq S^{-2p} 
\eqe 
with implicit constant depending only on $p$. 
\end{lem}

The key condition in Lemma~\ref{prop-good-radius} for the proof of Proposition~\ref{prop-lipschitz-path} is condition~\ref{item-good-radius-diam}, which says that restricting attention to paths in $\dot Q^{J_r}$ which do not cross the gluing interface increases distances along $\bdy \dot Q^{J_r}$ by a factor of at most $C$. Note that a path which does not cross the gluing interface is contained in either $Q_-$ or $Q_+$. This will be used in the proof of Proposition~\ref{prop-lipschitz-path} to re-route a $Q_{\op{zip}}$-geodesic in such a way that it crosses the gluing interface a constant order number of times and its length is increased by a factor of at most $C$.

It is crucial for our purposes that the probabilistic estimate in~\eqref{eqn-good-radius} holds for some $p > 1$.  The reason for this is as follows. We will eventually take a union bound over several different choices of the initial edge set $\BB A$ in order to cover an interval along the SAW by clusters of the form $\dot Q^{J_{R(C)}}$, for varying choices of $\BB A$, most of which do not contain either of the endpoints of the interval (see Lemma~\ref{prop-good-radius-exist}). The requirement that $p\in [1,3/2)$ comes from the fact that we get moments up to order $3/2$ in Proposition~\ref{prop-hull-moment}.
 
Our second lemma is needed for the proof of Proposition~\ref{prop-geodesic-away}. 

\begin{lem} \label{prop-good-radius'}
For $C> 1$, let $\wt R(C)$ be the smallest $r \geq (\#\BB A)^{1/2} $ for which the following are true. 
\begin{enumerate}
\item Each $Q_{\op{zip}}$-geodesic $\gamma$ from an edge of $Q_{\op{zip}}$ lying at $Q_{\op{zip}}$-graph distance at most $(\# \BB A)^{1/2}$ from $\BB A$ to an edge of $\bdy \dot Q^{J_r}$ hits a vertex of $Q_{\op{zip}}$ which lies at $Q_{\op{zip}}$-graph distance at least $C^{-1} r$ from $\bdy Q_-\cup \bdy Q_+$. \label{item-good-radius'-geo}
\item $\#\mcl E\left(  \dot Q^{J_r} \cap \left( \bdy Q_- \cup \bdy Q_+ \right) \right) \leq C^2 r^2$. \label{item-good-radius'-length}
\item $\op{diam}\left( \bdy \dot Q^{J_r} ; Q_{\op{zip}} \right) \leq C r$. \label{item-good-radius'-diam}
\end{enumerate}
For each $p \in [1,3/2)$ there exists $C = C(p)  > 1$ such that for each $S > 0$, 
\eqb \label{eqn-good-radius'}
\BB P\left[ \wt R(C)  > (\#\BB A)^{1/2} S \right] \preceq S^{-2p} 
\eqe 
with implicit constant depending only on $p$. 
\end{lem}

The key condition in Lemma~\ref{prop-good-radius'} for the proof of Proposition~\ref{prop-geodesic-away} is condition~\ref{item-good-radius'-geo}, which says that a $Q_{\op{zip}}$-geodesic started outside of $\dot Q^{J_r}$ cannot get close to $\BB A$ without first spending at least $C^{-1} r$ units of time away from $\bdy Q_-\cup \bdy Q_+$. In the proof of Proposition~\ref{prop-geodesic-away}, this will be used to show that such a geodesic cannot spend most of its time tracing $\bdy Q_-\cup \bdy Q_+$. As in Lemma~\ref{prop-good-radius}, it is crucial here that~\eqref{eqn-good-radius'} holds for some $p>1$.

\begin{remark}
Let $R_*(C)$ be the smallest $r\geq (\#\BB A)^{1/2}$ for which the conditions of Lemmas~\ref{prop-good-radius} and~\ref{prop-good-radius'} hold simultaneously.
It is possible to show that the conclusion of Lemma~\ref{prop-good-radius} (equivalently, that of Lemma~\ref{prop-good-radius'}) holds for $R_*(C)$, i.e., one can prove a single stronger lemma which supersedes both Lemmas~\ref{prop-good-radius} and~\ref{prop-good-radius'}. 
However, we find it more clear to treat Lemmas~\ref{prop-good-radius} and~\ref{prop-good-radius'} separately since (a) one is used in the proof of Proposition~\ref{prop-lipschitz-path} and the other is used in the proof of Proposition~\ref{prop-geodesic-away} and (b) the proof of Lemma~\ref{prop-good-radius'} is much more involved than that of Lemma~\ref{prop-good-radius}. 
\end{remark}

Lemmas~\ref{prop-good-radius} and~\ref{prop-good-radius'} are the only results in this subsection which are needed for the proofs of Propositions~\ref{prop-lipschitz-path} and~\ref{prop-geodesic-away}, so the reader does not have to read the proofs of either before reading the rest of the paper.

To prove the lemmas, we will work with scales of approximately exponential size in $k$ and prove that for a large enough choice of $C$, the conditions in the definitions of the times $R(C)$ and $\wt R(C)$ of Lemma~\ref{prop-good-radius} and~\ref{prop-good-radius'} have probability close to $1$ to be satisfied at each scale.

\begin{figure}[ht!]
 \begin{center}
\includegraphics[scale=1]{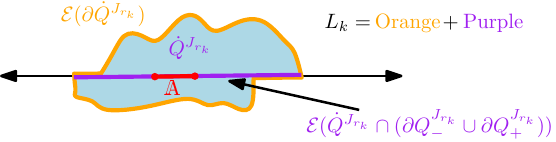} 
\caption{Illustration of the radii $r_k$ and the quantities $L_k$ defined in~\eqref{eqn-radius-iterate-r-L-def} for a single glued peeling cluster $\dot Q^{J_{r_k}}$. Note that in the definition of $L_k$, the overlap of the orange and purple arcs is double counted.}\label{fig-radius-iterate}
\end{center}
\end{figure}

More precisely, we will consider the following setup.  Let $r_0 = 0$ and $L_0 =  \#\BB A$. Inductively, if $k\in\BB N$ and $r_{k-1}$ and $L_{k-1}$ have been defined, let 
\eqb \label{eqn-radius-iterate-r-L-def}
r_k :=   2 r_{k-1} + \lceil L_{k-1}^{1/2} \rceil \quad\text{and}\quad L_k := \# \mcl E\left(\bdy \dot Q^{J_{r_k}} \right) +  \#\mcl E\left(  \dot Q^{J_{r_{k }}} \cap \left( \bdy Q_-^{J_{r_{k-1}}} \cup \bdy Q_+^{J_{r_{k-1}}} \right) \right).
\eqe
That is, $L_k$ gives the boundary length of the glued peeling cluster at the $k$th stage plus the number of edges of the gluing interface and the glued peeling cluster at the previous stage that have been separated from $\infty$ on at least one of the two sides of the gluing interface. 
See Figure~\ref{fig-radius-iterate} for an illustration. The term $\lceil L_{k-1}^{1/2} \rceil$ in the definition of $r_k$ is likely to be comparable to $r_{k-1}$ (Proposition~\ref{prop-hull-moment}), so it is likely that the graph distance across the annulus $\dot{Q}^{J_{r_k}} \setminus \dot{Q}^{J_{r_{k-1}}}$ is comparable to $r_{k-1}$.

The proof of Lemma~\ref{prop-good-radius} is given in Section~\ref{sec-good-radius}, and follows roughly the following outline. 
\begin{enumerate}[1.]
\item[Step 1:] We define for each $C> 1$ and $k\in\BB N$ a ``good" event $E_k(C)$ which belongs to the $\sigma$-algebra $\mcl F^{J_{r_k}}$ of~\eqref{eqn-peel-filtration} and such that  the conditions in the definition of $R(C)$ are satisfied with $r =r_{k-1}$ provided $E_k(C)$ occurs (Lemma~\ref{prop-good-radius-compare}). The event $E_k(C)$ is a slightly modified version of the conditions in the definition of $R(C)$ with $r= r_{k-1}$, with the modifications made so that $E_k(C) \in \mcl F^{J_{r_k}}$. 
\item[Step 2:] \label{item-good-radius-step-prob} We show that $\BB P[E_k(C) \,|\, \mcl F_{J_{r_{k-2}}} ]$ is close to 1 if the constant $C$ is chosen to be sufficiently large (Lemma~\ref{prop-iterate-cluster-good}). 
\item[Step 3:] We multiply the estimate from \hyperref[item-good-radius-step-prob]{Step~2} over a logarithmic number of values of $k$ (so that one of the $E_k(C)$ is very likely to occur or at least one scale) and apply Lemma~\ref{prop-iterate-cluster-length} just below to bound the value of $r_k$ at the last scale. This leads to Lemma~\ref{prop-good-radius}. 
\end{enumerate}
Lemma~\ref{prop-good-radius'} is proven in Section~\ref{sec-good-radius'} using a similar argument, with the same radii $r_k$ but with the events $E_k(C)$ replaced by different events. The main difference in this case is that the analog of \hyperref[item-good-radius-step-prob]{Step 2} (Lemma~\ref{prop-iterate-cluster-good'}) is much more involved, and requires us to use most of the estimates in Section~\ref{sec-hull-moment-misc}. 
The main estimate we need for the radii $r_k$ is the following lemma, which says that they typically grow at most an exponential rate.

\begin{lem} \label{prop-iterate-cluster-length}
For each $p\in [1,3/2)$, there exists a constant $A_p > 1$ depending only on $p$ such that for each $k\in\BB N$, 
\eqbn
\BB E\left[ r_k^{2p} \right] \leq  A_p^k (\#\BB A)^{p} .
\eqen
\end{lem}
\begin{proof}
We first observe that for each $k\in\BB N$, $\{\dot Q^{J_{r + r_{k-1}}} \setminus \dot Q^{J_{r_{k-1}}}\}_{r \geq 0}$ is the set of clusters of the glued peeling process in the glued map $Q_-^{J_{r_k-1}} \cup Q_+^{J_{r_k-1}} $ started from the initial edge set $\bdy \dot Q^{J_{r_{k-1}}}$, which has cardinality at most $  L_{k-1} \leq r_k^2$.  
The cluster $\dot Q^{J_{r_k}}$ is obtained by growing this peeling process up to radius $r_k - r_{k-1}  \leq r_k$.  
By Lemma~\ref{prop-peel-law} and Proposition~\ref{prop-hull-moment}, we can find $\wt A_p > 0$, depending only on $p$, such that for each $k\in \BB N$, 
\eqb \label{eqn-iterate-moment}
\BB E\left[  L_{k}^p \,|\, \mcl F^{J_{r_{k-1}}} \right] 
\leq \wt A_p  r_k^{2p}     .
\eqe 
Since $r_{k+1} = 2r_k + \lceil L_k^{1/2} \rceil$,   
\eqbn
\BB E\left[ r_{k+1}^{2p} \right] \leq 2^{4 p-1} \BB E\left[ r_k^{2p } \right] + 2^{2p -1} \BB E\left[ L_k^{p } \right]  \leq A_p \BB E\left[ r_k^{2p} \right] 
\eqen
for a constant $A_p > 1$ as in the statement of the lemma.  Since $r_1 =  (\#\BB A)^{1/2}$, iterating this estimate $k$ times yields the statement of the lemma. 
\end{proof} 

\subsubsection{Proof of Lemma~\ref{prop-good-radius}}
\label{sec-good-radius}

\begin{figure}[ht!]
 \begin{center}
\includegraphics[scale=.8]{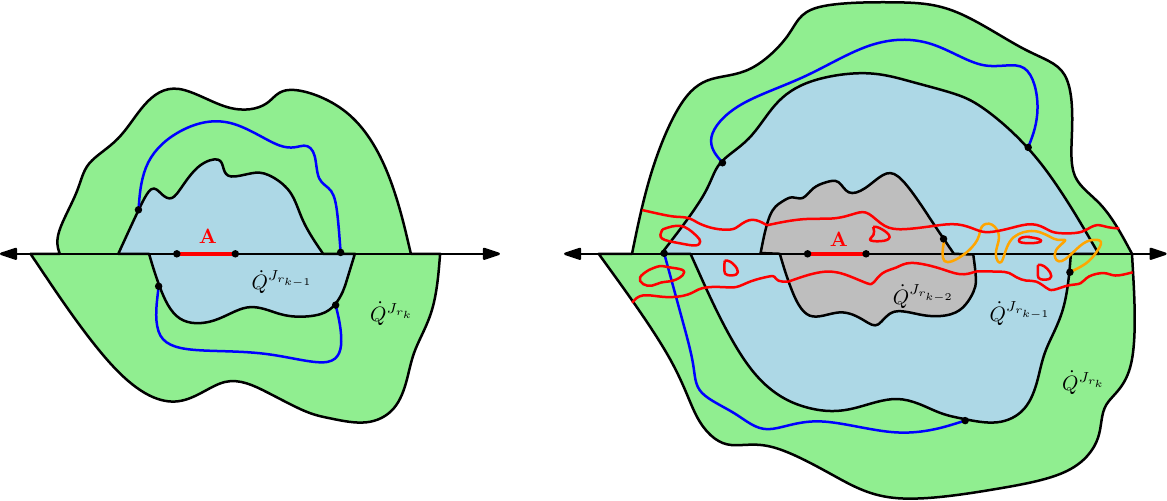} 
\caption[Events used in the proofs of Lemmas~\ref{prop-good-radius} and~\ref{prop-good-radius'}]{\textbf{Left:} Illustration of the event $E_k(C)$ used in the proof of Lemma~\ref{prop-good-radius}. If $E_k(C)$ occurs, then $L_k$ is not too much bigger than $r_{k-1}^2$ and any two points in $\bdy \dot Q^{J_{r_{k-1}}} \cap Q_\pm $ can be connected by a path of length at most $C r_{k-1}$ which stays in $\dot Q^{J_{r_k}}$ and does not cross $\bdy Q_- \cup \bdy Q_+$ (two such paths are shown in blue). The paths in the figure stay in the annulus $\dot Q^{J_{r_k}} \setminus \dot Q^{J_{r_{k-1}}}$; our proof shows that we can arrange for this to be the case, but it is not necessary for the proof of Proposition~\ref{prop-lipschitz-path}. A similar comment applies in the illustration on the right.
\textbf{Right:} Illustration of the event $\wt E_k(C)$ used in the proof of Lemma~\ref{prop-good-radius'}. On $\wt E_k(C)$, $L_k$ is not too much bigger than $r_{k-2}^2$; any two points in $\bdy \dot Q^{J_{r_{k-1}}} \cap Q_\pm$ can be connected by a path of length at most $C r_{k-1}$ which stays in $\dot Q^{J_{r_k}}$ and does not cross $\bdy Q_- \cup \bdy Q_+$; and every geodesic with respect to the internal graph metric on $\dot Q^{J_{r_k}}$ from a point of $\bdy \dot Q^{J_{r_{k-1}}}$ to a point of $\bdy \dot Q^{J_{r_{k-2}}}$ (such as the one shown in orange) must exit the $C^{-1/2} r_{k-2}$-neighborhood of $\dot Q^{J_{r_k}} \cap (\bdy Q_-\cup \bdy Q_+)$ (outlined in red in the figure).  
}\label{fig-good-scale}
\end{center}
\end{figure}

In order to prove Lemma~\ref{prop-good-radius}, we will consider the following events defined in terms of the quantities $r_k$ and $L_k$ of~\eqref{eqn-radius-iterate-r-L-def}. See Figure~\ref{fig-good-scale} for an illustration. 
For $k\in\BB N$ and $C>8$, let $E_k(C)$ be the event that the following are true. 
\begin{enumerate}
\item We have $L_{k-1} \leq \frac13 \left( C^2 - 8 \right) r_{k-1}^2$. \label{item-radius-iterate-length0}
\item The diameter of $\bdy \dot Q^{J_{r_{k-1} } } \cap Q_- $ with respect to the graph metric on $\dot Q^{J_{r_k}} \cap Q_-$ is at most $C r_{k-1} $; and the same is true with ``$+$" in place of ``$-$."  
\end{enumerate}
Note that $E_k(C) $ belongs to the $\sigma$-algebra $\mcl F^{J_{r_k}}$ defined as in~\eqref{eqn-peel-filtration} (which is why we measure distances with respect to the graph metric on $\dot Q^{J_{r_{k}}} \cap Q_\pm$, rather than that on all of $Q_\pm$).  Let $K(C)$ be the smallest $k\geq 2$ for which $E_k(C)$ occurs.

We now proceed to complete the proof of Lemma~\ref{prop-good-radius} following the outline indicated above by showing that the good radius $R(C)$ occurs before $r_{K(C)-1}$ (Lemma~\ref{prop-good-radius-compare}) and then by obtaining a uniform lower bound for the conditional probability of $E_k(C)$ given $\mcl F^{J_{r_{k-2}}}$, which allows us to stochastically dominate $K(C)$ by a geometric random variable.

\begin{lem} \label{prop-good-radius-compare}
For each $C>8$, we have $R(C) \leq r_{K(C)-1}$, with $R(C)$ as in Lemma~\ref{prop-good-radius}. 
\end{lem}
\begin{proof}
We will show that if $k\geq 2$ and $E_k(C)$ occurs, then the conditions in the definition of $R(C)$ are satisfied for $r = r_{k-1}$. By definition, $r_{k-1} \geq r_1 = (\#\BB A)^{1/2}$ for $k\geq 2$, so we just need to check conditions~\ref{item-good-radius-diam} and~\ref{item-good-radius-length} in the definition of $R(C)$. 

Since the graph metric on $Q_-$ restricted to $\dot Q^{J_{r_k}}\cap Q_-$ is bounded above by the graph metric on $\dot Q^{J_{r_k}} \cap Q_-$, if $E_k(C)$ occurs then
\eqbn
\op{diam}\left( \bdy \dot Q^{J_{r_{k-1}}} \cap Q_- ; Q_-\right) \leq   C r_{k-1}.
\eqen
Symmetrically, the same is true with ``$+$" in place of ``$-$". That is, condition~\ref{item-good-radius-diam} in the definition of $R(C)$ holds for $r = r_{k-1}$. 

We now check condition~\ref{item-good-radius-length} in the definition of $R(C)$ for $r = r_{k-1}$. 
For this purpose, we first observe that for any $k\in\BB N$, 
\eqb \label{eqn-iterate-radius-sum}
r_k = \sum_{i=0}^{k-1} 2^{k-1-i} \lceil L_{i}^{1/2} \rceil \geq \sum_{i=0}^{k-1}   L_{i}^{1/2}  \geq \left( \sum_{i=0}^{k-1} L_i \right)^{1/2} 
\eqe 
where in the last inequality we used that $x\mapsto x^{1/2}$ is concave, hence subadditive. 

We now argue that
\eqb \label{eqn-iterate-radius-contain}
\mcl E\left(  \dot Q^{J_{r_{k-1}}} \cap \left( \bdy Q_- \cup \bdy Q_+ \right) \right) \subset \bigcup_{i=1}^{k-1} \mcl E\left(  \dot Q^{J_{r_{i }}} \cap \left( \bdy Q_-^{J_{r_{i-1}}} \cup \bdy Q_+^{J_{r_{i-1}}} \right) \right). 
\eqe
Indeed, suppose $e\in  \mcl E\left(  \dot Q^{J_{r_{k-1}}} \cap \left( \bdy Q_- \cup \bdy Q_+ \right) \right)$ and let $i_e$ be the smallest $i\in [1,k-1]_{\BB Z}$ for which $e \in  \dot Q^{J_{r_{i }}}$. Then $e \notin \dot Q^{J_{r_{i_e-1}}}$ so since $e\in \bdy Q_- \cup \bdy Q_+$, $e$ must belong to $\bdy Q_-^{J_{r_{i_e-1}}} \cup \bdy Q_+^{J_{r_{i_e-1}}}$.

By~\eqref{eqn-radius-iterate-r-L-def}, \eqref{eqn-iterate-radius-sum}, and~\eqref{eqn-iterate-radius-contain},
\eqb \label{eqn-iterate-saw-sum}
\# \mcl E\left(  \dot Q^{J_{r_{k-1}}} \cap \left( \bdy Q_- \cup \bdy Q_+ \right) \right) \leq \sum_{i=0}^{k-1} L_i \leq r_k^2 .
\eqe  
By~\eqref{eqn-radius-iterate-r-L-def} and the elementary inequality $(a+b)^2 \leq 2a^2  + 2b^2$, we have $r_k^2  = (2r_{k-1} + \lceil L_{k-1}^{1/2} \rceil )^2 \leq 8 r_{k-1}^2 + 2\lceil L_{k-1}^{1/2} \rceil^2 \leq 8 r_{k-1}^2 + 3L_{k-1}^2 $. 
By combining this with~\eqref{eqn-iterate-saw-sum} and condition~\ref{item-radius-iterate-length0} in the definition of $E_k(C)$, we get that if $E_k(C)$ occurs, then
\eqbn
\# \mcl E\left(  \dot Q^{J_{r_{k-1}}} \cap \left( \bdy Q_- \cup \bdy Q_+ \right) \right) \leq   r_k^2 \leq 8 r_{k-1}^2 +  3 L_{k-1}  \leq C^2 r_{k-1}^2 , 
\eqen
which is condition~\ref{item-good-radius-length} in the definition of $R(C)$ for $r=r_{k-1}$. 
Thus the conditions in the definition of $R(C)$ are satisfied for $r = r_{k-1}$. The result follows by the minimality of $R(C)$.
\end{proof}

We next prove a lower bound for the probability of the events $E_k(C)$, which in particular implies that the time $K(C)$ is stochastically dominated by a geometric random variable with success probability which can be made arbitrarily close to $1$ by making $C$ sufficiently large.

\begin{lem} \label{prop-iterate-cluster-good}
For each $\alpha \in (0,1)$, there exists $C = C(\alpha) > 8$ such that for each $k \geq 2$,  
\eqbn
\BB P\left[ E_k(C) \,|\, \mcl F^{J_{r_{k-2}}} \right] \geq 1-\alpha .
\eqen
\end{lem}
\begin{proof}
By~\eqref{eqn-iterate-moment} and the Chebyshev inequality, we can find $C_0 = C_0(\alpha) \geq 1$ such that with conditional probability at least $1-\alpha/2$ given $\mcl F^{J_{r_{k-2}}}$, 
\eqb \label{eqn-cluster-length-good}
  L_{k-1} \leq C_0 r_{k-1}^2 . % \quad \op{and} \quad  L_{k } \leq C_0 r_k^2.
\eqe  
The first relation in~\eqref{eqn-cluster-length-good} immediately implies that condition~\ref{item-radius-iterate-length0} in the definition of $E_k(C)$ holds for each $C \geq \sqrt{3C_0+8}$.
 
Since the restriction of the graph metric on $Q_-^{J_{r_k-1}} \cup Q_+^{J_{r_k-1}} $ to $Q_-^{J_{r_{k-1}}}$ is bounded above by the graph metric on $Q_-^{J_{r_{k-1}}}$, since $r_k- r_{k-1} \geq r_{k-1}$, and by Lemma~\ref{prop-peel-ball},  
\eqbn
B_{r_{k-1} }\left(   \bdy \dot Q^{J_{r_{k-1}}} \cap Q_-^{J_{r_{k-1}} }   ; Q_-^{J_{r_{k-1}}} \right)  \subset  \dot Q^{J_{r_k}}  \cap Q_- .
\eqen
Since the conditional law of $ Q_-^{J_{r_{k-1}}}$ given $\mcl F^{J_{r_{k-1}}}$ is that of a UIHPQ$_{\op{S}}$, 
Lemma~\ref{prop-uihpq-bdy-holder} implies that we can find $C_1 = C_1(\alpha) \geq 0$ such that with conditional probability at least $1-\alpha/4$ given $\mcl F^{J_{r_{k-1}}}$, the diameter of $\bdy \dot Q^{J_{r_{k-1}}} \cap Q_-^{J_{r_{k-1}} }$ with respect to the graph metric on $\dot Q^{J_{r_k}} \cap Q_-$ is at most $C_1  L_{k-1}^{1/2} $. By symmetry, the same holds with ``$+$" in place of ``$-$." With conditional probability at least $1-\alpha$ given $\mcl F^{J_{r_{k-2}}}$, this condition holds for both $-$ and $+$ and the event in~\eqref{eqn-cluster-length-good} occurs. If this is the case, then   
\eqbn
\op{diam}\left( \bdy \dot Q^{J_{r_{k-1} } } \cap Q_\xi  ;  \dot Q^{J_{r_k}} \cap Q_\xi \right) \leq C_1 L_{k-1}^{1/2} \leq C_0^{1/2}  C_1 r_{k-1} ,  \quad \forall \xi \in \{-,+\} .
\eqen 
%and $L_k \leq 2 C_0\left( 4 + C_0 \right) r_{k-1}^2$. 
Hence $E_k(C)$ occurs for $C = \max\{C_0^{1/2} C_1 , \sqrt{3C_0+8}\}$.  
\end{proof}

\begin{proof}[Proof of Lemma~\ref{prop-good-radius}]
By Lemma~\ref{prop-good-radius-compare}, $R(C) \leq r_{K(C)}$ so it suffices to bound $r_{K(C)}$ for an appropriate $C = C(p) > 8$. 
Fix $1 < p  < p' < 3/2$ and let $A_{p'}$ be as in Lemma~\ref{prop-iterate-cluster-length} with $p'$ in place of $p$. 
Let $\alpha \in (0,1)$ be a small parameter, to be chosen later depending only on $p$ and $p'$, and let $C = C(\alpha)>8$ be as in Lemma~\ref{prop-iterate-cluster-good}. 
By Lemma~\ref{prop-iterate-cluster-good}, for each $k\in \BB N$,
\eqbn
\BB P\left[ K(C) > k \right] \leq \alpha^{\lfloor k /2 \rfloor} .
\eqen
For $S>1$, let
\eqbn
k_S = \frac{4p \log S }{\log\alpha^{-1} }  +1
\eqen
so that $\BB P\left[ K(C) > k_S \right] \leq S^{-2p} $.
By Lemma~\ref{prop-iterate-cluster-length} and the Chebyshev inequality,  
\eqbn
\BB P\left[ r_{k_S} > ( \#\BB A )^{1/2} S \right] \leq A_{p'}^{k_S}   S^{-2p'}  \leq S^{-2p' + o_\alpha(1)}  
\eqen 
where the rate at which the $o_\alpha(1)$ term tends $0$ as $\alpha \to 0$ depends only on $p$ and $p'$. By choosing $\alpha$ sufficiently small (and hence $C$ sufficiently large), depending only on $p$ and $p'$, we can arrange that this $o_\alpha(1)$ is smaller than $p'-p$. Recalling that $R(C) \leq r_{K(C)}$, we get
\alb
\BB P\left[ R(C) >  (\#\BB A)^{1/2} S  \right] 
\leq \BB P\left[ r_{K(C)} >  (\#\BB A)^{1/2} S  \right] 
\leq \BB P\left[ r_{k_S} >  (\#\BB A)^{1/2} S  \right]  + \BB P\left[ K(C) > k_S \right] 
\preceq S^{-2p}.
\ale 

\qedhere
\end{proof}

\subsubsection{Proof of Lemma~\ref{prop-good-radius'}}
\label{sec-good-radius'}

The proof of Lemma~\ref{prop-good-radius'} follows a similar outline as the proof of Lemma~\ref{prop-good-radius}, but we work with different events which are somewhat more complicated. See Figure~\ref{fig-good-scale} for an illustration of the definition of these events.

For $k\in\BB N$, define $r_k$ and $L_k$ as in~\eqref{eqn-radius-iterate-r-L-def}. 
Also let $d_k$ be the (internal) graph metric on $\dot Q^{J_{r_k } } $. For $C> 8$, let $\wt E_k(C)$ be the event that the following are true. 
\begin{enumerate}
\item $ L_{k-2} \vee L_{k-1} \vee L_k \leq \frac12 \left( C^2 - 8 \right) r_{k-2}^2$ and $r_k \leq C^{1/2} r_{k-2}$. \label{item-radius-iterate-length}  
\item The diameter of $\bdy \dot Q^{J_{r_{k-1} } } \cap Q_- $ with respect to the graph metric on $\dot Q^{J_{r_k}} \cap Q_-$ is at most $C r_{k-2} $; and the same is true with ``$+$" in place of ``$-$."  \label{item-radius-iterate-diam}
\item No $d_k$-geodesic from a vertex of $\bdy \dot Q^{J_{r_{k-1 }}}$ to a vertex of $\bdy \dot Q^{J_{r_{k-2}}}$ is contained in the $ C^{-1/2}  r_{k-2}   $-neighborhood (with respect to $d_k$) of $\dot Q^{J_{r_k}} \cap \left( \bdy Q_- \cup \bdy Q_+ \right)$. \label{item-radius-iterate-geo}
\end{enumerate}
As in the case of the event $E_k(C)$ from Section~\ref{sec-good-radius}, the event $\wt E_k(C) $ belongs to the $\sigma$-algebra $ \mcl F^{J_{r_k}}$ defined as in~\eqref{eqn-peel-filtration}.   

Let $\wt K(C)$ be the smallest $k\geq 3$ for which $\wt E_k(C)$ occurs.  
The following lemma, which is the analog of Lemma~\ref{prop-good-radius-compare} in this setting, is the reason for our interest in the above events.
 
\begin{lem} \label{prop-good-radius-compare'}
For each $C>8$, we have $\wt R(C) \leq r_{\wt K(C)-1}$, with $\wt R(C)$ as in Lemma~\ref{prop-good-radius'}. 
\end{lem}
\begin{proof}
Suppose $k\geq 3$ is such that $\wt E_k(C)$ occurs. We have $r_{k-1} \geq (\#\BB A)^{1/2}$ by definition. 
By condition~\ref{item-radius-iterate-length} in the definition of $\wt E_k(C)$ together with~\eqref{eqn-iterate-saw-sum},
\eqbn
\# \mcl E\left(  \dot Q^{J_{r_{k-1}}} \cap \left( \bdy Q_- \cup \bdy Q_+ \right) \right) \leq r_k^2 \leq 8 r_{k-1}^2 + 2 L_k \leq C^2 r_{k-1}^2 .  
\eqen
Thus condition~\ref{item-good-radius'-length} in the definition of $\wt R(C)$ is satisfied for $r = r_{k-1}$. It is clear from condition~\ref{item-radius-iterate-diam} in the definition of $\wt E_k(C)$ that condition~\ref{item-good-radius'-diam} in the definition of $\wt R(C)$ is also satisfied. 

Now we will check condition~\ref{item-good-radius'-geo}. 
Note that this is not quite immediate from condition~\ref{item-radius-iterate-geo} in the definition of $\wt E_k(C)$ since the definition of $E_k(C)$ involves $d_k$-distances instead of $Q_{\op{zip}}$-graph distances. Hence we will need to compare the two types of distances.
Let $\gamma$ be a $Q_{\op{zip}}$-geodesic from some edge of $Q_{\op{zip}}$ lying at $Q_{\op{zip}}$-graph distance at most $(\# \BB A)^{1/2}$ from $\BB A$ to some edge of $\bdy \dot Q^{J_r}$. Let $t_0$ (resp.\ $t_1$) be the largest $t\in [1,|\gamma|]_{\BB Z}$ such that $\gamma(t)$ has an endpoint in $\bdy \dot Q^{J_{r_{k-2}}}$ (resp.\ $\bdy \dot Q^{J_{r_{k-1}}}$). 
Since $r_{k-2} \geq r_1 \geq (\#\BB A)^{1/2}$, Lemma~\ref{prop-peel-ball} implies that $\gamma(1) \in \mcl E(\dot Q^{J_{r_{k-2}}})$ so $t_0$ and $t_1$ exist. 

The curve $\gamma |_{[t_0,t_1]}$ is a $d_k$-geodesic from a vertex of $\bdy \dot Q^{J_{r_{k-2}}}$ to a vertex of $\bdy \dot Q^{J_{r_{k-1}}}$. By condition~\ref{item-radius-iterate-geo} in the definition of $\wt E_k(C)$, there exists $t_* \in [t_0, t_1]_{\BB Z}$ such that 
\eqb  \label{eqn-geo-d_k}
d_k\left( \gamma (t_*) ,   \dot Q^{J_{r_k}} \cap \left( \bdy Q_- \cup \bdy Q_+ \right)\right) \geq C^{-1/2} r_{k-2}  .
\eqe 
By Lemma~\ref{prop-peel-ball} and the definition~\eqref{eqn-radius-iterate-r-L-def} of $r_k$, 
\eqb \label{eqn-geo-outside-dist}
\op{dist}\left( \bdy \dot Q^{J_{r_k}} ,  \dot Q^{J_{r_{k-1}}} ; Q_{\op{zip}} \right) \geq r_k - r_{k-1} \geq r_{k-1} \geq r_{k-2} .
\eqe 
Hence any path started from $\gamma(t_*)$ which exits $\dot Q^{J_{r_k}}$ must travel graph distance at least $r_{k-2}$. 
By this and the definition of $d_k$, we can upgrade~\eqref{eqn-geo-d_k} to the statement that
\eqb \label{eqn-geo-dist}
\op{dist}\left( \gamma (t_*) , \bdy Q_- \cup \bdy Q_+ ; Q_{\op{zip}} \right) \geq C^{-1/2} r_{k-2} .
\eqe 
By condition~\ref{item-radius-iterate-length} in the definition of $\wt E_k(C)$, 
\eqb \label{eqn-geo-r}
C^{-1/2} r_{k-2}  \geq C^{-1} r_k \geq C^{-1} r_{k-1}.
\eqe 
By~\eqref{eqn-geo-dist} and~\eqref{eqn-geo-r}, condition~\ref{item-good-radius'-geo} in the definition of $\wt R(C)$ is satisfied for $r = r_{k-1}$.
\end{proof}

We next have an analog of Lemma~\ref{prop-iterate-cluster-good} for the events $\wt E_k(C)$, which takes significantly more effort to prove. 

\begin{lem} \label{prop-iterate-cluster-good'}
For each $\alpha \in (0,1)$, there exists $C = C(\alpha) > 8$ such that if $k$ is at least some constant depending only on $\alpha$, then
\eqbn
\BB P\left[ \wt E_k(C) \,|\, \mcl F^{J_{r_{k-3}}} \right] \geq 1-\alpha .
\eqen
\end{lem}

Most of the proof of Lemma~\ref{prop-iterate-cluster-good'} will be carried out in Lemmas~\ref{prop-wtE0}--\ref{prop-cluster-saw-nbd} below. We give an outline before proceeding with the details. 
It is straightforward to obtain a lower bound for the probabilities of conditions~\ref{item-radius-iterate-length} and~\ref{item-radius-iterate-diam} in the definition of $\wt E_k(C)$, using the same argument as in the proof of Lemma~\ref{prop-iterate-cluster-good} (see Lemma~\ref{prop-wtE0}). 
The main difficulty is proving a lower bound for the probability of condition~\ref{item-radius-iterate-geo}.  
To this end, we will prove the following two statements, which correspond to Lemmas~\ref{prop-cluster-max} and~\ref{prop-cluster-saw-nbd}, respectively. 
\begin{enumerate}
\item The $\dot Q^{J_{r_k}}$-graph distance from any point of $\bdy \dot Q^{J_{r_{k-2}}}$ to any point of $\bdy \dot Q^{J_{r_{k-1}}}$ is typically at most some constant $A>0$ times $r_{k-2}$. 
\item For any given $A > 0$, we can find a small enough $\rho \in (0,1)$ such that with high conditional probability given $\mcl F^{J_{r_{k-3}}}$, every path from $\bdy \dot Q^{J_{r_{k-2}}}$ to $\bdy \dot Q^{J_{r_{k-1}}}$ which stays in the $\rho r_{k-2}$-neighborhood of $\bdy  Q_-^{J_{r_{k-2}}} \cup \bdy  Q_{+}^{J_{r_{k-2}}}$ has length at least $A r_{k-2}$. 
\end{enumerate} 
The first statement is proved using the upper bounds for graph distances arising from Lemma~\ref{prop-uihpq-bdy-holder} and Lemma~\ref{prop-peel-ball-upper}, and the second is proved using the lower bound for the lengths of paths which stay near $\bdy Q_- \cup \bdy Q_+$ from Lemma~\ref{prop-saw-neighborhood-dist}.
 
Combining the above two statements will show that with high conditional probability given $\mcl F^{J_{r_{k-3}}}$, a path from $\bdy \dot Q^{J_{r_{k-2}}}$ to $\bdy \dot Q^{J_{r_{k-1}}}$ which stays in the $\rho r_{k-2}$-neighborhood of $\bdy  Q_- \cup \bdy Q_+$ cannot be a geodesic. Indeed, such a path must have a sub-path which travels from $\bdy \dot Q^{J_{r_{k-2}}}$ to $\bdy \dot Q^{J_{r_{k-1}}}$ and stays in the $\rho r_{k-2}$-neighborhood of $\bdy  Q_-^{J_{r_{k-2}}} \cup \bdy  Q_+^{J_{r_{k-2}}}$, so the length of the original path must be larger than the $\dot Q^{J_{r_k}}$-graph distance between its endpoints. This gives condition~\ref{item-radius-iterate-geo} with $C = \rho^{-2}$.

We first give the (easy) proof that the first two conditions in the definition of $\wt E_k(C)$ hold with high conditional probability given $\mcl F^{J_{r_{k-3}}}$. 

\begin{lem} \label{prop-wtE0}
For $k\in\BB N$ and $C  >8$, let $\wt E_k^0(C  )$ be the event that conditions~\ref{item-radius-iterate-length} and~\ref{item-radius-iterate-diam} in the definition of $\wt E_k(C)$ are satisfied, i.e.,
\eqb \label{eqn-cluster-length-good'}
L_{k-2} \vee L_{k-1} \vee L_k \leq  \frac12(C^2 - 8 ) r_{k-2}^2\quad \op{and} \quad r_k \leq C^{1/2} r_{k-2}  
\eqe 
and  
\eqb \label{eqn-cluster-outer-event}
  \op{diam} \left( \bdy \dot Q^{J_{r_{k-1}}} \cap Q_\xi^{J_{r_{k-1}} } ; \dot Q^{J_{r_{k}}} \cap Q_\xi  \right) \leq C r_{k-2} ,\quad \forall \xi \in \{-,+\} .
\eqe   
For each $\alpha \in (0,1)$, there exists $C  =C(\alpha) > 8$ such that a.s.\ $\BB P\left[\wt E_k^0(C) \,|\, \mcl F^{J_{r_{k-3}}}\right] \geq 1-\alpha$.
\end{lem}
\begin{proof} 
By~\eqref{eqn-iterate-moment} and the Chebyshev inequality, we can find $C_1 = C_1(\alpha) > 1$ such that it is a.s.\ the case that with conditional probability at least $1-\alpha/2$ given $\mcl F^{J_{r_{k-3}}}$, 
\eqb \label{eqn-cluster-length-good2}
  L_{k-2} \leq C_1 r_{k-2}^2 , \quad  L_{k-1} \leq C_1 r_{k-1}^2 , \quad \op{and} \quad L_{k } \leq C_1 r_k^2    .
\eqe  
As in the proof of Lemma~\ref{prop-iterate-cluster-good}, if~\eqref{eqn-cluster-length-good2} holds then~\eqref{eqn-cluster-length-good'} holds for an appropriate $C > 8$ depending only on $C_1$. 
Again arguing as in the proof of Lemma~\ref{prop-iterate-cluster-good}, we deduce from Lemmas~\ref{prop-peel-ball} and~\ref{prop-uihpq-bdy-holder} that for a possibly larger choice of $C>8$, the relation~\eqref{eqn-cluster-outer-event} also a.s.\ holds with conditional probability at least $1-\alpha/2$ given $\mcl F^{J_{r_{k-3}}}$. 
\end{proof}

In the next two lemmas, we focus our attention on condition~\ref{item-radius-iterate-geo}. The first step is an upper bound on the maximal graph distance in $\dot Q^{J_{r_k}} $ between a vertex of $\bdy \dot Q^{J_{r_{k-1}}}$ and a vertex of $\bdy \dot Q^{J_{r_{k-2}}}$.

\begin{figure}[ht!]
 \begin{center}
\includegraphics[scale=.8]{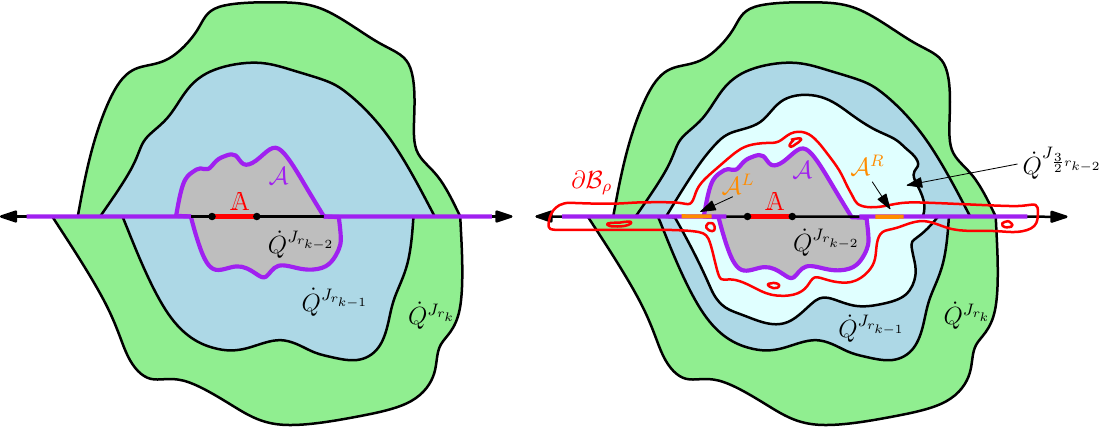} 
\caption{\textbf{Left:} Illustration of the proof of Lemma~\ref{prop-cluster-max}. We want an upper bound for the maximal distance between any point of $\bdy \dot Q^{J_{r_{k-1} }}$ and any point of $\bdy \dot Q^{J_{r_{k-2} }}$ along paths which stay in $\dot Q^{J_{r_{k-1}}}$. The set $\mcl A$ is the union of a boundary arc of the UIHPQ$_{\op{S}}$ $  Q_-^{J_{r_{k-2} }}$ and a boundary arc of the UIHPQ$_{\op{S}}$ $  Q_+^{J_{r_{k-2} }}$ which each have length at most $C_0^2 r_{k-2}^2$. The set $\mcl A$ contains $\bdy \dot Q^{J_{r_{k-2}}}$ on the event $\wt E_k^0(C_0)$ of Lemma~\ref{prop-wtE0}. Lemma~\ref{prop-uihpq-bdy-holder} (applied to each of $  Q_-^{J_{r_{k-2} }}$ and $  Q_+^{J_{r_{k-2} }}$) allows us to bound the $\dot Q^{J_{r_k}}$-graph distance diameter of $\bdy \dot Q^{J_{r_{k-2} }}$. Lemma~\ref{prop-peel-ball-upper} allows us to bound the maximal  $\dot Q^{J_{r_k}}$-distance from a given vertex $v_1 \in \bdy \dot Q^{J_{r_{k-1} }}$ to $ \bdy \dot Q^{J_{r_{k-2}}}$. 
\textbf{Right:} Illustration of the proof of Lemma~\ref{prop-cluster-saw-nbd}. The arcs $\mcl A^L ,\mcl A^R  \subset \bdy Q_-^{J_{r_{k-2}}}\cap  \bdy Q_+^{J_{r_{k-2} }}$ each have length equal to a small constant times $r_{k-2}^2$. The red set $\bdy \mcl B_\rho$ is the boundary of a small neighborhood of the purple set $\mcl A$. Lemma~\ref{prop-saw-neighborhood-dist} tells us that if $\rho$ is small, then a $d_k$-geodesic cannot spend a long time in this neighborhood (otherwise its length would be longer than the distance between its endpoints). This allows us to force such a geodesic to spend a positive fraction of its time away from $\bdy Q_-^{J_{r_{k-2}}}\cup  \bdy Q_+^{J_{r_{k-2} }}$. }\label{fig-cluster-max}
\end{center}
\end{figure}

\begin{lem} \label{prop-cluster-max}
For each $\alpha\in (0,1)$, there exists $C  = C (\alpha ) > 1$ such that for each $k\in\BB N$, it a.s.\ holds with conditional probability at least $1-\alpha/4$ given $\mcl F^{J_{r_{k-3}}}$ that 
\eqb \label{eqn-cluster-max}
\max\left\{ \op{dist}\left( v_1, v_2     ;  \dot Q^{J_{r_k}} \right) \,:\,  v_1 \in \mcl V\left(\bdy \dot Q^{J_{r_{k-1} }} \right) ,\, v_2 \in \mcl V\left( \bdy \dot Q^{J_{r_{k-2}}}  \right) \right\} \leq  C  r_{k-2}   
\eqe 
\end{lem}
\begin{proof}
See Figure~\ref{fig-cluster-max}, left, for an illustration of the proof. We will first use Lemmas~\ref{prop-wtE0} and~\ref{prop-uihpq-bdy-holder} to bound the $\dot Q^{J_{r_k}}$-graph distance diameter of $\dot \bdy Q^{J_{r_{k-2} }}$. Actually, we will prove a stronger estimate than we need, namely a bound for the $\dot Q^{J_{r_k }} \cap \bdy Q_\pm^{J_{r_{k-2}}}$-graph distance diameter of each of $\dot Q^{J_{r_{k-1} }} \cap \bdy Q_\pm^{J_{r_{k-2}}}$ (see~\eqref{eqn-cluster-intersect-diam}). The union of the latter two sets contains  $\bdy \dot Q^{J_{r_{k-2}}}  $. 
We will then use Lemma~\ref{prop-peel-ball-upper} to bound the maximal $\dot Q^{J_{r_k}}$-distance from a given vertex $v_1 \in \bdy \dot Q^{J_{r_{k-1} }}$ to $ \bdy \dot Q^{J_{r_{k-2}}}$. Combining these estimates will yield~\eqref{eqn-cluster-max}. 

Let $C_0 > 8$ be chosen so that the conclusion of Lemma~\ref{prop-wtE0} is satisfied with $C = C_0$ and $\alpha/3$ in place of $\alpha$.
Let $\mcl A $ be the set of edges $e$ of $\bdy Q_-^{J_{r_{k-2}}} \cup \bdy Q_+^{J_{r_{k-2}}}$
which can be connected to $\bdy \dot Q^{J_{r_{k-2}}}$ by an arc of $\bdy Q_-^{J_{r_{k-2}}}$ or $ \bdy Q_+^{J_{r_{k-2}}}$ with length at most $C_0^2 r_{k-2}^2$. 
By~\eqref{eqn-cluster-length-good'}, $L_k \leq \frac12 C_0^2 r_{k-2}^2$ on $\wt E_k^0(C_0)$ so on $\wt E_k^0(C_0)$, 
\eqb \label{eqn-cluster-arc-contain}
 \mcl E\left(  \dot Q^{J_{r_{k}}} \cap \left( \bdy Q_-^{J_{r_{k-2}}} \cup \bdy Q_+^{J_{r_{k-2}}} \right)  \right) \subset \mcl A  .
\eqe 

By Lemma~\ref{prop-peel-law}, the conditional law of $Q_-^{J_{r_{k-2}}}$ given $\mcl F^{J_{r_{k-2}}}$ is that of a UIHPQ$_{\op{S}}$, so by the last statement of Lemma~\ref{prop-uihpq-bdy-holder}, we can find a constant $C_1 = C_1(\alpha ) >0$ such that the conditional probability given $ \mcl F^{J_{r_{k-2}}}$ of the following event is a.s.\ at least $1-\alpha/3$:
\eqb \label{eqn-cluster-all-diam}
   \op{diam}\left( X ; B_{r_{k-2}}\left(X ; Q_-^{J_{r_{k-2}}} \right) \right) \leq C_1 r_{k-2}  ,\quad \forall X \subset \mcl E\left( \mcl A \cap \bdy Q_-^{J_{r_{k-2}}} \right) .
\eqe 
By symmetry, the same holds with ``$+$" in place of ``$-$". 

Henceforth assume that $\wt E_k^0(C_0)$ occurs and~\eqref{eqn-cluster-all-diam} is satisfied with both choices of sign, which happens with probability at least $1-\alpha$. 
We will verify~\eqref{eqn-cluster-max} for an appropriate choice of $C$. 

By Lemma~\ref{prop-peel-ball}, $\dot Q^{J_{r_k}}$ contains the $r_k - r_{k-1} \geq r_{k-2}$-neighborhood of $ \dot Q^{J_{r_{k-1} }} \cap \bdy Q_\pm^{J_{r_{k-2}}}$ with respect to the graph metric on $Q_-^{J_{r_{k-2}}} \cup Q_+^{J_{r_{k-2}}}$ (recall that we have identified $Q_-^{J_{r_{k-2}}} \cup Q_+^{J_{r_{k-2}}}$ with its image under the gluing map), which in turn contains the $ r_{k-2}$-neighborhood of $ \dot Q^{J_{r_{k-1} }} \cap \bdy Q_\pm^{J_{r_{k-2}}}$ with respect to the graph metric on $ Q_\pm^{J_{r_{k-2}}}$. Hence~\eqref{eqn-cluster-arc-contain} and~\eqref{eqn-cluster-all-diam} together imply that
\eqb \label{eqn-cluster-intersect-diam}
 \op{diam} \left(\dot Q^{J_{r_{k-1} }} \cap \bdy Q_-^{J_{r_{k-2}}}  ; \dot Q^{J_{r_k}} \cap Q_-^{J_{r_{k-2}}}  \right) 
\leq C_1 r_{k-2}  
\eqe
and the same holds with `$+$" in place of ``$-$".   
 
By Lemma~\ref{prop-peel-ball-upper}, $\bdy \dot Q^{J_{r_{k-1} }} \cap Q_-^{J_{r_{k-2}}}$ is contained in the $2(r_{k-1} - r_{k-2})$-neighborhood of $\dot Q^{J_{r_{k-1} }} \cap  \bdy Q_-^{J_{r_{k-2}}}  $ with respect to the graph metric on $\dot Q^{J_{r_{k-1}}} \cap Q_-^{J_{r_{k-2}}}$ and the same holds with `$+$" in place of ``$-$". 
By~\eqref{eqn-cluster-length-good'} from the definition of $\wt E_k^0(C_0)$, we have $r_{k-1} - r_{k-2} \leq  C_0^{1/2} r_{k-2}$. 

Consequently, each $v_1 \in \mcl V(\bdy \dot Q^{J_{r_{k-1}}})$ can be joined to a vertex $v_2'$ of $\dot Q^{J_{r_{k-1} }} \cap  \bdy Q_-^{J_{r_{k-2}}}  $ by a path of length at most $2C_0^{1/2} r_{k-2}$ which stays in $\dot Q^{J_{r_{k-1}}} \subset \dot Q^{J_{r_k}}$. By~\eqref{eqn-cluster-intersect-diam}, the $\dot Q^{J_{r_k}}$-graph distance diameter of $\dot Q^{J_{r_{k-1} }} \cap  \bdy Q_-^{J_{r_{k-2}}}  $ is at most $C_1 r_{k-2}$. Since $\bdy\dot Q^{J_{r_{k-2}}}\subset \dot Q^{J_{r_{k-1} }} \cap  \bdy Q_-^{J_{r_{k-2}}}$, the point $v_2'$ lies at $\dot Q^{J_{r_k}}$-graph distance at most $C_1 r_{k-2}$ from any given vertex $v_2$ of $\bdy\dot Q^{J_{r_{k-2}}}$. 
Hence~\eqref{eqn-cluster-max} is satisfied with $C = 2 C_0^{1/2} + C_1$. 
\end{proof}

The second step in our verification of condition~\ref{item-radius-iterate-geo} in the definition of $\wt E_k(C)$ is the following lower bound for distances when we restrict attention to paths which stay close to $\bdy Q_-^{J_{r_{k-2}}} \cup \bdy Q_+^{J_{r_{k-2}}}$, which will be deduced from Lemma~\ref{prop-saw-neighborhood-dist}.

\begin{lem} \label{prop-cluster-saw-nbd}
For $\rho \in (0,1)$, let $\mcl B_\rho$ be the $\rho r_{k-2}$-neighborhood of $ \dot Q^{J_{r_{k-1}}} \cap  \left( \bdy Q_-^{J_{r_{k-2}}} \cup \bdy Q_+^{J_{r_{k-2}}} \right) $ with respect to the internal graph metric $d_k$ on $\dot Q^{J_{r_k}}$.
For each $A > 0$ and each $\alpha \in (0,1)$, there exists $\rho = \rho(\alpha, A) \in (0,1)$ and $k_* = k_*(\alpha,A) \in (0,1)$ such that for each $k\geq k_*$, it is a.s.\ the case that the conditional probability given $\mcl F^{J_{r_{k-3}}}$ of the following event is at least $1-\alpha$:
\eqb \label{eqn-cluster-saw-nbd} 
\op{dist}\left( \bdy \dot Q^{J_{r_{k-1 }}} \cap \mcl B_\rho ,\,  \bdy \dot Q^{J_{r_{k-2}}} \cap \mcl B_\rho    ;   \mcl B_\rho \right) \geq A  r_{k-2}  .
\eqe  
\end{lem}
\begin{proof}
See Figure~\ref{fig-cluster-max}, right, for an illustration of the proof.
As in the proof of Lemma~\ref{prop-cluster-max}, we let $C_0 > 8$ be chosen so that the conclusion of Lemma~\ref{prop-wtE0} is satisfied with $C = C_0$ and $\alpha/3$ in place of $\alpha$ and we let $\mcl A $ be the set of edges $e \in \mcl E( \bdy Q_-^{J_{r_{k-2}}} \cup \bdy Q_+^{J_{r_{k-2}}} )$
which can be connected to $\bdy \dot Q^{J_{r_{k-2}}}$ by an arc of $\bdy Q_-^{J_{r_{k-2}}}$ or $ \bdy Q_+^{J_{r_{k-2}}}$ with length at most $C_0^2 r_{k-2}^2$.

We first establish a lower bound for the number of edges along $\bdy \dot Q_-^{J_{r_{k-2}}}$ between $ \bdy \dot Q^{J_{r_{k-2 }}} $ and $\bdy \dot Q^{J_{ \frac32 r_{k-2}}}$. 
One expects there to be of order $r_{k-2}^2$ such edges since the graph distance from $ \bdy \dot Q^{J_{r_{k-2 }}} $ to $\bdy \dot Q^{J_{ \frac32 r_{k-2}}}$ is of order at least $r_{k-2}^2$.

Fix a constant $C_1 > 0$ to be chosen later, in a manner depending only on $\alpha$, and let $\mcl A^L$ and $\mcl A^R$, respectively, be the set of edges in the length-$\lceil C_1^{-1} r_{k-2}^2 \rceil$ arc of $\bdy Q_-^{J_{r_{k-2}}} $ (equivalently, of $\bdy Q_+^{J_{r_{k-2}}}$) lying immediately to the left and right, respectively, of $\dot Q^{J_{r_{k-2}}}$.
Lemma~\ref{prop-uihpq-bdy-holder} applied to the UIHPQ$_{\op{S}}$ $Q_-^{J_{r_{k-2}}}$ implies that if $C_1$ is chosen sufficiently small, in a manner depending only on $\alpha$, then with conditional probability at least $1-\alpha/3$ given $\mcl F^{J_{r_{k-2}}}$, the $Q_-^{J_{r_{k-2}}}$-diameters of the segments $\mcl A^L$ and $\mcl A^R$ are each at most $\frac12 r_{k-2} -1$. If this is the case, then Lemma~\ref{prop-peel-ball} implies that  
\eqb \label{eqn-cluster-length-lower}
\mcl A^L\cup \mcl A^R \subset \mcl E\left( \dot Q^{J_{ \frac32 r_{k-2}-1}} \setminus  \dot Q^{J_{r_{k-2}}} \right)  
\eqe
note that $\frac32 r_{k-2} \leq r_{k-1} - \frac12 r_{k-2}$). Since the boundaries of $\dot Q^{J_{ \frac32 r_{k-1}-1}}$ and $\dot Q^{J_{ \frac32 r_{k-1} }}$ are disjoint,~\eqref{eqn-cluster-length-lower} implies that $L_{k-1} \geq C_1^{-1} r_{k-2}^2$ and
\eqb \label{eqn-cluster-length-lower'}
 \op{dist}\left( e ,  \bdy \dot Q^{J_{r_{k-2 }}} ;   \bdy Q_-^{J_{r_{k-2}}} \cup \bdy Q_+^{J_{r_{k-2}}}  \right) \geq C_1^{-1} r_{k-2}^2 , \quad \forall e \in \mcl E\left(  \bdy Q_-^{J_{r_{k-2}}} \cup \bdy Q_+^{J_{r_{k-2}}} \right) \setminus \mcl E\left( \dot Q^{J_{ \frac32 r_{k-2}}} \right) .
\eqe   
Note that here we are considering distances along paths in $\bdy Q_-^{J_{r_{k-2}}} \cup \bdy Q_+^{J_{r_{k-2}}}$. 

We now use Lemma~\ref{prop-saw-neighborhood-dist} to prove a lower bound for graph distances in a neighborhood of $\bdy Q_-^{J_{r_{k-2}}} \cup \bdy Q_+^{J_{r_{k-2}}}$ in terms of graph distances along $\bdy Q_-^{J_{r_{k-2}}} \cup \bdy Q_+^{J_{r_{k-2}}}$ itself.
For $\rho > 0$, let $\mcl B_\rho'$ be the $\rho r_{k-2}$-neighborhood of the set $\mcl A$ (defined at the beginning of the proof) with respect to the graph metric on the unexplored quadrangulation $Q_-^{J_{r_{k-2}}} \cup  Q_+^{J_{r_{k-2}}}$.  
Since $r_k \geq r_{k-1}+1$, we have $r_{k-2} \geq k-2$. 
By Lemma~\ref{prop-saw-neighborhood-dist} applied to the glued quadrangulation $Q_-^{J_{r_{k-2}}} \cup Q_+^{J_{r_{k-2}}}$ and with $r = \lfloor C_0^2 r_{k-2} \rfloor$ and $\zeta = 1/2$, we can find $\rho_* = \rho_*(\alpha) \in (0,1/2)$ such that for $\rho \in (0,\rho_*]$ and each $k \in \BB N$ which is at least some $\rho$-dependent constant, it holds with conditional probability at least $1-\alpha/5$ given $\mcl F^{J_{r_{k-2}}}$ that 
\eqb\label{eqn-cluster-saw-nbd0}
 \op{dist}\left(  e_1 , e_2  ;  \mcl B_\rho'  \right) \geq  C_0^{-12} \rho^{-1/2} r_{k-2} \left( \frac{1}{r_{k-2}^2} \op{dist}\left( e_1,e_2  ; \bdy Q_-^{J_{r_{k-2}}} \cup \bdy Q_+^{J_{r_{k-2}}} \right) \right)^{7/2}   ,\quad  \forall e_1,e_2\in \mcl A .
\eqe 

Henceforth assume that $\wt E_k^0(C_0)$ occurs,~\eqref{eqn-cluster-length-lower'} is satisfied, and~\eqref{eqn-cluster-saw-nbd0} is satisfied for some fixed $\rho\in (0,\rho_*]$, which happens with probability at least $1-\alpha$ provided $k$ is chosen sufficiently large, depending only on $\rho$. We will show that~\eqref{eqn-cluster-saw-nbd} holds provided $\rho$ is chosen sufficiently small, depending only on $\alpha$ and $A$. This will conclude the proof. 

We claim that  
\eqb \label{eqn-cluster-saw-nbd'} 
\op{dist}\left( \bdy \dot Q^{J_{r_{k-1 }}} \cap \mcl B_\rho',  \bdy \dot Q^{J_{r_{k-2}}}     ;    \mcl B_\rho'  \right) \geq  \left(    C_0^{-12} C_1^{-7/2} \rho^{-1/2} -  \rho \right) r_{k-2}   .
\eqe  
Note here that $\bdy \dot Q^{J_{r_{k-2}}}$ is contained in $\mcl B_\rho'$ by~\eqref{eqn-cluster-arc-contain}.
By~\eqref{eqn-cluster-arc-contain}, the occurrence of $\wt E_k^0(C_0)$ implies that then $\mcl E\left( \bdy \dot Q^{J_{r_{k-1}}} \cap \bdy Q_-^{J_{r_{k-2}}} \right)  \subset \mcl A$ and the same holds with ``$+$'' in place of ``$-$."
Consequently, if $v \in \mcl V(\bdy \dot Q^{J_{r_{k-1 }}} \cap \mcl B_\rho' )$, then there is an $e\in \mcl A$ with 
$\op{dist}\left( v , e ; \mcl B_\rho' \right) \leq \rho r_{k-2}$. 
Since $\rho \in (0,1/2)$ and by Lemma~\ref{prop-peel-ball},
\eqbn
\rho r_{k-2} < r_{k-1} - \frac32 r_{k-2} \leq \op{dist}\left( v , \dot Q^{J_{\frac32 r_{k-2}}} ; Q_{\op{zip}} \right) .
\eqen
Consequently, the edge $e$ cannot belong to $\dot Q^{J_{\frac32 r_{k-2}}}$.  
By~\eqref{eqn-cluster-length-lower'}, the distance from $ e$ to $\bdy \dot Q^{J_{ r_{k-2 }}}$ 
along $\bdy Q_-^{J_{r_{k-2}}} \cup \bdy Q_+^{J_{r_{k-2}}}$ is at least $C_1^{-1} r_{k-2}^2$.
By~\eqref{eqn-cluster-saw-nbd0} (applied with $e_1 =e$ and $e_2 \in \mcl E(\bdy \dot Q^{J_{ r_{k-2 }}})$) and the triangle inequality, we infer that~\eqref{eqn-cluster-saw-nbd'} holds. 

The relation~\eqref{eqn-cluster-arc-contain} implies that with $\mcl B_\rho$ as in the statement of the lemma, we have $\mcl B_\rho \setminus \dot Q^{J_{r_{k-2}}} \subset \mcl B_\rho'$. 
Hence~\eqref{eqn-cluster-saw-nbd'} implies~\eqref{eqn-cluster-saw-nbd} upon choosing $\rho$ sufficiently small that $C_0^{-12} C_1^{-7/2} \rho^{-1/2} - \rho \geq A$.
\end{proof}

\begin{proof}[Proof of Lemma~\ref{prop-iterate-cluster-good'}]
Given $\alpha\in (0,1)$, choose $C_0 = C_0(\alpha) > 8$ sufficiently large that the conclusions of Lemmas~\ref{prop-wtE0} and~\ref{prop-cluster-max} are satisfied with this choice of $C$ and with $\alpha/3$ in place of $\alpha$; and $\rho \in (0,1)$ and $k_* \in \BB N$ such that the conclusion of Lemma~\ref{prop-cluster-saw-nbd} is satisfied with $A = C_0$ and with $\alpha/3$ in place of $\alpha$. Then for each $k\geq k_*$, it a.s.\ holds with conditional probability at least $1-\alpha$ given $\mcl F^{J_{r_{k-3}}}$ that conditions~\ref{item-radius-iterate-length} and~\ref{item-radius-iterate-diam} in the definition of $\wt E_k(C_0)$ are satisfied and both~\eqref{eqn-cluster-max} and~\eqref{eqn-cluster-saw-nbd} hold. We now assume that this is the case for some $k\geq k_*$ and deduce condition~\ref{item-radius-iterate-geo} for an appropriate $C\geq C_0$. 

By~\eqref{eqn-cluster-max} and~\eqref{eqn-cluster-saw-nbd}, if we let $\mcl B_\rho$ be as in~\eqref{eqn-cluster-saw-nbd} then the $d_k$-distance from any vertex $v_1 \in \mcl V(\bdy \dot Q^{J_{r_{k-1}}})$ to any vertex $v_2 \in \mcl V(\bdy \dot Q^{J_{r_{k-2}}})$ is smaller than the distance from $\bdy \dot Q^{J_{r_{k-1}}} \cap \mcl B_\rho$ to $\bdy \dot Q^{J_{r_{k-2}}} \cap \mcl B_\rho$ along paths which stay in $\mcl B_\rho$. 
Therefore, any $d_k$-geodesic from $v_1$ to $v_2$ must exit $\mcl B_\rho$ before hitting $\bdy \dot Q^{J_{r_{k-2}}}$. 
Since the $\rho r_{k-2}$-neighborhood of $\dot Q^{J_{r_k}} \cap (\bdy Q_- \cup \bdy Q_+)$ with respect to $d_k$ is contained in $\mcl B_\rho \cup \dot Q^{J_{r_{k-2}}}$, we infer that condition~\ref{item-radius-iterate-geo} in the definition of $\wt E_k(C)$ holds for any $C \geq \rho^{-2}$. 
We thus obtain the statement of the lemma with $C = C_0\vee\rho^{-2}$. 
\end{proof}

\begin{proof}[Proof of Lemma~\ref{prop-good-radius'}]
This is deduced from Lemmas~\ref{prop-iterate-cluster-length},~\ref{prop-good-radius-compare'}, and~\ref{prop-iterate-cluster-good'} via exactly the same argument used in the proof of Lemma~\ref{prop-good-radius} (the requirement that $k$ is at least some $\alpha$-dependent constant only results in an additional $p$-dependent constant factor in~\eqref{eqn-good-radius'}).  
\end{proof}

\subsection{Proof of Propositions~\ref{prop-lipschitz-path} and~\ref{prop-geodesic-away}}
\label{sec-geo-proof}

In this subsection we will deduce Propositions~\ref{prop-lipschitz-path} and~\ref{prop-geodesic-away} from Lemmas~\ref{prop-good-radius} and~\ref{prop-good-radius'}, respectively. To this end we will use the following notation.
Fix $L > 0$. For $n\in\BB N$ and $\delta \in (0,1)$, let
\eqb \label{eqn-good-radius-partition}
\mcl I^n(\delta) := \left\{ [x-\delta^2 n^{1/2} , x]_{\BB Z} \,:\, x \in  [-L n^{1/2} , L n^{1/2}]  \cap \big(\lfloor \delta^2 n^{1/2}\rfloor \BB Z \big) \right\}
\eqe 
so an element of $\mcl I^n(\delta)$ is a discrete interval $I$ of length $\delta^2 n^{1/2}$ (up to rounding error). For $I\in \mcl I^n(\delta)$, let $\{\dot Q_I^j\}_{j\in \BB N_0}$ be the clusters of the glued peeling process of $Q_{\op{zip}}$ started from the initial edge set $\BB A = \lambda_-(I)$, where here we recall that $\lambda_-$ is the boundary path of $Q_-$. Also let $\{J_{I,r}\}_{r\in \BB N_0}$ be the stopping times as in Section~\ref{sec-glued-peeling} for these clusters and for $C>2$, let $R_I(C)$ and $\wt R_I(C)$ be the random ``good" radii defined in Lemmas~\ref{prop-good-radius} and~\ref{prop-good-radius'}, respectively, for these clusters. 
To lighten notation, we abbreviate the glued peeling clusters at the good radii by
\eqb \label{eqn-good-hull-abbrv}
  Q_I(C) := \dot Q_I^{J_{I,R_I(C)}} \quad \op{and} \quad \wt Q_I(C) :=  \dot Q_I^{J_{I,\wt R_I(C)}} .
\eqe 
We emphasize that the glued peeling processes for different choices of $I$ are defined separately, and are allowed to overlap. We make no claims about the correlation between them.

The following lemma tells us that the radii $R_I(C)$ and $\wt R_I(C)$ are typically not too big, uniformly over all $I\in\mcl I^n(\delta)$. It will eventually be used to show that most points of a specified boundary/SAW segment $\lambda_-( [z_0,z_1]_{\BB Z} )$ are contained in one of the good clusters $Q_I(C)$ or $\wt Q_I(C)$ which does not contain either of the endpoints $\lambda_-(z_0)$ or $\lambda_-(z_1)$. 
Note that for the proofs of Propositions~\ref{prop-lipschitz-path} and~\ref{prop-geodesic-away}, we do not care about the size of the clusters $Q_I(C)$ aside from the requirement that they do not contain the endpoints of the segment of interest.

\begin{lem} \label{prop-good-radius-exist}
For each $\zeta \in (0,1)$, there exists $C = C(\zeta) > 1$ such that the following is true. For each $L>0$, each $n\in\BB N$, each $\delta \in (0,1)$, and each $z \in [-L n^{1/2} , L n^{1/2}]_{\BB Z}$, we have (in the notation introduced just above)
\eqb \label{eqn-good-radius-exist}
\BB P\left[ \lambda_-(z) \notin \mcl E\left(  Q_I(C)     \right) , \:\forall I \in \mcl I^n(\delta) \: \text{with} \: \op{dist}(z, I) \geq \delta^{2- 3\zeta} n^{1/2}     \right] = 1 - O_\delta( \delta^{\zeta} ) ,
\eqe  
where here $\op{dist}(z,I)$ denotes one-dimensional Euclidean distance and the rate of the $O_\delta(\delta^\zeta)$ depends only on $L$ and $\zeta$.  The same holds with $\wt Q_I(C)$ in place of $Q_I(C)$. 
\end{lem} 
\begin{proof}  
We give the proof in the case of $Q_I(C)$; the proof for $\wt Q_I(C)$ is identical. 
Fix $L > 1$, $\delta \in (0,1)$, and $z \in [-L n^{1/2} , L n^{1/2}]_{\BB Z}$.
Also fix $p \in (1,3/2)$ and let $C = C(p)> 1$ be as in Lemma~\ref{prop-good-radius} for this choice of $p$. 

By condition~\ref{item-good-radius-length} in the definition of $R_I(C)$, if $I\in \mcl I^n(\delta)$ then each edge of $Q_I(C) \cap \bdy Q_-$ lies at $\bdy Q_-$-graph distance at most $C^2 R_I(C)^2$ from $I$. Hence if $R_I(C) \leq (2C)^{-1} \op{dist}(z,I)^{1/2}$ then $\lambda_-(z ) \notin \mcl E( Q_I(C) )$. 
By Lemma~\ref{prop-good-radius} (applied with $\#\BB A = \lfloor \delta^2 n^{1/2} \rfloor$ and $S = (2C)^{-1}\op{dist}(z,I)^{1/2} / (\#\BB A)^{1/2}    $), for $I \in \mcl I^n(\delta)$,
\eqb \label{eqn-good-radius-I}
\BB P\left[ \lambda_-(z ) \in \mcl E( Q_I(C) ) \right] \leq \BB P\left[ R_I(C) > (2C)^{-1}\op{dist}(z,I)^{1/2}  \right] \preceq  \frac{\delta^{2p} n^{p/2 }}{\op{dist}(z,I)^p}  
\eqe 
with the implicit constant depending only on $p$. 
For each $k\in\BB N$, there are at most 2 intervals $I \in \mcl I^n(\delta)$ with $\op{dist}(z,I)$ equal to $\delta^2 n^{1/2} k$ (up to rounding error). Summing the estimate~\eqref{eqn-good-radius-I} over all such intervals $I$ with $\op{dist}(z,I) \geq \delta^{2-3\zeta} n^{1/2}$ shows that the probability in~\eqref{eqn-good-radius-exist} is at most a constant (depending only on $p$, $L$, and $\zeta$) times
\eqb
\sum_{k= \lfloor \delta^{-3\zeta} \rfloor}^{\lceil 2 L \delta^{-2} \rceil} \frac{1}{k^p} \preceq \delta^{3(p-1) \zeta  } 
\eqe 
which is at most $\delta^\zeta$ provided we take $p>4/3$. 
\end{proof}

\subsubsection{Proof of Proposition~\ref{prop-lipschitz-path}}
\label{sec-lipschitz-path-proof}

Fix a small parameter $\zeta  \in (0,1)$ (for our purposes we can take $\zeta =1/4$, but we find it more transparent to allow for an arbitrary $\zeta$ since this makes the source of various exponents more clear). 
Let $C = C(\zeta) > 1$ be as in Lemma~\ref{prop-good-radius-exist} for this choice of $\zeta$. 
Also fix $n\in\BB N$, $L>0$, and $\delta \in (0,1)$. Let $\mcl I^n(\delta)$ be as in~\eqref{eqn-good-radius-partition} and for $I\in \mcl I^n(\delta)$, define the good radius $R_I(C)$ and the cluster $Q_I(C)$ as in~\eqref{eqn-good-hull-abbrv} and the discussion just preceding it.

Fix $z_0 , z_1 \in [-L n^{1/2} , L n^{1/2}]_{\BB Z}$. 
Throughout the proof, we truncate on the event $E^n$ defined as follows. 
\begin{enumerate}
\item  Neither $\lambda_-(z_0)$ nor $\lambda_-(z_1)$ belongs to $\mcl E\left( Q_I(C) \right) $ for each $I \in \mcl I^n(\delta) $ with $\op{dist}(z_i , I) \geq \delta^{2- 3\zeta} n^{1/2} $. \label{item-use-good-radius-lip}
\item For each $z\in \BB Z$ with $|z - z_0| \leq 2\delta^{2-3\zeta} n^{1/2}$, we have $\op{dist}(\lambda_-(z_0) , \lambda_-(z_1) ; Q_-) \leq \delta^{1-2\zeta} n^{1/4}$. The same is true with ``$+$" in place of ``$-$" and ``$z_1$" in place of ``$z_0$". \label{item-good-radius-reg}  
\end{enumerate}
By Lemmas~\ref{prop-uihpq-bdy-holder} and~\ref{prop-good-radius-exist}, for each $\alpha \in (0,1)$ there exists $\delta_* = \delta_*(\alpha,L,\zeta) > 0$ such that for each $\delta \in (0,\delta_*]$, we have $\BB P[E^n] \geq 1-\alpha$ for large enough $n\in\BB N$. 
Hence it suffices to prove the existence of a path $\wt\gamma$ as in the proposition statement on the event $E^n$.

The rest of the argument is purely deterministic. Let $\gamma$ be a $Q_{\op{zip}}$-geodesic from $\lambda_-(z_0)$ to $\lambda_-(z_1)$, chosen in some measurable manner.
We will prove Proposition~\ref{prop-lipschitz-path} via the following steps.  
\begin{itemize}
\item[Step 1:] We replace $\gamma $ by a slightly modified path $\gamma'$ which does not cross $\lambda_-([-Ln^{1/2} , Ln^{1/2}]_{\BB Z})$ in any of the ``bad" intervals $\lambda_-(I)$ for $I\in\mcl I^n(\delta)$ with $\op{dist}(\{z_0,z_1\} , I) < \delta^{2- 3\zeta} n^{1/2} $ (i.e., the ones for which we don't know that $\lambda_-(z_0) , \lambda_-(z_1) \notin \mcl E(Q_I(C))$). This makes it so that $\gamma'$ has to cross between $\bdy Q_I(C)$ and $I$ for some $I\in\mcl I^n(\delta)$ every time it crosses  $\lambda_-([-Ln^{1/2} , Ln^{1/2}]_{\BB Z})$. 
\item[Step 2:] We iteratively replace the segment of $\gamma'$ inside one of the at most $2L\delta^{-2}$ good clusters $Q_I(C)$ for $I\in\mcl I^n(\delta)$ with $\op{dist}(\{z_0,z_1\} , I) \geq \delta^{2- 3\zeta} n^{1/2} $ by a new curve segment which crosses $\lambda_-([-Ln^{1/2} , Ln^{1/2}]_{\BB Z})$ at most once and whose length is at most $2 C$ times the length of the original segment. The reason we are able to do this is condition~\ref{item-good-radius-diam} in the definition of $R_I(C)$ from Lemma~\ref{prop-good-radius}. One of the replacement steps is shown in Figure~\ref{fig-lipschitz-path}.  
\item[Step 3:] We argue that the final curve after all of these replacement steps crosses $\lambda_-([-Ln^{1/2} , Ln^{1/2}]_{\BB Z})$ at most $2L\delta^{-2}$ times and has length at most $2C|\gamma|$ (plus a small error). %Basically, the reason for this is that we perform at most $2 L\delta^{-2}$ re-routing steps (i.e., at most one for each $I\in\mcl I^n(\delta)$) and each step involves replacing a segment of $\gamma'$ by a new segment which crosses $\lambda_-([-Ln^{1/2} , Ln^{1/2}]_{\BB Z})$ at most once and whose length is at most $2 C$ times the length of the original segment. 
\end{itemize}

\medskip

\begin{figure}[ht!]
 \begin{center}
\includegraphics[scale=1]{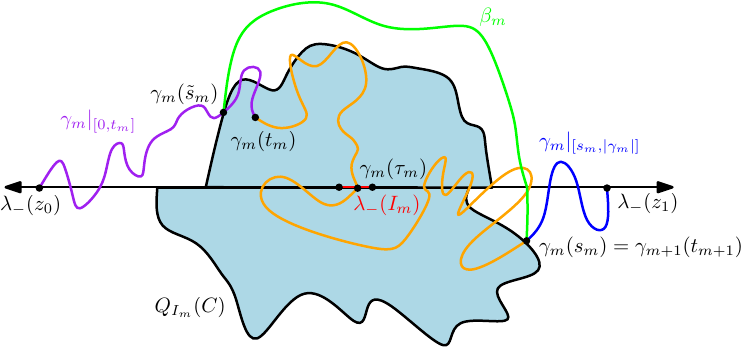} 
\caption[Re-routing procedure used in the proof of Proposition~\ref{prop-lipschitz-path}]{Illustration of the proof of Proposition~\ref{prop-lipschitz-path}. We iteratively construct paths $\gamma_m$, each of which agrees with the slightly modified $Q_{\op{zip}}$-geodesic $\gamma'$ after a certain time $t_m$ and crosses $\lambda_-([-L n^{1/2} , Ln^{1/2}]_{\BB Z})$ at most $m$ times. Here, the path $\gamma_m$ is the concatenation of the purple, orange, and blue segments; the orange and blue segments are also part of $\gamma'$. To construct $\gamma_{m+1}$, we run $\gamma_m$ up to the first time $\tau_m$ after $t_m$ at which it hits $\lambda_-([-L n^{1/2} , Ln^{1/2}]_{\BB Z})$ and consider the $\delta^2 n^{1/2}$-length interval $\lambda_-(I_m)$ which it hits at this time. Choose a path $\beta_m$ (green) between the points at which $\gamma_m$ which enters and exits the cluster $Q_{I_m}(C)$ (light blue) which crosses $\lambda_-([-L n^{1/2} , Ln^{1/2}]_{\BB Z})$ at most once and whose length is most $2C R_{I_m}(C)$. Such a path exists by definition of $ R_{I_m}(C)$. The path $\gamma_{m+1}$ is the concatenation of (part of the) purple, green, and blue segments. Since $\gamma_m$ takes at least $ R_{I_m}(C)$ units of time to get from $\bdy  \dot Q_{I_m}(C)$ to $I_m$ (by Lemma~\ref{prop-peel-ball}), $\beta_m$ is at most $2C$ times as long as the segment of $\gamma'$ it replaces.  
Iterating this procedure until we get to $\lambda_-(z_1)$ constructs the path $\wt\gamma$ in the proposition statement. 
}\label{fig-lipschitz-path}
\end{center}
\end{figure}

\noindent\textit{Step 1: modifying $\gamma$ near its endpoints.} We now deal with the small number of intervals $I\in \mcl I^n(\delta)$ near $z_0$ and $z_1$ for which $\lambda_-(z_0)$ or $\lambda_-(z_1)$ might belong to $\mcl E\left(Q_I(C)\right)$ by replacing $\gamma$ by a slightly different path $\gamma'$ whose length is not too much longer than $|\gamma|$ (see~\eqref{eqn-gamma'-length}).  
Let $T_0$ (resp.\ $T_1$) be the largest (resp.\ smallest) $t\in [1 , |\gamma|]_{\BB Z}$ for which $\gamma(t)$ is incident to an edge $\lambda_-(z)$ for $z\in\BB Z$ with $\op{dist}(\{z_0,z_1\} , I) < 2\delta^{2- 3\zeta} n^{1/2} $. We will replace $\gamma|_{[0,T_1]_{\BB Z}}$ and $\gamma|_{[T_1,|\gamma|]_{\BB Z}}$ by different paths which do not cross the gluing interface. 
To this end, let $\gamma_0'$ (resp.\ $\gamma_1'$) be a $Q_\xi$-graph distance geodesic from $\lambda_-(z_0)$ (resp.\ $\gamma(T_0)$) to $\gamma(T_1)$ (resp.\ $\lambda_-(z_1)$), with the sign $\xi \in \{\pm\}$ chosen so that $\gamma(T_0)$ (resp.\ $\gamma(T_1)$) belongs to $Q_\xi$.
 
To produce the desired path $\gamma'$, we first concatenate $\gamma_0'$, $\gamma|_{[T_0, T_1]_{\BB Z}}$, and $\gamma_1'$ and then erase (chronologically) any loops that this concatenation makes. Then $\gamma'$ is a simple path from $\lambda_-(z_0)$ to $\lambda_-(z_1)$ whose image is a subset of $\gamma_0' \cup \gamma([T_0,T_1]_{\BB Z}) \cup \gamma_1'$. By condition~\ref{item-good-radius-reg} in the definition of $E^n$ and since $\gamma_0'$ and $\gamma_1'$ are one-sided geodesics, $|\gamma_0'|\vee |\gamma_1'| \leq \delta^{1-2\zeta} n^{1/4}$. Hence
\eqb \label{eqn-gamma'-length}
|\gamma'| \leq |\gamma| + 2\delta^{1-2\zeta} n^{1/4} =  \op{dist}\left(\lambda_-(z_0) , \lambda_-(z_1) ; Q_{\op{zip}} \right) +  2\delta^{1-2\zeta} n^{1/4}    .
\eqe

We now introduce some notation to describe $\gamma'$.
For each edge $e$ hit by $\gamma'$, let
\eqb \label{eqn-gamma'-hit}
\sigma_e  :=  (\gamma')^{-1}(e) \in [1,|\gamma'|]_{\BB Z} .
\eqe
Note that this is well-defined since $\gamma'$ is simple. 
Analogously to the definitions above, let $T_0'$ (resp.\ $T_1'$) be the largest (resp.\ smallest) $t\in [1 , |\gamma'|]_{\BB Z}$ for which $\gamma'(t)$ is incident to an edge  $\lambda_-(z)$ for $z\in\BB Z$ with $\op{dist}(\{z_0,z_1\} , I) < 2\delta^{2- 3\zeta} n^{1/2} $. By the definition of $\gamma'$, 
\eqb \label{eqn-gamma'-times}
\gamma'([0,T_0']_{\BB Z} ) \subset \gamma_0' \quad \op{and}\quad \gamma'([T_1',|\gamma'|]_{\BB Z}) \subset \gamma_1' ,
\eqe 
so in particular $\gamma'$ does not cross $\lambda_-([-Ln^{1/2}, Ln^{1/2}]_{\BB Z})$ before time $T_0'$ or after time $T_1'$ (it is still allowed to touch  $\lambda_-([-Ln^{1/2}, Ln^{1/2}]_{\BB Z})$).
\medskip
 
\noindent\textit{Step 2: inductive construction of paths.} To construct the path $\wt\gamma$ in the proposition statement, we will inductively define paths $\gamma_m$ from $\lambda_-(z_0)$ to $\lambda_-(z_1)$ and times $t_m \in [1,|\gamma_m|]_{\BB Z}$ for $m\in\BB N_0$ with the following properties.
\begin{enumerate}
\item (Path after time $t_m$ has not been modified) $\gamma_m|_{[t_m , |\gamma_m|]_{\BB Z} }$ coincides with the final segment of $\gamma'$ of the same time length. \label{item-geo-approx-coincide}
\item (Bound for number of crossings) The number of times that $\gamma_m|_{[1,t_m]_{\BB Z} }$ crosses $\lambda_-([-Ln^{1/2}, Ln^{1/2}]_{\BB Z})$ is at most $m$. \label{item-geo-approx-hit}
\item (Bound for length of initial segment of the path) With $\sigma_{\gamma_m(t_m)}$ as in~\eqref{eqn-gamma'-hit}, we have $t_m \leq 2C \sigma_{\gamma_m(t_m)}  $, i.e., the modification steps we have made so far have increased the length of the part of $\gamma'$ which has been modified by a factor of at most $2C$. \label{item-geo-approx-length} 
\end{enumerate} 
We will eventually take $\wt\gamma = \gamma_M$, where $M$ is the time, to be defined below, after which all of the paths $\gamma_m$ are equal. 
A typical step of the inductive construction is illustrated in Figure~\ref{fig-lipschitz-path}. 

Recall the times $T_0'$ and $T_1'$ defined just above~\eqref{eqn-gamma'-times}. 
Let $\gamma_0 = \gamma'$ and $t_0 = T_0'$. Inductively, suppose $m\in\BB N_0$ and $\gamma_{m }$ and $t_{m }$ have been defined. We will define $\gamma_{m+1}$ and $t_{m+1}$.
Let 
\eqb \label{eqn-next-hit}
\tau_{m } := \inf\left\{ t \in [t_{m } + 1 , |\gamma_{m }|]_{\BB Z}  : \text{$\gamma_{m }(t)$ is incident to $\lambda_-([-L n^{1,2} , Ln^{1/2}]_{\BB Z})$} \right\} ,
\eqe 
or $\tau_m = |\gamma_m|$ if no such $t$ exists.
If $|\gamma_m| - \tau_m \leq |\gamma'| - T_1'$, i.e., the edge $\gamma_m(\tau_m)$ is part of the segment of $\gamma'$ traced after time $T_1'$, we set $\gamma_m = \gamma_{m+1}$ (which terminates the induction). 

Now suppose $|\gamma_m| - \tau_m  >|\gamma'| - T_1'$.  Let $I_{m } \in \mcl I^n(\delta)$ be chosen so that $\gamma_m(t_m) \in \lambda_-(I_m)$.
By condition~\ref{item-geo-approx-coincide} in the inductive hypothesis, the definition of $T_1'$, and condition~\ref{item-use-good-radius-lip} in the definition of $E^n$, it follows that neither $\lambda_-(z_0) = \gamma_m(1)$ nor $\lambda_-(z_1) = \gamma_m(|\gamma_m|)$ belongs to the glued peeling cluster $    Q_{I_m}(C)$.  
Consequently, if we let $\wt s_m$ (resp.\ $s_m$) be the first (resp.\ last) time $s \in [1,|\gamma_m|]_{\BB Z}$ such that $\gamma_m(s) $ is incident to $  Q_{I_m}(C)$, then $1 < \wt s_m < s_m < |\gamma_m|$.  

We claim that there exists a path $\beta_m$ from $\gamma_m(\wt s_m)$ to $\gamma_m(s_m)$ which crosses $\lambda_-([-Ln^{1/2} ,Ln^{1/2}]_{\BB Z})$ at most once and which has length $|\beta_m| \leq 2 C R_{I_m}(C)$. 
To see this, suppose without loss of generality that $\gamma_m(\wt s_m) \in \mcl E\left(Q_- \right)$ (the case when $\gamma_m(\wt s_m) \in Q_+$ is treated similarly). 
By definition of $R_{I_m}(C)$ (recall Lemma~\ref{prop-good-radius}), there exists a path in $Q_-$ from $\gamma_m(\wt s_m)$ to any other given edge $e$ of $  \bdy  Q_{I_m}(C) \cap Q_-$ with length at most $ C R_{I_m}(C)$. If $\gamma_m( s_m) \in \mcl E(Q_-)$, we take $\beta_m$ to be such a path for $e = \gamma_m(s_m)$. Otherwise, if $\gamma_m(s_m) \in Q_+$, we let $e$ be an edge of $\bdy  Q_{I_m}(C) \cap Q_-$ which is incident to an edge $e'$ of $\bdy Q_{I_m}(C)\cap Q_+$ and concatenate a path in $Q_-$ of length at most $ C R_{I_m}(C)$ from $\gamma_m(\wt s_m)$ to $e$ and a path in $Q_+$ of length at most $C R_{I_m}(C)$ from $e'$ to $\gamma_m(s_m)$. 

Let $\gamma_{m+1}$ be the path obtained from $\gamma_m$ by replacing $\gamma_m|_{[\wt s_m , s_m]}$ with $\beta_m$, i.e.\ the concatenation of $\gamma_m|_{[1,\wt s_m-1]_{\BB Z}}$, $\beta_m$, and $\gamma_m|_{[s_m+1,|\gamma_m|]_{\BB Z}}$. Also let $t_{m+1}$ be the time for $\gamma_{m+1}$ at which it finishes tracing $\beta_m$, so that $\gamma_{m+1}(t_{m+1}) = \gamma_m(s_m)$. By the inductive hypothesis $\gamma_{m+1}$ is a path from $\lambda_-(z_0)$ to $\lambda_-(z_1)$ which coincides with $\gamma$ after time $t_{m+1}$, so condition~\ref{item-geo-approx-coincide} in the inductive hypothesis is satisfied with $m+1$ in place of~$m$. 

By our choice of $\beta_m$, the path $\gamma_{m+1}|_{[1, t_{m+1}]_{\BB Z}}$ crosses $\lambda_-([-Ln^{1/2}, Ln^{1/2}]_{\BB Z})$ at most one more time than $\gamma_{m}|_{[1, t_{m}]_{\BB Z}}$, so condition~\ref{item-geo-approx-hit} in the inductive hypothesis is satisfied with $m+1$ in place of $m$. 
 
We will now check condition~\ref{item-geo-approx-length}. For this, we recall the times $\sigma_e = (\gamma')^{-1}(e)$ from~\eqref{eqn-gamma'-hit}.
By condition~\ref{item-geo-approx-length} in the inductive hypothesis, 
$t_m \leq 2C \sigma_{\gamma_m(t_m)}$.
By condition~\ref{item-geo-approx-coincide} in the inductive hypothesis, $\gamma_m$ traces the final segment of $\gamma'$ after time $t_m$.
Combining the two preceding sentences shows that the time when $\gamma_m$ hits any given edge $e$ is at most $2C$ times the time when $\gamma'$ hits $e$.
In particular, 
\eqb \label{eqn-geo-trace-initial}
t_m \vee \wt s_m \leq 2C \sigma_{\gamma_m(t_m \vee \wt s_m )} .
\eqe   
Furthermore, using that $\gamma_m$ traces $\gamma'$ after time $t_m$ and $\gamma_m(s_m) = \gamma_{m+1}(t_{m+1})$, we get that
\eqb \label{eqn-geo-trace-initial'}
s_m - (t_m  \vee \wt s_m) = \sigma_{\gamma_m(s_m)} - \sigma_{\gamma_m(t_m \vee \wt s_m)} =  \sigma_{\gamma_{m+1}(t_{m+1})} - \sigma_{\gamma_m(t_m \vee \wt s_m)} .
\eqe

Lemma~\ref{prop-peel-ball} implies that each edge of $\bdy Q_{I_m}(C)$ lies at $Q_{\op{zip}}$-graph distance at least $R_{I_m}(C) $ from $\lambda_-(I_m)$. In particular, since $\tau_m \in [t_m \vee \wt s_m , s_m]$, we have $s_m - (t_m  \vee \wt s_m ) \geq  s_m - \tau_m \geq  R_{I_m}(C)$. Combining this with the definition of $\beta_m$ shows that
\eqb \label{eqn-time-gain}
t_{m+1} \leq \wt s_m + |\beta_m| \leq \wt s_m + 2 C R_{I_m}(C) \leq (t_m \vee \wt s_m)  + 2C \left(s_m - (t_m \vee \wt s_m )\right)  .
\eqe  
By combining this with~\eqref{eqn-geo-trace-initial} and~\eqref{eqn-geo-trace-initial'}, we get
\alb
t_{m+1} 
 \leq 2C \sigma_{\gamma_m(t_m \vee \wt s_m )}  + 2C \left( \sigma_{\gamma_{m+1}(t_{m+1})} - \sigma_{\gamma_m(t_m \vee \wt s_m)} \right)
 = 2C \sigma_{\gamma_{m+1}(t_{m+1})}  ,
\ale
which is condition~\ref{item-geo-approx-length} with $m+1$ in place of $m$. This completes the induction.
\medskip

\noindent\textit{Step 3: definition of the path $\wt\gamma$.} To conclude the proof, let $M$ be the smallest $m\in \BB N$ for which $ |\gamma_m| - \tau_m  \leq |\gamma'| - T_1'$, equivalently the smallest $m\in\BB N$ for which $\gamma_m = \gamma_M$ for each $m\geq M$. 
We will now check that the conditions in the statement of the lemma are satisfied for $\wt\gamma = \gamma_M$. 
It is clear that $\gamma_M$ is a path from $\lambda_-(z_0)$ to $\lambda_-(z_1)$. 
With the discrete interval $I_m$ as above, the path $\gamma'$ does not hit $\lambda_-(I_m)$ after hitting $\gamma_{m+1}(t_{m+1}) = \gamma_m(s_m)$ by the definition of $s_m$. Since $\gamma_{m+1}$ and $\gamma'$ coincide after hitting $\gamma_{m+1}(t_{m+1})$, it follows from the definition of $I_m$ that $I_m \not= I_{m+1}$ unless $m \geq M-1$. Since there are at most $2L \delta^{-2}$ elements of $\mcl I^n(\delta)$, we infer that $M \leq 2 L \delta^{-2}$. 

By condition~\eqref{item-geo-approx-hit} for $m=M$, the path $\gamma_M|_{[1 , t_M]_{\BB Z}}$ crosses $\lambda_-([-Ln^{1/2}, Ln^{1/2}]_{\BB Z})$ at most $2 L \delta^{-2}$ times.  By~\eqref{eqn-gamma'-times} the path $\gamma_M$ traces a segment of the one-sided geodesic $\gamma_1'$ after time $t_m$, so does not cross $\lambda_-([-Ln^{1/2}, Ln^{1/2}]_{\BB Z})$ after time $t_M$. 
Therefore $\gamma_M$ crosses $\lambda_-([-Ln^{1/2}, Ln^{1/2}]_{\BB Z})$ at most $2 L \delta^{-2}$ times.
By condition~\ref{item-geo-approx-length} for $m=M$ and~\eqref{eqn-gamma'-length}, 
\eqbn
|\gamma_M| \leq 2C |\gamma'| \leq  2C \op{dist}\left(\lambda_-(z_0) , \lambda_-(z_1) ; Q_{\op{zip}} \right) + 4C \delta^{1-2\zeta} n^{1/4}     .
\eqen 
Setting $\zeta = 1/4$ and possibly shrinking $\delta_*$ yields~\eqref{eqn-lipschitz-path} with $C = 2C(1/4)$.  \qed

\subsubsection{Proof of Proposition~\ref{prop-geodesic-away}}
\label{sec-geodesic-away-proof}

\begin{figure}[ht!]
 \begin{center}
\includegraphics[scale=1]{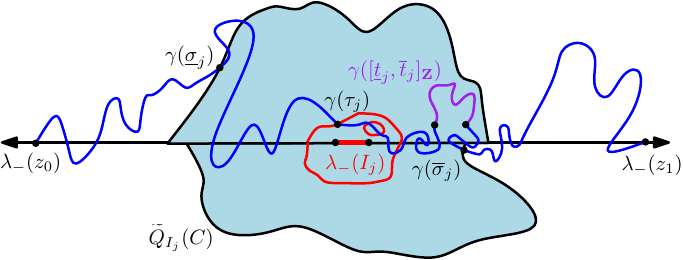} 
\caption[Illustration of the proof of Proposition~\ref{prop-geodesic-away}]{Illustration of the proof of Proposition~\ref{prop-geodesic-away}. Let $\gamma$ be a $Q_{\op{zip}}$-geodesic from $\lambda_-(z_0)$ to $\lambda_-(z_1)$. Given an interval $[  \delta n^{1/4} (j-1) , \delta n^{1/4} j]_{\BB Z}$ which contains a time $t$ for which $\gamma(t)$ lies within distance $\frac14 C^{-1} \delta n^{1/4}$ of the SAW, let $\tau_j$ be the smallest such time and choose a $\delta^2 n^{1/2}$-length SAW segment $\lambda_-(I_j)$ which lies within distance $\frac14 C^{-1} n^{1/4}$ of $\gamma(\tau_j)$ (the $\frac14 C^{-1} \delta n^{1/4}$-neighborhood of this SAW segment is outlined in red). 
Also let $\ul\sigma_j$ and $\ol\sigma_j$ be the first and last times at which $\gamma$ enters the glued peeling cluster $\wt{Q}_{I_j}(C)$ (light blue), which are finite for most choices of $j$ by Lemma~\ref{prop-good-radius-exist}. 
By definition of $\wt R_{I_j}(C)$ (recall Lemma~\ref{prop-good-radius'}), there must be a time interval $[\ul t_j , \ol t_j]_{\BB Z}\subset [\ul\sigma_j , \ol\sigma_j]_{\BB Z}$ of length at least $\frac14 C^{-1} \wt R_{I_j}(C)$ during which $\gamma$ is at distance at least $(2C)^{-1} \wt R_{I_j}(C) \geq (2C)^{-1} (\# I_j)^{1/2} \geq \frac14 C^{-1} \delta n^{1/4}$ away from the SAW (the image of this time interval under $\gamma$ is shown in purple). Furthermore, we have $\ol\sigma_j - \ul\sigma_j \leq 2 C \wt R_{I_j}(C)$. 
This leads to the conclusion $\gamma$ spends at least a $\frac12 C^{-2}$-fraction of its time (minus a small error) at distance at least $\frac14 C^{-1} \delta n^{1/4}$ away from our given segment of $\bdy Q_-$.
}\label{fig-geodesic-away}
\end{center}
\end{figure}

See Figure~\ref{fig-geodesic-away} for an illustration of the proof. 
Fix a small parameter $\zeta \in (0,1/4]$ (for our purposes we can take $\zeta =1/4$, but we find it more transparent to allow for an arbitrary $\zeta$ since this makes the source of various exponents more clear).
Fix $L>0$, $z_0,z_1 \in [-L n^{1/2} , L n^{1/2}]_{\BB Z}$. 
For $\delta \in (0,1)$ and $n\in\BB N$, define the set of intervals $\mcl I^n(\delta)$ as in~\eqref{eqn-good-radius-partition} and the clusters $\wt Q_I(C)$ for $I \in \mcl I^n(\delta)$ and $C>1$ as in~\eqref{eqn-good-hull-abbrv}. 

Also recall from the statement of Proposition~\ref{prop-geodesic-away} that for $\beta\in (0,1)$, $T_\gamma^\beta(\delta)$ denotes the set of ``good" times at which the geodesic $\gamma$ lies at $Q_{\op{zip}}$-distance at least $\beta\delta n^{1/4}$ from $\lambda_-([-Ln^{1/2} , L n^{1/2}]_{\BB Z})$. 
 
The main idea of the proof is to use condition~\ref{item-good-radius'-geo} of Lemma~\ref{prop-good-radius'} to argue that whenever a $Q_{\op{zip}}$-geodesic $\gamma$ from $\lambda_-(z_0)$ to $\lambda_-(z_1)$ gets close to one of the intervals $I\in\mcl I^n(\delta)$, it must subsequently spend a positive-length interval of time away from $\bdy Q_- \cup \bdy Q_+$. The proof is divided into five steps.
\begin{itemize}
\item[Step 1:] We define a regularity event $G^n$ which we will truncate on for most of the proof. On $G^n$, not too many of the glued peeling clusters $\wt Q_I(C)$ can contain the endpoints $\lambda_-(z_0)$ and $\lambda_-(z_1)$ and the graph-distance diameters of the boundary intervals $\lambda_-(I)$ for $I\in\mcl I^n(\delta)$ are bounded above. We then show that $G^n$ occurs with high probability when $\delta$ is small using Lemmas~\ref{prop-uihpq-bdy-holder} and~\ref{prop-good-radius-exist}. 
\item[Step 2:] We define a set $\mcl J$ of ``bad" indices $j\in [0,\delta^{-1} n^{-1/4} |\gamma|]_{\BB Z}$ corresponding to time intervals $[(j-1)\delta n^{1/4} , j\delta n^{1/4}]_{\BB Z}$ during which our $Q_{\op{zip}}$-geodesic $\gamma$ gets within distance $\frac14 C^{-1} \delta n^{1/4}$ of $\lambda_-([-Ln^{1/2} , Ln^{1/2}]_{\BB Z})$. We then show that for $\beta \in (0,\frac14 C^{-1}]$, the quantity $\#T_\gamma^\beta(\delta)$ can be bounded below by $|\gamma| - \delta n^{1/4} \#\mcl J$ minus a small error corresponding to possible pathologies near the starting and ending points of $\gamma$. See~\eqref{eqn-geodesic-away-upper}. \label{item-bad-index-step}  
\item[Step 3:] Using the definition of the clusters $\wt Q_I(C)$, we show that for each ``bad" index $j\in\mcl J$, there must exist a corresponding ``good" time interval $[\ul t_j , \ol t_j]_{\BB Z}$ of length at least $\frac12 C^{-1} \delta n^{1/4}$ during which $\gamma$ is not too close to $\lambda_-([-Ln^{1/2} , Ln^{1/2}]_{\BB Z})$. 
\item[Step 4:] We prove for each $j\in\mcl J$ an upper bound for the number of $j'\in\mcl J$ for which the corresponding intervals $[\ul t_j , \ol t_j]_{\BB Z}$ and $[\ul t_{j'} , \ol t_{j'}]_{\BB Z}$ overlap. See~\eqref{eqn-overlap-bound}. 
\item[Step 5:] Using the previous step, we prove that $\# T_\gamma^\beta(\delta)$ can be bounded below by a small constant times $\delta n^{1/4} \#\mcl J$ (see~\eqref{eqn-geodesic-away-lower}). Combining this with the lower bound from \hyperref[item-bad-index-step]{Step 2} and considering the worst possible value of $\#\mcl J$, we deduce our desired lower bound for $\# T_\gamma^\beta(\delta)$. 
\end{itemize}

\medskip
\noindent\textit{Step 1: regularity event.}
We first define a regularity event which we will work on for most of the proof. 
For $C>1$, let $G^n  = G^n(C , \delta , \zeta , L, z_0,z_1)$ be the event that the following are true. 
\begin{enumerate}
\item $z_0,z_1 \notin \wt Q_I(C)$ for each $I\in \mcl I^n(\delta)$ with $\op{dist}(I , \, \{z_0 , z_1\}) \geq \delta^{2-\zeta/4} n^{1/2}$.  \label{item-geo-away-reg-hull}
\item $\op{diam}\left( \lambda_-(I) ; Q_- \right) \leq \delta^{1-\zeta/4} n^{1/4}$ for each $I \in \mcl I^n(\delta)$. \label{item-geo-away-reg-diam}
\end{enumerate}
By Lemma~\ref{prop-uihpq-bdy-holder} (applied with $\delta^{-\zeta/8}$ in place of $C$, $\lfloor n^{1/2}\rfloor$ in place of $r$, and recalling the fact that $\#I  \approx \delta^2 n^{1/2}$ for $I\in\mcl I^n(\delta)$) and Lemma~\ref{prop-good-radius-exist} (applied with $\zeta/100$ in place of $\zeta$, say), for each $\alpha \in (0,1)$ there exists $\delta_* = \delta_*(\alpha,L,\zeta)$ and $C = C(\zeta)  > 1$ such that for $\delta \in (0,\delta_*]$, we have $\BB P[G^n] \geq 1-\alpha$ for large enough values of $n$. 
Henceforth fix such a $C$. We will prove that for an appropriate choice of $\beta$ and small enough $\delta \in (0,\delta_*]$ (depending only on $\alpha,L$, and $\zeta$), the condition in the proposition statement is satisfied whenever $G^n$ occurs. 
\medskip
 
\noindent\textit{Step 2: the set of ``bad" indices.} Assume now that $G^n$ occurs and let $\gamma$ be a $Q_{\op{zip}}$-geodesic from $\lambda_-(z_0)$ to $\lambda_-(z_1)$. 
For $j\in [0,\delta^{-1} n^{-1/4} |\gamma|]_{\BB Z}$, let 
\eqb \label{eqn-s_j-def}
s_j := \lfloor   \delta n^{1/4} j \rfloor . 
\eqe
Let $\mcl J$ be the set of ``bad" indices $j\in [1,\delta^{-1} n^{-1/4} |\gamma|]_{\BB Z}$ for which there exists $t \in  [s_{j-1}+1, s_j]_{\BB Z} $ and $I_j \in \mcl I^n(\delta)$ for which
\eqb \label{eqn-geodesic-close-time}
z_0,z_1 \notin \wt Q_{I_j}(C) \quad\op{and} \quad \op{dist}\left( \gamma(t) , \lambda_-(I_j) ; Q_{\op{zip}} \right) \leq \frac14 C^{-1} \delta n^{1/4} .
\eqe 
 
If we let $\beta \in  (0,   \frac14 C^{-1}]$ and let $T_\gamma^\beta(\delta)$ be as in Proposition~\ref{prop-geodesic-away}, then by condition~\ref{item-geo-away-reg-hull} in the definition of $G^n$, each $t \in [1,|\gamma|]_{\BB Z} \setminus T_\gamma^\beta(\delta)$ is either contained in $[s_{j-1} +1 , s_j]_{\BB Z}$ for some $j\in \mcl J$ or satisfies $\op{dist}(\gamma(t) , I; Q_{\op{zip}}) \leq \frac14 C^{-1} \delta n^{1/4}$ for one of the intervals $I\in \mcl I^n(\delta)$ with $\op{dist}(I , \, \{z_0 , z_1\}) \leq \delta^{2-\zeta/4} n^{1/2}$.
There are at most $4\delta^{-\zeta/4}$ intervals $I\in \mcl I^n(\delta)$ with $\op{dist}(I , \, \{z_0 , z_1\}) \leq \delta^{2-\zeta/4} n^{1/2}$ and by condition~\ref{item-geo-away-reg-diam} in the definition of $G^n$ each has $Q_{\op{zip}}$-diameter at most $\delta^{1-\zeta/4} n^{1/4}$. Since $\gamma$ is a geodesic, the total amount of time that $\gamma$ spends in the $\frac14 C^{-1} \delta n^{1/4}$-neighborhood of the union of these $\delta^{-\zeta/4}$ intervals is at most 
\eqbn
\left(\frac12 C^{-1} \delta n^{1/4}  +  \delta^{1-\zeta/4} n^{1/4} \right) \times 4 \delta^{-\zeta/4} \leq 8 \delta^{1-\zeta/2} n^{1/4} .
\eqen
In other words, the total number of times $[0,|\gamma|]_{\BB Z} \setminus T_\gamma^\beta(\delta)$ which are \emph{not} contained in $[s_{j-1} +1 , s_j]_{\BB Z}$ for some $j\in \mcl J$ is at most $ 8 \delta^{1-\zeta/2} n^{1/4} $. Since each interval in $\mcl J$ has length at most $\delta n^{1/4}$, it follows that
\eqb \label{eqn-geodesic-away-upper}
\# T_\gamma^\beta(\delta)  
\geq |\gamma| - \delta n^{1/4} \left( \#\mcl J  + 8 \delta^{-\zeta/2} \right) . 
\eqe 
We will now use the definition of $\wt R_I(C)$ from Lemma~\ref{prop-good-radius'} to prove another lower bound for $\# T_\gamma^\beta(\delta)$ which is increasing in $\# \mcl J$. Combining this bound with~\eqref{eqn-geodesic-away-upper} and considering the worst possible value of $\# \mcl J$ will give~\eqref{eqn-geodesic-away}. 
\medskip

\noindent\textit{Step 3: bad indices give rise to intervals during which $\gamma$ is far from $\bdy Q_-\cup \bdy Q_+$.} For $j\in \mcl J$, let $\tau_j$ be the smallest $t \in [s_{j-1}+1,s_j]_{\BB Z}$ for which~\eqref{eqn-geodesic-close-time} holds.
By Lemma~\ref{prop-peel-ball} and since $C>1$, $\gamma(t) \in \dot Q_{I_j}^{J_{I_j , \lfloor \delta n^{1/4} \rfloor}}$, which is contained in $\wt Q_{I_j}(C)$ since $\wt R_{I_j}(C) \geq \# (\lambda_-(I_j))^{1/2} \geq \lfloor \delta n^{1/4} \rfloor$ by definition (recall Lemma~\ref{prop-good-radius'}). 

Also let $\ul\sigma_j$ (resp.\ $\ol\sigma_j$) be the first (resp.\ last) time $t \in [1,|\gamma|]_{\BB Z}$ for which $\gamma(t)$ is incident to an edge of $\bdy \wt Q_{I_j}(C)$.
By~\eqref{eqn-geodesic-close-time}, $z_0,z_1 \notin \wt Q_{I_j}(C)$ so $\ul\sigma_j$ and $\ol\sigma_j$ are well-defined and (by the preceding paragraph) $\tau_j \in [\ul\sigma_j , \ol\sigma_j]_{\BB Z}$. 

By condition~\ref{item-good-radius'-geo} in the definition of $\wt R_{I_j}(C)$, for each $j\in\mcl J$ the path $\gamma$ must hit a vertex of $Q_{\op{zip}}$ which lies at $Q_{\op{zip}}$-graph distance at least $C^{-1} \wt R_{I_j}(C)$ from $ \lambda_-([-L n^{1/2}   , L n^{1/2}  ]_{\BB Z})$ between times $\tau_j$ and $\ol\sigma_j$, so must spend at least $\frac12 C^{-1} \wt R_{I_j}(C)$ units of time at distance greater than $\frac14 C^{-1 } \wt R_{I_j}(C)$ from $\lambda_-([-L n^{1/2}   , L n^{1/2}  ]_{\BB Z})$. 
Consequently, if we let $[\ul t_j , \ol t_j]_{\BB Z} \subset [\tau_j , |\gamma|]_{\BB Z}$ be the largest discrete interval contained in $[\tau_j , |\gamma|]_{\BB Z}$ such that
\eqb \label{eqn-good-interval-choice} 
\op{dist}\left(  \gamma([\ul t_j , \ol t_j]_{\BB Z})  ,   \,  \lambda_-([-L n^{1/2}   , L n^{1/2}  ]_{\BB Z}) \right) \geq \frac14 C^{-1} \delta n^{1/4} ,
\eqe 
then
\eqb \label{eqn-good-interval-length}
\ol t_j - \ul t_j \geq \frac12 C^{-1} \wt R_{I_j}(C) \quad \op{and} \quad  \ul t_j \leq \ol\sigma_j .
\eqe
\medskip
 
\noindent\textit{Step 4: bounding the overlap.} 
By~\eqref{eqn-good-interval-choice}, each of the intervals $[\ul t_j , \ol t_j]_{\BB Z}$ is contained in $T_\gamma^\beta(\delta)$ for $\beta \in  (0,   \frac14 C^{-1}]$, but it is possible that these intervals overlap. In order to get a lower bound for $\#T_\gamma^\beta(\delta)$, we need to prove an upper bound for the amount of overlap.\footnote{One might think to try to bound the overlap of the $[\ul t_j , \ol t_j]_{\BB Z}$'s by bounding the overlap of the $\wt Q_{I }(C)$'s. This is not helpful, however, since the overlap of the $\wt Q_{I }(C)$'s can be quite complicated. For example, it is even possible for one of these sets to contain another one.} 

We first argue that if $j,j' \in \mcl J$ with $[\ul t_{j'} , \ol t_{j'}]_{\BB Z} \cap  [\ul t_j , \ol t_j]_{\BB Z} \not=\emptyset $, then $[\ul t_{j'} , \ol t_{j'}]_{\BB Z}  =  [\ul t_j , \ol t_j]_{\BB Z}  $. Indeed, by the maximality of $[\ul t_j , \ol t_j]_{\BB Z}$ we must have $\ol t_{j'} = \ol t_j$. On the other hand, since $\gamma(\tau_{j'})$ lies at $Q_{\op{zip}}$-distance at most $\frac14 C^{-1} \delta n^{1/4}$ from $  \lambda_-([-L n^{1/2}   , L n^{1/2}  ]_{\BB Z})$, we have $\tau_{j'} \notin [\ul t_j , \ol t_j]_{\BB Z}$ and similarly with $j$ and $j'$ interchanged, so again by maximality $\ul t_{j'} = \ul t_j$. 

We next claim that
\eqb \label{eqn-overlap-bound}
\# \left\{j' \in \mcl J \,:\, [\ul t_{j'} , \ol t_{j'}]_{\BB Z} \cap  [\ul t_j , \ol t_j]_{\BB Z} \not=\emptyset  \right\}  \leq   \frac{4C^2}{\delta  n^{1/4}}( \ol t_j - \ul t_j)     , \quad \forall j \in \mcl J .
\eqe
Indeed, suppose $j\in\mcl J$. Condition~\ref{item-good-radius'-diam} in the definition of $\wt R_{I_j}(C)$ implies that the $Q_{\op{zip}}$-diameter of $\bdy \wt Q_{I_j}(C)$ is at most $C \wt R_{I_j}(C)$, so since $\gamma$ is a geodesic and $\gamma(\ul\sigma_j) $ and $ \gamma(\ol\sigma_j)$ are incident to $\bdy \wt Q_{I_j}(C)$, 
\eqb \label{eqn-good-interval-sep}
0 \leq \ul t_j - s_{j-1}  \leq   \ol\sigma_j - \ul\sigma_j    \leq 2 C \wt R_{I_j}(C)  .
\eqe 
If $j' \in \mcl J$ for which $[\ul t_{j'} , \ol t_{j'}]_{\BB Z} \cap  [\ul t_j , \ol t_j]_{\BB Z} \not=\emptyset $ (equivalently, $[\ul t_{j'} , \ol t_{j'}]_{\BB Z} =   [\ul t_j , \ol t_j]_{\BB Z}$ by the preceding paragraph) then we have that, 
\begin{align*}
0 \leq \ul t_j - s_{j'-1}
&= \ul t_{j'} - s_{j'-1} \quad \text{(since $\ul t_j = \ul t_{j'}$)}\\
&\leq 2 C \wt R_{I_{j'}}(C) \quad\text{(by \eqref{eqn-good-interval-sep} with $j'$ instead of $j$)}\\
&\leq 4 C^2 ( \ol t_{j'} - \ul t_{j'} ) \quad\text{(by \eqref{eqn-good-interval-length} with $j'$ instead of $j$)}\\
&= 4 C^2 ( \ol t_j - \ul t_j) \quad\text{(since $\ul t_j = \ul t_{j'}$ and $\ol t_j = \ol t_{j'}$)}.
\end{align*}
Therefore, every such $j' \in \mcl J$ satisfies $s_{j'-1} \in [\ul t_j - 4 C^2 (\ol t_j - \ul t_j) , \ul t_j]$ so (recalling~\eqref{eqn-s_j-def}) the number of such $j'$ is at most the right side of~\eqref{eqn-overlap-bound}. 
\medskip

\noindent\textit{Step 5: conclusion.} Since $[\ul t_j , \ol t_j]_{\BB Z} \subset T_\gamma^\beta(\delta)$ for each $j\in \mcl J$ and each $\beta \in  (0,   \frac14 C^{-1}]$, we infer from~\eqref{eqn-overlap-bound} that for each such $\beta$, 
\eqb \label{eqn-geodesic-away-lower}
\# T_\gamma^\beta(\delta)
\geq \sum_{j\in \mcl J} \frac{ \ol t_j - \ul t_j }{ \# \left\{j' \in \mcl J \,:\, [\ul t_{j'} , \ol t_{j'}]_{\BB Z} \cap  [\ul t_j , \ol t_j]_{\BB Z} \not=\emptyset  \right\}  }
\geq \frac{\delta n^{1/4}}{4 C^2}  \# \mcl J .
\eqe  
By~\eqref{eqn-geodesic-away-upper} and~\eqref{eqn-geodesic-away-lower}, 
\eqbn
\# T_\gamma^\beta(\delta) \geq \max\left\{ \frac{\delta n^{1/4}}{4C^2}  \# \mcl J ,\, |\gamma| - \delta n^{1/4} (\# \mcl J + 8 \delta^{-\zeta/2}) \right\} .
\eqen
By considering the value of $\#\mcl J$ for which the right side is minimized, we get
\eqbn
\# T_\gamma^\beta(\delta) \geq \frac{ |\gamma|  -    8 \delta^{1- \zeta/2} n^{1/4}  }{1 + 4C^2 }  .
\eqen
We now conclude by choosing $\beta \leq \min\left\{\frac14 C^{-1} ,   (1+ 4C^2)^{-1} \right\}$ and shrinking $\delta_*$ in such a way that $8 (1+ 4 C^2)^{-1} \delta^{1- \zeta/2} \leq \delta^{1/2}$ for each $\delta \in (0,\delta_*]$. \qed
 
%Solve[(1/wtC) delta n^(1/4) jj ==  gamma -  delta n^(1/4) (jj + 8 delta^(-zeta/2)), jj]FullSimplify[(1/wtC) delta n^(1/4) jj /. %]

\section{Proof of main theorems}
\label{sec-saw-conv}

In this section we will complete the proof of Theorem~\ref{thm-saw-conv-wedge}.  At the end, we will briefly remark on the minor adaptations necessary to prove Theorem~\ref{thm-saw-conv-2side} and Theorem~\ref{thm-saw-conv-cone} in Remark~\ref{remark-other-thm}. See Figure~\ref{fig-gluing-maps} for an illustration of the objects involved in the proof.
 
We will begin in Section~\ref{sec-saw-conv-tight} by introducing some notation and establishing tightness of the $4$-tuples $(Q_{\op{zip}}  , d_{\op{zip}}^n , \mu_{\op{zip}}^n , \eta_{\op{zip}}^n)$ in the local GHPU topology.  By the Prokhorov theorem and since we already know from~\cite{gwynne-miller-uihpq} that both of the $4$-tuples $(Q_\pm  , d_\pm^n , \mu_\pm^n , \eta_\pm^n)$ converge to Brownian half-planes in the local GHPU topology, we can find a random element $(\wt X , \wt d , \wt\mu,\wt\eta) \in \BB M_\infty^{\op{GHPU}}$ coupled with the Brownian half-planes $(X_\pm , d_\pm)$ and a subsequence $\mcl N$ along which the joint law of $(Q_{\op{zip}}  , d_{\op{zip}}^n , \mu_{\op{zip}}^n , \eta_{\op{zip}}^n)$, $(Q_-  , d_-^n , \mu_-^n , \eta_-^n)$, and $(Q_+  , d_+^n , \mu_+^n , \eta_+^n)$ converges in the local GHPU topology to the joint law of $(\wt X , \wt d , \wt \eta,\wt\mu)$, $(X_-, d_-, \mu_- , \eta_-)$, and $(X_+, d_+, \mu_+ , \eta_+)$.

In the remainder of the section we fix such a subsequential limit and aim to show that $(\wt X , \wt d , \wt\mu,\wt\eta) $ is equivalent (as a curve-decorated metric measure space) to our desired limiting space $ (X_{\op{zip}} , d_{\op{zip}} , \mu_{\op{zip}} , \eta_{\op{zip}})$, which we recall is obtained as the metric quotient of $ (X_-, d_-)$ and $(X_+ ,d_+)$ identified along their positive boundary rays, as in Section~\ref{sec-main-result} and Figure~\ref{fig-thick-gluing}. The measure $\mu_{\op{zip}}$ is the sum of the pushforwards of $\mu_-$ and $\mu_+$ under the quotient map and the curve $\eta_{\op{zip}}$ is the gluing interface (i.e., the image of the positive boundary rays under the quotient map). 
 We now give a brief overview of how we show that $(\wt X , \wt d , \wt\mu,\wt\eta)   =   (X_{\op{zip}} , d_{\op{zip}} , \mu_{\op{zip}} , \eta_{\op{zip}})$.

We begin in Section~\ref{sec-saw-conv-estimate} by using the results of Section~\ref{sec-hull-moment-misc} to establish some qualitative properties of the subsequential limiting curve $\wt \eta$. Namely, we will deduce that $\wt \eta$ is necessarily transient (i.e., it is a.s.\ the case that the amount of time it spends in any compact set is finite), simple, and satisfies $\wt\mu(\wt\eta)=0$.
 
In Section~\ref{sec-1side-map}, we show that there is a bijective, 1-Lipschitz, curve-preserving, measure-preserving map $f_{\op{zip}} : X_{\op{zip}} \rta \wt X$ which preserves the length of each path in $X_{\op{zip}}$ which does not hit the gluing interface $\eta_{\op{zip}}$. 
Roughly speaking, the reason that such a map $f_{\op{zip}}$ exists is as follows. 
Using that $(\wt X , \wt d , \wt \eta,\wt\mu)$ is a subsequential limit of the glued maps $(Q_{\op{zip}}  , d_{\op{zip}}^n , \mu_{\op{zip}}^n , \eta_{\op{zip}}^n)$, one can produce $1$-Lipschitz, measure-preserving maps $f_\pm : (X_\pm , d_\pm) \rta (\wt X ,\wt d)$ which preserve the lengths of curves which do not touch the gluing interface and satisfy $f_-(X_-) \cup f_+(X_+) = \wt X$, $f_-(X_-) \cap f_+(X_+) = \wt\eta$, and $f_-(\eta_-(t)) = f_+(\eta_+(t))$ for each $t\geq 0$ (see Figure~\ref{fig-gluing-maps} for an illustration of these maps). The existence of $f_\pm$ is established using limiting arguments and elementary metric space theory. 
The maps $f_-$ and $f_+$ give rise to a 1-Lipschitz map from the disjoint union of $X_-$ and $X_+$ to $\wt X$. By the universal property of the quotient metric (Remark~\ref{remark-quotient-universal}) and the condition that $f_-(\eta_-(t)) = f_+(\eta_+(t))$ for each $t\geq 0$, we can factor the 1-Lipschitz map of the preceding sentence through the quotient map $X_-\cup X_+ \rta X_{\op{zip}}$ to get the desired map $f_{\op{zip}} : X_{\op{zip}} \rta \wt X$. 

In Section~\ref{sec-saw-proof}, we conclude the proof by showing that the above map $f_{\op{zip}}$ is in fact an isometry, so that $(\wt X , \wt d , \wt\mu,\wt\eta)   =   (X_{\op{zip}} , d_{\op{zip}} , \mu_{\op{zip}} , \eta_{\op{zip}})$ as curve-decorated metric measure spaces. This is the most interesting part of the argument, and is based on the results of Section~\ref{sec-geodesic-properties}.
In particular, we first deduce from Proposition~\ref{prop-lipschitz-path} that $f_{\op{zip}}^{-1}$ is a.s.\ $C$-Lipschitz for some \emph{deterministic} constant $C\geq 1$ and from Proposition~\ref{prop-geodesic-away} that there a.s.\ exists a $\wt d$-geodesic between and two fixed points of $\wt\eta$ which spends at least a $\beta$-fraction of its time outside of $\wt \eta$ for some $\beta \in (0,1)$. 

We then look at such a $\wt d$-geodesic $\gamma$ between two fixed points of $\wt\eta$ and decompose it as the union of finitely many segments during which it does not hit $\wt\eta$, with total length $a |\gamma| \geq (\beta/2) |\gamma|$; and finitely many complementary segments with total length $(1-a)|\gamma| \leq (1-\beta/2)|\gamma|$ (see Figure~\ref{fig-lip-geo} for an illustration), where here $a \in [\beta/2,1)$ is random.  The union of the latter collection of segments contains the intersection of $\gamma$ with $\wt{\eta}$.  Since $f_{\op{zip}}$ preserves the lengths of paths which do not hit $\eta_{\op{zip}}$, the total $d_{\op{zip}}$-length of the images of the first collection of segments under $f_{\op{zip}}^{-1}$ is equal to $a|\gamma|$. Since $f_{\op{zip}}^{-1}$ is $C$-Lipschitz, the total $d_{\op{zip}}$-length of the images of the second collection of segments under $f_{\op{zip}}^{-1}$ is at most $C (1-a) |\gamma|$. Hence $|f_{\op{zip}}^{-1}(\gamma)| \leq a |\gamma| + C(1-a)|\gamma|$.  Since $C \geq 1$, the right hand side is maximized when we make $a \geq \beta/2$ as small as possible.  Consequently, $|f_{\op{zip}}^{-1}(\gamma)| \leq (\beta/2) |\gamma| + C(1-\beta/2)|\gamma|$.  Since $\gamma$ was an arbitrary geodesic, this implies that $f_{\op{zip}}^{-1}$ is a.s.\ $(\beta/2 + C(1-\beta/2))$-Lipshitz.  As $\beta \in (0,1)$, we can therefore take $C=1$.

\subsection{Setup and tightness}
\label{sec-saw-conv-tight}

For the proof of Theorem~\ref{thm-saw-conv-wedge}, we will use a slightly modified version of the notation used in the theorem statement where we define the $n$th rescaled objects with respect to a different copy of the map. The reason for this is that we will be applying the Skorokhod representation theorem for weak convergence, so that we can couple everything together on the same probability space. 
  
For $n\in\BB N$, let $Q_{\op{zip}}^n = Q_-^n \cup Q_+^n$ be a copy of the quadrangulation $Q_{\op{zip}}$ from Theorem~\ref{thm-saw-conv-wedge}.  We view $Q_{\op{zip}}^n$ and $Q_\pm^n$ as connected metric spaces by identifying each edge with an isometric copy of the unit interval in $\BB R$, as in Remark~\ref{remark-ghpu-graph}. 
Let $d_-^n$, $d_+^n$, and~$d^n_{\op{zip}}$ be the graph metrics on~$Q_-^n$, $Q_+^n$, and~$Q_{\op{zip}}^n$, respectively, re-scaled by $(9/8)^{1/4} n^{-1/4}$.  
Let~$\mu_-^n$, $\mu_+^n$, and~$\mu_{\op{zip}}^n$ be the measures on $Q_-^n$, $Q_+^n$, and~$Q_{\op{zip}}^n$, respectively, which assign to each vertex a mass equal to $(4n)^{-1}$ times its degree.  
 
Let $\lambda_\pm^n$ be the boundary paths of $Q_\pm^n$, respectively, started from the root edge, viewed as paths from $\BB R$ to $Q_\pm^n$ in the manner of Remark~\ref{remark-ghpu-graph}. For $t\in\BB R$, let $\eta_\pm^n(t) := \lambda_\pm\left( \frac{2^{3/2}}{3} n^{1/2} t \right)$. Also let $\eta_{\op{zip}}^n := \eta_-^n|_{[0,\infty)}$, which is equal to $ \eta_+^n|_{[0,\infty)}$ since $Q_\pm^n$ are glued together along $\lambda_\pm^n([0,\infty))$ to obtain $Q_{\op{zip}}^n$. 

As in Theorem~\ref{thm-saw-conv-wedge}, let $(X_- , d_- )$ and $(X_+ , d_+)$ be a pair of independent Brownian half-planes. Let $\mu_\pm$ be the area measure on $X_\pm$ and let $\eta_\pm : \BB R \rta \bdy X_\pm$ be the parameterization of $\bdy X_\pm$ according to boundary length, normalized so that $\eta_\pm(0)$ is the marked point of $\bdy X_\pm$ (see Section~\ref{sec-bhp-prelim} for a definition of the boundary length measure on the Brownian half-plane). 
Let $(X_{\op{zip}} , d_{\op{zip}})$ be the metric space quotient of the disjoint union of $(X_- ,d_-)$ and $(X_+ , d_+)$ under the equivalence relation which identifies $\eta_-(t)$ with $\eta_+(t)$ for each $t\geq 0$.
  Also let $\mu_{\op{zip}}$ be the measure on $X_{\op{zip}}$ which restricts to the pushforward of $\mu_\pm$ under the quotient map on the image of $X_\pm$ under the quotient map.  Let $\eta_{\op{zip}} : [0,\infty) \rta X_{\op{zip}}$ be the path which is the image of $\eta_-([0,\infty))$ (equivalently $\eta_+([0,\infty))$) under the quotient map.  
  
Define the curve-decorated metric measure spaces
\begin{align} \label{eqn-cmm-spaces}
&\frk Q_{\op{zip}}^n := \left( Q_{\op{zip}}^n , d_{\op{zip}}^n , \mu_{\op{zip}}^n , \eta_{\op{zip}}^n \right) \qquad
\frk Q_{\pm}^n := \left( Q_{\pm}^n , d_{\pm}^n , \mu_{\pm}^n , \eta_{\pm}^n \right) \notag \\
&\frk X_{\op{zip}}  := \left( X_{\op{zip}} , d_{\op{zip}} , \mu_{\op{zip}} , \eta_{\op{zip}} \right)  \qquad
\frk X_\pm  := \left( X_\pm , d_\pm , \mu_\pm , \eta_\pm \right)  .
\end{align}
See Figure~\ref{fig-gluing-maps} for an illustration of the above objects (plus some additional objects, to be introduced later).

\begin{figure}[ht!]
 \begin{center}
\includegraphics[scale=.85]{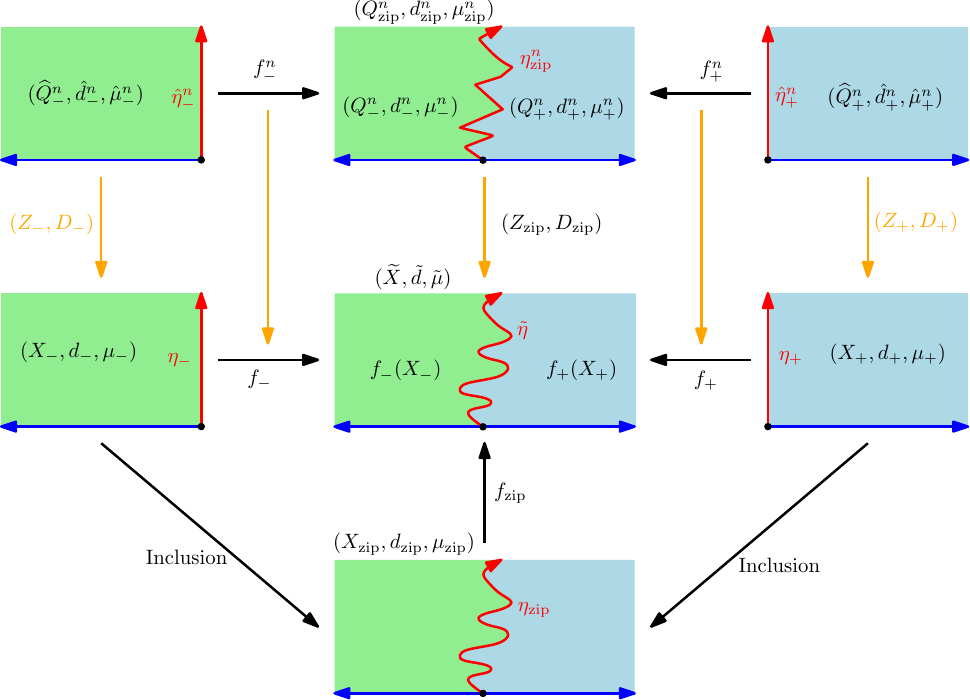} 
\caption{Illustration of the objects used in the proof of Theorem~\ref{thm-saw-conv-wedge}. Orange arrows indicate convergence and black arrows indicate functions. The spaces on the left (resp.\ right) are isometrically embedded into $(Z_-, D_-)$ (resp.\ $(Z_+ , D_+)$) and the top two spaces are isometrically embedded into $(Z_{\op{zip}}, D_{\op{zip}} )$. The functions $f_\pm$ (resp.\ $f_{\op{zip}}$) are shown to exist in Lemma~\ref{prop-map-limit} (resp.\ Proposition~\ref{prop-zip-map}). The orange arrows between spaces indicate convergence in the HPU topology of Definition~\ref{def-hpu-local} (along the subsequence $\mcl N$) for the larger space into which they are embedded. The orange arrows between maps indicate convergence in the sense of Lemma~\ref{prop-map-limit} (along a subsequence $\mcl N'\subset\mcl N$). To prove Theorem~\ref{thm-saw-conv-wedge} we will show that $f_{\op{zip}}$ is an isometry which pushes forward $\eta_{\op{zip}}$ to $\wt\eta$ and $\mu_{\op{zip}}$ to $\wt\mu$. 
}\label{fig-gluing-maps}
\end{center}
\end{figure}

By~\cite[Theorem~1.12]{gwynne-miller-uihpq}, the joint law of $(\frk Q_-^n , \frk Q_+^n)$ converges in the local GHPU topology to the joint law of $(\frk X_- , \frk X_+)$. 
We want to establish tightness of $\{\frk X_{\op{zip}}^n\}_{n\in\BB N}$ in the local GHPU topology.
For this purpose we record the following estimate which will also be used several times in later subsections.
 
\begin{lem} \label{prop-ball-contain}
For $n\in\BB N$ and $R , \rho > 0$, let $G^n(\rho , R)$ be the event that the following are true. 
\begin{enumerate}
\item $\eta_{\op{zip}}^n\left([R^2,\infty) \right)  \cap B_\rho\left( \eta_{\op{zip}}^n(0) ; d_{\op{zip}}^n \right) = \emptyset$. \label{item-2side-ball-contain} 
\item $B_\rho\left( \eta_{\op{zip}}^n(0) ; d_{\op{zip}}^n \right) \subset B_R\left( \eta_-^n(0) ; d_-^n \right) \cup B_R\left( \eta_+^n(0) ; d_+^n \right)$. \label{item-full-ball-contain}
\end{enumerate}
For each $\ep \in (0,1)$ and each $\rho > 0$, there exists $R = R(\rho,\ep)> \rho$ such that for each $n\in \BB N$,  
\eqbn
\BB P\left[ G^n(\rho,R)\right] \geq 1-\ep .
\eqen
\end{lem}
\begin{proof}
This is immediate from Proposition~\ref{prop-hull-moment} and Lemma~\ref{prop-hull-diam}. 
\end{proof}

\begin{lem} \label{prop-tight}
The laws of the curve-decorated metric measure spaces $\frk Q_{\op{zip}}^n$ are tight in the local GHPU topology.
\end{lem}
\begin{proof} 
Fix $\epsilon \in (0,1)$ and $\rho > 0$. Let $R = R(\rho,\ep) > \rho$ be chosen so that the conclusion of Lemma~\ref{prop-ball-contain} is satisfied and let $G^n(\rho , R)$ for $n\in\BB N$ be as in that lemma. 
Since $(Q_\pm^n , d_\pm^n ,\mu^n , \eta_\pm^n)$ each converge to non-degenerate limits in the local GHPU topology, we can find $ C , N , \delta>0$ depending only on $\rho$ and $\ep$ such that for each $n\in\BB N$, it holds with probability at least $1-\ep$ that the following are true. 
\begin{enumerate} 
\item $G^n(\rho, R)$ occurs. \label{item-tight-contain} 
\item $B_R\left( \eta_-^n(0) ; d_-^n \right)$ can be covered by at most $N$ $d_-^n$-metric balls of radius $\ep$. \label{item-tight-bded}
\item $\mu_-^n \left( B_R\left( \eta_-^n(0) ; d_-^n \right) \right) \leq C$. \label{item-tight-measure}
\item $d_-^n(\eta_-^n(s) , \eta_-^n(t)) \leq \delta$ whenever $s,t\in [0 , R^2]$ with $|s-t| \leq \ep$. \label{item-tight-path}
\item The previous three conditions also hold with ``$+$" in place of ``$-$." 
\end{enumerate}
Since $\rho$ and $\ep$ are arbitrary the desired tightness therefore follows from~\cite[Lemma~2.12]{gwynne-miller-uihpq}.  
\end{proof}  
 
By Lemma~\ref{prop-tight},~\cite[Theorem~1.12]{gwynne-miller-uihpq}, and the Prokhorov theorem, for any sequence of positive integers tending to $\infty$, there exists a subsequence $\mcl N$, a curve-decorated metric measure space $\wt{\frk X} = (\wt X , \wt d , \wt \mu, \wt \eta) \in \BB M_\infty^{\op{GHPU}}$, and a coupling of $\wt{\frk X}$ with $(\frk X_-,\frk X_+)$ such that 
\eqb \label{eqn-subsequential-ghpu}
\left( \frk Q_{\op{zip}}^n ,\frk Q_-^n , \frk Q_+^n \right) \rta \left( \wt{\frk X} , \frk X_-, \frk X_+ \right) 
\eqe 
in law in the local GHPU topology as $\mcl N \ni n \rta\infty$. Henceforth fix $\mcl N$, $\wt{\frk X}$, and a coupling as above. We will show that in any such coupling, $\wt{\frk X} = \frk X_{\op{zip}}$ a.s.

By the Skorokhod representation theorem, we can couple the sequence $\left\{(Q_{\op{zip}}^n,  Q_-^n , Q_+^n)\right\}_{n\in\mcl N}$ with $(  \wt{\frk X} , \frk X_-,\frk X_+)$ in such a way that the convergence in~\eqref{eqn-subsequential-ghpu} holds a.s. For our proofs, it will be convenient to embed the spaces $\frk Q_{\op{zip}}^n$ (resp.\ $\frk Q_-^n$, $\frk Q_+^n$) for $n\in\mcl N$ together with their limiting space $\wt{\frk X}$ (resp.\ $\frk X_-$, $\frk X_+$) into a common metric space.

By Proposition~\ref{prop-ghpu-embed-local}, we can a.s.\ find (random) boundedly compact (i.e., closed bounded subsets are compact) metric spaces $(Z_{\op{zip}} , D_{\op{zip}})$, $(Z_- , D_-)$, and $(Z_+ , D_+)$ and isometric embeddings 
\alb
&\iota_{\op{zip}}^n : (Q_{\op{zip}}^n , d_{\op{zip}}^n ) \rta (Z_{\op{zip}} , D_{\op{zip}}) ,\quad 
\iota_{\op{zip}} :     (\wt X , \wt d) \rta (Z_{\op{zip}} , D_{\op{zip}}) ,\\
&\iota_\pm^n : (Q_\pm^n , d_\pm^n) \rta (Z_\pm , D_\pm) , \quad \op{and} \quad 
\iota_\pm : \left(X_\pm , d_\pm \right) \rta (Z_\pm , D_\pm)
\ale
such that a.s.\ 
\alb
\iota_{\op{zip}}^n\left( \frk Q_{\op{zip}}^n \right) \rta \iota_{\op{zip}}\left( \wt{\frk X} \right) \quad \op{and} \quad 
\iota_{\pm}^n\left( \frk Q_\pm^n \right) \rta \iota_{\pm}\left(  \frk X_\pm \right)
\ale
in the $D_{\op{zip}}$- and $D_\pm$-local Hausdorff-Prokhorov-uniform (HPU) topologies, respectively (Definition~\ref{def-hpu-local}). 

To lighten notation, we henceforth identify $\frk Q_{\op{zip}}^n$, $\frk Q_-^n$, $\frk Q_+^n$, and $\wt{\frk X}$ with their images under $\iota_{\op{zip}}^n$ and $\iota_{\op{zip}}$, respectively (so that $Q_{\op{zip}}^n = Q_-^n \cup Q_+^n$ and $\wt X$ are subsets of $Z_{\op{zip}}$, etc.). We similarly identify $\frk X_\pm$ with their images under $\iota_\pm$ and (since we have already identified $\frk Q_\pm^n$ with their images under $\iota_{\op{zip}}^n$) we define
\eqbn
\wh{\frk Q}_\pm^n =  \left( \wh Q_\pm^n , \wh d_\pm^n  ,\wh \mu_\pm^n , \wh \eta_\pm^n \right) := \iota_\pm^n \left( \frk Q_\pm^n \right) \subset Z_\pm .
\eqen
Let\footnote{Here, to define $(\iota_\pm^n)^{-1}$ we restrict the co-domain of $\iota_\pm^n$ to $\wh Q_\pm^n \subset Z_\pm$ instead of $Z_\pm$.}
\eqb \label{eqn-one-side-maps}
f_\pm^n :=  \iota_{\op{zip}}^n \circ (\iota_\pm^n )^{-1} : \wh Q_\pm^n \rta Q_\pm^n  
\eqe 
be the map taking us from the embedding of $\wh Q_\pm^n$ into $Z_\pm$ to its embedding into $Z_{\op{zip}}$. 
Since $\iota_\pm^n$ is an isometry and $d_{\op{zip}}^n|_{Q_\pm^n}$ is dominated by the internal metric $d_\pm^n$, each $f_\pm^n$ is $1$-Lipschitz with respect to the metric $\wh d_\pm^n$ on the domain and the restriction of the metric $d_{\op{zip}}^n$ on the range. Furthermore, $f_\pm^n\left( \wh{\frk Q}_\pm^n \right) = \frk Q_\pm^n$. 

We will continue to use the objects introduced just above throughout the remainder of this section. See Figure~\ref{fig-gluing-maps} for an illustration of these objects.

\subsection{Basic properties of the subsequential limiting curve}
\label{sec-saw-conv-estimate}

Throughout the remainder of this section we continue to use the notation introduced in Section~\ref{sec-saw-conv-tight}. We will prove some basic facts about the curve $\wt\eta$ associated with the subsequential limiting object $\wt{\frk X}  $ which follow relatively easily from estimates for $Q_{\op{zip}}$ proven earlier in the paper. In particular, we will show that it is simple, transient, and has zero $\wt\mu$-mass. 

We first show that the subsequential limiting curve $\wt\eta$ is simple, which is a consequence of Lemma~\ref{prop-reverse-holder}. 

\begin{lem} \label{prop-simple-curve}
Almost surely, the curve $\wt\eta$ is simple. In fact, for each $L > 0$ and each $p  > 3/2$, there a.s.\ exists $c  >0$ such that for each $\tau_1,\tau_2\in [0,L]$,
\eqb \label{eqn-simple-curve}
\wt d\left( \wt\eta(\tau_1) , \wt\eta(\tau_2) \right) \geq c |\tau_1-\tau_2|^p .
\eqe 
\end{lem}
 
Once we have concluded the proof of Theorem~\ref{thm-saw-conv-wedge}, Lemma~\ref{prop-simple-curve} will imply that the limiting curve $\eta_{\op{zip}}$ in that theorem is simple and that its inverse is locally H\"older continuous of any exponent strictly smaller than $2/3$. The simplicity of $\eta_{\op{zip}}$ can also be deduced from~\cite[Corollary~1.2]{gwynne-miller-gluing}, but the H\"older continuity statement is not obvious from either the Brownian or the $\sqrt{8/3}$-LQG descriptions of $(X_{\op{zip}} , d_{\op{zip}} , \mu_{\op{zip}} , \eta_{\op{zip}} )$. 
 
\begin{proof}[Proof of Lemma~\ref{prop-simple-curve}]
Fix $L > 0$ and $p  > 3/2$ and let $\beta = 1/p \in (0,2/3)$. 
Since $\eta_{\op{zip}}^n\rta \wt\eta$ uniformly on compact subsets of $[0,\infty)$, we can take the scaling limit of the estimate of Lemma~\ref{prop-reverse-holder} to find that $\BB P[E_\delta] \geq 1 - \delta^{\frac32 (2-\beta) - 2 + o_\delta(1)}$, where 
\eqbn
E_\delta := 
\left\{
\wt d\left( \wt\eta(\tau_1) , \wt\eta(\tau_2) \right) \geq \delta ,\,
\forall \tau_1,\tau_2\in [0,L] \: \text{with $|\tau_1-\tau_2| \geq \delta^\beta $} 
\right\} .
\eqen
By the Borel-Cantelli lemma, there a.s.\ exists $K\in\BB N$ such that $E_{2^{-k}}$ occurs for each $k\geq K$. 
If this is the case and $\tau_1,\tau_2\in [0,L]$ are distinct times with $|\tau_1-\tau_2|\leq 2^{-\beta K}$, choose $k\geq K$ such that $2^{-\beta (k+1)} \leq |\tau_1-\tau_2|\leq 2^{-\beta k}$. Then since $E_{2^{-k}}$ occurs,
\eqbn
\wt d\left( \wt\eta(\tau_1) , \wt\eta(\tau_2) \right) \geq 2^{-k} \succeq |\tau_1-\tau_2|^p 
\eqen
with implicit constant depending only on $\beta$. On the other hand, if $|\tau_1-\tau_2| \geq 2^{-\beta K}$, then since $E_{2^{-K}}$ occurs,
\eqbn
\wt d\left( \wt\eta(\tau_1) , \wt\eta(\tau_2) \right) \geq 2^{-K} \succeq |\tau_1-\tau_2|^p
\eqen
with implicit constant depending on $K$. Hence~\eqref{eqn-simple-curve} holds.
\end{proof}

Next we check transience of $\wt\eta$.

\begin{lem} \label{prop-transient}
Almost surely, the curve $\wt\eta$ is transient.  That is, for each $\rho >0$ there exists $R  > 0$ such that $B_\rho(\wt\eta(0) ; \wt d)\cap \wt\eta([R^2 ,\infty)) =\emptyset$.  
\end{lem}
\begin{proof}
Fix $\rho > 0$ and $\ep \in (0,1)$. By Lemma~\ref{prop-ball-contain}, we can find $R > 0$ such that for each $n\in\BB N$, the probability of the event $G^n(\rho ,R)$ defined in that lemma is at least $1-\ep$. Hence except on an event of probability at most $\ep$, there is a sequence $\mcl N'\subset\mcl N$ such that $G^n(\rho,R)$ occurs for each $n\in\mcl N'$. If this is the case, then for each $n\in\mcl N'$ and each $T\geq R^2 $, we have $  B_\rho(\eta_{\op{zip}}^n(0) ; d_{\op{zip}}^n) \cap \eta_{\op{zip}}^n([R^2 , T])  =\emptyset$. Since $\frk Q_{\op{zip}}^n \rta \wt{\frk X}$ in the $D_{\op{zip}}$-local HPU topology, it must be the case that $  B_\rho(\wt\eta(0) ; \wt d) \cap \wt\eta([R^2 , T])  =\emptyset$ for each $T\geq R^2$. Thus $B_\rho(\wt\eta(0) ; \wt d)\cap \wt\eta([R^2,\infty)) =\emptyset$.
\end{proof} 

Finally we show that $\wt\eta$ has zero mass.

\begin{lem} \label{prop-0mass}
Almost surely, $\wt\mu(\wt\eta) = 0$. 
\end{lem} 
\begin{proof}
Fix $T>0$ and $\ep \in (0,1)$. By Lemma~\ref{prop-uihpq-bdy-holder}, we can find $\rho =\rho(T,\ep) > 0$ such that for each $n\in\BB N$, it holds with probability at least $1-\ep/3$ that
\eqb \label{eqn-0mass-ball}
\eta_{\op{zip}}^n([0,T]) \subset B_\rho\left( \eta_{\op{zip}}^n(0) ; d_{\op{zip}}^n \right) .
\eqe 
By Lemma~\ref{prop-ball-contain}, we can find $R  = R(T,\ep) > 0$ such that for each $n\in\BB N$, it holds with probability at least $1-\ep/3$ that
\eqb \label{eqn-0mass-curve}
\eta_{\op{zip}}^n([R^2,\infty)) \cap B_{\rho +1}\left( \eta_{\op{zip}}^n(0) ; d_{\op{zip}}^n \right) = \emptyset .
\eqe 
If $n\in\BB N$ is such that~\eqref{eqn-0mass-ball} and~\eqref{eqn-0mass-curve} hold, $\delta \in (0,1)$, and $x \in Q_\pm^n \cap B_\rho\left( \eta_{\op{zip}}^n(0) ; d_{\op{zip}}^n \right)$, then any path in $Q_{\op{zip}}^n$ with $d_{\op{zip}}^n$-length at least $\delta  $ from $x$ to a point of $Q_\mp^n$ must hit $\eta_{\op{zip}}^n([0,R^2])$. Therefore, with probability at least $1-2\ep/3$, 
\eqbn
B_\delta\left( \eta_{\op{zip}}^n([0,T]) ; d_{\op{zip}}^n \right)  
\subset B_\delta\left( \eta_-^n([0,R^2]) ; d_-^n \right) \cup  B_\delta\left( \eta_+^n([0,R^2]) ; d_+^n \right) ,\quad \forall \delta\in (0,1).
\eqen
Since $\frk X_\pm^n \rta \frk X_\pm$ in law in the local GHPU topology and $\mu_\pm(\eta_\pm) = 0$, there exists $\delta \in (0,1)$ such that with probability at least $1-\ep/3$, 
\eqbn
\mu_-^n\left( B_\delta\left( \eta_-^n([0,R^2]) ; d_-^n \right) \right) + \mu_+^n\left(  B_\delta\left( \eta_-^n([0,R^2]) ; d_+^n \right) \right) \leq \ep .
\eqen
Hence with probability at least $1-\ep $,
\eqbn
\mu_{\op{zip}}^n\left( B_\delta\left( \eta_{\op{zip}}^n([0,T]) ; d_{\op{zip}}^n \right)   \right) \leq \ep .
\eqen
Since $\frk X_{\op{zip}}^n \rta \frk X_{\op{zip}}$ in law in the local GHPU topology as $\mcl N\ni n\rta\infty$, we see that 
\eqbn
\BB P\left[ B_\delta\left( \wt\eta([0,T]) ; \wt d \right)  \leq \ep \right] \geq 1-\ep .
\eqen
The lemma follows since we took $T > 0$ and $\epsilon\in (0,1)$ arbitrary.  
\end{proof}

\subsection{A $1$-Lipschitz map from $X_{\op{zip}}$ to $\widetilde X$}
\label{sec-1side-map}

A priori, we do not have any explicit relationship between the Brownian half-planes $X_\pm$ and the subsequential limiting metric space $\wt X$. In this subsection, we will prove that there is a map from the desired limiting space $X_{\op{zip}}$ (which we recall is the metric gluing of $X_\pm$) to $\wt X$ which satisfies several properties.

\begin{prop} \label{prop-zip-map}
Almost surely, there exists a bijective $1$-Lipschitz map $f_{\op{zip}} : ( X_{\op{zip}} ,d_{\op{zip}})  \rta (\wt X ,\wt d)$ such that $f_{\op{zip}} \circ \eta_{\op{zip}} = \wt\eta$ and $(f_{\op{zip}})_*\mu_{\op{zip}} = \wt\mu$. Furthermore, the $d_{\op{zip}}$-length of any curve $\gamma$ in $X_{\op{zip}}$ which does not intersect $\eta_{\op{zip}}$ is the same as the $\wt d$-length of $f_{\op{zip}}(\gamma)$. 
\end{prop}

Once Proposition~\ref{prop-zip-map} is established, the only remaining step to show that $\frk X_{\op{zip}} = \wt{\frk X}$ as curve-decorated metric measure spaces, and thereby finish the proof of Theorem~\ref{thm-saw-conv-wedge}, is to show that $f_{\op{zip}}$ is an isometry. Equivalently, we need to show that it preserves the lengths of paths which are allowed to $\eta_{\op{zip}}$, not just paths which do not cross $\eta_{\op{zip}}$. This will be accomplished in Section~\ref{sec-saw-proof} using Propositions~\ref{prop-lipschitz-path} and~\ref{prop-geodesic-away}. Section~\ref{sec-saw-proof} does not use anything from the present subsection except Proposition~\ref{prop-zip-map}. 

The proof of Proposition~\ref{prop-zip-map} uses only the qualitative properties of the curve $\wt\eta$ established in Section~\ref{sec-saw-conv-estimate} and elementary metric space theory. 
The main step in the proof is establishing the existence of $f_\pm : X_\pm \rta \wt X$ which are subsequential limits of the maps $f_\pm^n : \wh Q_\pm^n \rta Q_{\op{zip}}^n$ (see Figure~\ref{fig-gluing-maps}) and which give us a decomposition of $\wt X = f_-(X_-) \cup f_+(X_+)$ analogous to the decomposition $X_{\op{zip}} = X_-\cup X_+$.

\begin{prop} \label{prop-map-isometry}
Almost surely, there exist $1$-Lipschitz homeomorphisms $f_\pm : (X_\pm , d_\pm)  \rta (\wt X,\wt d)$ such that the following are true. 
\begin{enumerate}
\item We have $f_-(X_-) \cup f_+(X_+) = \wt X$ and $f_-(X_-) \cap f_+(X_+) = \wt\eta$.  \label{item-map-iso-decomp}
\item $(f_\pm)_* \mu_\pm = \wt\mu|_{f_\pm(X_\pm)}$ and $f_\pm(\eta_\pm(t)) =\wt\eta(t)$ for each $t\geq 0$.  \label{item-map-iso-push}
\item Let $\wt d_\pm$ be the internal metric of $\wt d$ on $f_\pm(X_\pm) \setminus \wt\eta$. Then each $f_\pm$ is an isometry from $(X_\pm \setminus \eta_\pm([0,\infty)) , d_\pm)$ to $ (f_\pm(X_\pm) \setminus \wt\eta , \wt d_\pm)$. \label{item-map-iso-internal}
\end{enumerate}
\end{prop}

Before we give the proof of Proposition~\ref{prop-map-isometry}, let us explain why it implies Proposition~\ref{prop-zip-map}.
See Figure~\ref{fig-gluing-maps} for an illustration.

\begin{proof}[Proof of Proposition~\ref{prop-zip-map}]
Fix maps $f_\pm : X_\pm \rta \wt X$ satisfying the conditions of Proposition~\ref{prop-map-isometry}.
Endow the disjoint union $X_-\sqcup X_+$ with the metric $d_-\sqcup d_+$ which restricts to $d_\pm$ on $X_\pm$ and satisfies $(d_-\sqcup d_+)(x_-,x_+) = \infty$ for $x_-\in X_-$ and $x_+ \in X_+$. 
Let $f_-\sqcup f_+$ be the map from $X_-\sqcup X_+$ to $\wt X$ which restricts to $f_\pm$ on $X_\pm$. Then $f_-\sqcup f_+$ is $1$-Lipschitz from $(X_- \sqcup X_+ , d_- \sqcup d_+)$ to $(\wt X ,\wt d)$. By condition~\ref{item-map-iso-push} of Proposition~\ref{prop-map-isometry}, $f_-(\eta_-(t)) = f_+(\eta_+(t)) = \wt\eta(t)$ for each $t\geq 0$. That is, $(f_-\sqcup f_+)(x)  = (f_-\sqcup f_+)(y)$ whenever $x,y\in X_-\sqcup X_+$ are identified under the equivalence relation defining $X_{\op{zip}}$.
Let $q : X_-\sqcup X_+ \rta X_{\op{zip}}$ be the quotient map, which sends $\eta_\pm(t)$ to $\eta_{\op{zip}}(t)$ for each $t\geq 0$. By the universal property of the quotient metric (Remark~\ref{remark-quotient-universal}), there exists a $1$-Lipschitz map $f_{\op{zip}} : ( X_{\op{zip}}  , d_{\op{zip}}) \rta (\wt X , \wt d)$ satisfying $f_{\op{zip}} \circ q = f_- \sqcup f_+$. 

By condition~\ref{item-map-iso-decomp} of Proposition~\ref{prop-map-isometry}, $f_{\op{zip}}$ is surjective.
Since each $f_\pm$ is injective and $\wt\eta$ is a simple curve (Lemma~\ref{prop-simple-curve}), $f_{\op{zip}}$ is injective.
By condition~\ref{item-map-iso-push}, $f_{\op{zip}} \circ \eta_{\op{zip}} = \wt\eta$ and by this same condition together with Lemma~\ref{prop-0mass}, $(f_{\op{zip}})_* \mu_{\op{zip}} = \wt\mu$. 
The length-preserving condition for $f_{\op{zip}}$ follows from condition~\ref{item-map-iso-internal} of Proposition~\ref{prop-map-isometry} since each curve in $X_{\op{zip}}$ which does not intersect $\eta_{\op{zip}}$ is contained in either $X_-$ or $X_+$ and since each of $f_-\sqcup f_+$ and $q$ preserve the lengths of curves which are entirely contained in one of $X_-$ or $X_+$.
\end{proof}

In the rest of this subsection we prove Proposition~\ref{prop-map-isometry}. The proof is elementary but somewhat technical since we need to check that several properties of the maps $f_\pm^n$ of~\eqref{eqn-one-side-maps} are preserved under taking subsequential limits. Since Proposition~\ref{prop-zip-map} is the only result from this subsection used in Section~\ref{sec-saw-proof}, the reader can safely skip the rest of this subsection on a first read.
We start by establishing existence of subsequential limits of the maps $f_\pm^n$ and proving some basic properties.

\begin{lem} \label{prop-map-limit}
Almost surely, there exist $1$-Lipschitz maps $f_\pm : (X_\pm , d_\pm ) \rta (\wt X , \wt d) $ and a random subsequence $\mcl N'\subset\mcl N$ such that the following are true.
\begin{enumerate}
\item (Convergence $f_\pm^n \rta f_\pm$) The maps $f_\pm^n$ of~\eqref{eqn-one-side-maps} converge to $f_\pm$ as $\mcl N' \ni n \rta\infty$ in the following sense. Suppose given a subsequence $\mcl N''\subset \mcl N'$, a sequence of points $x^n\in \wh{ Q}_-^n$ for $n\in \mcl N''$, and $x \in X_-$ such that $D_-(x^n ,x ) \rta 0$ as $\mcl N'' \ni n \rta\infty$. Then $D_{\op{zip}}(f_-^n(x^n) , f_-(x)) \rta 0$ as $\mcl N''\ni n\rta\infty$; and the same holds with ``$+$" in place of ``$-$".  \label{item-map-limit-conv}
\item (Convergence of distances to the gluing interface) For each sequence $x^n\rta x$ as in condition~\ref{item-map-limit-conv},
\eqbn
\lim_{\mcl N'' \ni n\rta\infty} \wh d_\pm^n\left( x^n , \wh\eta_\pm^n \right)
= \lim_{\mcl N'' \ni n\rta\infty} d_{\op{zip}}^n\left( f_\pm^n( x^n ) ,  \eta_{\op{zip}}^n \right) 
=  d_\pm\left( x ,   \eta_-([0,\infty)) \right) 
= \wt d\left( f_-(x) , \wt\eta \right) .
\eqen
\label{item-map-limit-dist}
\item (Pre-images of convergent sequences are contained in compact sets) For each subsequence $\mcl N''\subset\mcl N'$ and each sequence of points $y^n\in Q_\pm^n$ for $n\in\mcl N''$ such that $y^n \rta y \in \wt X$ (with respect to $D_{\op{zip}}$) as $\mcl N''\ni n \rta\infty$, we can find a compact subset $A$ of $Z_\pm$ such that $(f_\pm^n)^{-1}(y^n) \in A$ for each $n\in\mcl N''$. \label{item-map-limit-compact}
\item (Images of $f_\pm$ cover $\wt X$) We have $f_-(X_-) \cup f_+(X_+) = \wt X$. In fact, for each $\rho > 0$, there exists $R> 0$ such that 
\eqb  \label{eqn-map-limit-cover}
 B_\rho\left( \wt\eta(0) ; \wt d \right) \subset f_-\left( B_R( \eta_-(0) ; d_- ) \right) \cup f_+\left( B_R(  \eta_+(0)    ; d_+ )  \right) .
\eqe  \label{item-map-limit-cover} 
\end{enumerate}
\end{lem}

To establish the existence of the maps $f_\pm$ of condition~\ref{item-map-limit-conv} in Lemma~\ref{prop-map-limit}, we will use the following general lemma about subsequential limits of 1-Lipschitz maps, which is~\cite[Lemma~2.1]{gwynne-miller-uihpq}.

\begin{lem} \label{prop-lip-limit}
Let $(W, D  , w )$ be a separable pointed metric space and let $(\wh W , \wh D)$ be any metric space. Let $\{X^n\}_{n\in\BB N}$ and $X$ be closed subsets of $W$ and for $n\in\BB N$, let $f^n : X^n \rta \wh W$ be a 1-Lipschitz map. Suppose that the following are true.
\begin{enumerate}
\item\label{item-lip-limit-conv0}  For each $r>0$, $B_r(w ; D)\cap X^n \rta B_r(w;D)\cap X$ in the $D$-Hausdorff metric. 
\item\label{item-lip-limit-image0}   For each $r> 0$, there exists a compact set $\wh W_r \subset \wh W$ such that $f^n\left(B_r(w ; D) \cap X^n \right) \subset \wh W_r$ for each $n\in\BB N$.  
\end{enumerate}  
Then there is a sequence $\mcl N'$ of positive integers tending to $\infty$ and a 1-Lipschitz map $f : X \rta \wh W$ such that $f^n \rta f$ as $\mcl N' \ni n \rta\infty$ in the following sense. For any $x \in X$, any subsequence $\mcl N''$ of $\mcl N'$, and any sequence of points $x^n \in X^n$ for $n\in\mcl N'$ with $x^n \rta x $, we have $f^n (x^n) \rta f(x)$ as $\mcl N''\ni n \rta\infty$.
%Moreover, if each $f^n$ is an isometry onto its image, then $f$ is also an isometry onto its image. 
\end{lem}

\begin{proof}[Proof of Lemma~\ref{prop-map-limit}]
The proof can be summarized as follows. The existence of the maps $f_\pm$ of condition~\ref{item-map-limit-conv} follows from Lemma~\ref{prop-lip-limit} applied to each of $f_\pm^n$.
In fact, we can arrange that a.s.\ there is a subsequence of $\mcl N'$ of $\mcl N$ such that condition~\ref{item-map-limit-conv} holds and also a certain regularity event, which is a minor variant of the event $G^n(\rho,R)$ of Lemma~\ref{prop-ball-contain}, holds for each $n\in\mcl N'$. Once this is established, the rest of the conditions in the lemma follow from the convergence $f_\pm^n\rta f_\pm$ and elementary limiting arguments. Condition~\ref{item-2side-ball-contain} of Lemma~\ref{prop-ball-contain} is used in the proof of condition~\ref{item-map-limit-dist} of the present lemma to allow us to restrict attention to a finite-length segment of $\wt\eta$. Condition~\ref{item-full-ball-contain} of Lemma~\ref{prop-ball-contain} is the main ingredient in the proof of condition~\ref{item-map-limit-compact} of the present lemma. 

Fix $\ep \in (0,1)$. We will show that the objects in the statement of the lemma exist on an event of probability at least $1-\ep$. 
For $n\in\mcl N$ and $\rho , R > 0$, define the event $G^n(\rho , R)$ as in Lemma~\ref{prop-ball-contain}. 
By that lemma, for each $k\in\BB N$, there exists $R_k > 0$ such that
\eqbn
\BB P\left[ G^n(k , R_k )\right] \geq 1-2^{-k} \ep .
\eqen
Let
\eqbn
\wt G^n := \bigcap_{k=1}^\infty G^n(k , R_k) \quad \op{and} \quad \wt G := \bigcap_{n\in \mcl N} \bigcup_{\substack{m\in \mcl N\\ m \geq n}} \wt G^n 
\eqen
so that $\wt G$ is the event that $\wt G^n$ occurs for infinitely many $n\in\mcl N$ and $\wt G^n$ and $\wt G$ each have probability at least $1-\ep$. We will check that the conditions in the statement of the lemma on $\wt G$. 
\medskip
 
\noindent\textit{Step 1: existence of $f_\pm$ and proof of condition~\ref{item-map-limit-conv}.} The maps $f_\pm^n$ are $1$-Lipschitz and for each $\rho > 0$, we have $B_\rho(\wh\eta_\pm^n(0) ; \wh d^n) \rta B_\rho(\eta_\pm(0) ; d_\pm)   $ in the $D_\pm$-Hausdorff metric. 
Furthermore, for each $n\in\mcl N$ and each $\rho  > 0$ we have $f_\pm^n\left( B_\rho(\wh\eta_\pm^n(0) ; \wh d^n) \right)\subset B_\rho(\eta_{\op{zip}}^n(0) ; d_{\op{zip}}^n)$, which converges to $B_\rho(\wt\eta(0) ; \wt d)$ in the $D_{\op{zip}}$-Hausdorff metric. 
In particular, each $f_\pm^n\left( B_\rho(\wh\eta_\pm^n(0) ; \wh d^n) \right)$ is contained in an $n$-independent compact subset of $Z_{\op{zip}}$. 
On the event $\wt G$, we apply~\cite[Lemma~2.1]{gwynne-miller-uihpq} along a subsequence of $\mcl N$ for which $\wt G^n$ occurs and with $(W,D,w) = (Z_\pm , D_\pm , \eta_\pm(0)) $, $(\wh W , \wh D)  = (Z_{\op{zip}} , D_{\op{zip}})$, and $X^n = Q_\pm^n \subset Z_{\op{zip}}$. 
This shows that on the event $\wt G$ there exists a subsequence $\mcl N' \subset \mcl N$ and maps $f_\pm$ as in the statement of the lemma such that condition~\ref{item-map-limit-conv} is satisfied and $\wt G^n$ occurs for each $n\in\mcl N'$. 
\medskip

\noindent\textit{Step 2: proof of condition~\ref{item-map-limit-dist}.} By symmetry it suffices to check condition~\ref{item-map-limit-dist} for $f_-$. 
Suppose given a subsequence $\mcl N''$ of $\mcl N'$, a sequence of points $x^n\in \wh{  Q}_-^n$ for $n\in \mcl N''$, and an $x\in X_-$ with $x^n\rta x$. 
We know by condition~\ref{item-map-limit-conv} that $f_-^n(x^n) \rta f_-(x)$.
Since each $f_-^n$ is an isometry from $(\wh Q_-^n , \wh d_-^n)$ to $(Q_-^n , d_-^n)$ and each path from $Q_-^n$ to $Q_+^n$ in $Q_{\op{zip}}^n$ must pass through $\eta_{\op{zip}}^n$, for each $n\in\mcl N''$ we have
\eqb \label{eqn-path-dist-agree}
d_{\op{zip}}^n\left( f_-^n( x^n ) ,  \eta_{\op{zip}}^n \right)  = \wh d_-^n\left( x^n , \wh\eta_-^n \right).
\eqe 
Since $x^n \rta x$ there exists $k \in \BB N$ such that $x^n\in B_k^n(\wh\eta_-^n(0) ; \wh d_-^n)$ for each $n\in\mcl N''$. 
Since $d_-^n \geq d_{\op{zip}}^n$, also $f_-^n(x^n) \in B_k^n(\eta_{\op{zip}}^n(0) ; d_{\op{zip}}^n)$ for each $n\in\mcl N''$.  
By condition~\ref{item-2side-ball-contain} in Lemma~\ref{prop-ball-contain} for $G^n(2k,R_{2k})$, for each $n\in\mcl N'' $ and each $\wt R \geq R_{2k}$, each of the quantities~\eqref{eqn-path-dist-agree} is equal to
\eqb \label{eqn-partial-path-dist}
 \wh d_-^n\left( x^n , \wh\eta_-^n([0, \wt R^2] ) \right)  =    d_{\op{zip}}^n\left( x^n ,  \eta_-^n([0, \wt R^2] ) \right) . 
\eqe 
By the transience of $\wt\eta$ (Lemma~\ref{prop-simple-curve}) and of $\eta_-$ (which is immediate from the definition of the Brownian half-plane), we can a.s.\ find $\wt R\geq R_{2k}$ such that 
\eqbn
 d_-\left( x ,  \eta_-([0,\infty)) \right) =  d_-\left( x ,  \eta_-([0,\wt R^2]) \right)
\quad \op{and} \quad \wt d\left( f_-(x) , \wt\eta \right) = \wt d\left( f_-(x) , \wt\eta([0,\wt R^2]) \right) .
\eqen
Since $\wh{\frk Q}_-^n \rta \frk X_-$ in the $D_-$-local HPU topology and $\frk Q_{\op{zip}}^n \rta \wt{\frk X}$ in the $D_{\op{zip}}$-local HPU topology, we can take a limit along the subsequence $\mcl N''$ in~\eqref{eqn-partial-path-dist} to get condition~\ref{item-map-limit-dist}. 
\medskip

\noindent\textit{Step 3: proof of condition~\ref{item-map-limit-compact}.} We check the condition for $f_-$ (which, again, suffices by symmetry). Suppose we are given $\mcl N''\subset\mcl N'$ and $y^n \in Q_-^n$ for $n\in\mcl N''$ such that $y^n \rta y \in \wt X$. 
Since $Q_-^n\subset Q_{\op{zip}}^n$ and the latter converges to $\wt X$ in the $D_{\op{zip}}$-local Hausdorff metric, we can find $k \in \BB N$ such that
 $y^n \in B_k(\eta_{\op{zip}}^n(0) ; d_{\op{zip}}^n) \cap Q_-^n$ for each $n\in\mcl N''$. 
 
By condition~\ref{item-full-ball-contain} of Lemma~\ref{prop-ball-contain} for the event $G^n(k,R_k)$, we have $y^n\in B_{R_k}\left( \eta_-^n(0) ; d_-^n \right) $ for each $n\in \mcl N''$, so $(f^n)^{-1}(y^n) \in B_{R_k}\left( \wh\eta_-^n(0) ; \wh d_-^n \right)$ for each such $n$. Since $\wh Q_-^n \rta X_-$ in the $D_-$-local Hausdorff metric, there is a compact subset $A$ of $Z_-$ such that $B_{R_k}\left( \wh\eta_-^n(0) ; \wh d_-^n \right) \subset A$ for each $n\in\mcl N''$. Thus condition~\ref{item-map-limit-compact} is satisfied.
\medskip
 
\noindent\textit{Step 4: proof of condition~\ref{item-map-limit-cover}.} 
Suppose given $\rho > 0$ and $y\in B_\rho(\wt\eta(0) ; \wt d)$. Since $Q_{\op{zip}}^n \rta \wt X$ in the $D_{\op{zip}}$-Hausdorff metric, we can find a sequence of points $y^n \in Q_{\op{zip}}^n$ for $n\in\mcl N'$ such that $y^n\rta y$.  Either there is a subsequence $\mcl N''$ of $\mcl N'$ such that $y^n \in Q_-^n$ for each $n\in\mcl N''$, or the same is true with ``$+$" in place of ``$-$". Suppose we are in the former situation. Then condition~\ref{item-map-limit-compact} implies that after passing to a further subsequence, we can arrange that $(f_-^n)^{-1}(y^n) $ converges to some $x\in X_-$ with respect to $D_-$. By condition~\ref{item-map-limit-conv}, $f_-(x) = y$, which gives the first part of condition~\ref{item-map-limit-cover}. To obtain~\eqref{eqn-map-limit-cover}, choose $k\in \BB N$ with $k > \rho$. Then $y^n\in B_\rho\left(\eta_{\op{zip}}^n(0) ; d_{\op{zip}}^n \right)$ for large enough $n\in\mcl N''$, so since $\wt G^n$ occurs for each $n\in\mcl N''$,  
$(f_-^n)^{-1}(y^n) \in   B_{R_k}\left( \wh\eta_-^n(0) ; \wh d_-^n \right)$. Therefore $y \in f_-\left(B_{R_k}\left(  \eta_-(0) ; d_- \right)\right)$. 
\end{proof}

From Lemma~\ref{prop-map-limit}, we can deduce the following further properties of the maps $f_\pm$, again using the analogous properties in the discrete setting and elementary limiting arguments. 

\begin{lem} \label{prop-map-properties}
Let $f_\pm : (X_\pm , d_\pm) \rta (\wt X , \wt d)$ and $\mcl N' \subset \mcl N$ be $1$-Lipschitz maps and a subsequence satisfying the conditions of Lemma~\ref{prop-map-limit}. 
Almost surely, the following conditions are satisfied.
\begin{enumerate}
\item $f_- \circ \eta_-|_{[0,\infty)} =  f_+ \circ \eta_+|_{[0,\infty)} =   \wt\eta$. \label{item-map-curve} 
\item $f_-(X_-) \cap f_+(X_+) = \wt\eta$. \label{item-map-intersect} 
\item The maps $f_\pm$ are local isometries away from $\eta_\pm$, i.e., for each $x\in X_\pm \setminus \eta_\pm([0,\infty)) $ and each $0 < \rho < \frac13  d_\pm(x , \eta_\pm([0,\infty)) )$, the map $f_\pm|_{B_\rho(x ; d_\pm)}$ is an isometry onto $B_\rho( f_\pm(x) ; \wt d)$ (with the metric $d_\pm$ on the domain and the metric $\wt d$ on the range).  \label{item-map-local}
\item For $x$ and $\rho$ as in condition~\ref{item-map-local}, we have $\mu_\pm(A) = \wt\mu(f_\pm(A))$ for each Borel set $A\subset B_\rho(x ; d_\pm)$. \label{item-map-measure}
\end{enumerate}
\end{lem} 
\begin{proof}
\noindent\textit{Step 1: proof of condition~\ref{item-map-curve}.} each $t \geq 0$ and each $n\in\mcl N'$, we have $f^n(\wh\eta_\pm^n(t)) = \eta_\pm^n(t) = \eta_{\op{zip}}^n(t)$. Furthermore, $D_\pm\left( \wh\eta_-^n(t)  , \eta_-(t) \right) \rta 0$ and $D_{\op{zip}}\left( \eta_{\op{zip}}^n(t) , \wt\eta(t)\right) \rta 0$ as $\mcl N'\ni n \rta \infty$. Therefore, condition~\ref{item-map-limit-conv} of Lemma~\ref{prop-map-limit} implies that $f_\pm(\eta_\pm(t)) = \wt\eta(t)$, i.e.\ condition~\ref{item-map-curve} holds. 
\medskip
 
\noindent\textit{Step 2: proof of condition~\ref{item-map-intersect}.} By condition~\ref{item-map-curve}, we have $f_-(X_-) \cap f_+(X_+) \supset \wt\eta$, so we just need to check the reverse inclusion. If $z \in f_-(X_-)\cap f_+(X_+)$, then there exists $x_-^n \in \wh Q_-^n$ and $x_+^n \in \wh Q_+^n$ for $n\in\mcl N'$ such that $f_-^n(x_-^n) \rta z$ and $f_+^n(x_+^n) \rta z$. This implies that $d_{\op{zip}}^n(f_-^n(x_-^n)) , f_+^n(x_+^n)) \rta 0$, so since $Q_\pm^n$ intersect only along $\eta_-^n$, 
\eqb \label{eqn-conv-to-curve}
\wh d_\pm^n\left( x_\pm^n  ,  \wh\eta_\pm^n([0,\infty)) \right) = d_\pm^n\left(f_\pm^n(x_\pm^n  ) , \eta_{\op{zip}}^n \right)  \rta 0. 
\eqe 
By condition~\ref{item-map-limit-compact} of Lemma~\ref{prop-map-limit}, after possibly passing to a subsequence of the $x_\pm^n$'s we can find $x_\pm \in X_\pm$ such that $x_\pm^n \rta x_\pm$. 
By condition~\ref{item-map-limit-conv} of Lemma~\ref{prop-map-limit}, $f_\pm(x_\pm) = z$.
By~\eqref{eqn-conv-to-curve} and condition~\ref{item-map-limit-dist} of Lemma~\ref{prop-map-limit}, $z \in \wt\eta$. 
\medskip
 
\noindent\textit{Step 3: proof of condition~\ref{item-map-local}.} By symmetry it suffices to check this for $f_-$. Let $x\in  X_- \setminus \eta_-([0,\infty))$ and $0 < \rho < \frac13  d_-(x , \eta_-([0,\infty)) )$ and choose $\ep \in (0,1)$ such that $0  <   3\ep <  \frac13  d_-(x , \eta_-([0,\infty)) ) - \rho$. 
Let $y_1,y_2\in B_\rho(x ; d_-)$ and choose points $x^n , y_1^n , y_2^n \in \wh Q_-^n$ for $n\in\mcl N'$ such that $D_-(x^n , x)\rta 0$ and $D_-(y^n_i , y_i) \rta 0$ for $i\in \{1,2\}$.
By condition~\ref{item-map-limit-conv} in Lemma~\ref{prop-map-limit}, $D_{\op{zip}}(f_-^n( y_i^n )  , f_-( y_i) ) \rta 0$.
 
Since $\wh{\frk Q}_-^n \rta \frk X_-$ in the $D_-$-local HPU topology and by condition~\ref{item-map-limit-dist} of Lemma~\ref{prop-map-limit}, for large enough $n\in\mcl N'$,
\eqbn
\wh d_-^n\left(x^n , \wh\eta_-^n([0,\infty)) \right) > 3 \rho  + 3\ep , \quad \op{and} \quad
\wh d_-^n\left(x^n , y_i^n \right)  < \rho +\ep  ,\:\forall i \in \{1,2\}  .
\eqen
If this is the case, then $y_1^n$ and $y_2^n$ are $\wh d_-^n$-closer to each other than to $\eta_-^n([0,\infty))$, so by the triangle inequality and since every path from $Q_-^n$ to $Q_+^n$ in $Q_{\op{zip}}^n$ must pass through $\eta_-^n([0,\infty))$,
\eqbn
\wh d_-^n(y_1^n , y_2^n) = d_-^n \left(  f_-^n(y_1^n) , f_-^n(y_2^n) \right) = d_{\op{zip}}^n \left(  f_-^n(y_1^n) , f_-^n(y_2^n) \right)  . 
\eqen
Taking a limit as $n\rta\infty$ shows that $d_-(y_1,y_2) = \wt d(f_-(y_1) , f_-(y_2))$. 
Therefore $f_-$ is distance-preserving on $B_\rho(x ; d_-)$. 

We still need to show that $f_-(B_\rho(x ; d_-)) = B_\rho( f_-(x) ; \wt d)$. It is clear from the preceding paragraph that $f_-(B_\rho(x ; d_-)) \subset B_\rho( f_-(x) ; \wt d)$, so we just need to prove the reverse inclusion.  
Since $f_-^n(x^n) \rta f_-(x)$ and $\frk Q_{\op{zip}}^n \rta \wt{\frk X} $ in the $D_{\op{zip}}$-local HPU topology,  
\eqb \label{eqn-map-ball-conv}
B_\rho\left( f_-^n(x^n)  ; d_{\op{zip}}^n \right) \rta B_\rho\left( f_-(x) ; \wt d \right) 
\eqe
in the $D_{\op{zip}}$-Hausdorff topology. By~\eqref{eqn-map-ball-conv}, for each $z\in  B_\rho\left( f_-(x) ; \wt d \right) $, there exists a sequence of points $z^n \in B_\rho\left( f_-^n(x^n)  ; d_{\op{zip}}^n \right)$ for $n\in\mcl N'$ such that $z^n \rta z$. 
By condition~\ref{item-map-limit-dist} of Lemma~\ref{prop-map-limit} and our choice of $\rho$, for large enough $n\in\mcl N' $, $z^n$ is $d_{\op{zip}}^n$-closer to $f_-(x^n)$ than to $\eta_{\op{zip}}^n$, so $z^n\in Q_-^n$ and 
\eqb \label{eqn-ball-dist-compare}
d_{\op{zip}}^n\left(f_-^n(x^n) , z^n \right) = d_-^n(f_-^n(x^n) , z^n) = \wh d_-^n\left( x^n , (f_-^n)^{-1}(z^n)  \right) .
\eqe 
By condition~\ref{item-map-limit-compact} of Lemma~\ref{prop-map-limit}, there is a subsequence $\mcl N''$ of $\mcl N'$ and a $y \in X_-$ such that $(f_-^n)^{-1}(z^n) \rta y$ as $\mcl N'' \ni n \rta\infty$. By condition~\ref{item-map-limit-conv} of Lemma~\ref{prop-map-limit}, $f_-(y) = z$. 
The left side of~\eqref{eqn-ball-dist-compare} converges to $d_{\op{zip}}^n(f_-(x) , z) \leq \rho$ and the right side converges to $d_-(x,y)$. Therefore $y\in B_\rho(x ; d_-)$ so since our initial choice of $z$ was arbitrary, we obtain condition~\ref{item-map-local}.
\medskip

\noindent\textit{Step 4: proof of condition~\ref{item-map-measure}.} 
Let $x$, $\rho$, and $x^n \in \wh Q_-^n$ be as above. Since $\mu_{\op{zip}}$ and $\wt\mu$ are locally finite measures, we can choose $\rho' > \rho$ such that $\rho' < \frac13 d_-(x , \eta_-^n([0,\infty)) )$ and 
\eqbn
\mu_{\op{zip}}\left( \bdy B_{\rho'}\left(x ; d_- \right) \right) =    \wt\mu\left( \bdy B_{\rho'}\left( f_-(x) ; \wt d \right) \right) = 0 .
\eqen
By this condition together with the local HPU convergence $\wh{\frk Q}_-^n \rta \frk X_-$ and $\frk Q_{\op{zip}}^n \rta \wt{\frk X}$ as $\mcl N'\ni n \rta\infty$, 
\eqb \label{eqn-ball-measure-conv}
\wh\mu_-^n|_{ B_{\rho'}(x^n ; \wh d_-^n )} \rta \mu_-|_{ B_{\rho'}\left(x ; d_-\right)}
\eqe 
in the $D_-$-Prokhorov metric and 
\eqb \label{eqn-ball-measure-conv'}
 \mu_{\op{zip}}^n|_{ B_{\rho'}(f_-(x^n) ; d_{\op{zip}}^n )} \rta \wt\mu|_{ B_{\rho'}\left( f_-(x) ; \wt d\right)}
\eqe 
in the $D_{\op{zip}}$-Prokhorov metric.  

We now want to use condition~\ref{item-map-limit-conv} of Lemma~\ref{prop-map-limit} to study the pushforward of $\mu_-$ under $f_-$. 
For this purpose, we need to produce a convergent sequence, which we do by means of the Skorokhod representation theorem as follows. 
Conditional on everything else, for $n\in \mcl N'$ let $w^n$ be sampled uniformly from $\wh\mu_-^n|_{ B_{\rho'}(x^n ; \wh d_-^n )} $ (normalized to be a probability measure) and let $w$ be sampled uniformly from $\mu_{-}|_{ B_{\rho'}\left(x ; d_-\right)}$ (normalized to be a probability measure). By~\eqref{eqn-ball-measure-conv} $w^n\rta w$ in law, so by the Skorokhod representation theorem, we can couple together $\{ w^n\}_{n\in\mcl N'} $ and $w$ (still conditioning on everything else) in such a way that a.s.\ $w^n \rta w$ as $\mcl N' \ni n \rta\infty$. By condition~\ref{item-map-limit-conv} of Lemma~\ref{prop-map-limit}, $f_-^n(w^n) \rta f_-(w)$. Since $\rho' < \frac13 d_-(x , \eta_-^n([0,\infty)) )$, it follows from condition~\ref{item-map-limit-dist} of Lemma~\ref{prop-map-limit} that for each sufficiently large $n \in \mcl N'$, we have
$B_{\rho'}(f_-(x^n) ; d_{\op{zip}}^n ) = B_{\rho'}(f_-(x^n) ; d_-^n )$.
For such an $n$ the law of $f_-^n(w^n)$ is that of a uniform sample from $ \mu_{\op{zip}}^n|_{ B_{\rho'}(f_-(x^n) ; d_{\op{zip}}^n )}$. By~\eqref{eqn-ball-measure-conv'}, the conditional law of $f_-(w)$ given $(\wt X,\wt d , \wt\mu,\wt\eta)$ is that of a uniform sample from $ \wt\mu|_{ B_{\rho'}\left( f_-(x) ; \wt d\right)}$. We similarly infer from~\eqref{eqn-ball-measure-conv} and~\eqref{eqn-ball-measure-conv'} that  
\eqbn
\mu_{-}\left(  B_{\rho'}\left(x ; d_-\right) \right) = \wt\mu\left( B_{\rho'}\left( f_-(x) ; \wt d\right) \right) .
\eqen
Therefore,  
\eqbn
(f_-)_*\left(   \mu_{-}|_{ B_{\rho'}\left(x ; d_-\right)} \right) =  \wt\mu|_{ B_{\rho'}\left( f_-(x) ; \wt d\right)} ,
\eqen
which implies condition~\ref{item-map-measure} for $f_-$. By symmetry, the analogous relation holds for $f_+$. 
\end{proof}

Now we can establish the main desired properties of the maps $f_\pm$.

\begin{proof}[Proof of Proposition~\ref{prop-map-isometry}]
Let $f_\pm : (X_\pm , d_\pm) \rta (\wt X , \wt d)$ and $\mcl N' \subset \mcl N$ be $1$-Lipschitz maps and a subsequence satisfying the conditions of Lemma~\ref{prop-map-limit}. We will check the conditions of the proposition statement for $f_-$; the statement for $f_+$ follows by symmetry. 

Condition~\ref{item-map-iso-decomp} follows from condition~\ref{item-map-limit-cover} of Lemma~\ref{prop-map-limit} together with condition~\ref{item-map-intersect} of Lemma~\ref{prop-map-properties}. By condition~\ref{item-map-curve} of Lemma~\ref{prop-map-properties}, $f_-\circ \eta_-|_{[0,\infty)} = \wt\eta$. By condition~\ref{item-map-measure} in Lemma~\ref{prop-map-properties}, $(f_-)_* \mu_- = \wt\mu|_{f_-(X_-) \setminus \wt\eta}$ and by Lemma~\ref{prop-0mass}, $\wt\mu(\wt\eta) = 0$. Therefore $(f_-)_* \mu_- = \wt\mu |_{f_-(X_-)}$, i.e.\ condition~\ref{item-map-iso-push} holds.

Next we check that $f_-$ is a homeomorphism onto its image.
We first argue that $f_-$ is injective.
Indeed, condition~\ref{item-map-local} of Lemma~\ref{prop-map-properties} implies that $f_-(x) \not= f_-(y)$ whenever $x, y \in X_- $ and either $x$ or $y$ does not belong to $X_- \setminus \eta_-([0,\infty))$. By condition~\ref{item-map-curve} of Lemma~\ref{prop-map-properties} and Lemma~\ref{prop-simple-curve}, $f_-|_{\eta_-([0,\infty))}$ is injective, so $f_-$ is injective. 

The relation~\eqref{eqn-map-limit-cover} of Lemma~\ref{prop-map-limit} implies that if $\{x^j\}_{j\in\BB N}$ is a sequence of points in $X_-$ which tends to $\infty$, then for each $\rho > 0$, $ f_-(x^j) $ lies outside of $B_\rho(\wt\eta(0) ; \wt d )$ for large enough $j\in\BB N$. 
Therefore $f_-$ is a homeomorphism from $ X_-$ to $ f_-(X_-)  $ (equipped with the restriction of $\wt d$, not $\wt d_-$).
In particular, $f_-$ restricts to a homeomorphism from $X_-\setminus \eta_-([0,\infty))$ to $f_-(X_-)\setminus\wt\eta$. 

Finally, we check condition~\ref{item-map-iso-internal}. 
Given $x\in X_- \setminus \eta_-([0,\infty))$, let $0 < \rho < \frac13 d_-(x ,\eta_-([0,\infty)))$.  
By condition~\ref{item-map-local} of Lemma~\ref{prop-map-properties}, $f_-$ maps $B_\rho(x ; d_-)$ isometrically onto $B_\rho(f_-(x) ; \wt d)$. 
The image of any finite continuous path $\gamma$ in $f_-(X_-)$ which does not hit $\wt\eta$ can be covered by finitely many balls of the form $B_\rho(f_-(x) ; \wt d)$ for $x\in X_- \setminus \eta_-([0,\infty))$ and $0 < \rho < \frac13 d_-(x ,\eta_-([0,\infty)))$. 
The $\wt d_-$-length of $\gamma$ is determined by its restriction to the time intervals which it spends in these balls. 
Consequently, this $\wt d_-$ length is the same as the $d_-$-length of $f_-^{-1}(\gamma)$. 
Therefore, $f_-$ is an isometry from $(X_-\setminus \eta_-([0,\infty)) , d_-)$ to $(f_-(X_-) \setminus \wt\eta , \wt d_-)$. 
\end{proof}

\subsection{Proof of Theorem~\ref{thm-saw-conv-wedge}}
\label{sec-saw-proof}

In this subsection we will conclude the proof of Theorem~\ref{thm-saw-conv-wedge} by showing that $\wt{\frk X} = \frk X_{\op{zip}}$ as elements of $\BB M_\infty^{\op{GHPU}}$.  
In order to prove this, it remains only to show that the map $f_{\op{zip}}$ of Proposition~\ref{prop-zip-map} does not decrease distances. This will be accomplished using the results of Section~\ref{sec-geodesic-properties}. 
We first use Proposition~\ref{prop-lipschitz-path} to show that $f_{\op{zip}}^{-1}$ is a.s.\ Lipschitz with a \emph{deterministic} Lipschitz constant.

\begin{figure}[ht!]
 \begin{center}
\includegraphics[scale=.8]{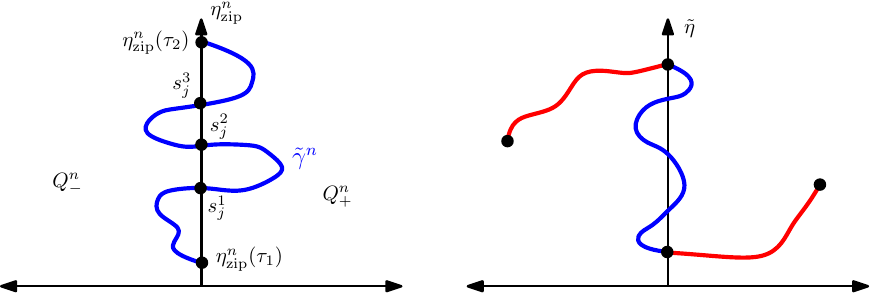} 
\caption{Illustration of the proof of Lemma~\ref{prop-map-lipschitz}. \textbf{Left:} We fix times $0 < \tau_1 < \tau_2 < \infty$. Proposition~\ref{prop-lipschitz-path} allows us to find a subsequence of $n$ values tending to $\infty$ and paths $\wt\gamma^n$ from $\eta_{\op{zip}}^n(\tau_1)$ to $\eta^n(\tau_2)$ of $d_{\op{zip}}^n$-length at most $C d_{\op{zip}}^n(\eta_{\op{zip}}^n(\tau_1) , \eta^n(\tau_2))$ which cross the gluing interface $\eta_{\op{zip}}^n$ at most a constant order number of times. Passing to the limit at recalling the definition of the quotient metric shows that $d_{\op{zip}}(\eta_{\op{zip}}(\tau_1) , \eta_{\op{zip}}(\tau_2) ) \leq C \wt d(\wt\eta(\tau_1) , \wt\eta(\tau_2))$. This gives the desired Lipschitz property for points along the gluing interface. \textbf{Right:} To get the Lipschitz property in general, we decompose a path between to arbitrary points of $\wt X$ with near-minimal $\wt d$ length into two red segments which do not cross the interface (whose $\wt d$-lengths are the same as the $d_{\op{zip}}$-lengths of their pre-images under $f_{\op{zip}}$) and a blue segment which does cross the interface (whose $\wt d$-length is at most $C$ times the $d_{\op{zip}}$-length of its pre-image under $f_{\op{zip}}$).  
}\label{fig-map-lipschitz}
\end{center}
\end{figure}

\begin{lem} \label{prop-map-lipschitz}
Let $f_{\op{zip}} : X_{\op{zip}} \rta \wt X$ be as in Proposition~\ref{prop-zip-map}. 
There is a universal constant $C \geq 1$ such that almost surely
\eqb \label{eqn-map-lipschitz}
  d_{\op{zip}}\left( x,y \right) \leq C \wt d\left( f_{\op{zip}}(x), f_{\op{zip}}(y) \right)  ,\quad \forall x,y\in X_{\op{zip}} .
\eqe  
\end{lem}
\begin{proof} 
See Figure~\ref{fig-map-lipschitz} for an illustration and outline of the proof.
Let $C \geq 1$ be chosen so that the conclusion of Proposition~\ref{prop-lipschitz-path} holds. Also let $\tau_1,\tau_2 \geq 0$.
We will take limits of the paths produced in Proposition~\ref{prop-lipschitz-path} to show that almost surely 
\eqb \label{eqn-saw-lipschitz}
 d_{\op{zip}}\left( \eta_{\op{zip}}(\tau_1 ) , \eta_{\op{zip}}(\tau_2)  \right) \leq C \wt d\left(\wt\eta(\tau_1) , \wt\eta(\tau_2) \right) .
\eqe  
\medskip

\noindent\textit{Step 1: constructing discrete paths via Proposition~\ref{prop-lipschitz-path}.}
To this end, fix $\alpha \in (0,1)$. By Lemma~\ref{prop-ball-contain}, we can find $R = R(\tau_1,\tau_2,\alpha)  > 0$ and $L = L(\tau_1,\tau_2,\alpha) > 0$ such that for each $n$ in our original subsequence $\mcl N$, it holds with probability at least $1-\alpha/2$ that 
\eqb \label{eqn-lipschitz-ball}
d_{\op{zip}}^n\left( \eta_{\op{zip}}^n(\tau_i) , Q_{\op{zip}}^n \setminus B_R\left( \eta_{\op{zip}}^n(0) ; d_{\op{zip}}^n \right) \right) \geq 2C d_{\op{zip}}^n\left( \eta_{\op{zip}}^n(\tau_1)  , \eta_{\op{zip}}^n(\tau_2) \right) ,\quad \forall i \in \{1,2\} .
\eqe 
and
\eqb \label{eqn-lipschitz-length}
B_R\left( \eta_{\op{zip}}^n(0) ; d_{\op{zip}}^n \right) \cap \eta_{\op{zip}}^n([ L , \infty)) = \emptyset .
\eqe 
The use of these two conditions is that together they imply that no path in $Q_{\op{zip}}$ from $\tau_1$ to $\tau_2$ with $d_{\op{zip}}^n$-length at most $C d_{\op{zip}}^n\left( \eta_{\op{zip}}^n(\tau_1)  , \eta_{\op{zip}}^n(\tau_2) \right) $ can hit $\eta_{\op{zip}}^n([ L , \infty))$.
By Proposition~\ref{prop-lipschitz-path} and our choice of $C$, we can find $\delta_* = \delta_*(\tau_1,\tau_2,\alpha) > 0$ such that for each $\delta \in (0,\delta_*]$ and each sufficiently large $n\in\mcl N$, it holds with probability at least $1-\alpha/2$ that there is a path $\wt\gamma^n$ from $\eta_{\op{zip}}^n(\tau_1)$ to $\eta_{\op{zip}}^n(\tau_2)$ in $Q_{\op{zip}}^n$ which crosses $\eta_{\op{zip}}^n([0,L])$ at most $2L \delta^{-2}$ times and which has $d_{\op{zip}}^n$-length at most $C d_{\op{zip}}^n(\eta_{\op{zip}}^n(\tau_1) , \eta_{\op{zip}}^n(\tau_2)) + \delta^{1/2}$. Let $E_\delta^n$ be the event that~\eqref{eqn-lipschitz-ball} and~\eqref{eqn-lipschitz-length} hold and such a path $\wt\gamma^n$ exists; and let $E_\delta$ be the event that $E_\delta^n$ occurs for infinitely many $n\in\mcl N$, so that (by downward continuity of measure)
\eqb \label{eqn-downward-cont}
\BB P[E_\delta ] = \BB P\left[ \bigcap_{m=1}^\infty\bigcup_{\substack{n\in\mcl N \\ n\geq m}} E_\delta^n \right] \geq 1-\alpha  .
\eqe 
\medskip

\noindent\textit{Step 2: decomposing discrete paths into excursions away from $\eta_{\op{zip}}^n$.}
Now fix $\delta\in (0,\delta_*]$ and suppose that $E_\delta$ occurs. 
Let $N := \lfloor 2 L \delta^{-2} \rfloor$. For $n\in\mcl N$ for which $E_\delta^n$ occurs, let $\wt\gamma^n$ be a path as in the definition of $E_\delta^n$. By~\eqref{eqn-lipschitz-ball} and~\eqref{eqn-lipschitz-length}, if we choose $\delta < C^2$ then $\wt\gamma^n \cap \eta_{\op{zip}}^n \subset \eta_{\op{zip}}^n([0,L])$. Let $s_0^n = 0$, $s_N^n = \op{len}\left(\wt\gamma^n ; d_{\op{zip}}^n \right)$, and for $j \in [1, N-1]_{\BB Z}$ let $s_j^n$ be the $j$th smallest time $s\in \left[0,  \op{len}\left(\wt\gamma^n ; d_{\op{zip}}^n \right) \right]$ at which $\wt\gamma^n$ crosses $\eta_{\op{zip}}^n$; or $s_j^n =   \op{len}\left(\wt\gamma^n ; d_{\op{zip}}^n \right)$ if there are fewer than $j$ such times. 
 
By our choice of $\wt\gamma^n$, the times $s_j^n$ include all of the times at which $\wt\gamma^n$ crosses $\eta_{\op{zip}}^n$.
Hence each of the path segments $\wt\gamma^n([s_{j-1}^n ,s_j^n])$ does not cross $\eta_{\op{zip}}^n$, so for each $j\in [1,N]_{\BB Z}$ we can choose $\xi_j^n \in \{-,+\}$ such that $\wt\gamma^n([s_{j-1}^n ,s_j^n]) \subset Q_{\xi_j^n}^n$. Then for $n\in\mcl N$ such that $E_\delta^n$ occurs, the definition of $\wt\gamma^n$ shows that
\begin{align} \label{eqn-lipschitz-path-decomp}
C d_{\op{zip}}^n\left(\eta_{\op{zip}}^n(\tau_1) , \eta_{\op{zip}}^n(\tau_2)\right) + \delta^{1/2}
\geq \op{len}\left(\wt\gamma^n ; d_{\op{zip}}^n \right)
&\geq \sum_{j=1}^N   d_{\xi_j^n}^n\left( \wt\gamma^n(s_{j-1}^n) , \wt\gamma^n(s_j^n)    \right)  \notag\\
&= \sum_{j=1}^N   \wh d_{\xi_j^n}^n\left( (f_{\xi_j^n}^n)^{-1} \left( \wt\gamma^n(s_{j-1}^n)  \right) , (f_{\xi_j^n}^n)^{-1} \left( \wt\gamma^n(s_j^n)  \right)    \right)  ,
\end{align}
where $f_{\xi_j^n}^n  : \wh Q_\pm^n \rta Q_\pm^n \subset Q_{\op{zip}}^n  $ are the maps defined in~\eqref{eqn-one-side-maps}.
 \medskip

\noindent\textit{Step 3: passing to the scaling limit.}
By compactness and since each point $\wt\gamma^n(s_j^n)$ lies in $\eta_{\op{zip}}^n([0,L])$, 
on the event $E_\delta$ we can a.s.\ find a random subsequence $\mcl N'\subset\mcl N$ such $E_\delta^n$ occurs for each $n\in\mcl N'$ and for each $j\in [0,N]_{\BB Z}$,  there exists $\xi_j \in \{-,+\}$ and $t_j \in [0,L]_{\BB Z}$ such that the following is true. We have $\xi_j^n =\xi_j^{n'} = \xi_j$ for each $n,n' \in \mcl N'$ and each $j\in [0,N]_{\BB Z}$ and
\eqb \label{eqn-lipschitz-path-points}
\lim_{\mcl N'\ni n \rta 0} D_{\xi_j} \left(  (f_{\xi_j}^n)^{-1}\left( \wt\gamma^n(s_j^n) \right) ,   \wh\eta_{\xi_j}(t_j) \right) = \lim_{\mcl N'\ni n \rta 0} D_{\xi_j} \left(  (f_{\xi_j}^n)^{-1}\left( \wt\gamma^n(s_{j-1}^n) \right) ,   \wh\eta_{\xi_j}(t_{j-1}) \right) = 0
\eqe 
for each $j\in [1,N]_{\BB Z}$. 
Note that the reason why we can take the second limit to be zero as well (even though we might not have $\xi_{j-1} = \xi_j$) is that $\eta_-^n(t) = \eta_+^n(t) = \eta_{\op{zip}}^n(t)$ for each $t\geq 0$.  
   
We have 
\eqbn
d_{\op{zip}}^n\left(\eta_{\op{zip}}^n(\tau_1) , \eta_{\op{zip}}^n(\tau_2)\right) \rta \wt d \left(\wt\eta(\tau_1) , \wt\eta (\tau_2)\right) \quad \text{as} \quad \mcl N\ni n \rta\infty
\eqen
so taking the limit of the left and right sides of~\eqref{eqn-lipschitz-path-decomp} along the subsequence $\mcl N'$ and applying~\eqref{eqn-lipschitz-path-points} shows that on $E_\delta$,
\eqbn
C \wt d \left(\wt\eta(\tau_1) , \wt\eta (\tau_2)\right) + \delta^{1/2} \geq \sum_{j=1}^N   d_{\xi_j} \left(  \eta_{\xi_j}(t_{j-1}) ,   \eta_{\xi_j}(t_j) \right)  .
\eqen
The right side of this inequality is at least 
\eqbn
\sum_{j=1}^N   d_{\op{zip}} \left(  \eta_{\op{zip}}(t_{j-1}) ,  \eta_{\op{zip}}(t_j) \right) \geq d_{\op{zip}}\left(  \eta_{\op{zip}}(\tau_1)  , \eta_{\op{zip}}(\tau_2)  \right).
\eqen
Sending $\delta \rta 0$ and then $\alpha \rta 0$ shows that~\eqref{eqn-saw-lipschitz} holds a.s.\ for each fixed $\tau_1,\tau_2 \geq 0$. 
 \medskip

\noindent\textit{Step 4: concluding the proof from~\eqref{eqn-saw-lipschitz}.}
The relation~\eqref{eqn-saw-lipschitz} holds a.s.\ for a dense set of times $\tau_1,\tau_2 \geq 0$, so by continuity it holds a.s.\ for all such times simultaneously, i.e.\ the left inequality in~\eqref{eqn-map-lipschitz} holds whenever $x,y\in \eta_{\op{zip}}$. By the last statement in Proposition~\ref{prop-zip-map}, the $d_{\op{zip}}$-length of any path in $X_{\op{zip}}$ which does not hit $\eta_{\op{zip}}$ except at its endpoints is the same as the $\wt d$-length of its image under $f_{\op{zip}}$. By decomposing a geodesic between given points $x,y\in X_{\op{zip}}$ into two paths which hit $\eta_{\op{zip}}$ only at their endpoints and a path between two points of $\eta_{\op{zip}}$, we obtain~\eqref{eqn-map-lipschitz}. 
\end{proof}

In order to show that the Lipschitz constant in Lemma~\ref{prop-map-lipschitz} is equal to 1, we will use a lower bound for the amount of time that a $\wt d$-geodesic spends away from $\wt\eta$. This bound will be deduced from Proposition~\ref{prop-geodesic-away}.  

\begin{lem} \label{prop-map-geodesic}
There is a universal constant $\beta \in (0,1)$ such that the following is true. 
Fix distinct $\tau_1,\tau_2 \geq 0$. Almost surely, there exists a $\wt d$-geodesic $\gamma$ from $\wt\eta(\tau_1)$ to $\wt\eta(\tau_2)$ such that with $T_\gamma = \{t\in [0,\wt d(\wt\eta(\tau_1) , \wt\eta(\tau_2))] \,:\, \gamma(t) \notin \wt\eta\}$,  
\eqb \label{eqn-map-geodesic}
|T_\gamma| \geq \beta \wt d(\wt\eta(\tau_1) , \wt\eta(\tau_2)) ,
\eqe 
where here $|\cdot|$ denotes Lebesgue measure. 
\end{lem}
\begin{proof}
Let $\beta \in (0,1)$ be chosen so that the conclusion of Proposition~\ref{prop-geodesic-away} holds. We will show that~\eqref{eqn-map-geodesic} holds with $\beta/2$ in place of~$\beta$ by applying Proposition~\ref{prop-geodesic-away} and using that geodesics behave well under Gromov-Hausdorff convergence.

Since Proposition~\ref{prop-geodesic-away} is proven only for a fixed time interval $[-L n^{1/2} , L n^{1/2}]_{\BB Z}$, we first need to ensure that we can restrict attention to such a time interval.  By Lemma~\ref{prop-ball-contain}, we can find $R = R(\tau_1,\tau_2,\alpha)  > 0$ and $L = L(\tau_1,\tau_2,\alpha) > 0$ such that for each $n\in\mcl N$, it holds with probability at least $1-\alpha/2$ that 
\eqb \label{eqn-geo-ball}
d_{\op{zip}}^n\left( \eta_{\op{zip}}^n(\tau_i) , Q_{\op{zip}}^n \setminus B_R\left( \eta_{\op{zip}}^n(0) ; d_{\op{zip}}^n \right) \right) \geq 2  d_{\op{zip}}^n\left( \eta_{\op{zip}}^n(\tau_1)  , \eta_{\op{zip}}^n(\tau_2) \right) ,\quad \forall i \in \{1,2\} .
\eqe 
and 
\eqb \label{eqn-geo-length}
B_R\left( \eta_{\op{zip}}^n (0) ;  d_{\op{zip}}^n \right) \cap \eta_{\op{zip}}^n([ L , \infty)) = \emptyset .
\eqe 
By Proposition~\ref{prop-geodesic-away} and our choice of $\beta$, we can find $\delta_* = \delta_*(\tau_1,\tau_2,\alpha) > 0$ such that for each $\delta \in (0,\delta_*]$ and each sufficiently large $n\in\mcl N$, it holds with probability at least $1-\alpha/2$ that every $d_{\op{zip}}^n$-geodesic $\gamma^n$ from $\eta_{\op{zip}}^n(\tau_1)$ to $\eta_{\op{zip}}^n(\tau_2)$ spends at least $\beta d_{\op{zip}}^n\left(\eta_{\op{zip}}^n(\tau_1), \eta_{\op{zip}}^n(\tau_2) \right) - \delta^{1/2}$ units of time at $d_{\op{zip}}^n$-distance at least $\beta \delta$ away from $\eta_{\op{zip}}^n([ 0,L])$. 
Let $F_\delta^n$ be the event that this is the case and~\eqref{eqn-geo-ball} and~\eqref{eqn-geo-length} hold; and let $F_\delta$ be the event that $F_\delta^n$ occurs for infinitely many $n\in\mcl N$, so that $\BB P[F_\delta ] \geq 1-\alpha $ by the same argument as in~\eqref{eqn-downward-cont}.

On $F_\delta$, choose for each $n\in\mcl N$ such that $F_\delta^n$ occurs a $d_{\op{zip}}^n$-geodesic $\gamma^n$ from $\eta_{\op{zip}}^n(\tau_1)$ to $\eta_{\op{zip}}^n(\tau_2)$.  
The $d_{\op{zip}}^n$-geodesics are 1-Lipschitz functions from $[0,d_{\op{zip}}^n(\eta_{\op{zip}}^n(\tau_1) , \eta_{\op{zip}}^n(\tau_2))]$ to $(X_{\op{zip}}^n ,d_{\op{zip}}^n)$.
Since $(X_{\op{zip}}^n ,d_{\op{zip}}^n)$ is isometrically embedded into $(Z_{\op{zip}} , D_{\op{zip}})$, these geodesics are also 1-Lipschitz functions from $[0,d_{\op{zip}}^n(\eta_{\op{zip}}^n(\tau_1) , \eta_{\op{zip}}^n(\tau_2))]$ to $(Z_{\op{zip}} , D_{\op{zip}})$. 
By the convergence $\eta_{\op{zip}}^n \rta \wt\eta$, the intervals $[0,d_{\op{zip}}^n(\eta_{\op{zip}}^n(\tau_1) , \eta_{\op{zip}}^n(\tau_2))]$ on which $\gamma^n$ is defined are all contained in some fixed compact subset of $\BB R$. 

By the Arz\'ela-Ascoli theorem, on $F_\delta$ we can a.s.\ find a random subsequence $\mcl N'\subset\mcl N$ such that $F_\delta^n$ occurs for each $n\in\mcl N'$ and a curve $\gamma$ from $\wt\eta(\tau_1)$ to $\wt\eta(\tau_2)$ in $\wt X$ such that $\gamma^n\rta \gamma$ in the $D_{\op{zip}}$-uniform topology. 
Since each $\gamma^n$ is a $d_{\op{zip}}^n$-geodesic, it is easily seen that $\gamma$ is a $\wt d$-geodesic. 
This geodesic $\gamma $ spends at least $\beta \wt d(\wt\eta(\tau_1) , \wt\eta(\tau_2)) - \delta^{1/2}$ units of time away from $\wt\eta([0,L])$.  By passing to the limit in~\eqref{eqn-geo-ball} and~\eqref{eqn-geo-length}, we see that $\gamma$ cannot hit $\wt\eta([L,\infty))$, so $\gamma$ spends at least $\beta \wt d(\wt\eta(\tau_1) , \wt\eta(\tau_2)) - \delta^{1/2}$ units of time away from $\wt\eta$.

For each $\delta \in (0,\delta_*]$, a geodesic $\gamma$ as in the preceding paragraph exists with probability at least $1-\alpha$. 
By Lemma~\ref{prop-simple-curve}, we can choose $\delta \in (0,\delta_*]$ such that with probability at least $1-\alpha$, we have $\delta^{1/2} \leq (\beta/2) \wt d(\wt\eta(\tau_1) , \wt\eta(\tau_2))$. Then with probability at least $1-2\alpha$, there exists a $\wt d$-geodesic $\gamma$ satisfying~\eqref{eqn-map-geodesic}. Since $\alpha$ is arbitrary, we conclude.
\end{proof}

\begin{figure}[ht!]
 \begin{center}
\includegraphics[scale=1]{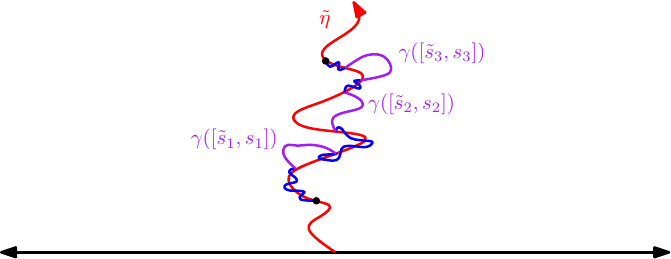} 
\caption[Re-routing procedure in the proof of Theorem~\ref{thm-saw-conv-wedge}]{Illustration of the proof of Theorem~\ref{thm-saw-conv-wedge}. Fix times $\tau_1,\tau_2 \geq 0$ and consider a $\wt d$-geodesic $\gamma$ from $\wt\eta(\tau_1)$ to $\wt\eta(\tau_2)$.  By Lemma~\ref{prop-map-geodesic} there exists $\beta \in (0,1)$ such that $\gamma$ a.s.\ spends at least a $\beta$-fraction of its time away from $\wt\eta$, so we can find finitely many excursion intervals $[\wt s_1, s_1],\dots , [\wt s_N , s_N]$ during which $\gamma$ does not cross $\wt\eta$ whose total length is at least $\beta/2$ times the $\wt d$-length of $\gamma$ (purple). The $\wt d$-length of the restriction of $\gamma$ to each such excursion interval is equal to the $d_{\op{zip}}$-length of its pre-image under $f_{\op{zip}}$ since $f_{\op{zip}}$ preserves the lengths of paths which stay entirely on one of the two sides of $\wt\eta $ (Proposition~\ref{prop-zip-map}). The $\wt d$-length of each of the blue intermediate segments is at most $C$ times the $d_{\op{zip}}$-length of its pre-image under $f_{\op{zip}}$ by Lemma~\ref{prop-map-lipschitz}. 
Hence the $d_{\op{zip}}$-length of $f_{\op{zip}}^{-1}(\gamma)$ is at most $(1-\beta/2) C + \beta/2$ times the $\wt d$-length of $\gamma$. 
Since this is true for almost every pair of times $(\tau_1,\tau_2)$, we get $C \leq (1-\beta/2) C + \beta/2$, so $C=1$.  
}\label{fig-lip-geo}
\end{center}
\end{figure}

We are now ready to prove our main theorem. See Figure~\ref{fig-lip-geo} for an illustration of the proof. 

\begin{proof}[Proof of Theorem~\ref{thm-saw-conv-wedge}]
The map $f_{\op{zip}} : X_{\op{zip}} \rta \wt X$ constructed in Proposition~\ref{prop-zip-map} is $1$-Lipschitz, surjective, and pushes forward $\mu_{\op{zip}}$ to $\wt\mu$ and $\eta_{\op{zip}}$ to $\wt\eta$. 
We will show that $f_{\op{zip}}$ does not decrease distances, so is an isometry. 
This will imply that $\wt{\frk X}$ and $\frk X_{\op{zip}}$ are equivalent elements $\BB M_\infty^{\op{GHPU}}$. Since $\frk X_{\op{zip}}^n \rta \wt{\frk X}$ in the local GHPU topology as $\mcl N \ni n \rta\infty$ and our initial choice of subsequence (from which $\mcl N$ was extracted) was arbitrary, this will imply that $\frk X_{\op{zip}}^n \rta  \frk X_{\op{zip}}$ a.s.\ in the local GHPU topology. 

By Lemma~\ref{prop-map-lipschitz}, there is a universal constant $C \geq 1$ such that a.s.\ 
\eqb  \label{eqn-optimal-lipschitz}
  d_{\op{zip}}\left( x,y \right) \leq  C \wt d\left( f_{\op{zip}}(x), f_{\op{zip}}(y) \right) ,\quad \forall x , y\in X_{\op{zip}} .
\eqe 
Suppose that $C \geq 1$ is the smallest universal constant for which this is the case. We will show that in fact $C=1$. 

Let $\beta \in (0,1)$ be chosen so that the conclusion of Lemma~\ref{prop-map-geodesic} is satisfied. Almost surely, for each distinct $\tau_1,\tau_2 \geq 0$ there is a $\wt d$-geodesic $\gamma$ from $\wt\eta(\tau_1)$ to $\wt\eta(\tau_2)$ which spends at least a $\beta$-fraction of its time away from $\wt\eta$. For such a geodesic $\gamma$, we can choose finitely many times 
\eqbn
 0 = s_0 < \wt s_1 < s_1 <  \dots < \wt s_{N } < s_{N }<  \wt s_{N+1} = \wt d(\wt\eta(\tau_1) , \wt\eta(\tau_2))
\eqen
such that the following hold. 
For each $j\in [0,N]_{\BB Z}$, there exists $\wt t_j \geq 0$ such that $\gamma(\wt s_j) = \wt\eta(\wt t_j)$; for each $j\in [1,N+1]_{\BB Z}$, there exists $t_j \geq 0$ such that $\gamma(s_j ) = \wt\eta(t_j)$; each segment $\gamma((\wt s_j , s_j))$ for $j\in [1,N]_{\BB Z}$ is contained in either $f_{\op{zip}}(X_-)\setminus\wt\eta$ or $f_{\op{zip}}(X_+)\setminus\wt\eta$; and  
\eqb \label{eqn-good-interval-sum}
\sum_{j=1}^N (s_j - \wt s_j) \geq (\beta/2)  \wt d(\wt\eta(\tau_1) , \wt\eta(\tau_2)) . 
\eqe
We note that the total length of the complementary segments of $\gamma$ (during which it might cross $\wt\eta$ many times) satisfies
\eqb \label{eqn-geodesic-complement}
\sum_{j=0}^N (  \wt s_{j+1}  - s_j)  =      \wt d(\wt\eta(\tau_1) , \wt\eta(\tau_2))  -\sum_{j=1}^N (s_j - \wt s_j)   .
\eqe

By~\eqref{eqn-optimal-lipschitz} and since $\gamma$ is a $\wt d$-geodesic,
\eqb \label{eqn-bad-interval-bound}
d_{\op{zip}}\left( \eta_{\op{zip}}(t_j) , \eta_{\op{zip}}(\wt t_{j+1}) \right) \leq C (\wt s_{j+1} - s_j ) ,\quad \forall j \in [0,N]_{\BB Z} .
\eqe 
The curve $\gamma$ is a $\wt d$-geodesic and each segment $\gamma((\wt s_j , s_j))$ is contained in one of $f_{\op{zip}}(X_\pm)\setminus\wt\eta$.
By the last assertion of Proposition~\ref{prop-zip-map}, $f_{\op{zip}}|_{X_\pm\setminus \eta_{\op{zip}}}$ is an isometry with respect to the internal metric of $d_{\op{zip}}$ on $X_\pm\setminus \eta_{\op{zip}}$. Therefore,
\eqb \label{eqn-good-interval-bound}
d_{\op{zip}}\left( \eta_{\op{zip}}(\wt t_j) , \eta_{\op{zip}}(t_j) \right) = s_j - \wt s_j   ,\quad \forall j \in [1,N]_{\BB Z} .
\eqe 
By~\eqref{eqn-geodesic-complement},~\eqref{eqn-bad-interval-bound}, and~\eqref{eqn-good-interval-bound},  
\begin{align*}
d_{\op{zip}}\left( \eta_{\op{zip}}(\tau_1) , \eta_{\op{zip}}(\tau_2) \right) 
&\leq C \sum_{j=0}^N  (\wt s_{j+1} - s_j ) + \sum_{j=1}^N (s_j - \wt s_j)\\
&= C \left( \wt d(\wt\eta(\tau_1) , \wt\eta(\tau_2))  - \sum_{j=1}^N (s_j - \wt s_j) \right)   + \sum_{j=1}^N (s_j - \wt s_j).
\end{align*}
Since $C \geq 1$, the above expression only gets larger when we replace the term $\sum_{j=1}^N (s_j - \wt{s}_j)$ by its lower bound of $(\beta/2)\wt d(\wt\eta(\tau_1) , \wt\eta(\tau_2))$ from~\eqref{eqn-good-interval-sum}.  We therefore have that
\begin{align}
d_{\op{zip}}\left( \eta_{\op{zip}}(\tau_1) , \eta_{\op{zip}}(\tau_2) \right)  &\leq C (1-\beta/2)  \wt d(\wt\eta(\tau_1) , \wt\eta(\tau_2)) +  (\beta/2)  \wt d(\wt\eta(\tau_1) , \wt\eta(\tau_2)) . \label{eqn-better-constant}
\end{align}

The inequality~\eqref{eqn-better-constant} holds a.s.\ for a dense set of pairs of times $\tau_1,\tau_2\geq 0$. By the same argument used at the end of the proof of Lemma~\ref{prop-map-lipschitz}, we infer that~\eqref{eqn-optimal-lipschitz} holds a.s.\ with $(1-\beta/2)C + \beta/2$ in place of $C$. By the minimality of $C \geq 1$,
\eqbn
C \leq (1-\beta/2) C + \beta/2,
\eqen
therefore $C=1$. 
\end{proof}

\begin{figure}[t!]
 \begin{center}
\includegraphics[scale=.8]{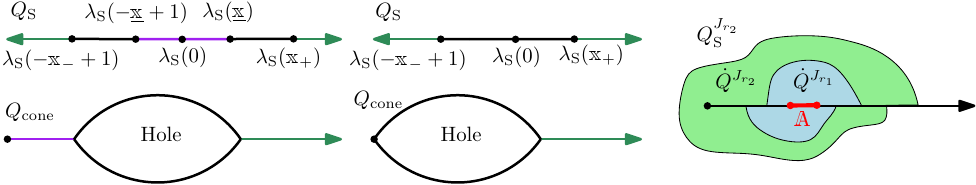}
\vspace{-0.01\textheight}
\caption{ \textbf{Left:} Illustration of the gluing procedure for a single UIHPQ$_{\op{S}}$ used in the analog of Section~\ref{sec-glued-peeling} for the proof of Theorem~\ref{thm-saw-conv-cone}. The purple (resp.\ green) boundary arcs of $Q_{\op{S}}$ to the left and right of $\lambda_{\op{S}}(0)$ are identified, but the black arcs are not identified, which results in a hole. \textbf{Middle:} Illustration of the same gluing procedure in the case when $\ul{\BB x} =0$. \textbf{Right:} Illustration of the variant of the glued peeling process used in the proof of Theorem~\ref{thm-saw-conv-cone}, which is a peeling process of the single UIHPQ$_{\op{S}}$ $Q_{\op{S}}$. Note that at time $J_{r_2}$, the clusters started from the two sides of the initial edge set $\BB A$ have merged. This makes it so that the boundary of $Q_{\op{S}}^{J_{r_2}}$ is glued to itself in the manner of the middle figure. After time $J_{r_2}$, the glued peeling process behaves like a standard peeling-by-layers process of $Q_{\op{S}}^{J_{r_2}}$ started from the boundary arc $\bdy \dot Q^{J_{r_2}} \subset \bdy Q_{\op{S}}^{J_{r_2}}$. }\label{fig-cone-peeling}
\end{center}
\vspace{-1em}
\end{figure}

\begin{remark}[Whole-plane cases]
 \label{remark-other-thm}
In this remark we explain the modifications to our proof of Theorem~\ref{thm-saw-conv-wedge} which are needed to prove Theorems~\ref{thm-saw-conv-2side} and~\ref{thm-saw-conv-cone}. 
The proofs of both of these theorems are essentially identical to the proof of Theorem~\ref{thm-saw-conv-wedge}, with only minor cosmetic modifications. In the case of Theorem~\ref{thm-saw-conv-2side}, the proof is almost verbatim the same --- we just always work with a pair of UIHPQ$_{\op{S}}$'s glued together along their entire boundaries instead of just their positive boundary rays. In the case of Theorem~\ref{thm-saw-conv-cone}, we only have a single UIHPQ$_{\op{S}}$ so a few more cosmetic changes are necessary, which we now describe. 

We start with a UIHPQ$_{\op{S}}$ $(Q_{\op{S}} , \BB e_{\op{S}})$ with boundary path $\lambda_{\op{S}} : \BB Z\rta \mcl E(\bdy Q_{\op{S}})$ and define a way of gluing its left and right boundary rays in such a way that one is left with a ``hole", analogously to Section~\ref{sec-glued-peeling}. See Figure~\ref{fig-cone-peeling}, left and middle, for an illustration.
In particular, we fix \emph{gluing times} $\ul{\BB x} , \BB x_- , \BB x_+ \in \BB N_0$ with $\ul{\BB x} \leq \BB x_- \wedge \BB x_+$ and let $Q_{\op{cone}}$ be the planar map obtained from $Q_-$ and $Q_+$ by identifying $\lambda_{\op{S}}(x)$ with $\lambda_{\op{S}}(-x +1)$ for each $x\in [0,\ul{\BB x} -1 ]_{\BB Z}$ and $\lambda_{\op{S}}(- \BB x_- - y + 1)$ with $\lambda_+(\BB x_+ + y)$ for each $y \in \BB N_0$. Note that we allow $\ul{\BB x} = 0 $, in which case the root edge $\lambda_{\op{S}}(0)=  \BB e_{\op{S}}$ is not identified to any other edge. For a fixed choice of connected initial edge set $\BB A$ which contains the boundary of the ``hole" of $Q_{\op{cone}}$, the glued peeling process started from $\BB A$ as described in Section~\ref{sec-glued-peeling} still makes sense in this setting: it is simply a peeling process of the single UIHPQ$_{\op{S}}$ $Q_{\op{S}}$. In particular, the Markov property of Proposition~\ref{prop-peel-law} continues to hold, except that there is only a single unexplored UIHPQ$_{\op{S}}$ $Q_{\op{S}}^j$, instead of two unexplored UIHPQ$_{\op{S}}$'s $Q_\pm^j$. 

We note that after running the glued peeling process for a sufficiently long time, the clusters being peeled from the two sides of the gluing interface will merge into one another and we will be left with only one ``side", as illustrated in the right panel of Figure~\ref{fig-cone-peeling} (after this time, the unexplored UIHPQ$_{\op{S}}$ is glued to itself in such a way that $\ul{\BB x} = 0$). This is not a problem for our purposes since the definition of the glued peeling process still makes sense after this time. It simply behaves like a peeling-by-layers process of the unexplored UIHPQ$_{\op{S}}$ based at a single boundary interval, namely the one which corresponds to the boundary of the glued peeling cluster.

The aforementioned Markov property enables us to apply the results of Section~\ref{sec-peeling-prelim} to estimate the glued peeling process. The statements and proofs in Sections~\ref{sec-peeling-glued},~\ref{sec-peeling-moment}, and~\ref{sec-geodesic-properties} carry through essentially verbatim in the setting of a single UIHPQ$_{\op{S}}$ with its boundaries identified. We can then apply the results of these sections together with exactly the same argument given in Section~\ref{sec-saw-conv} to prove Theorem~\ref{thm-saw-conv-cone}. The only difference is that we only have a single UIHPQ$_{\op{S}}$ and a single Brownian half-plane, so there is only one function $f^n$ instead of two functions $f_\pm^n$, etc.  
\end{remark}

%Furthermore, the two sides of the gluing interface are not independent (since they are part of the same UIHPQ$_{\op{S}}$), but this does not present a problem; we only need the Markov property for a single peeling step, which still holds when we are only gluing a single UIHPQ$_{\op{S}}$ to itself. 

\appendix

\section{Index of notation}
\label{sec-index}
 
Here we record some commonly used symbols in the paper, along with their meaning and the location where they are first defined. 

\begin{multicols}{2}
\begin{itemize}
\item $Q_-,Q_+,Q_{\op{S}}$: UIHPQ$_{\op{S}}$'s; Section~\ref{sec-main-result} (c.f.\ Section~\ref{sec-quad-prelim}).
\item $\lambda_- , \lambda_+ , \lambda_{\op{S}}$: boundary path of the UIHPQ$_{\op{S}}$; Section~\ref{sec-main-result}.
\item $Q_{\op{zip}}$: SAW-decorated map obtained by gluing $Q_\pm$; Section~\ref{sec-main-result} (c.f.\ Section~\ref{sec-peeling-moment}).
\item $\lambda_{\op{zip}}$: gluing interface (=SAW); Section~\ref{sec-main-result}. 
\item $\frk X_\pm = (X_\pm , d_\pm , \mu_\pm , \eta_\pm)$: Brownian half-planes; Section~\ref{sec-main-result}.
\item $\frk Q_{\op{zip}}^n = (Q_{\op{zip}}^n , d_{\op{zip}}^n , \mu_{\op{zip}}^n , \eta_{\op{zip}}^n)$: re-scaled glued map; Section~\ref{sec-main-result} (c.f.~\eqref{eqn-cmm-spaces}).
\item $\frk X_{\op{zip}} = (X_{\op{zip}}, d_{\op{zip}}, \mu_{\op{zip}},\eta_{\op{zip}})$: space obtained by gluing $X_\pm$; Section~\ref{sec-main-result} (c.f.~\eqref{eqn-cmm-spaces}).
\item $\frk P(Q,e)$: peeling indicator;~\eqref{eqn-peeling-indicator}.
\item $\frk f(Q,e)$: peeled quadrilateral; Section~\ref{sec-general-peeling}.
\item $\op{Peel}(Q,e)$: unbounded component of $Q\setminus \frk f(Q,e)$; Section~\ref{sec-general-peeling}.
\item $\frk F(Q,e)$: union of bounded components of $Q\setminus \frk f(Q,e)$; Section~\ref{sec-general-peeling}.
\item $\op{Co}(Q,e)$: $\#$ of edges of $\bdy Q$ disconnected from $\infty$ by $\frk f(Q,e)$; Section~\ref{sec-general-peeling}.
\item $\op{Ex}(Q,e)$: $\#$ of exposed edges of $\frk f(Q,e)$; Section~\ref{sec-general-peeling}.
\item $\BB A$: subset of $\bdy Q_-\cup \bdy Q_+$ where we start glued peeling process; Section~\ref{sec-general-peeling}.
\item $\dot Q^j$: glued peeling cluster; Section~\ref{sec-glued-peeling}.
\item $Q^j_\pm$: unexplored UIHPQ$_{\op{S}}$'s in the glued peeling procedure; Section~\ref{sec-glued-peeling}.
\item $J_r$: number of peeled quadrilaterals before the glued peeling process reaches radius $r$; Section~\ref{sec-glued-peeling}. 
\item $\xi^j$: sign indicating whether $Q_-$ or $Q_+$ is peeled at step $j$; Section~\ref{sec-glued-peeling}. 
\item $\mcl F^j$: filtration for glued peeling process; Section~\ref{sec-glued-peeling}. 
\item $\wh Y^j$: number of edges of $\dot Q^j\cap (\bdy Q_-\cup \bdy Q_+)$; \eqref{eqn-saw-edge-count'}.
\item $\wh Y^j_n$: sum of jumps of $\wh Y^j$ truncated at level $n$; \eqref{eqn-jump-truncated}.
\item $X^j_\pm $: outer boundary length of $\dot Q^j\cap Q_\pm$; \eqref{eqn-interior-count}.
\item $Y^j_\pm$: number of edges of $Q_\pm$ disconnected from $\infty$ by $\dot Q^j$; \eqref{eqn-bdy-count}.
\item $X^j$ and $Y^j$: $X_-^j + X_+^j$ and $Y_-^j + Y_+^j$; \eqref{eqn-2side-count}. 
\item $Z^j$: $X^j - Y^j$; \eqref{eqn-2side-count}.   
\item $R(C)$: good radius used in proof of Proposition~\ref{prop-lipschitz-path}; Lemma~\ref{prop-good-radius}.
\item $\wt R(C)$: good radius used in proof of Proposition~\ref{prop-geodesic-away}; Lemma~\ref{prop-good-radius'}.
\item $r_k$ and $L_k$: radii and boundary lengths for clusters used in Section~\ref{sec-geo-iterate}; \eqref{eqn-radius-iterate-r-L-def}.
\item $\mcl I^n(\delta)$: set of $\lceil 2L \delta^{-2} \rceil$ discrete intervals of length $\lfloor \delta^2 n^{1/2} \rfloor$ in $[- Ln^{1/2} , L n^{1/2}]_{\BB Z}$;~\eqref{eqn-good-radius-partition}.
\item $Q_I(C)$ and $\wt Q_I(C)$: good-radius glued peeling cluster for $I\in \mcl I^n(\delta)$;~\eqref{eqn-good-hull-abbrv}. 
\item $\wt{\mcl N}$: subsequence along which we have GHPU convergence; Section~\ref{sec-saw-conv-tight}.
\item $\wt{\frk X} = (\wt X , \wt d , \wt\mu,\wt\eta)$: subsequential limiting space; Section~\ref{sec-saw-conv-tight}. 
\item $(Z_\pm , D_\pm)$, $(Z_{\op{zip}} , D_{\op{zip}})$: big metric spaces obtained using Proposition~\ref{prop-ghpu-embed-local}; Section~\ref{sec-saw-conv-tight}. 
\item $\wh{\frk Q}_\pm^n = (\wh Q_\pm^n , \wh d_\pm^n , \wh\mu_\pm^n ,\wh\eta_\pm^n)$: re-scaled UIHPQ$_{\op{S}}$'s viewed as subspaces of $(Z_\pm , D_\pm)$; Section~\ref{sec-saw-conv-tight}. 
\item $f_\pm^n $: identity map $Z_\pm \supset \wh Q_\pm^n \rta Q_\pm^n \subset Z_{\op{zip}}$; Section~\ref{sec-saw-conv-tight}. 
\item $f_\pm$: subsequential limits of $f_\pm^n$; Lemma~\ref{prop-map-limit}. 
\item $f_{\op{zip}}$: $1$-Lipschitz map from $X_{\op{zip}}$ to $\wt X$; Proposition~\ref{prop-zip-map}. 
\end{itemize}
\end{multicols}

\bibliography{cibiblong,cibib}

\def\cprime{$'$}
\begin{thebibliography}{BDCGS12}

\bibitem[AC15]{angel-curien-uihpq-perc}
O.~Angel and N.~Curien.
\newblock Percolations on random maps {I}: {H}alf-plane models.
\newblock {\em Ann. Inst. Henri Poincar\'e Probab. Stat.}, 51(2):405--431,
  2015, \arxiv{1301.5311}. \MR{3335009}

\bibitem[ADJ97]{adj-quantum-geometry}
J.~Ambj{\o}rn, B.~Durhuus, and T.~Jonsson.
\newblock {\em Quantum geometry}.
\newblock Cambridge Monographs on Mathematical Physics. Cambridge University
  Press, Cambridge, 1997.
\newblock A statistical field theory approach. \MR{1465433}

\bibitem[Ang03]{angel-peeling}
O.~Angel.
\newblock Growth and percolation on the uniform infinite planar triangulation.
\newblock {\em Geom. Funct. Anal.}, 13(5):935--974, 2003, \arxiv{0208123}.
  \MR{2024412}

\bibitem[AR15]{angel-ray-classification}
O.~Angel and G.~Ray.
\newblock Classification of half-planar maps.
\newblock {\em Ann. Probab.}, 43(3):1315--1349, 2015, \arxiv{1303.6582}.
  \MR{3342664}

\bibitem[BBG12]{bbg-recursive-approach}
G.~{Borot}, J.~{Bouttier}, and E.~{Guitter}.
\newblock {A recursive approach to the O(n) model on random maps via nested
  loops}.
\newblock {\em Journal of Physics A Mathematical General}, 45(4):045002,
  February 2012, \arxiv{1106.0153}.

\bibitem[BBS16]{slade-4d-survey}
R.~{Bauerschmidt}, D.~C. {Brydges}, and G.~{Slade}.
\newblock {Renormalisation group analysis of 4D spin models and self-avoiding
  walk}.
\newblock {\em ArXiv e-prints}, February 2016, \arxiv{1602.04048}.

\bibitem[BC13]{benjamini-curien-uipq-walk}
I.~Benjamini and N.~Curien.
\newblock Simple random walk on the uniform infinite planar quadrangulation:
  subdiffusivity via pioneer points.
\newblock {\em Geom. Funct. Anal.}, 23(2):501--531, 2013, \arxiv{1202.5454}.
  \MR{3053754}

\bibitem[BDCGS12]{saw-lectures}
R.~Bauerschmidt, H.~Duminil-Copin, J.~Goodman, and G.~Slade.
\newblock Lectures on self-avoiding walks.
\newblock In {\em Probability and statistical physics in two and more
  dimensions}, volume~15 of {\em Clay Math. Proc.}, pages 395--467. Amer. Math.
  Soc., Providence, RI, 2012, \arxiv{1206.2092}. \MR{3025395}

\bibitem[Bet15]{bet-disk-tight}
J.~Bettinelli.
\newblock Scaling limit of random planar quadrangulations with a boundary.
\newblock {\em Ann. Inst. Henri Poincar\'e Probab. Stat.}, 51(2):432--477,
  2015, \arxiv{1111.7227}. \MR{3335010}

\bibitem[BG09]{bg-simple-quad}
J.~Bouttier and E.~Guitter.
\newblock Distance statistics in quadrangulations with a boundary, or with a
  self-avoiding loop.
\newblock {\em J. Phys. A}, 42(46):465208, 44, 2009, \arxiv{0906.4892}.
  \MR{2552016}

\bibitem[BLR17]{blr-exponents}
N.~Berestycki, B.~Laslier, and G.~Ray.
\newblock Critical exponents on {F}ortuin-{K}asteleyn weighted planar maps.
\newblock {\em Comm. Math. Phys.}, 355(2):427--462, 2017, \arxiv{1502.00450}.
  \MR{3681382}

\bibitem[BM17]{bet-mier-disk}
J.~Bettinelli and G.~Miermont.
\newblock Compact {B}rownian surfaces {I}: {B}rownian disks.
\newblock {\em Probab. Theory Related Fields}, 167(3-4):555--614, 2017,
  \arxiv{1507.08776}. \MR{3627425}

\bibitem[BMR16]{bmr-uihpq}
E.~{Baur}, G.~{Miermont}, and G.~{Ray}.
\newblock {Classification of scaling limits of uniform quadrangulations with a
  boundary}.
\newblock {\em ArXiv e-prints}, August 2016, \arxiv{1608.01129}.

\bibitem[BS85]{brydges-spencer-saw5}
D.~Brydges and T.~Spencer.
\newblock Self-avoiding walk in {$5$} or more dimensions.
\newblock {\em Comm. Math. Phys.}, 97(1-2):125--148, 1985. \MR{782962}

\bibitem[BS01]{benjamini-schramm-topology}
I.~Benjamini and O.~Schramm.
\newblock Recurrence of distributional limits of finite planar graphs.
\newblock {\em Electron. J. Probab.}, 6:no. 23, 13 pp. (electronic), 2001,
  \arxiv{0011019}. \MR{1873300 (2002m:82025)}

\bibitem[CC15]{caraceni-curien-uihpq}
A.~{Caraceni} and N.~{Curien}.
\newblock {Geometry of the Uniform Infinite Half-Planar Quadrangulation}.
\newblock {\em ArXiv e-prints}, August 2015, \arxiv{1508.00133}.

\bibitem[CC16]{caraceni-curien-saw}
A.~{Caraceni} and N.~{Curien}.
\newblock {Self-Avoiding Walks on the UIPQ}.
\newblock {\em ArXiv e-prints}, September 2016, \arxiv{1609.00245}.

\bibitem[CL14]{curien-legall-plane}
N.~Curien and J.-F. {Le Gall}.
\newblock The {B}rownian plane.
\newblock {\em J. Theoret. Probab.}, 27(4):1249--1291, 2014, \arxiv{1204.5921}.
  \MR{3278940}

\bibitem[CLG17]{curien-legall-peeling}
N.~Curien and J.-F. Le~Gall.
\newblock Scaling limits for the peeling process on random maps.
\newblock {\em Ann. Inst. Henri Poincar\'e Probab. Stat.}, 53(1):322--357,
  2017, \arxiv{1412.5509}. \MR{3606744}

\bibitem[CM15]{curien-miermont-uihpq}
N.~Curien and G.~Miermont.
\newblock Uniform infinite planar quadrangulations with a boundary.
\newblock {\em Random Structures Algorithms}, 47(1):30--58, 2015,
  \arxiv{1202.5452}. \MR{3366810}

\bibitem[DCS12]{dc-smirnov-connective-constant}
H.~Duminil-Copin and S.~Smirnov.
\newblock The connective constant of the honeycomb lattice equals
  {$\sqrt{2+\sqrt{2}}$}.
\newblock {\em Ann. of Math. (2)}, 175(3):1653--1665, 2012, \arxiv{1007.0575}.
  \MR{2912714}

\bibitem[DK88]{dup-kos-saw}
B.~Duplantier and I.~Kostov.
\newblock Conformal spectra of polymers on a random surface.
\newblock {\em Phys. Rev. Lett.}, 61:1433--1437, Sep 1988.

\bibitem[DK90]{dup-kos-saw-long}
B.~Duplantier and I.~Kostov.
\newblock Geometrical critcal phenomena on a random surface of arbitrary genus.
\newblock {\em Nucl. Phys.B}, 340:491--541, 1990.

\bibitem[DMS14]{wedges}
B.~{Duplantier}, J.~{Miller}, and S.~{Sheffield}.
\newblock {Liouville quantum gravity as a mating of trees}.
\newblock {\em ArXiv e-prints}, September 2014, \arxiv{1409.7055}.

\bibitem[DS11]{shef-kpz}
B.~Duplantier and S.~Sheffield.
\newblock Liouville quantum gravity and {KPZ}.
\newblock {\em Invent. Math.}, 185(2):333--393, 2011, \arxiv{1206.0212}.
  \MR{2819163 (2012f:81251)}

\bibitem[Dur10]{durrett}
R.~Durrett.
\newblock {\em Probability: theory and examples}.
\newblock Cambridge Series in Statistical and Probabilistic Mathematics.
  Cambridge University Press, Cambridge, fourth edition, 2010. \MR{2722836
  (2011e:60001)}

\bibitem[Flo53]{flory-book}
P.~Flory.
\newblock {\em {Principles of Polymer Chemistry}}.
\newblock {Cornell University Press}, 1953.

\bibitem[GL13]{gl-connective3}
G.~R. {Grimmett} and Z.~{Li}.
\newblock {Counting self-avoiding walks}.
\newblock {\em ArXiv e-prints}, April 2013, \arxiv{1304.7216}.

\bibitem[GL17]{gl-saw-survey-new}
G.~R. {Grimmett} and Z.~{Li}.
\newblock {Self-avoiding walks and connective constants}.
\newblock {\em ArXiv e-prints}, April 2017, 1704.05884.

\bibitem[GM16]{gwynne-miller-gluing}
E.~{Gwynne} and J.~{Miller}.
\newblock {Metric gluing of Brownian and $\sqrt{8/3}$-Liouville quantum gravity
  surfaces}.
\newblock {\em Annals of {P}robability}, to appear, 2016, \arxiv{1608.00955}.

\bibitem[GM17a]{gwynne-miller-perc}
E.~{Gwynne} and J.~{Miller}.
\newblock {Convergence of percolation on uniform quadrangulations with boundary
  to SLE$_{6}$ on $\sqrt{8/3}$-Liouville quantum gravity}.
\newblock {\em ArXiv e-prints}, January 2017, \arxiv{1701.05175}.

\bibitem[GM17b]{gwynne-miller-uihpq}
E.~Gwynne and J.~Miller.
\newblock Scaling limit of the uniform infinite half-plane quadrangulation in
  the {G}romov-{H}ausdorff-{P}rokhorov-uniform topology.
\newblock {\em Electron. J. Probab.}, 22:1--47, 2017, \arxiv{1608.00954}.

\bibitem[GM19]{gwynne-miller-simple-quad}
E.~Gwynne and J.~Miller.
\newblock Convergence of the free {B}oltzmann quadrangulation with simple
  boundary to the {B}rownian disk.
\newblock {\em Ann. Inst. Henri Poincar\'{e} Probab. Stat.}, 55(1):551--589,
  2019, \arxiv{1701.05173}. \MR{3901655}

\bibitem[Ham57]{hammersley-cc}
J.~M. Hammersley.
\newblock Percolation processes. {II}. {T}he connective constant.
\newblock {\em Proc. Cambridge Philos. Soc.}, 53:642--645, 1957. \MR{0091568}

\bibitem[HS92]{hara-slade-saw-bm}
T.~Hara and G.~Slade.
\newblock Self-avoiding walk in five or more dimensions. {I}. {T}he critical
  behaviour.
\newblock {\em Comm. Math. Phys.}, 147(1):101--136, 1992. \MR{1171762}

\bibitem[Ken02]{kennedy-saw-sim}
T.~Kennedy.
\newblock {Monte Carlo Tests of Stochastic Loewner Evolution Predictions for
  the 2D Self-Avoiding Walk}.
\newblock {\em Phys. Rev. Lett.}, 88:130601, Mar 2002, \arxiv{math/0112246}.

\bibitem[KPZ88]{kpz-scaling}
V.~Knizhnik, A.~Polyakov, and A.~Zamolodchikov.
\newblock {Fractal structure of 2D-quantum gravity}.
\newblock {\em {Modern Phys. Lett A}}, 3(8):819--826, 1988.

\bibitem[{Le }13]{legall-uniqueness}
J.-F. {Le Gall}.
\newblock Uniqueness and universality of the {B}rownian map.
\newblock {\em Ann. Probab.}, 41(4):2880--2960, 2013, \arxiv{1105.4842}.
  \MR{3112934}

\bibitem[{Le }14]{legall-sphere-survey}
J.-F. {Le Gall}.
\newblock {Random geometry on the sphere}.
\newblock {\em Proceedings of the {ICM}}, 2014, \arxiv{1403.7943}.

\bibitem[LSW03]{lsw-restriction}
G.~Lawler, O.~Schramm, and W.~Werner.
\newblock Conformal restriction: the chordal case.
\newblock {\em J. Amer. Math. Soc.}, 16(4):917--955 (electronic), 2003,
  \arxiv{math/0209343v2}. \MR{1992830 (2004g:60130)}

\bibitem[LSW04]{lsw-saw}
G.~F. Lawler, O.~Schramm, and W.~Werner.
\newblock On the scaling limit of planar self-avoiding walk.
\newblock In {\em Fractal geometry and applications: a jubilee of {B}eno\^\i t
  {M}andelbrot, {P}art 2}, volume~72 of {\em Proc. Sympos. Pure Math.}, pages
  339--364. Amer. Math. Soc., Providence, RI, 2004, \arxiv{math/0204277}.
  \MR{2112127}

\bibitem[Mie09]{miermont-survey}
G.~Miermont.
\newblock Random maps and their scaling limits.
\newblock In {\em Fractal geometry and stochastics {IV}}, volume~61 of {\em
  Progr. Probab.}, pages 197--224. Birkh\"auser Verlag, Basel, 2009.
  \MR{2762678 (2012a:60017)}

\bibitem[Mie13]{miermont-brownian-map}
G.~Miermont.
\newblock The {B}rownian map is the scaling limit of uniform random plane
  quadrangulations.
\newblock {\em Acta Math.}, 210(2):319--401, 2013, \arxiv{1104.1606}.
  \MR{3070569}

\bibitem[Mie14]{miermont-st-flour}
G.~Miermont.
\newblock {\em Aspects of random maps}.
\newblock 2014.
\newblock St. Flour lecture notes. Available at
  \href{http://perso.ens-lyon.fr/gregory.miermont/coursSaint-Flour.pdf}{http://perso.ens-lyon.fr/gregory.miermont/coursSaint-Flour.pdf}.

\bibitem[MS93]{ms-saw-book}
N.~Madras and G.~Slade.
\newblock {\em The self-avoiding walk}.
\newblock Probability and its Applications. Birkh\"auser Boston, Inc., Boston,
  MA, 1993. \MR{1197356}

\bibitem[MS15a]{tbm-characterization}
J.~{Miller} and S.~{Sheffield}.
\newblock {An axiomatic characterization of the Brownian map}.
\newblock {\em ArXiv e-prints}, June 2015, \arxiv{1506.03806}.

\bibitem[MS15b]{lqg-tbm1}
J.~{Miller} and S.~{Sheffield}.
\newblock {Liouville quantum gravity and the Brownian map I: The QLE(8/3,0)
  metric}.
\newblock {\em Inventiones Mathematicae}, to appear, 2015, \arxiv{1507.00719}.

\bibitem[MS16a]{lqg-tbm2}
J.~{Miller} and S.~{Sheffield}.
\newblock {Liouville quantum gravity and the Brownian map II: geodesics and
  continuity of the embedding}.
\newblock {\em ArXiv e-prints}, May 2016, \arxiv{1605.03563}.

\bibitem[MS16b]{lqg-tbm3}
J.~{Miller} and S.~{Sheffield}.
\newblock {Liouville quantum gravity and the Brownian map III: the conformal
  structure is determined}.
\newblock {\em ArXiv e-prints}, August 2016, \arxiv{1608.05391}.

\bibitem[MS16c]{ig1}
J.~Miller and S.~Sheffield.
\newblock Imaginary geometry {I}: interacting {SLE}s.
\newblock {\em Probab. Theory Related Fields}, 164(3-4):553--705, 2016,
  \arxiv{1201.1496}. \MR{3477777}

\bibitem[MS16d]{qle}
J.~Miller and S.~Sheffield.
\newblock Quantum {L}oewner evolution.
\newblock {\em Duke Math. J.}, 165(17):3241--3378, 2016, \arxiv{1312.5745}.
  \MR{3572845}

\bibitem[MS17]{ig4}
J.~Miller and S.~Sheffield.
\newblock Imaginary geometry {IV}: interior rays, whole-plane reversibility,
  and space-filling trees.
\newblock {\em Probab. Theory Related Fields}, 169(3-4):729--869, 2017,
  \arxiv{1302.4738}. \MR{3719057}

\bibitem[MS19]{sphere-constructions}
J.~Miller and S.~Sheffield.
\newblock Liouville quantum gravity spheres as matings of finite-diameter
  trees.
\newblock {\em Ann. Inst. Henri Poincar\'{e} Probab. Stat.}, 55(3):1712--1750,
  2019, \arxiv{1506.03804}. \MR{4010949}

\bibitem[Ric15]{richier-perc}
L.~Richier.
\newblock Universal aspects of critical percolation on random half-planar maps.
\newblock {\em Electron. J. Probab.}, 20:Paper No. 129, 45, 2015,
  \arxiv{1412.7696}. \MR{3438743}

\bibitem[Sch97]{schaeffer-bijection}
G.~Schaeffer.
\newblock Bijective census and random generation of {E}ulerian planar maps with
  prescribed vertex degrees.
\newblock {\em Electron. J. Combin.}, 4(1):Research Paper 20, 14 pp.\
  (electronic), 1997. \MR{1465581 (98g:05074)}

\bibitem[Sch00]{schramm0}
O.~Schramm.
\newblock Scaling limits of loop-erased random walks and uniform spanning
  trees.
\newblock {\em Israel J. Math.}, 118:221--288, 2000, \arxiv{math/9904022}.
  \MR{1776084 (2001m:60227)}

\bibitem[She07]{shef-gff}
S.~Sheffield.
\newblock Gaussian free fields for mathematicians.
\newblock {\em Probab. Theory Related Fields}, 139(3-4):521--541, 2007,
  \arxiv{math/0312099}. \MR{2322706 (2008d:60120)}

\bibitem[She16]{shef-zipper}
S.~Sheffield.
\newblock Conformal weldings of random surfaces: {SLE} and the quantum gravity
  zipper.
\newblock {\em Ann. Probab.}, 44(5):3474--3545, 2016, \arxiv{1012.4797}.
  \MR{3551203}

\bibitem[Sla11]{slade-saw-survey}
G.~Slade.
\newblock The self-avoiding walk: a brief survey.
\newblock In {\em Surveys in stochastic processes}, EMS Ser. Congr. Rep., pages
  181--199. Eur. Math. Soc., Z\"urich, 2011. \MR{2883859}

\bibitem[SS13]{ss-contour}
O.~Schramm and S.~Sheffield.
\newblock A contour line of the continuum {G}aussian free field.
\newblock {\em Probab. Theory Related Fields}, 157(1-2):47--80, 2013,
  \arxiv{math/0605337}. \MR{3101840}

\end{thebibliography}
\bibliographystyle{hmralphaabbrv}

\end{document}